\documentclass[EJP]{ejpecp} 

\usepackage{amsmath}
\usepackage{amssymb}
\usepackage{hyperref}
\usepackage{yfonts}
\usepackage{mathrsfs}
\usepackage {latexsym}
\usepackage{graphics}
\usepackage{tikz}
\usetikzlibrary{snakes}
\usetikzlibrary{arrows.meta}
\usetikzlibrary{decorations.pathmorphing}
\usetikzlibrary{er, positioning}

\theoremstyle{remark}
\theoremstyle{plain} 

\newcommand{\joinR}{\hspace{-.1em}}
\newcommand{\RomanI}{I}
\newcommand{\RomanII}{\mbox{\RomanI\joinR\RomanI}}
\newcommand{\RomanIII}{\mbox{\RomanI\joinR\RomanII}}
\newcommand{\RomanIV}{\mbox{\RomanI\joinR\RomanV}}
\newcommand{\RomanV}{V}
\newcommand{\RomanVI}{\mbox{\RomanV\joinR\RomanI}}
\newcommand{\RomanVII}{\mbox{\RomanV\joinR\RomanII}}
\newcommand{\RomanVIII}{\mbox{\RomanV\joinR\RomanIII}}
\newcommand{\RomanIX}{\mbox{\RomanI\joinR\RomanX}}
\newcommand{\RomanX}{X}
\newcommand{\RomanXI}{\mbox{\RomanX\joinR\RomanI}}
\newcommand{\RomanXII}{\mbox{\RomanX\joinR\RomanII}}

\allowdisplaybreaks 

\DeclareMathSymbol{:}{\mathord}{operators}{"3A}

\SHORTTITLE{Magnetohydrodynamics system}

\TITLE{Three-dimensional magnetohydrodynamics system forced by space-time white noise}

\AUTHORS{%
  Kazuo Yamazaki\footnote{Texas Tech University, Department of Mathematics and Statistics, Lubbock, TX 79409, U.S.A.    \EMAIL{kyamazak@ttu.edu}}
 }


\AMSSUBJ{35B65; 35Q85; 35R60} 

\SUBMITTED{April 8, 2022} 
\ACCEPTED{March 2, 2023} 

\VOLUME{0}
\YEAR{2020}
\PAPERNUM{0}
\DOI{10.1214/YY-TN}

\ABSTRACT{We consider the three-dimensional magnetohydrodynamics system forced by noise that is white in both time and space. Its complexity due to four non-linear terms makes its analysis very intricate. Nevertheless, taking advantage of its structure and adapting the theory of paracontrolled distributions from \cite{GIP15}, we prove its local well-posedness. A first challenge is to find an appropriate paracontrolled ansatz which must consist of both the velocity and the magnetic fields. Second challenge is that for some non-linear terms, renormalizations cannot be achieved individually; we overcome this obstacle by strategically coupling certain terms together rather than separately. Our proof is also inspired by the work of \cite{ZZ15}. }

\begin{document}

\tableofcontents

\section{Introduction}

When solutions to a system of partial differential equations (PDEs) lack sufficient regularity, a common remedy is to multiply by a sufficiently smooth function, integrate by parts to rid of any derivative on the solution, and only ask that its integral formulation is well-defined; this is the standard definition of a weak solution. However, if the PDEs are non-linear, then the lack of regularity creates difficulty in understanding any product of the solution with itself because there is no universal agreement on the definition of a product of distributions. Some physically meaningful models which have found rich applications in the real world were forced by a term that is white in both space and time, so-called space-time white noise (STWN). We refer to e.g., \cite{KPZ86} for the Kardar-Parisi-Zhang (KPZ) equation \eqref{4}, \cite{YO86} for the Navier-Stokes equations (NSE) \eqref{1} and Burgers' equation forced by STWN, as well as \cite{ACHS81, GP75, HS92, SH77} concerning the Boussinesq system forced by STWN. While considering the mild solution formulation typically solved the issue in case the noise is white only in time, the STWN leads to a lack of spatial regularity of the solution, and the construction of a solution has created a significant obstacle because the non-linear term seemed to be ill-defined in the classical sense. Let us describe recent developments that ultimately led to the two novel approaches of the theory of regularity structures by Hairer \cite{H14} and the theory of paracontrolled distributions by Gubinelli et al. \cite{GIP15} (see also \cite{GP17}). 

Following the notations of Young \cite[pg. 258]{Y36}, let us denote by $V_{p}(f)$ the $p$-variation of $f$ and write $f \in W_{p}$ if $V_{p} (f) < \infty$. Furthermore, we denote by $C^{\alpha}$ the space of H$\ddot{\mathrm{o}}$lder continuous functions with exponent $\alpha \geq 0$ (e.g.,  \cite[Definition 1.49]{BCD11}). Young \cite[pg. 265]{Y36} proved an important theorem in which if $f \in W_{p}, g \in W_{q}$ where $p, q > 0, \frac{1}{p} + \frac{1}{q} > 1$, and they have no common discontinuities, then their Lebesgue-Stieltjes integral $\int g(x) df(x)$ still exists. In order to understand its implication, let us introduce the NSE. Let us denote by $u: \hspace{1mm} \mathbb{T}^{N} \times \mathbb{R}_{+} \mapsto \mathbb{R}^{N}$ and $\pi: \hspace{1mm} \mathbb{T}^{N} \times \mathbb{R}_{+} \mapsto \mathbb{R}$ the $N$-dimensional ($N$-d) velocity vector field and the pressure scalar field, respectively. Additionally, by denoting by $\nu \geq 0$ the viscous diffusivity and $\partial_{t} \triangleq \frac{\partial}{\partial t}, x = (x_{1}, \hdots, x_{N})$, $\partial_{x_{j}} \triangleq \frac{\partial}{\partial x_{j}}$ and $\partial_{x_{j}}^{k} \triangleq \frac{ \partial^{k}}{(\partial x_{j})^{k}}$ for $j \in \{1,\hdots, N\}$, the NSE can be written as 
\begin{equation}\label{1}  
\partial_{t}u + (u\cdot\nabla) u + \nabla \pi - \nu \Delta u = \xi^{u}, \hspace{3mm} \nabla\cdot u = 0, 
\end{equation} 
with initial data $u^{\text{in}}(x) \triangleq u(x,0)$, where $\xi^{u}$ is the Gaussian field that is white in both time and space; i.e. $\mathbb{E} [ \xi^{u}(x,t)\xi^{u}(y,s)] = \delta(x-y)\delta(t-s)$. We will also need the definition of the H$\ddot{\mathrm{o}}$lder space with negative exponent; for this purpose, let us recall the basic background of Besov spaces (\cite{GIP15} and also \cite{I99} on how the Littlewood-Paley theory on $\mathbb{R}^{3}$ may be transferred to $\mathbb{T}^{3}$). Let us use the notation of $A \lesssim_{a,b} B$ in case there exists a non-negative constant $C = C(a,b)$ that depends on $a,b$ such that $A \leq C B$; similarly let us write $A \approx_{a,b} B$ in case $A = CB$. Moreover, unless elaborated in detail, we denote $\sum_{k \in \mathbb{Z}^{3}}$ by $\sum_{k}$. First we recall the Fourier transform 
\begin{equation*}
\hat{f}(k) \triangleq \mathcal{F}_{\mathbb{T}^{3}}(f)(k) \triangleq \frac{1}{(2\pi)^{\frac{3}{2}}} \int_{\mathbb{T}^{3}} f(x) e^{-i x \cdot k} dx 
\end{equation*} 
with its inverse denoted by $\mathcal{F}_{\mathbb{T}^{3}}^{-1}$, let $\mathcal{D}$ be the set of all smooth functions with compact support on $\mathbb{T}^{3}$, $\mathcal{D}'$ its dual. We let $\chi, \rho \in \mathcal{D}$ be non-negative, radial such that the support of $\chi$ is contained in a ball while that of $\rho$ in an annulus and satisfy 
\begin{align*} 
&\chi(\xi) + \sum_{j \geq 0} \rho(2^{j} \xi) = 1 \hspace{1mm} \forall \hspace{1mm} \xi, \hspace{1mm} \text{supp} (\chi) \cap \text{supp} (\rho(2^{-j} \cdot)) = \emptyset \hspace{1mm} \forall \hspace{1mm} j \geq 1,\\
&\text{supp} (\rho(2^{-i} \cdot)) \cap \text{supp} (\rho(2^{-j} \cdot)) = \emptyset \text{ for } \lvert i - j \rvert > 1.
\end{align*}
We see that $\chi(\cdot) = \rho(2^{-1} \cdot)$ and define Littlewood-Paley operator as $\Delta_{j} f \triangleq \mathcal{F}_{\mathbb{T}^{3}}^{-1}(\rho_{j} \mathcal{F}_{\mathbb{T}^{3}} (f))$ where $\rho_{j} \triangleq \rho(2^{-j}\cdot)$. We also write $S_{j} f \triangleq \sum_{i \leq j-1} \Delta_{j} f$. Now for $\alpha \in \mathbb{R}$ and $p, q \in [1,\infty]$, we may define the inhomogeneous Besov space 
\begin{equation*}
B_{p,q}^{\alpha} (\mathbb{T}^{3}) \triangleq \{ f \in \mathcal{D}' (\mathbb{T}^{3}): \hspace{1mm} \lVert f \rVert_{B_{p,q}^{\alpha} (\mathbb{T}^{3})} \triangleq  \lVert 2^{j\alpha}\lVert \Delta_{j} f \rVert_{L^{p}(\mathbb{T}^{3})} \rVert_{l^{q}(\{j \geq -1\})}  < \infty \}. 
\end{equation*}  
The H$\ddot{\mathrm{o}}$lder-Besov space $\mathcal{C}^{\alpha}(\mathbb{T}^{3})$ is the special case when $p = q = \infty$; i.e. 
$\mathcal{C}^{\alpha}(\mathbb{T}^{3}) = B_{\infty, \infty}^{\alpha}(\mathbb{T}^{3})$. For $\alpha \in (0,\infty) \setminus \mathbb{N}$, $\mathcal{C}^{\alpha}(\mathbb{T}^{3}) = C^{\alpha}(\mathbb{T}^{3})$ (\cite[pg. 99]{BCD11}). We point out that 
\begin{equation}\label{2}
\lVert \cdot \rVert_{\mathcal{C}^{\beta}} \lesssim \lVert \cdot \rVert_{L^{\infty}} \lesssim \lVert \cdot \rVert_{\mathcal{C}^{\alpha}} \hspace{1mm} \text{ if } \hspace{1mm} \beta \leq 0 \leq \alpha \text{ and } \lVert S_{j} \cdot \rVert_{L^{\infty}} \lesssim 2^{-j\alpha} \lVert \cdot \rVert_{\mathcal{C}^{\alpha}} \hspace{1mm} \forall \hspace{1mm} \alpha < 0. 
\end{equation} 

Now for simplicity let us consider the 1-d analogue of $(u\cdot\nabla) u$ in the NSE (\ref{1}), specifically $u \partial_{x}u$ corresponding to the non-linear term of the Burgers' equation which was studied by Da Prato et al. \cite{DDT94}. Following the discussion of \cite[pg. 1548]{H11}, assuming that its solution $u \in \mathcal{C}^{\alpha}$ for $\alpha > \frac{1}{2}$, we may multiply this non-linear term by a smooth periodic function $\psi$ and understand it as 
\begin{equation}\label{3} 
\int_{\mathbb{T}} \psi(x) u(x) du(x)
\end{equation} 
which is well-defined as a Young's integral because $\psi u \in \mathcal{C}^{\alpha}$ for $\alpha > \frac{1}{2}$ and $f \in C^{\frac{1}{p}}$ for $p \in (0, \infty)$ implies $f \in W_{p}$ in general. Unfortunately, the assumption of $u\in \mathcal{C}^{\alpha}$ for $\alpha > \frac{1}{2}$ turns out to be a wishful thinking. In fact, in the general case when the spatial dimension is $N$, considering that the space-time dimension is $N + 1$ so that the scaling $\mathcal{S} \in \mathbb{N}^{N}$ is $\mathcal{S} = (\mathcal{S}_{1}, \hdots, \mathcal{S}_{N+1}) = (2, 1, \hdots, 1)$ with the first entry informally representing the dimension of time due to $\partial_{t}$ and $\Delta$, we actually know that $\xi \in \mathcal{C}^{\alpha}(\mathbb{T}^{N})$ for $\alpha < - \frac{\lvert \mathcal{S} \rvert}{2}$ where $\lvert \mathcal{S} \rvert = N + 2$ by \cite[Lemma 10.2]{H14} (see also \cite[Lemma 3.20]{H14} and \cite{BK17}). This leads to $u \in \mathcal{C}^{\alpha}(\mathbb{T}^{N})$ for $\alpha < 2 - \frac{N+2}{2}$ due to regularization from the diffusion (see \cite[pg. 417, 481]{H14}). Therefore, the Young's integral (\ref{3}) is ill-defined even in case $N = 1$. 

Although one may turn to the theory of stochastic integrals such as the It$\hat{\mathrm{o}}$'s integrals at this point, its limitations have also been noticed over decades (e.g.,  \cite[pg. 6]{GIP15}, \cite[pg. 1548]{H11}).  In order to complement the theory of It$\hat{\mathrm{o}}$'s integrals, Lyons developed a theory of rough path (\cite{L98, LQ02}).  Subsequently, Gubinelli \cite{G04} extended the Lyon's rough path theory; we refer to \cite{FV10a, FV10b, GT10, H11, HW13} for further study and applications of rough path theory.  As one of the most prominent examples of a result inspired from the rough path theory, let us briefly discuss recent developments of the KPZ equation (\ref{4}). The KPZ equation, an interface model of flame propagation, was derived in \cite[Equation (1)]{KPZ86} as 
\begin{equation}\label{4}
\partial_{t}h = \partial_{x}^{2} h + \lambda (\partial_{x} h)^{2} + \xi^{h}
\end{equation} 
where $h(x,t)$ represents the interface height, $\lambda > 0$ is the coupling strength, $x \in \mathbb{S}^{1}$ and $\xi^{h}$ is the STWN. Following \cite[pg. 562]{H13}, let us consider a multiplicative stochastic heat equation $dZ = \partial_{x}^{2} Z dt + \lambda Z dW$ where $\partial_{t}W = \xi^{h}$. We denote by $Z^{\epsilon}$ the solution to the same equation with $W$ replaced by a mollified noise $W^{\epsilon}$, which is obtained from multiplying the $k$-th Fourier component of $W$ by $f(k\epsilon)$ for a smooth cut-off function $f$ with compact support such that $f(0) = 1$. Then It$\hat{\mathrm{o}}$'s formula shows that $h^{\epsilon} (x,t) \triangleq \frac{1}{\lambda} \ln Z^{\epsilon}(x,t)$ (see \cite{M91} on the positivity of $Z^{\epsilon}$) solves 
\begin{equation}\label{5} 
\partial_{t} h^{\epsilon} = \partial_{x}^{2} h^{\epsilon} + \lambda (\partial_{x} h^{\epsilon})^{2} - \frac{\lambda}{2} \sum_{k \in \mathbb{Z}} f^{2}(k \epsilon) + \xi^{h,\epsilon}
\end{equation} 
where $\sum_{k \in \mathbb{Z}} f^{2}(k\epsilon) \approx \frac{1}{\epsilon} \int_{\mathbb{R}} f^{2}(x) dx \to \infty$. This simple computation displays the necessity to rely on techniques from quantum field theory  (e.g.,  \cite[Section 4]{MR99}) such as renormalization, which amounts to strategically subtracting off a large constant from a regularized equation, and replacing a standard product by Wick product (e.g.,  \cite[pg. 23]{J97}). These techniques actually have long history of its utility in stochastic quantization. In particular, Da Prato and Debussche \cite{DD03} proved the existence of a unique strong solution to the 2-d stochastic quantization equation for almost all initial data with respect to the invariant measure using such techniques (see also \cite{BG97, DDT07}). Without delving into further details, we mention that Hairer \cite{H13} in particular discovered two additional logarithmically divergent constants beside the $\frac{1}{\epsilon}$ in (\ref{5}) and successfully introduced a completely new concept of a solution to the KPZ equation (\ref{4}) using rough path theory (see also \cite{HQ18}).  

Let us now discuss this direction of research in the case of the NSE (\ref{1}). To the best of the author's knowledge, Flandoli and Gozzi \cite{FG98} were the first to consider the 2-d NSE in $\mathbb{T}^{2}$ with the forcing that is not regular; they proved in \cite[Theorem 4.3]{FG98} that the Kolmogorov equation associated to the NSE with covariance operator that is an identity has a weak solution. However, due to the spatial roughness of the noise, the authors in \cite{FG98} were not able to make the connection to the original equation. Subsequently, Da Prato and Debussche  \cite{DD02} overcame this difficulty using techniques of renormalization and Wick products. 

At this point let us introduce the magnetohydrodynamics (MHD) system of main concern because the failure to apply the proofs of \cite{DD02, FG98}, which we will explain shortly, clearly displays the complexity of the MHD system in contrast to the NSE. We denote the magnetic $N$-d vector field by $b: \hspace{1mm} \mathbb{T}^{N} \times \mathbb{R}_{+} \mapsto \mathbb{R}^{N}$ and the magnetic diffusivity by $\eta \geq 0$, where $N \in \{2,3,4\}$. Then the MHD system reads as 
\begin{subequations}
\begin{align}
&\partial_{t}u + (u\cdot\nabla) u + \nabla \pi = \nu \Delta u + (b\cdot\nabla) b + \xi^{u}, \hspace{3mm} \nabla\cdot u = 0, \label{6a}\\
& \partial_{t}b+ (u\cdot\nabla) b = \eta \Delta b + (b\cdot\nabla) u + \xi^{b}, \hspace{14mm} \nabla\cdot b=  0, \label{6b}
\end{align}
\end{subequations} 
for which we write the solution as $y \triangleq (y_{1}, \hdots, y_{6}) \triangleq (u,b) \triangleq (u_{1}, u_{2}, u_{3}, b_{1}, b_{2}, b_{3})$, with initial data $y^{\text{in}}(x) \triangleq (u^{\text{in}}, b^{\text{in}})(x) = (u,b)(x,0)$, and $\xi \triangleq (\xi^{u}, \xi^{b})$ where $\xi^{u} \triangleq (\xi_{1}^{u}, \xi_{2}^{u}, \xi_{3}^{u}) = (\xi_{1}, \xi_{2}, \xi_{3})$ and $\xi^{b} \triangleq (\xi_{1}^{b}, \xi_{2}^{b}, \xi_{3}^{b}) = (\xi_{4}, \xi_{5}, \xi_{6})$, is a Gaussian field which is white in both space and time. For simplicity of computation, let us assume that $\nu = \eta = 1$ as well as that $\int_{\mathbb{T}^{3}} \xi^{u} dx = \int_{\mathbb{T}^{3}} \xi^{b} dx = 0$ which  in turn allows us to assume that $(u,b)$ are also mean zero; this may be justified via a standard scaling argument of the solution to the MHD system. Such MHD system forced by STWN has been studied by physicists for decades; e.g., Camargo and Tasso \cite{CT92} applied the renormalization group theory to the MHD system forced by STWN and determined the effective viscosity and magnetic resistivity without solving the system.

\begin{remark}\label{Remark 1.1}
As a STWN, the correlation of $\xi^{u}$ and that of $\xi^{b}$ are both products of a delta function in $x$ with another delta function in $t$. In the literature on Boussinesq system such as \cite[Equation (3)]{GP75}, the authors make an assumption corresponding to the MHD system that the correlation of $\xi^{u}$ and $\xi^{b}$ vanish; i.e. $\mathbb{E} [\xi_{i}^{u} \xi_{j}^{b}] = 0$ for all $i, j \in \{1,2,3\}$.  Considering that there is no physical reason why $\xi^{u}$ and $\xi^{b}$ should have any independence, in this manuscript we shall assume that the correlation of $\xi^{u}$ and $\xi^{b}$ is also a product of a delta function in $x$ with another delta function in $t$ (see (\ref{116}) which is a corollary of this assumption). Our computations are thus more general. Indeed, it is easy to recover the case $\mathbb{E}[\xi_{i}^{u} \xi_{j}^{b}] = 0$  for all $i, j \in \{1,2,3\}$ because many terms within our proof vanish due to the mixed non-linear terms such as  $(u\cdot\nabla)b$ and $(b\cdot\nabla)u$. This is an interesting difference from the case of the NSE; the computations of the mixed non-linear terms can be actually much simpler than the case of the NSE under the assumption of the zero correlation among $\xi^{u}$ and $\xi^{b}$. 
\end{remark} 

It is well-known that if we take the $L^{2}(\mathbb{T}^{N})$-inner products of (\ref{1}) with $u$, then the non-linear term, as well as the pressure term, both vanish by divergence-free property; e.g.,  $\int_{\mathbb{T}^{3}} (u\cdot\nabla) u \cdot u dx = \frac{1}{2} \int_{\mathbb{T}^{3}} (u\cdot\nabla) \lvert u \rvert^{2} dx = 0$. An analogous attempt of taking $L^{2}$-inner products on (\ref{6a}) with $u$ fails because 
\begin{equation}\label{7}
\int_{\mathbb{T}^{3}} (b\cdot\nabla) b \cdot u dx \neq 0 
\end{equation} 
in general. Yet, if we take $L^{2}(\mathbb{T}^{N})$-inner products on (\ref{6b}) with $b$ simultaneously and add the two resulting equations, then all the non-linear terms and the pressure term in (\ref{6a})-(\ref{6b}) do vanish because $\int_{\mathbb{T}^{3}} (u\cdot\nabla) b \cdot b dx = \frac{1}{2} \int_{\mathbb{T}^{3}} (u\cdot\nabla) \lvert b \rvert^{2} dx = 0$ and 
\begin{equation}\label{8}
\int_{\mathbb{T}^{3}} (b\cdot\nabla) b \cdot u + (b\cdot\nabla) u \cdot b dx = 0. 
\end{equation} 
Even though there exist some extensions of techniques on the NSE to the MHD system such as this, attempts to modify the proofs of \cite{DD02, FG98} on the 2-d NSE to the 2-d MHD system face a non-trivial difficulty. In both works of \cite{DD02, FG98}, the authors relied on the following key identity:
\begin{equation}\label{9}
\int_{\mathbb{T}^{2}} (u\cdot\nabla) u \cdot \Delta u dx = 0.
\end{equation} 
In fact, one of the reasons why the authors admitted that extending to other boundary conditions beside $\mathbb{T}^{2}$ is not easy (e.g.,  \cite[pg. 312]{FG98}) is exactly this identity (\ref{9}). The identity (\ref{9}) was used in \cite[pg. 328]{FG98} and \cite[pg. 190]{DD02}, and it actually fails in the case of the MHD system because $\int_{\mathbb{T}^{3}} [(u\cdot\nabla) u - (b\cdot\nabla) b ] \cdot \Delta u dx \neq 0$ and even if we add similarly to (\ref{8}), 
\begin{equation}\label{10}
\int_{\mathbb{T}^{3}} [(u\cdot\nabla) u - (b\cdot\nabla) b ] \cdot \Delta u + [(u\cdot\nabla) b - (b\cdot\nabla) u] \cdot \Delta b dx \neq 0 
\end{equation} 
in general. In fact, the identity (\ref{9}), which is equivalent to $ \int_{\mathbb{T}^{2}} (u\cdot\nabla) (\nabla \times u) \cdot (\nabla \times u) dx = 0$, has also been used crucially in various other works on the NSE (e.g.,  \cite{HM06}), many of which have not been extended to the MHD system with (\ref{10}) being one of the sources of the technical issues. As we will elaborate in Remark \ref{Remark 3.5}, interestingly we will need to renormalize certain term togethers very similarly to \eqref{8}. 

Zhu and Zhu \cite{ZZ15} gave a very nice discussion of how the proof within \cite{DD02} cannot be extended to the 3-d NSE and thus most certainly has no chance of being extended to the 3-d MHD system; let us recollect it here. Da Prato and Debussche \cite{DD02} considered (\ref{1}) in $\mathbb{T}^{2}$, $z$ to be the solution to the Stokes equation forced by the fixed STWN $\xi^{u}$ and the equation solved by $v \triangleq u - z, q \triangleq \pi - p$, specifically
\begin{align*}
& \partial_{t} z = \Delta z - \nabla p + \xi^{u}, \hspace{3mm} \nabla\cdot z = 0, \\
& \partial_{t} v = \Delta v - \nabla q - \frac{1}{2} \text{div} [(v + z) \otimes (v + z)], \hspace{3mm} \nabla\cdot v = 0. 
\end{align*}
Similarly to the discussion of the Burgers' equation in (\ref{3}), due to \cite[Lemma 10.2]{H14} (see also \cite[Lemma 3.20]{H14}) the solution $z$ is very rough, and only in $\mathcal{C}^{\alpha}(\mathbb{T}^{N})$ for $\alpha < 1 - \frac{N}{2}$. Thus, if $N = 2$, then $z \in \mathcal{C}^{\alpha}(\mathbb{T}^{2})$ for $\alpha < 0$ and considering $\text{div} (z \otimes z) \in \mathcal{C}^{\alpha}(\mathbb{T}^{2})$ for $\alpha < -1$, the diffusion leads to $v \in \mathcal{C}^{\alpha}(\mathbb{T}^{2})$ for $\alpha < 1$. This implies that according to Bony's estimates (see Lemma \ref{Lemma 1.1} (4)) the product $v \otimes v$ and even $v \otimes z$ can be well-defined, leaving only $z \otimes z$ for which one can turn to Wick products. However, in the case $N = 3$ same computations show that not only $z\otimes z$ but even $z \otimes v$ is ill-defined. 

Two novel approaches have been developed to bring about a resolution to such an issue, specifically the theory of regularity structures due to Hairer \cite{H14} and that of paracontrolled distributions due to Gubinelli et al. \cite{GIP15}. The work of Hairer \cite{H14} allows one to construct a regularity structure endowed with a whole set of calculus operations such as multiplication, integration and differentiation, so that one can recover a fixed point theory, and finally rely on the reconstruction theorem to conclude the existence and uniqueness of a solution to the original problem (see \cite{CFG17, HM18, HM18a} for further discussions). On the other hand, the theory of paracontrolled distributions relies heavily on the Bony's decomposition (e.g.,  \cite[pg. 86]{BCD11}) beside the rough path theory, which we now describe briefly. The purpose of the Bony's decomposition is to split $fg$ in parts where the frequency of $f$ and $g$ are low and high, specifically 
\begin{equation*}
fg = \sum_{i,j \geq -1} \Delta_{i} f \Delta_{j} g = \pi_{<}(f,g) + \pi_{>} (f,g) + \pi_{0}(f,g) \text { where }
\end{equation*} 
\begin{equation*}
\pi_{<}(f,g) = \sum_{j \geq -1} S_{j} f \Delta_{j} g, \pi_{>}(f,g) = \sum_{j \geq -1} \Delta_{j} f S_{j} g, \pi_{0}(f,g) =  \sum_{j, l \geq -1: \lvert l-j \rvert \leq 1} \Delta_{j} f \Delta_{l} g.
\end{equation*}
The terms $\pi_{<}(f,g)$ and $\pi_{>}(f,g)$ are called paraproducts while $\pi_{0}(f,g)$ the remainder. The key observation by Bony was that $\pi_{<}(f,g)$ and similarly $\pi_{>}(f,g)$ are well-defined distributions such that the mapping $(f,g) \mapsto \pi_{<}(f,g)$ is a bounded bi-linear operator from $\mathcal{C}^{\alpha}(\mathbb{T}^{N}) \times \mathcal{C}^{\beta}(\mathbb{T}^{N})$ to $\mathcal{C}^{\beta}(\mathbb{T}^{N})$ if $\alpha  > 0, \beta \in \mathbb{R}$. Heuristically, $\pi_{<}(f,g)$ behaves at large frequencies similarly to $g$, and $f$ provides only a modulation of $g$ at large scales. We will rely heavily on the following lemma:
\begin{lemma}\label{Lemma 1.1}
\rm{(\cite[Lemma 2.1]{GIP15}, \cite[Lemma 2.1]{GP17}, \cite[Proposition 2.3]{CC18})} Let $\alpha, \beta \in \mathbb{R}$. Then 
\begin{enumerate}
\item $\lVert \pi_{<}(f,g) \rVert_{\mathcal{C}^{\beta}} \lesssim \lVert f \rVert_{L^{\infty}} \lVert g \rVert_{\mathcal{C}^{\beta}} \text{ for } f \in L^{\infty}(\mathbb{T}^{N}), g \in \mathcal{C}^{\beta}(\mathbb{T}^{N}),$
\item $\lVert \pi_{>}(f,g) \rVert_{\mathcal{C}^{\alpha + \beta}} \lesssim \lVert f \rVert_{\mathcal{C}^{\alpha}} \lVert g \rVert_{\mathcal{C}^{\beta}} \text{ for } \beta < 0, f \in \mathcal{C}^{\alpha}(\mathbb{T}^{3}), g \in \mathcal{C}^{\beta}(\mathbb{T}^{3})$, 
\item $\lVert \pi_{0}(f,g) \rVert_{\mathcal{C}^{\alpha + \beta}} \lesssim \lVert f \rVert_{\mathcal{C}^{\alpha}} \lVert g \rVert_{\mathcal{C}^{\beta}} \text{ for } \alpha + \beta > 0, f \in \mathcal{C}^{\alpha}(\mathbb{T}^{3}),g \in \mathcal{C}^{\beta}(\mathbb{T}^{3})$. 
\item $fg$ is well-defined for $f \in \mathcal{C}^{\alpha}(\mathbb{T}^{3}), g \in \mathcal{C}^{\beta}(\mathbb{T}^{3})$ if $\alpha + \beta > 0$ and $\lVert f g\rVert_{\mathcal{C}^{\text{min} \{\alpha, \beta, \alpha + \beta\}}} \lesssim \lVert f \rVert_{\mathcal{C}^{\alpha}} \lVert g \rVert_{\mathcal{C}^{\beta}}$. 
\end{enumerate} 
\end{lemma}  
By our discussion, only difficulty in defining the product $fg$ boils down to $\pi_{0}(f,g)$, and for this purpose, Gubinelli et al. in \cite{GIP15} relied on a paracontrolled ansatz (see (\ref{29}) and (\ref{32})) and a commutator lemma (see Lemma \ref{Lemma 2.4}). 

Beside the work of Zhu and Zhu in \cite{ZZ15}, we wish to mention the work of Catellier and Chouk \cite{CC18}, by which our work was inspired. The purpose of this manuscript is to prove the local existence of a unique solution to the MHD system forced by the STWN (\ref{6a})-(\ref{6b}); i.e., 
\begin{subequations} 
\begin{align}
&\partial_{t} u_{i} - \Delta u_{i} = \sum_{i_{1} = 1}^{3} \mathcal{P}_{ii_{1}} \xi_{i_{1}}^{u} - \frac{1}{2} \sum_{i_{1}, j = 1}^{3} \mathcal{P}_{ii_{1}} \partial_{x_{j}}(u_{i} u_{j}) + \frac{1}{2} \sum_{i_{1}, j=1}^{3} \mathcal{P}_{ii_{1}} \partial_{x_{j}}(b_{i}b_{j}), \label{11a}\\
&\partial_{t} b_{i} - \Delta b_{i} = \sum_{i_{1} =1}^{3} \mathcal{P}_{ii_{1}} \xi_{i_{1}}^{b} - \frac{1}{2} \sum_{i_{1}, j=1}^{3} \mathcal{P}_{ii_{1}} \partial_{x_{j}}(b_{i} u_{j}) + \frac{1}{2} \sum_{i_{1}, j=1}^{3} \mathcal{P}_{ii_{1}} \partial_{x_{j}} (u_{i} b_{j}), \label{11b} \\
& u(x, 0) = \mathcal{P} u^{\text{in}}(\cdot), \hspace{3mm} b(x, 0) = \mathcal{P} b^{\text{in}}(\cdot),\label{11c} 
\end{align}
\end{subequations}
for $i \in \{1,2,3\}$, where $\widehat{\mathcal{P}_{lm}}(k) = \delta(l-m) - \frac{k_{l}k_{m}}{\lvert k \rvert^{2}}$ so that $\mathcal{P}$ represents the Leray projection onto the space of divergence-free vector fields. For brevity we define $L \triangleq \partial_{t} - \Delta$. 

\begin{theorem}\label{Theorem 1.2}
Let $\delta_{0} \in (0, \frac{1}{2})$ and then $z \in (\frac{1}{2}, \frac{1}{2} +\delta_{0})$, as well as $y^{\text{in}} = (u^{\text{in}}, b^{\text{in}}) \in \mathcal{C}^{-z}(\mathbb{T}^{3})$. Suppose that $\xi^{\epsilon} = \sum_{k} f(\epsilon k) \hat{\xi}(k) e_{k}$ for $\epsilon > 0$ and $f$ is a smooth radial cut-off function with compact support such that $f(0) = 1$, and $y^{\epsilon} = (u^{\epsilon}, b^{\epsilon})$ is the maximal unique solution to 
\begin{align}\label{261}
Ly^{\epsilon}_{i}= 
\begin{pmatrix}
\sum_{i_{1} =1}^{3} \mathcal{P}_{ii_{1}} \xi_{i_{1}}^{u,\epsilon} - \frac{1}{2} \sum_{i_{1}, j=1}^{3} \mathcal{P}_{ii_{1}}  \partial_{x_{j}}  (u^{\epsilon}_{i} u^{\epsilon}_{j}) + \frac{1}{2} \sum_{i_{1}, j=1}^{3} \mathcal{P}_{ii_{1}}  \partial_{x_{j}} (b^{\epsilon}_{i}b^{\epsilon}_{j})\\
\sum_{i_{1} =1}^{3} \mathcal{P}_{ii_{1}} \xi_{i_{1}}^{b, \epsilon} - \frac{1}{2} \sum_{i_{1}, j=1}^{3} \mathcal{P}_{ii_{1}}  \partial_{x_{j}}  (b^{\epsilon}_{i} u^{\epsilon}_{j}) + \frac{1}{2} \sum_{i_{1}, j=1}^{3} \mathcal{P}_{ii_{1}}  \partial_{x_{j}}  (u^{\epsilon}_{i} b^{\epsilon}_{j}),
\end{pmatrix}
\end{align} 
such that $u^{F,\epsilon}, b^{F,\epsilon}$, which is constructed identically to \eqref{16}-\eqref{21} except that \eqref{16} has $\xi^{\epsilon} = (\xi^{u,\epsilon}, \xi^{b,\epsilon})$ rather than $\xi = (\xi^{u}, \xi^{b})$, belong to $C([0,T^{\epsilon}); \mathcal{C}^{\frac{1}{2} - \delta_{0}})$. Then there exists $y \in C([0, \tau); \mathcal{C}^{-z})^{2}$ and $\{\tau_{L}\}_{L}$, specifically defined in \eqref{260}, such that $\tau_{L}$ increases to the explosion time $\tau$ of $y = (u,b)$ that satisfies 
\begin{equation}\label{267} 
\sup_{t \in [0, \tau_{L}]} \lVert y^{\epsilon} - y \rVert_{\mathcal{C}^{-z}} \to 0 \text{ as } \epsilon \to 0 \text{ in probability}.
\end{equation}  
\end{theorem}  

\begin{remark}\label{Remark 1.3}
We emphasize two new novelty of this work in comparison to the approaches of \cite{CC18, ZZ15}. First, let us acknowledge that a nonlinearly coupled systems of equations forced by STWN have been studied before, e.g., a multi-component KPZ equation 
\begin{equation*}
\partial_{t}h_{i} = \partial_{x}^{2} h_{i} + S_{jki} \partial_{x}h_{j} \partial_{x}h_{k} + \xi_{i} 
\end{equation*}
in \cite[Equation (5.12)]{HM18a} where each $\xi_{i}$ is an independent STWN on $\mathbb{R} \times \mathbb{T}$ and $S_{jki} \in \mathbb{R}$. We point out that the equations of $h_{i}$ is essentially identical while those of $u$ and $b$ in \eqref{6a}-\eqref{6b} differ significantly, leading to the need to carefully take advantage of its structure as follows.
\begin{itemize}
\item We need to define correct paracontrolled ansatz; see \eqref{29} and \eqref{32} for velocity and magnetic fields, respectively. The correct choices \eqref{29} and \eqref{32} display clearly the complexity of the MHD system due to the four mixed non-linear terms (see Remark \ref{Remark 3.2}).  
\item Certain renormalizations must be ``coupled'' together. This major issue is elaborated in detail in Remark \ref{Remark 3.5}. Interestingly, the nature of this problem is same as those of \eqref{7}-\eqref{8}. 
\end{itemize}
Other differences from \cite{ZZ15} are mentioned in Remark \ref{new remark}. 

Let us also emphasize that there are many results on the NSE which have not been extended to the MHD system despite much effort by many mathematicians. As already mentioned, the work of Hairer and Mattingly \cite{HM06} on the ergodicity of the 2-d NSE seems difficult to be extended to the 2-d MHD system. In the deterministic case, there exist also abundance of results for which an extension from the case of the NSE to the MHD system is a challenging open problem. For example, although Yudovich \cite{Y63} over 55 years ago proved the global regularity of the solution to the 2-d NSE with zero viscous diffusion, which is the Euler equations, its extension to the 2-d MHD system with zero viscous diffusion remains open despite extensive interest from many mathematicians (e.g.,  \cite{CWY14, FMMNZ14, JZ14a, Y18a}). 
\end{remark} 

\begin{remark}\label{Remark 1.5}
We point out an interesting open problem of extending our result to the Hall-MHD system:
\begin{subequations}
\begin{align}
&\partial_{t}u + (u\cdot\nabla) u + \nabla \pi = \Delta u + (b\cdot\nabla) b + \xi^{u}, \hspace{25mm} \nabla\cdot u = 0, \label{12a}\\
& \partial_{t}b+ (u\cdot\nabla) b =  \Delta b + (b\cdot\nabla) u - \epsilon \nabla \times ((\nabla \times b) \times b)+ \xi^{b}, \hspace{3mm} \nabla\cdot b=  0, \label{12b}
\end{align}
\end{subequations} 
where $\epsilon \geq 0$ is the Hall parameter. We note that the case $\epsilon = 0$ reduces (\ref{12a})-(\ref{12b}) to the MHD system (\ref{6a})-(\ref{6b}). Since this system was introduced by Lighthill \cite{L60} over 75 years ago, it has found rich applications in astrophysics, geophysics and plasma physics; we refer to \cite{ADFL11, CDL14} for its study in the deterministic case and \cite{Y17a, YM18} in the stochastic case. By definition from \cite[Assumption 8.3]{H14}, the $N$-d Hall-MHD system is not locally subcritical for any $N \geq 2$. We believe that extending Theorem \ref{Theorem 1.2} to the Hall-MHD system, which is quasi-linear, is a  mathematically challenging and physically meaningful open problem. 
\end{remark}
 
\begin{remark}\label{Remark 1.6}
All the previous work on the MHD and related systems forced by random force have been devoted to the case the noise is white in only time and not space (e.g.,  \cite{BD07, S10a, SS99, Y18b}). Theorem \ref{Theorem 1.2} sheds light on the MHD system forced by STWN that has been studied in the physics literature (e.g., \cite{CT92}), and it has become clearer how to establish similar results for other systems such as the Boussinesq system for which its study with STWN has also been suggested by physicists for decades (\cite{ACHS81, GP75, HS92, SH77}). Moreover, it will be interesting to study a system of PDEs forced partially by STWN, e.g., the Boussinesq system with only the equation of the temperature forced by noise that is white only in time in \cite{FGRT15}.
\end{remark} 

\begin{remark}\label{Remark 1.7}
This work was initially completed in 2019. Subsequently in 2021, strong Feller property of the 3D MHD system forced by STWN was proven \cite{Y21a} via the approach of \cite{HM18a} using the theory of regularity structures (see also \cite{Y21b}). In comparison to the theory of regularity structures, the theory of paracontrolled distributions offers simpler approach that has led to results which do not seem accessible yet via the theory of regularity structures. One example is \cite{GKO18, GKP22} in which the authors successfully employed the theory of paracontrolled distributions to the stochastic nonlinear wave equations forced by STWN that falls outside the scope of the theory of regularity structures. Second important example is the very recent application of convex integration to the 3-d NSE forced by STWN \cite{HZZ22}. Let us briefly elaborate on this topic considering its relevance to our current work. The convex integration is a new revolutionary technique in deterministic hydrodynamic PDEs that led to, among many other breakthroughs, non-uniqueness of the Euler equations in any dimension \cite{DS09}, resolution of Onsager's conjecture \cite{I18}, and non-uniqueness of weak solutions to the 3-d NSE \cite{BV19a}. The impact of convex integration has reached the stochastic community as well and very recently, Hofmanov$\acute{a}$, Zhu, and Zhu \cite{HZZ19} proved non-uniqueness in law of the 3-d NSE forced by either additive or linear multiplicative noise that is white only in time (see also \cite{BFH20, CFF19}); subsequently, the author in \cite{Y21c} extended this result to the 3-d MHD system forced by either additive or linear multiplicative noise, although its diffusion $-\Delta u, -\Delta b$ had to be replaced by $(-\Delta)^{m_{1}} u, (-\Delta)^{m_{2}} b$ for any $m_{1}, m_{2} \in (0,1)$ due to technical reasons. Remarkably, Hofmanov$\acute{a}$, Zhu, and Zhu \cite{HZZ22} extended \cite{HZZ19} to the case of STWN, and here, they crucially relied on the approach of paracontrolled distributions rather than the theory of regularity structures. Proving non-uniqueness of singular stochastic PDEs forced by STWN via probabilistic convex integration rather than well-posedness is a completely new approach that has great potential, especially for singular PDEs that are not locally subcritical and fall outside the scope of the theory of regularity structures or the paracontrolled distributions, e.g., the stochastic Yang-Mills equation in dimension beyond three (see \cite{CCHS22a, CCHS22b}). 
\end{remark}

\section{Proof of Theorem \ref{Theorem 1.2}: Fixed Point Procedure} 
Hereafter, we denote $\mathcal{C}^{\alpha}(\mathbb{T}^{3})$ by simply $\mathcal{C}^{\alpha}$. We consider $\{\xi^{\epsilon} \}_{\epsilon > 0}$, a family of smooth approximations of $\xi = (\xi^{u}, \xi^{b})$, to be specified subsequently, and study the MHD system corresponding to $\xi^{\epsilon}$; we should formally denote its solution as $y^{\epsilon} \triangleq (u^{\epsilon}, b^{\epsilon})$ but for brevity omit it until (\ref{111}) when it is clear. We recall that $L \triangleq \partial_{t} - \Delta$ and study the following system: 
\begin{subequations}\label{15} 
\begin{align}
Lu_{i} =& \sum_{i_{1} = 1}^{3} \mathcal{P}_{i i_{1}} \xi_{i_{1}}^{u}
 - \frac{1}{2} \sum_{i_{1}, j=1}^{3} \mathcal{P}_{i i_{1}} \partial_{x_{j}} (u_{i_{1}} u_{j})  + \frac{1}{2} \sum_{i_{1}, j=1}^{3} \mathcal{P}_{i i_{1}} \partial_{x_{j}} ( b_{i_{1}} b_{j}),\\
L b_{i} =& \sum_{i_{1} =1}^{3} \mathcal{P}_{ii_{1}} \xi_{i_{1}}^{b}
 - \frac{1}{2} \sum_{i_{1}, j=1}^{3} \mathcal{P}_{ii_{1}} \partial_{x_{j}}(b_{i_{1}} u_{j}) + \frac{1}{2} \sum_{i_{1}, j=1}^{3} \mathcal{P}_{ i i_{1}} \partial_{x_{j}} (u_{i_{1}} b_{j}),\\
y(\cdot, 0) =& \mathcal{P}(u^{\text{in}},b^{\text{in}})(\cdot) \in \mathcal{C}^{-z}, 
\end{align} 
\end{subequations}
where $\xi \triangleq (\xi^{u}, \xi^{b})$ are periodic, independent STWN. 
\subsection{Paracontrolled Ansatz}
Let us approximate (\ref{11a})-(\ref{11b}) as follows. We start with the linear equations forced by  noise first: 
\begin{equation}\label{16}
Lu^{\scalebox{0.6}{

}}_{i} = \sum_{i_{1} =1}^{3} \mathcal{P}_{ii_{1}} \xi_{i_{1}}^{b}. 
\end{equation}

\begin{remark}
Informally, we denoted by $\scalebox{0.6}{%
}$
respectively the STWN of $\xi^{u}$ and $\xi^{b}$ and by a downward line an integration after applying $e^{-t\Delta} \mathcal{P}_{ii_{1}}$. Moreover, a zigzag line will represent an integration after applying $e^{-t\Delta} \mathcal{P}_{ii_{1}} \partial_{x_{j}}$, as we will see next in \eqref{17}-\eqref{18}. In the equations \eqref{17}-\eqref{18} we chose light colors of green 
$\scalebox{0.6}{%
}$ 
to informally represent the velocity field $u$, and dark colors of violet
$\scalebox{0.6}{%
}$
to represent the magnetic field $b$. Finally, we define $u^{F}$ and $b^{F}$ respectively in \eqref{20} and \eqref{21} where we chose ``F'' only because it is the final piece such that the sum satisfies the original system \eqref{15}.
\end{remark}  
We proceed as follows. If we temporarily define $v^{\scalebox{0.6}{%
}}_{i}$ and study the equation of $Lv_{i}^{u}$, then considering that $\xi_{i_{1}}^{u}, \xi_{i_{1}}^{b} \in \mathcal{C}^{\alpha}$ for $\alpha < - \frac{5}{2}$ so that $u^{\scalebox{0.6}{%
}}_{i} \in \mathcal{C}^{\alpha}$ for $\alpha < - \frac{1}{2}$, we see that within the equation of $Lv^{
\scalebox{0.6}{%
}}_{i}$ there are nonlinear terms $-\frac{1}{2} \sum_{i_{1}, j = 1}^{3} \mathcal{P}_{ii_{1}} \partial_{x_{j}} (u^{\scalebox{0.6}{%
}}_{j})$ which are ill-defined according to Lemma \ref{Lemma 1.1} (4). This leads to the equation of $Lu^{
\scalebox{0.18}{%
}}_{i}$ in \eqref{17a} with \eqref{24b} below and repeating this procedure also leads to \eqref{17b}, \eqref{18}-\eqref{21}: 
\begin{subequations}\label{17} 
\begin{align}
Lu^{ 
\scalebox{0.18}{%
}}(\cdot, 0)  = 0,  
\end{align}
\end{subequations} 
and finally with initial data of 
\begin{equation}\label{19}
u^{F}(\cdot, 0) = \mathcal{P} u^{\text{in}}(\cdot) - u^{\scalebox{0.6}{%
}}_{j} + b^{F}_{j})]. 
\end{align}
\begin{remark}\label{Remark 3.1} 
As we agreed to write $y = (u,b)$ for brevity, let us also write 
\begin{equation}\label{22}
y^{\scalebox{0.6}{%
}}), 
\hspace{5mm} y^{F} \triangleq (u^{F}, b^{F}).
\end{equation}  
Let us observe that $y^{
\scalebox{0.18}{%
}}$ may be solved in (\ref{17})  using that $y^{\scalebox{0.6}{%
}}$ may be solved in (\ref{18}) using that $y^{\scalebox{0.6}{%
}}$ are known, but $y^{F}$ in (\ref{19})-(\ref{21}) are the unknown. We also point out that another important feature of this construction is that $u^{
\scalebox{0.18}{%
}} + u^{F})(\cdot, 0) = \mathcal{P} u^{\text{in}}(\cdot)$, and an analogous statement for the equation of magnetic field can be made. Finally, let us observe that 
$\lVert y^{F}(\cdot, 0) \rVert_{\mathcal{C}^{-z}} \lesssim \lVert y^{\text{in}}(\cdot) \rVert_{\mathcal{C}^{-z}} + \lVert y^{\scalebox{0.6}{%
}}(\cdot, 0) \rVert_{\mathcal{C}^{-z}} \lesssim 1$ by the hypothesis of Theorem \ref{Theorem 1.2} that $y^{\text{in}} \in \mathcal{C}^{-z}, z \in (\frac{1}{2}, \frac{1}{2} + \delta_{0}), \delta_{0} \in (0, \frac{1}{2})$ and $y^{\scalebox{0.6}{%
}} \in \mathcal{C}^{\alpha}$ for $\alpha <- \frac{1}{2}$ due to (\ref{16}) being a linear heat equation so that $y^{\scalebox{0.6}{%
}} \in \mathcal{C}^{-z}$ indeed; this will be crucially used in (\ref{108}). 
\end{remark}

We now specify that 
\begin{subequations}
\begin{align}
& u^{\scalebox{0.6}{%
}}_{j}) - C_{1,4}^{\epsilon, ij}; \label{25b}
\end{align}
\end{subequations} 
we postpone specific description of the constants; e.g.,  $C_{0,1}^{\epsilon, ij}, C_{2,3}^{\epsilon, ij}$ and $C_{1,3}^{\epsilon, ij}$ are given in (\ref{117}), (\ref{162}) and (\ref{195}), respectively. Now we consider the following equations 
\begin{equation}\label{26}
L K_{i}^{u} = u^{\scalebox{0.6}{%
}}_{j} ).\nonumber
\end{align}
\end{subequations}
We also define a paracontrolled ansatz of 
\begin{align}\label{29}
u^{F}_{i} =& - \frac{1}{2} \sum_{i_{1}, j_{1} = 1}^{3} \mathcal{P}_{i i_{1}} \partial_{x_{j_{1}}}  [ \pi_{<} (u^{
\scalebox{0.16}{%
}}_{j_{1}} + b^{F}_{j_{1}}, K_{i_{1}}^{b}) ] + u^{\sharp}_{i};
\end{align} 
additionally, we define 
\begin{align*}
& \pi_{0, \diamond} (\mathcal{P}_{ii_{1}} \partial_{x_{j}} K_{j}^{u}, u^{\scalebox{0.6}{%
}}_{j_{2}}).
\end{align*}
Similarly we may define $\pi_{0,\diamond}(b^{F}_{i}, b^{
\scalebox{0.6}{%
}}_{j})$ is identical to $\pi_{0, \diamond} (P_{ii_{1}} \pi_{<} (u^{
\scalebox{0.16}{%
}}_{j})$ is defined as $\pi_{0, \diamond} (\mathcal{P}_{ii_{1}} \pi_{<} (b^{
\scalebox{0.16}{%
}}_{j}$ and $K_{j_{1}}^{b}$ replaced by $K_{j_{1}}^{u}$. We also define a paracontrolled ansatz of 
\begin{align}\label{32}
b^{F}_{i} =& - \frac{1}{2} \sum_{i_{1}, j_{1} = 1}^{3} \mathcal{P}_{i i_{1}} \partial_{x_{j_{1}}}  [- \pi_{<} (u^{
\scalebox{0.16}{%
}}_{j_{1}} + b^{F}_{j_{1}}, K_{i_{1}}^{u}) ] + b^{\sharp}_{i}; 
\end{align} 
additionally we define
\begin{align*} 
& \pi_{0, \diamond} (\mathcal{P}_{ii_{1}} \partial_{x_{j}} K_{j}^{b}, b^{
\scalebox{0.6}{%
}}_{j_{2}}).
\end{align*}
\begin{remark}\label{Remark 3.2}
This step is absolutely crucial and even following the case of the NSE in \cite{ZZ15}, particularly the signs of the four terms within (\ref{32}) are not clear at first sight. We chose (\ref{32}) in order to make the proof work, particularly bearing in mind the crucial steps at (\ref{21}), (\ref{51}), and (\ref{31}). 
\end{remark} 

For $\pi_{0, \diamond} (u^{F}_{i}, b^{
\scalebox{0.6}{%
}}_{j})$ of (\ref{23f}), it is essentially identical to $\pi_{0, \diamond} (u^{F}_{i}, u^{\scalebox{0.6}{%
}}_{j}$ because $u^{F}_{i}$ has already been defined in (\ref{29}). We leave details here for completeness:
\begin{align*} 
\pi_{0, \diamond} (u^{F}_{i}, b^{
\scalebox{0.6}{%
}}_{j})$ is identical to $\pi_{0, \diamond} (\mathcal{P}_{ i i_{1}} \pi_{<} (u^{
\scalebox{0.16}{%
}}_{j})$ is defined as $\pi_{0, \diamond} ( \mathcal{P}_{i i_{1}} \pi_{<} (b^{
\scalebox{0.16}{%
}}_{j} )$ of (\ref{23h}), it is also identical to $\pi_{0, \diamond} (b^{F}_{i}, b^{
\scalebox{0.6}{%
}}_{j} $, which is automatic because we already defined $b^{F}_{i}$ in (\ref{32}). Now from (\ref{26}), for all $\delta \in [0, 4]$ we may compute 
\begin{subequations}\label{33} 
\begin{align}
& \lVert K_{i}^{u} (t) \rVert_{\mathcal{C}^{\frac{3}{2} - \delta}} \lesssim \int_{0}^{t} (t-s)^{- \frac{(2 - \frac{\delta}{2}}{2})} \lVert u^{\scalebox{0.6}{%
}}_{i} (s) \rVert_{\mathcal{C}^{- \frac{1}{2} - \frac{\delta}{2}}} t^{\frac{\delta}{4}} 
\end{align}
\end{subequations} 
by (\ref{26}) and Lemma \ref{Lemma 2.6}. We fix 
\begin{equation}\label{35}
0 < \delta < \delta_{0} \wedge \frac{1- 2 \delta_{0}}{3} \wedge \frac{1-z}{4} \wedge (2z-1).
\end{equation} 
Let us assume that 
\begin{subequations}
\begin{align}
& u^{\scalebox{0.6}{%
}}_{j_{1}}) \in C([0,T]; \mathcal{C}^{-\delta})\label{36f}
\end{align}
\end{subequations} 
for all $i, j, i_{1}, j_{1} \in \{1,2,3\}$ so that we may define a finite number of 
\begin{align}\label{37}
C_{\xi}^{\epsilon} \triangleq & \sup_{t \in [0,T]} [ \sum_{i=1}^{3} \lVert (u^{\scalebox{0.6}{%
}}_{j_{1}})) \rVert_{\mathcal{C}^{-\delta}}];
\end{align} 
let us write $C_{\xi}$ in case $\epsilon = 0$. We mention in particular the inclusion of the last two summations in (\ref{37}) will be crucial in (\ref{70}) and (\ref{76}). Now from (\ref{17}) we see that 
\begin{align}\label{38}
\sup_{t \in [0,T]} \lVert (u^{
\scalebox{0.18}{%
}}_{j} ) \rVert_{\mathcal{C}^{- 1 - \frac{\delta}{2}}} ds  \lesssim C_{\xi} T^{\frac{\delta}{4}}
\end{align} 
by Lemmas \ref{Lemma 2.7} and \ref{Lemma 2.6}, (\ref{35}) and (\ref{37}). Similarly from (\ref{18}), by relying on Lemma \ref{Lemma 2.6}, (\ref{35}) and (\ref{37}) we may compute 
\begin{align*}
&\sup_{t \in [0,T]} \lVert (u^{
\scalebox{0.16}{%
}} \rVert_{C([0,T]; \mathcal{C}^{\frac{1}{2} - \delta})} \lesssim C_{\xi} T^{\frac{\delta}{4}}. 
\end{equation} 
Next, from (\ref{19})-(\ref{21}), we may compute 
\begin{equation}\label{40}
\sup_{t \in [0, T]} t^{\frac{ \frac{1}{2} - \delta_{0} + z}{2}} \lVert (u^{F}_{i}, b^{F}_{i})(t) \rVert_{\mathcal{C}^{\frac{1}{2} - \delta_{0}}}\lesssim I_{T}^{1} + I_{T}^{2} 
\end{equation} 
for 
\begin{subequations}
\begin{align}
I_{T}^{1} \triangleq& \sup_{t \in [0,T]} t^{\frac{ \frac{1}{2} - \delta_{0} + z}{2}} \lVert P_{t} (\mathcal{P} y_{i}^{\text{in}} - (u^{\scalebox{0.6}{\begin{tikzpicture}
\draw[black, thick] (0,0.5) -- (0,0);
\filldraw[red] (0,0.5) circle (2pt); 
\end{tikzpicture}
}}_{i}, b^{
\scalebox{0.6}{\begin{tikzpicture}
\draw[black, thick] (0,0.5) -- (0,0);
\filldraw[blue] (0,0.5) circle (2pt); 
\end{tikzpicture}
}}_{i})(0)) \rVert_{\mathcal{C}^{\frac{1}{2} - \delta_{0}}}, \\
I_{T}^{2} \triangleq& \sup_{t \in [0,T]} t^{\frac{ \frac{1}{2} - \delta_{0} + z}{2}} \sum_{i_{1}, j =1}^{3} \int_{0}^{t} \lVert P_{t-s} (u^{\scalebox{0.6}{\begin{tikzpicture}
\draw[black, thick] (0,0.5) -- (0,0);
\filldraw[red] (0,0.5) circle (2pt); 
\end{tikzpicture}
}}_{i_{1}}  \diamond (u^{
\scalebox{0.16}{\begin{tikzpicture}
\draw[black, thick] (-0.7,0.9) -- (0,0);
\draw[black, thick] (0.7,0.9) -- (0,0);
\draw[black, thick] (0.7,0) -- (0,-0.9);
\draw[snake=zigzag](0,0) -- (0,-0.9);
\draw[snake=zigzag](0,-0.9) -- (0,-1.8);
\filldraw[pink] (-0.7,0.9) circle (7pt); 
\filldraw[pink] (0.7,0.9) circle (7pt); 
\filldraw[pink] (0.7,0) circle (7pt); 
\end{tikzpicture}
}}_{j} + u^{F}_{j}) + (u^{
\scalebox{0.16}{\begin{tikzpicture}
\draw[black, thick] (-0.7,0.9) -- (0,0);
\draw[black, thick] (0.7,0.9) -- (0,0);
\draw[black, thick] (0.7,0) -- (0,-0.9);
\draw[snake=zigzag](0,0) -- (0,-0.9);
\draw[snake=zigzag](0,-0.9) -- (0,-1.8);
\filldraw[pink] (-0.7,0.9) circle (7pt); 
\filldraw[pink] (0.7,0.9) circle (7pt); 
\filldraw[pink] (0.7,0) circle (7pt); 
\end{tikzpicture}
}}_{i_{1}} + u^{F}_{i_{1}}) \diamond u^{\scalebox{0.6}{\begin{tikzpicture}
\draw[black, thick] (0,0.5) -- (0,0);
\filldraw[red] (0,0.5) circle (2pt); 
\end{tikzpicture}
}}_{j}    \\
& \hspace{2mm} + u^{
\scalebox{0.18}{\begin{tikzpicture}
\draw[black, thick] (-0.7,0.9) -- (0,0);
\draw[black, thick] (0.7,0.9) -- (0,0);
\draw[snake=zigzag](0,0) -- (0,-0.9);
\filldraw[green] (-0.7,0.9) circle (6pt); 
\filldraw[green] (0.7,0.9) circle (6pt); 
\end{tikzpicture}
}}_{i_{1}} \diamond u^{
\scalebox{0.18}{\begin{tikzpicture}
\draw[black, thick] (-0.7,0.9) -- (0,0);
\draw[black, thick] (0.7,0.9) -- (0,0);
\draw[snake=zigzag](0,0) -- (0,-0.9);
\filldraw[green] (-0.7,0.9) circle (6pt); 
\filldraw[green] (0.7,0.9) circle (6pt); 
\end{tikzpicture}
}}_{j} + u^{
\scalebox{0.18}{\begin{tikzpicture}
\draw[black, thick] (-0.7,0.9) -- (0,0);
\draw[black, thick] (0.7,0.9) -- (0,0);
\draw[snake=zigzag](0,0) -- (0,-0.9);
\filldraw[green] (-0.7,0.9) circle (6pt); 
\filldraw[green] (0.7,0.9) circle (6pt); 
\end{tikzpicture}
}}_{i_{1}}(u^{
\scalebox{0.16}{\begin{tikzpicture}
\draw[black, thick] (-0.7,0.9) -- (0,0);
\draw[black, thick] (0.7,0.9) -- (0,0);
\draw[black, thick] (0.7,0) -- (0,-0.9);
\draw[snake=zigzag](0,0) -- (0,-0.9);
\draw[snake=zigzag](0,-0.9) -- (0,-1.8);
\filldraw[pink] (-0.7,0.9) circle (7pt); 
\filldraw[pink] (0.7,0.9) circle (7pt); 
\filldraw[pink] (0.7,0) circle (7pt); 
\end{tikzpicture}
}}_{j} + u^{F}_{j}) + u^{
\scalebox{0.18}{\begin{tikzpicture}
\draw[black, thick] (-0.7,0.9) -- (0,0);
\draw[black, thick] (0.7,0.9) -- (0,0);
\draw[snake=zigzag](0,0) -- (0,-0.9);
\filldraw[green] (-0.7,0.9) circle (6pt); 
\filldraw[green] (0.7,0.9) circle (6pt); 
\end{tikzpicture}
}}_{j}(u^{
\scalebox{0.16}{\begin{tikzpicture}
\draw[black, thick] (-0.7,0.9) -- (0,0);
\draw[black, thick] (0.7,0.9) -- (0,0);
\draw[black, thick] (0.7,0) -- (0,-0.9);
\draw[snake=zigzag](0,0) -- (0,-0.9);
\draw[snake=zigzag](0,-0.9) -- (0,-1.8);
\filldraw[pink] (-0.7,0.9) circle (7pt); 
\filldraw[pink] (0.7,0.9) circle (7pt); 
\filldraw[pink] (0.7,0) circle (7pt); 
\end{tikzpicture}
}}_{i_{1}} + u^{F}_{i_{1}})\nonumber\\
& \hspace{2mm} + (u^{
\scalebox{0.16}{\begin{tikzpicture}
\draw[black, thick] (-0.7,0.9) -- (0,0);
\draw[black, thick] (0.7,0.9) -- (0,0);
\draw[black, thick] (0.7,0) -- (0,-0.9);
\draw[snake=zigzag](0,0) -- (0,-0.9);
\draw[snake=zigzag](0,-0.9) -- (0,-1.8);
\filldraw[pink] (-0.7,0.9) circle (7pt); 
\filldraw[pink] (0.7,0.9) circle (7pt); 
\filldraw[pink] (0.7,0) circle (7pt); 
\end{tikzpicture}
}}_{i_{1}} + u^{F}_{i_{1}})(u^{
\scalebox{0.16}{\begin{tikzpicture}
\draw[black, thick] (-0.7,0.9) -- (0,0);
\draw[black, thick] (0.7,0.9) -- (0,0);
\draw[black, thick] (0.7,0) -- (0,-0.9);
\draw[snake=zigzag](0,0) -- (0,-0.9);
\draw[snake=zigzag](0,-0.9) -- (0,-1.8);
\filldraw[pink] (-0.7,0.9) circle (7pt); 
\filldraw[pink] (0.7,0.9) circle (7pt); 
\filldraw[pink] (0.7,0) circle (7pt); 
\end{tikzpicture}
}}_{j} + u^{F}_{j}) - b^{
\scalebox{0.6}{\begin{tikzpicture}
\draw[black, thick] (0,0.5) -- (0,0);
\filldraw[blue] (0,0.5) circle (2pt); 
\end{tikzpicture}
}}_{i_{1}}\diamond (b^{
\scalebox{0.16}{\begin{tikzpicture}
\draw[black, thick] (-0.7,0.9) -- (0,0);
\draw[black, thick] (0.7,0.9) -- (0,0);
\draw[black, thick] (0.7,0) -- (0,-0.9);
\draw[snake=zigzag](0,0) -- (0,-0.9);
\draw[snake=zigzag](0,-0.9) -- (0,-1.8);
\filldraw[gray] (-0.7,0.9) circle (7pt); 
\filldraw[gray] (0.7,0.9) circle (7pt); 
\filldraw[gray] (0.7,0) circle (7pt); 
\end{tikzpicture}
}}_{j} + b^{F}_{j}) - (b^{
\scalebox{0.16}{\begin{tikzpicture}
\draw[black, thick] (-0.7,0.9) -- (0,0);
\draw[black, thick] (0.7,0.9) -- (0,0);
\draw[black, thick] (0.7,0) -- (0,-0.9);
\draw[snake=zigzag](0,0) -- (0,-0.9);
\draw[snake=zigzag](0,-0.9) -- (0,-1.8);
\filldraw[gray] (-0.7,0.9) circle (7pt); 
\filldraw[gray] (0.7,0.9) circle (7pt); 
\filldraw[gray] (0.7,0) circle (7pt); 
\end{tikzpicture}
}}_{i_{1}} + b^{F}_{i_{1}}) \diamond b^{
\scalebox{0.6}{\begin{tikzpicture}
\draw[black, thick] (0,0.5) -- (0,0);
\filldraw[blue] (0,0.5) circle (2pt); 
\end{tikzpicture}
}}_{j}\nonumber\\
& \hspace{2mm} - b^{
\scalebox{0.18}{\begin{tikzpicture}
\draw[black, thick] (-0.7,0.9) -- (0,0);
\draw[black, thick] (0.7,0.9) -- (0,0);
\draw[snake=zigzag](0,0) -- (0,-0.9);
\filldraw[violet] (-0.7,0.9) circle (6pt); 
\filldraw[violet] (0.7,0.9) circle (6pt); 
\end{tikzpicture}
}}_{i_{1}} \diamond b^{
\scalebox{0.18}{\begin{tikzpicture}
\draw[black, thick] (-0.7,0.9) -- (0,0);
\draw[black, thick] (0.7,0.9) -- (0,0);
\draw[snake=zigzag](0,0) -- (0,-0.9);
\filldraw[violet] (-0.7,0.9) circle (6pt); 
\filldraw[violet] (0.7,0.9) circle (6pt); 
\end{tikzpicture}
}}_{j} - b^{
\scalebox{0.18}{\begin{tikzpicture}
\draw[black, thick] (-0.7,0.9) -- (0,0);
\draw[black, thick] (0.7,0.9) -- (0,0);
\draw[snake=zigzag](0,0) -- (0,-0.9);
\filldraw[violet] (-0.7,0.9) circle (6pt); 
\filldraw[violet] (0.7,0.9) circle (6pt); 
\end{tikzpicture}
}}_{i_{1}} (b^{
\scalebox{0.16}{\begin{tikzpicture}
\draw[black, thick] (-0.7,0.9) -- (0,0);
\draw[black, thick] (0.7,0.9) -- (0,0);
\draw[black, thick] (0.7,0) -- (0,-0.9);
\draw[snake=zigzag](0,0) -- (0,-0.9);
\draw[snake=zigzag](0,-0.9) -- (0,-1.8);
\filldraw[gray] (-0.7,0.9) circle (7pt); 
\filldraw[gray] (0.7,0.9) circle (7pt); 
\filldraw[gray] (0.7,0) circle (7pt); 
\end{tikzpicture}
}}_{j} + b^{F}_{j}) - b^{
\scalebox{0.18}{\begin{tikzpicture}
\draw[black, thick] (-0.7,0.9) -- (0,0);
\draw[black, thick] (0.7,0.9) -- (0,0);
\draw[snake=zigzag](0,0) -- (0,-0.9);
\filldraw[violet] (-0.7,0.9) circle (6pt); 
\filldraw[violet] (0.7,0.9) circle (6pt); 
\end{tikzpicture}
}}_{j} (b^{
\scalebox{0.16}{\begin{tikzpicture}
\draw[black, thick] (-0.7,0.9) -- (0,0);
\draw[black, thick] (0.7,0.9) -- (0,0);
\draw[black, thick] (0.7,0) -- (0,-0.9);
\draw[snake=zigzag](0,0) -- (0,-0.9);
\draw[snake=zigzag](0,-0.9) -- (0,-1.8);
\filldraw[gray] (-0.7,0.9) circle (7pt); 
\filldraw[gray] (0.7,0.9) circle (7pt); 
\filldraw[gray] (0.7,0) circle (7pt); 
\end{tikzpicture}
}}_{i_{1}} + b^{F}_{i_{1}}) - (b^{
\scalebox{0.16}{\begin{tikzpicture}
\draw[black, thick] (-0.7,0.9) -- (0,0);
\draw[black, thick] (0.7,0.9) -- (0,0);
\draw[black, thick] (0.7,0) -- (0,-0.9);
\draw[snake=zigzag](0,0) -- (0,-0.9);
\draw[snake=zigzag](0,-0.9) -- (0,-1.8);
\filldraw[gray] (-0.7,0.9) circle (7pt); 
\filldraw[gray] (0.7,0.9) circle (7pt); 
\filldraw[gray] (0.7,0) circle (7pt); 
\end{tikzpicture}
}}_{i_{1}} + b^{F}_{i_{1}})(b^{
\scalebox{0.16}{\begin{tikzpicture}
\draw[black, thick] (-0.7,0.9) -- (0,0);
\draw[black, thick] (0.7,0.9) -- (0,0);
\draw[black, thick] (0.7,0) -- (0,-0.9);
\draw[snake=zigzag](0,0) -- (0,-0.9);
\draw[snake=zigzag](0,-0.9) -- (0,-1.8);
\filldraw[gray] (-0.7,0.9) circle (7pt); 
\filldraw[gray] (0.7,0.9) circle (7pt); 
\filldraw[gray] (0.7,0) circle (7pt); 
\end{tikzpicture}
}}_{j} + b^{F}_{j}), \nonumber\\
& \hspace{2mm} b^{
\scalebox{0.6}{\begin{tikzpicture}
\draw[black, thick] (0,0.5) -- (0,0);
\filldraw[blue] (0,0.5) circle (2pt); 
\end{tikzpicture}
}}_{i_{1}} \diamond (u^{
\scalebox{0.16}{\begin{tikzpicture}
\draw[black, thick] (-0.7,0.9) -- (0,0);
\draw[black, thick] (0.7,0.9) -- (0,0);
\draw[black, thick] (0.7,0) -- (0,-0.9);
\draw[snake=zigzag](0,0) -- (0,-0.9);
\draw[snake=zigzag](0,-0.9) -- (0,-1.8);
\filldraw[pink] (-0.7,0.9) circle (7pt); 
\filldraw[pink] (0.7,0.9) circle (7pt); 
\filldraw[pink] (0.7,0) circle (7pt); 
\end{tikzpicture}
}}_{j} + u^{F}_{j}) +(b^{
\scalebox{0.16}{\begin{tikzpicture}
\draw[black, thick] (-0.7,0.9) -- (0,0);
\draw[black, thick] (0.7,0.9) -- (0,0);
\draw[black, thick] (0.7,0) -- (0,-0.9);
\draw[snake=zigzag](0,0) -- (0,-0.9);
\draw[snake=zigzag](0,-0.9) -- (0,-1.8);
\filldraw[gray] (-0.7,0.9) circle (7pt); 
\filldraw[gray] (0.7,0.9) circle (7pt); 
\filldraw[gray] (0.7,0) circle (7pt); 
\end{tikzpicture}
}}_{i_{1}}+  b^{F}_{i_{1}}) \diamond u^{\scalebox{0.6}{\begin{tikzpicture}
\draw[black, thick] (0,0.5) -- (0,0);
\filldraw[red] (0,0.5) circle (2pt); 
\end{tikzpicture}
}}_{j}  + b^{
\scalebox{0.18}{\begin{tikzpicture}
\draw[black, thick] (-0.7,0.9) -- (0,0);
\draw[black, thick] (0.7,0.9) -- (0,0);
\draw[snake=zigzag](0,0) -- (0,-0.9);
\filldraw[violet] (-0.7,0.9) circle (6pt); 
\filldraw[violet] (0.7,0.9) circle (6pt); 
\end{tikzpicture}
}}_{i_{1}} \diamond u^{
\scalebox{0.18}{\begin{tikzpicture}
\draw[black, thick] (-0.7,0.9) -- (0,0);
\draw[black, thick] (0.7,0.9) -- (0,0);
\draw[snake=zigzag](0,0) -- (0,-0.9);
\filldraw[green] (-0.7,0.9) circle (6pt); 
\filldraw[green] (0.7,0.9) circle (6pt); 
\end{tikzpicture}
}}_{j}\nonumber \\
& \hspace{2mm} + b^{
\scalebox{0.18}{\begin{tikzpicture}
\draw[black, thick] (-0.7,0.9) -- (0,0);
\draw[black, thick] (0.7,0.9) -- (0,0);
\draw[snake=zigzag](0,0) -- (0,-0.9);
\filldraw[violet] (-0.7,0.9) circle (6pt); 
\filldraw[violet] (0.7,0.9) circle (6pt); 
\end{tikzpicture}
}}_{i_{1}} (u^{
\scalebox{0.16}{\begin{tikzpicture}
\draw[black, thick] (-0.7,0.9) -- (0,0);
\draw[black, thick] (0.7,0.9) -- (0,0);
\draw[black, thick] (0.7,0) -- (0,-0.9);
\draw[snake=zigzag](0,0) -- (0,-0.9);
\draw[snake=zigzag](0,-0.9) -- (0,-1.8);
\filldraw[pink] (-0.7,0.9) circle (7pt); 
\filldraw[pink] (0.7,0.9) circle (7pt); 
\filldraw[pink] (0.7,0) circle (7pt); 
\end{tikzpicture}
}}_{j} + u^{F}_{j}) + u^{
\scalebox{0.18}{\begin{tikzpicture}
\draw[black, thick] (-0.7,0.9) -- (0,0);
\draw[black, thick] (0.7,0.9) -- (0,0);
\draw[snake=zigzag](0,0) -- (0,-0.9);
\filldraw[green] (-0.7,0.9) circle (6pt); 
\filldraw[green] (0.7,0.9) circle (6pt); 
\end{tikzpicture}
}}_{j}(b^{
\scalebox{0.16}{\begin{tikzpicture}
\draw[black, thick] (-0.7,0.9) -- (0,0);
\draw[black, thick] (0.7,0.9) -- (0,0);
\draw[black, thick] (0.7,0) -- (0,-0.9);
\draw[snake=zigzag](0,0) -- (0,-0.9);
\draw[snake=zigzag](0,-0.9) -- (0,-1.8);
\filldraw[gray] (-0.7,0.9) circle (7pt); 
\filldraw[gray] (0.7,0.9) circle (7pt); 
\filldraw[gray] (0.7,0) circle (7pt); 
\end{tikzpicture}
}}_{i_{1}} + b^{F}_{i_{1}}) + (b^{
\scalebox{0.16}{\begin{tikzpicture}
\draw[black, thick] (-0.7,0.9) -- (0,0);
\draw[black, thick] (0.7,0.9) -- (0,0);
\draw[black, thick] (0.7,0) -- (0,-0.9);
\draw[snake=zigzag](0,0) -- (0,-0.9);
\draw[snake=zigzag](0,-0.9) -- (0,-1.8);
\filldraw[gray] (-0.7,0.9) circle (7pt); 
\filldraw[gray] (0.7,0.9) circle (7pt); 
\filldraw[gray] (0.7,0) circle (7pt); 
\end{tikzpicture}
}}_{i_{1}} + b^{F}_{i_{1}}) (u^{
\scalebox{0.16}{\begin{tikzpicture}
\draw[black, thick] (-0.7,0.9) -- (0,0);
\draw[black, thick] (0.7,0.9) -- (0,0);
\draw[black, thick] (0.7,0) -- (0,-0.9);
\draw[snake=zigzag](0,0) -- (0,-0.9);
\draw[snake=zigzag](0,-0.9) -- (0,-1.8);
\filldraw[pink] (-0.7,0.9) circle (7pt); 
\filldraw[pink] (0.7,0.9) circle (7pt); 
\filldraw[pink] (0.7,0) circle (7pt); 
\end{tikzpicture}
}}_{j} + u^{F}_{j}) \nonumber\\
& \hspace{2mm} - u^{\scalebox{0.6}{\begin{tikzpicture}
\draw[black, thick] (0,0.5) -- (0,0);
\filldraw[red] (0,0.5) circle (2pt); 
\end{tikzpicture}
}}_{i_{1}}  \diamond (b^{
\scalebox{0.16}{\begin{tikzpicture}
\draw[black, thick] (-0.7,0.9) -- (0,0);
\draw[black, thick] (0.7,0.9) -- (0,0);
\draw[black, thick] (0.7,0) -- (0,-0.9);
\draw[snake=zigzag](0,0) -- (0,-0.9);
\draw[snake=zigzag](0,-0.9) -- (0,-1.8);
\filldraw[gray] (-0.7,0.9) circle (7pt); 
\filldraw[gray] (0.7,0.9) circle (7pt); 
\filldraw[gray] (0.7,0) circle (7pt); 
\end{tikzpicture}
}}_{j} + b^{F}_{j}) - (u^{
\scalebox{0.16}{\begin{tikzpicture}
\draw[black, thick] (-0.7,0.9) -- (0,0);
\draw[black, thick] (0.7,0.9) -- (0,0);
\draw[black, thick] (0.7,0) -- (0,-0.9);
\draw[snake=zigzag](0,0) -- (0,-0.9);
\draw[snake=zigzag](0,-0.9) -- (0,-1.8);
\filldraw[pink] (-0.7,0.9) circle (7pt); 
\filldraw[pink] (0.7,0.9) circle (7pt); 
\filldraw[pink] (0.7,0) circle (7pt); 
\end{tikzpicture}
}}_{i_{1}} + u^{F}_{i_{1}}) \diamond b^{
\scalebox{0.6}{\begin{tikzpicture}
\draw[black, thick] (0,0.5) -- (0,0);
\filldraw[blue] (0,0.5) circle (2pt); 
\end{tikzpicture}
}}_{j} - u^{
\scalebox{0.18}{\begin{tikzpicture}
\draw[black, thick] (-0.7,0.9) -- (0,0);
\draw[black, thick] (0.7,0.9) -- (0,0);
\draw[snake=zigzag](0,0) -- (0,-0.9);
\filldraw[green] (-0.7,0.9) circle (6pt); 
\filldraw[green] (0.7,0.9) circle (6pt); 
\end{tikzpicture}
}}_{i_{1}} \diamond b^{
\scalebox{0.18}{\begin{tikzpicture}
\draw[black, thick] (-0.7,0.9) -- (0,0);
\draw[black, thick] (0.7,0.9) -- (0,0);
\draw[snake=zigzag](0,0) -- (0,-0.9);
\filldraw[violet] (-0.7,0.9) circle (6pt); 
\filldraw[violet] (0.7,0.9) circle (6pt); 
\end{tikzpicture}
}}_{j} \nonumber\\
& \hspace{2mm} - u^{
\scalebox{0.18}{\begin{tikzpicture}
\draw[black, thick] (-0.7,0.9) -- (0,0);
\draw[black, thick] (0.7,0.9) -- (0,0);
\draw[snake=zigzag](0,0) -- (0,-0.9);
\filldraw[green] (-0.7,0.9) circle (6pt); 
\filldraw[green] (0.7,0.9) circle (6pt); 
\end{tikzpicture}
}}_{i_{1}} (b^{
\scalebox{0.16}{\begin{tikzpicture}
\draw[black, thick] (-0.7,0.9) -- (0,0);
\draw[black, thick] (0.7,0.9) -- (0,0);
\draw[black, thick] (0.7,0) -- (0,-0.9);
\draw[snake=zigzag](0,0) -- (0,-0.9);
\draw[snake=zigzag](0,-0.9) -- (0,-1.8);
\filldraw[gray] (-0.7,0.9) circle (7pt); 
\filldraw[gray] (0.7,0.9) circle (7pt); 
\filldraw[gray] (0.7,0) circle (7pt); 
\end{tikzpicture}
}}_{j} + b^{F}_{j}) - b^{
\scalebox{0.18}{\begin{tikzpicture}
\draw[black, thick] (-0.7,0.9) -- (0,0);
\draw[black, thick] (0.7,0.9) -- (0,0);
\draw[snake=zigzag](0,0) -- (0,-0.9);
\filldraw[violet] (-0.7,0.9) circle (6pt); 
\filldraw[violet] (0.7,0.9) circle (6pt); 
\end{tikzpicture}
}}_{j} (u^{
\scalebox{0.16}{\begin{tikzpicture}
\draw[black, thick] (-0.7,0.9) -- (0,0);
\draw[black, thick] (0.7,0.9) -- (0,0);
\draw[black, thick] (0.7,0) -- (0,-0.9);
\draw[snake=zigzag](0,0) -- (0,-0.9);
\draw[snake=zigzag](0,-0.9) -- (0,-1.8);
\filldraw[pink] (-0.7,0.9) circle (7pt); 
\filldraw[pink] (0.7,0.9) circle (7pt); 
\filldraw[pink] (0.7,0) circle (7pt); 
\end{tikzpicture}
}}_{i_{1}} + u^{F}_{i_{1}}) -(u^{
\scalebox{0.16}{\begin{tikzpicture}
\draw[black, thick] (-0.7,0.9) -- (0,0);
\draw[black, thick] (0.7,0.9) -- (0,0);
\draw[black, thick] (0.7,0) -- (0,-0.9);
\draw[snake=zigzag](0,0) -- (0,-0.9);
\draw[snake=zigzag](0,-0.9) -- (0,-1.8);
\filldraw[pink] (-0.7,0.9) circle (7pt); 
\filldraw[pink] (0.7,0.9) circle (7pt); 
\filldraw[pink] (0.7,0) circle (7pt); 
\end{tikzpicture}
}}_{i_{1}} + u^{F}_{i_{1}}) (b^{
\scalebox{0.16}{\begin{tikzpicture}
\draw[black, thick] (-0.7,0.9) -- (0,0);
\draw[black, thick] (0.7,0.9) -- (0,0);
\draw[black, thick] (0.7,0) -- (0,-0.9);
\draw[snake=zigzag](0,0) -- (0,-0.9);
\draw[snake=zigzag](0,-0.9) -- (0,-1.8);
\filldraw[gray] (-0.7,0.9) circle (7pt); 
\filldraw[gray] (0.7,0.9) circle (7pt); 
\filldraw[gray] (0.7,0) circle (7pt); 
\end{tikzpicture}
}}_{j} + b^{F}_{j}) )\rVert_{\mathcal{C}^{\frac{3}{2} - \delta_{0}}} ds\nonumber
\end{align}
\end{subequations} 
by Lemma \ref{Lemma 2.7} where it is immediate that we may estimate for $\epsilon \in (0,1)$ fixed, 
\begin{align*} 
I_{T}^{1}\lesssim& \sup_{t \in [0, T]} t^{ \frac{ \frac{1}{2} - \delta_{0} - z}{2}}   t^{- \frac{ (\frac{1}{2} - \delta_{0} + z)}{2}} ( \lVert \mathcal{P} y_{i}^{\text{in}} \rVert_{\mathcal{C}^{-z}} + \lVert (u^{\scalebox{0.6}{\begin{tikzpicture}
\draw[black, thick] (0,0.5) -- (0,0);
\filldraw[red] (0,0.5) circle (2pt); 
\end{tikzpicture}
}}_{i}, b^{
\scalebox{0.6}{\begin{tikzpicture}
\draw[black, thick] (0,0.5) -- (0,0);
\filldraw[blue] (0,0.5) circle (2pt); 
\end{tikzpicture}
}}_{i})(0) \rVert_{\mathcal{C}^{-z}} ) \lesssim 1 
\end{align*}
due to Lemma \ref{Lemma 2.6}, (\ref{35}) and Remark \ref{Remark 3.1}. Thus, we now focus on $I_{T}^{2}$. First we may estimate also for $\epsilon \in (0,1)$ fixed, 
\begin{align}\label{41}
& \sup_{t \in [0, T]} t^{\frac{ \frac{1}{2} - \delta_{0} + z}{2}} \int_{0}^{t} \lVert P_{t-s} (b^{
\scalebox{0.18}{\begin{tikzpicture}
\draw[black, thick] (-0.7,0.9) -- (0,0);
\draw[black, thick] (0.7,0.9) -- (0,0);
\draw[snake=zigzag](0,0) -- (0,-0.9);
\filldraw[violet] (-0.7,0.9) circle (6pt); 
\filldraw[violet] (0.7,0.9) circle (6pt); 
\end{tikzpicture}
}}_{i_{1}} \diamond u^{
\scalebox{0.18}{\begin{tikzpicture}
\draw[black, thick] (-0.7,0.9) -- (0,0);
\draw[black, thick] (0.7,0.9) -- (0,0);
\draw[snake=zigzag](0,0) -- (0,-0.9);
\filldraw[green] (-0.7,0.9) circle (6pt); 
\filldraw[green] (0.7,0.9) circle (6pt); 
\end{tikzpicture}
}}_{j})  \rVert_{\mathcal{C}^{\frac{3}{2} - \delta_{0}}} ds\nonumber \\
\lesssim& \sup_{t \in [0, T]} t^{\frac{ \frac{1}{2} - \delta_{0} + z}{2}} \int_{0}^{t}(t-s)^{ - \frac{ (\frac{3}{2} - \delta_{0} + \delta)}{2}} \lVert b^{
\scalebox{0.18}{\begin{tikzpicture}
\draw[black, thick] (-0.7,0.9) -- (0,0);
\draw[black, thick] (0.7,0.9) -- (0,0);
\draw[snake=zigzag](0,0) -- (0,-0.9);
\filldraw[violet] (-0.7,0.9) circle (6pt); 
\filldraw[violet] (0.7,0.9) circle (6pt); 
\end{tikzpicture}
}}_{i_{1}} \diamond u^{
\scalebox{0.18}{\begin{tikzpicture}
\draw[black, thick] (-0.7,0.9) -- (0,0);
\draw[black, thick] (0.7,0.9) -- (0,0);
\draw[snake=zigzag](0,0) -- (0,-0.9);
\filldraw[green] (-0.7,0.9) circle (6pt); 
\filldraw[green] (0.7,0.9) circle (6pt); 
\end{tikzpicture}
}}_{j} \rVert_{\mathcal{C}^{-\delta}} ds \lesssim 1
\end{align} 
by Lemma \ref{Lemma 2.6}, (\ref{35}) and (\ref{36d}). Second, e.g.,  we may also estimate 
\begin{align}
&\sup_{t \in [0, T]} t^{\frac{ \frac{1}{2} - \delta_{0} + z}{2}} \int_{0}^{t} \lVert P_{t-s} (u^{F}_{i_{1}} b^{F}_{j}) \rVert_{\mathcal{C}^{\frac{3}{2} - \delta_{0}}} ds \nonumber \\
\lesssim& \sup_{t \in [0,T]} t^{\frac{ \frac{1}{2} - \delta_{0} + z}{2}} \int_{0}^{t} (t-s)^{- \frac{1}{2}} \lVert u^{F}_{i_{1}} \rVert_{\mathcal{C}^{\frac{1}{2} - \delta_{0}}} \lVert b^{F}_{j} \rVert_{\mathcal{C}^{\frac{1}{2} - \delta_{0}}} ds \nonumber\\
\lesssim& (\sup_{t \in [0, T]} t^{ \frac{\frac{1}{2} - \delta_{0} + z}{2}} \lVert y^{F}(t) \rVert_{\mathcal{C}^{\frac{1}{2} - \delta_{0}}} )^{2} T^{\frac{ \frac{1}{2} + \delta_{0} -z}{2}} \lesssim 1 \label{42}
\end{align}
by Lemma \ref{Lemma 2.6} and Lemma \ref{Lemma 1.1} (4). Similar computations on other terms in $I_{T}^{2}$ of (\ref{40}) show that for all $\epsilon \in (0,1)$ fixed, there exists a maximal existence time $T_{\epsilon } > 0$ and $(u^{F}, b^{F}) \in C([0, T_{\epsilon}); \mathcal{C}^{ \frac{1}{2} - \delta_{0}})$ such that $(u^{F}, b^{F})$ satisfies (\ref{19})-(\ref{21}) and 
\begin{equation}\label{43}
\sup_{t \in [0, T_{\epsilon})} t^{ \frac{ \frac{1}{2} - \delta_{0} + z}{2}} \lVert y^{F}(t) \rVert_{\mathcal{C}^{\frac{1}{2} - \delta_{0}}} = + \infty. 
\end{equation} 
Now we set 
\begin{equation}\label{44} 
\frac{\delta}{2} < \beta < z + 2 \delta - \frac{1}{2}< \frac{1}{2} - 2 \delta
\end{equation} 
and realize that in the computation of (\ref{42}), we could have instead estimated
\begin{align}\label{45} 
& t^{ \frac{ \frac{1}{2} + \beta + z}{2}} \int_{0}^{t} \lVert P_{t-s} (u^{F}_{i_{1}} b^{F}_{j}) \rVert_{\mathcal{C}^{\frac{3}{2} + \beta}} ds 
\lesssim t^{\frac{ \frac{1}{2} + \beta + z}{2}} \int_{0}^{t}(t-s)^{ - ( \frac{1+\beta + \delta_{0}}{2})} \lVert u^{F}_{i_{1}} \rVert_{\mathcal{C}^{\frac{1}{2} - \delta_{0}}} \lVert b^{F}_{j} \rVert_{\mathcal{C}^{\frac{1}{2} - \delta_{0}}} ds \nonumber\\
& \hspace{50mm} \lesssim t^{ \frac{ \frac{1}{2} + \delta_{0} -z}{2}} \left( \sup_{s \in [0,t]} s^{\frac{ \frac{1}{2} - \delta_{0} + z}{2}} \lVert y^{F}(s) \rVert_{\mathcal{C}^{\frac{1}{2} - \delta_{0}}} \right)^{2} 
\end{align} 
by Lemma \ref{Lemma 2.6}, (\ref{44}), (\ref{35}) and Lemma \ref{Lemma 1.1} (4). Thus, similar computations on other terms in $I_{T}^{1}$ and $I_{T}^{2}$ of (\ref{40}) lead to 
\begin{align}\label{46}
t^{ \frac{ \frac{1}{2} + \beta + z}{2}} \lVert y^{F}(t) \rVert_{\mathcal{C}^{\frac{1}{2} + \beta}} \lesssim& C (\epsilon,  \lVert y^{\text{in}} \rVert_{\mathcal{C}^{-z}}, y^{\scalebox{0.6}{\begin{tikzpicture}
\draw[black, thick] (0,0.5) -- (0,0);
\filldraw[black] (0,0.5) circle (2pt); 
\end{tikzpicture}
}}, y^{
\scalebox{0.18}{\begin{tikzpicture}
\draw[black, thick] (-0.7,0.9) -- (0,0);
\draw[black, thick] (0.7,0.9) -- (0,0);
\draw[snake=zigzag](0,0) -- (0,-0.9);
\filldraw[black] (-0.7,0.9) circle (6pt); 
\filldraw[black] (0.7,0.9) circle (6pt); 
\end{tikzpicture}
}}, y^{
\scalebox{0.16}{\begin{tikzpicture}
\draw[black, thick] (-0.7,0.9) -- (0,0);
\draw[black, thick] (0.7,0.9) -- (0,0);
\draw[black, thick] (0.7,0) -- (0,-0.9);
\draw[snake=zigzag](0,0) -- (0,-0.9);
\draw[snake=zigzag](0,-0.9) -- (0,-1.8);
\filldraw[black] (-0.7,0.9) circle (7pt); 
\filldraw[black] (0.7,0.9) circle (7pt); 
\filldraw[black] (0.7,0) circle (7pt); 
\end{tikzpicture}
}}) \nonumber\\
& \hspace{5mm} + t^{ \frac{ \frac{1}{2} + \delta_{0} - z}{2}} \left( \sup_{s \in [0,t]} s^{ \frac{ \frac{1}{2} - \delta_{0} + z}{2}} \lVert y^{F}(s) \rVert_{\mathcal{C}^{\frac{1}{2} - \delta_{0}}} \right)^{2} 
\end{align} 
for all $t \in (0, T_{\epsilon})$. This shows that $(u^{\sharp}_{i}, b^{\sharp}_{i} ) (t) \in \mathcal{C}^{ \frac{1}{2} + \beta}$ for all $t \in (0, T_{\epsilon})$ due to (\ref{43}). This leads us to the next estimate of 
\begin{align}\label{47}
& \lVert u^{F}_{i} \rVert_{\mathcal{C}^{\frac{1}{2} - \delta}} + \lVert b^{F}_{i} \rVert_{\mathcal{C}^{\frac{1}{2} - \delta}} \nonumber \\
\lesssim& \sum_{i_{1}, j_{1}=1}^{3} \lVert \mathcal{P}_{i i_{1}}  \partial_{x_{j_{1}}}  [ \pi_{<} (u^{
\scalebox{0.16}{\begin{tikzpicture}
\draw[black, thick] (-0.7,0.9) -- (0,0);
\draw[black, thick] (0.7,0.9) -- (0,0);
\draw[black, thick] (0.7,0) -- (0,-0.9);
\draw[snake=zigzag](0,0) -- (0,-0.9);
\draw[snake=zigzag](0,-0.9) -- (0,-1.8);
\filldraw[pink] (-0.7,0.9) circle (7pt); 
\filldraw[pink] (0.7,0.9) circle (7pt); 
\filldraw[pink] (0.7,0) circle (7pt); 
\end{tikzpicture}
}}_{i_{1}} + u^{F}_{i_{1}}, K_{j_{1}}^{u}) + \pi_{<} ( u^{
\scalebox{0.16}{\begin{tikzpicture}
\draw[black, thick] (-0.7,0.9) -- (0,0);
\draw[black, thick] (0.7,0.9) -- (0,0);
\draw[black, thick] (0.7,0) -- (0,-0.9);
\draw[snake=zigzag](0,0) -- (0,-0.9);
\draw[snake=zigzag](0,-0.9) -- (0,-1.8);
\filldraw[pink] (-0.7,0.9) circle (7pt); 
\filldraw[pink] (0.7,0.9) circle (7pt); 
\filldraw[pink] (0.7,0) circle (7pt); 
\end{tikzpicture}
}}_{j_{1}} + u^{F}_{j_{1}}, K_{i_{1}}^{u})] \rVert_{\mathcal{C}^{\frac{1}{2} - \delta}}\nonumber \\
& + \lVert \mathcal{P}_{ ii_{1}}  \partial_{x_{j_{1}}}  [ \pi_{<} ( b^{
\scalebox{0.16}{\begin{tikzpicture}
\draw[black, thick] (-0.7,0.9) -- (0,0);
\draw[black, thick] (0.7,0.9) -- (0,0);
\draw[black, thick] (0.7,0) -- (0,-0.9);
\draw[snake=zigzag](0,0) -- (0,-0.9);
\draw[snake=zigzag](0,-0.9) -- (0,-1.8);
\filldraw[gray] (-0.7,0.9) circle (7pt); 
\filldraw[gray] (0.7,0.9) circle (7pt); 
\filldraw[gray] (0.7,0) circle (7pt); 
\end{tikzpicture}
}}_{i_{1}} + b^{F}_{i_{1}}, K_{j_{1}}^{b}) + \pi_{<} ( b^{
\scalebox{0.16}{\begin{tikzpicture}
\draw[black, thick] (-0.7,0.9) -- (0,0);
\draw[black, thick] (0.7,0.9) -- (0,0);
\draw[black, thick] (0.7,0) -- (0,-0.9);
\draw[snake=zigzag](0,0) -- (0,-0.9);
\draw[snake=zigzag](0,-0.9) -- (0,-1.8);
\filldraw[gray] (-0.7,0.9) circle (7pt); 
\filldraw[gray] (0.7,0.9) circle (7pt); 
\filldraw[gray] (0.7,0) circle (7pt); 
\end{tikzpicture}
}}_{j_{1}} + b^{F}_{j_{1}}, K_{i_{1}}^{b}) ] \rVert_{\mathcal{C}^{\frac{1}{2} - \delta}} + \lVert u^{\sharp}_{i} \rVert_{\mathcal{C}^{\frac{1}{2} - \delta}} \nonumber\\
&+ \lVert  \mathcal{P}_{ ii_{1}}  \partial_{x_{j_{1}}}  [- \pi_{<} ( u^{
\scalebox{0.16}{\begin{tikzpicture}
\draw[black, thick] (-0.7,0.9) -- (0,0);
\draw[black, thick] (0.7,0.9) -- (0,0);
\draw[black, thick] (0.7,0) -- (0,-0.9);
\draw[snake=zigzag](0,0) -- (0,-0.9);
\draw[snake=zigzag](0,-0.9) -- (0,-1.8);
\filldraw[pink] (-0.7,0.9) circle (7pt); 
\filldraw[pink] (0.7,0.9) circle (7pt); 
\filldraw[pink] (0.7,0) circle (7pt); 
\end{tikzpicture}
}}_{i_{1}} + u^{F}_{i_{1}}, K_{j_{1}}^{b}) + \pi_{<} ( u^{
\scalebox{0.16}{\begin{tikzpicture}
\draw[black, thick] (-0.7,0.9) -- (0,0);
\draw[black, thick] (0.7,0.9) -- (0,0);
\draw[black, thick] (0.7,0) -- (0,-0.9);
\draw[snake=zigzag](0,0) -- (0,-0.9);
\draw[snake=zigzag](0,-0.9) -- (0,-1.8);
\filldraw[pink] (-0.7,0.9) circle (7pt); 
\filldraw[pink] (0.7,0.9) circle (7pt); 
\filldraw[pink] (0.7,0) circle (7pt); 
\end{tikzpicture}
}}_{j_{1}} + u^{F}_{j_{1}}, K_{i_{1}}^{b})] \rVert_{\mathcal{C}^{\frac{1}{2} - \delta}} \nonumber\\
&+ \lVert \mathcal{P}_{ii_{1}}  \partial_{x_{j_{1}}}  [ \pi_{<} (b^{
\scalebox{0.16}{\begin{tikzpicture}
\draw[black, thick] (-0.7,0.9) -- (0,0);
\draw[black, thick] (0.7,0.9) -- (0,0);
\draw[black, thick] (0.7,0) -- (0,-0.9);
\draw[snake=zigzag](0,0) -- (0,-0.9);
\draw[snake=zigzag](0,-0.9) -- (0,-1.8);
\filldraw[gray] (-0.7,0.9) circle (7pt); 
\filldraw[gray] (0.7,0.9) circle (7pt); 
\filldraw[gray] (0.7,0) circle (7pt); 
\end{tikzpicture}
}}_{i_{1}} + b^{F}_{i_{1}}, K_{j_{1}}^{u}) - \pi_{<} ( b^{
\scalebox{0.16}{\begin{tikzpicture}
\draw[black, thick] (-0.7,0.9) -- (0,0);
\draw[black, thick] (0.7,0.9) -- (0,0);
\draw[black, thick] (0.7,0) -- (0,-0.9);
\draw[snake=zigzag](0,0) -- (0,-0.9);
\draw[snake=zigzag](0,-0.9) -- (0,-1.8);
\filldraw[gray] (-0.7,0.9) circle (7pt); 
\filldraw[gray] (0.7,0.9) circle (7pt); 
\filldraw[gray] (0.7,0) circle (7pt); 
\end{tikzpicture}
}}_{j_{1}} + b^{F}_{j_{1}}, K_{i_{1}}^{u})] \rVert_{\mathcal{C}^{\frac{1}{2}  -\delta}}  + \lVert b^{\sharp}_{i} \rVert_{\mathcal{C}^{\frac{1}{2} - \delta}}
\end{align} 
by the paracontrolled ansatz (\ref{29}) and (\ref{32}). First, we may estimate 
\begin{subequations}\label{48} 
\begin{align}
&  \lVert \mathcal{P}_{ii_{1}}  \partial_{x_{j_{1}}}  [ \pi_{<}(u^{
\scalebox{0.16}{\begin{tikzpicture}
\draw[black, thick] (-0.7,0.9) -- (0,0);
\draw[black, thick] (0.7,0.9) -- (0,0);
\draw[black, thick] (0.7,0) -- (0,-0.9);
\draw[snake=zigzag](0,0) -- (0,-0.9);
\draw[snake=zigzag](0,-0.9) -- (0,-1.8);
\filldraw[pink] (-0.7,0.9) circle (7pt); 
\filldraw[pink] (0.7,0.9) circle (7pt); 
\filldraw[pink] (0.7,0) circle (7pt); 
\end{tikzpicture}
}}_{i_{1}} + u^{F}_{i_{1}}, K_{j_{1}}^{u}) + \pi_{<} ( u^{
\scalebox{0.16}{\begin{tikzpicture}
\draw[black, thick] (-0.7,0.9) -- (0,0);
\draw[black, thick] (0.7,0.9) -- (0,0);
\draw[black, thick] (0.7,0) -- (0,-0.9);
\draw[snake=zigzag](0,0) -- (0,-0.9);
\draw[snake=zigzag](0,-0.9) -- (0,-1.8);
\filldraw[pink] (-0.7,0.9) circle (7pt); 
\filldraw[pink] (0.7,0.9) circle (7pt); 
\filldraw[pink] (0.7,0) circle (7pt); 
\end{tikzpicture}
}}_{j_{1}} + u^{F}_{j_{1}}, K_{i_{1}}^{u})] \rVert_{\mathcal{C}^{\frac{1}{2}  -\delta}} \nonumber \\
\lesssim&  \lVert u^{ 
\scalebox{0.16}{\begin{tikzpicture}
\draw[black, thick] (-0.7,0.9) -- (0,0);
\draw[black, thick] (0.7,0.9) -- (0,0);
\draw[black, thick] (0.7,0) -- (0,-0.9);
\draw[snake=zigzag](0,0) -- (0,-0.9);
\draw[snake=zigzag](0,-0.9) -- (0,-1.8);
\filldraw[pink] (-0.7,0.9) circle (7pt); 
\filldraw[pink] (0.7,0.9) circle (7pt); 
\filldraw[pink] (0.7,0) circle (7pt); 
\end{tikzpicture}
}}_{i_{1}} + u^{F}_{i_{1}} \rVert_{\mathcal{C}^{\frac{1}{2} - \delta_{0}}} \lVert K_{j_{1}}^{u} \rVert_{\mathcal{C}^{\frac{3}{2} - \delta}} + \lVert u^{
\scalebox{0.16}{\begin{tikzpicture}
\draw[black, thick] (-0.7,0.9) -- (0,0);
\draw[black, thick] (0.7,0.9) -- (0,0);
\draw[black, thick] (0.7,0) -- (0,-0.9);
\draw[snake=zigzag](0,0) -- (0,-0.9);
\draw[snake=zigzag](0,-0.9) -- (0,-1.8);
\filldraw[pink] (-0.7,0.9) circle (7pt); 
\filldraw[pink] (0.7,0.9) circle (7pt); 
\filldraw[pink] (0.7,0) circle (7pt); 
\end{tikzpicture}
}}_{j_{1}} + u^{F}_{j_{1}} \rVert_{\mathcal{C}^{\frac{1}{2} - \delta_{0}}}\lVert K_{i_{1}}^{u} \rVert_{\mathcal{C}^{\frac{3}{2} - \delta}}, \\
& \lVert \mathcal{P}_{i i_{1}}  \partial_{x_{j_{1}}}  [ \pi_{<} ( b^{
\scalebox{0.16}{\begin{tikzpicture}
\draw[black, thick] (-0.7,0.9) -- (0,0);
\draw[black, thick] (0.7,0.9) -- (0,0);
\draw[black, thick] (0.7,0) -- (0,-0.9);
\draw[snake=zigzag](0,0) -- (0,-0.9);
\draw[snake=zigzag](0,-0.9) -- (0,-1.8);
\filldraw[gray] (-0.7,0.9) circle (7pt); 
\filldraw[gray] (0.7,0.9) circle (7pt); 
\filldraw[gray] (0.7,0) circle (7pt); 
\end{tikzpicture}
}}_{i_{1}} + b^{F}_{i_{1}}, K_{j_{1}}^{b}) + \pi_{<} ( b^{
\scalebox{0.16}{\begin{tikzpicture}
\draw[black, thick] (-0.7,0.9) -- (0,0);
\draw[black, thick] (0.7,0.9) -- (0,0);
\draw[black, thick] (0.7,0) -- (0,-0.9);
\draw[snake=zigzag](0,0) -- (0,-0.9);
\draw[snake=zigzag](0,-0.9) -- (0,-1.8);
\filldraw[gray] (-0.7,0.9) circle (7pt); 
\filldraw[gray] (0.7,0.9) circle (7pt); 
\filldraw[gray] (0.7,0) circle (7pt); 
\end{tikzpicture}
}}_{j_{1}} + b^{F}_{j_{1}}, K_{i_{1}}^{b}) ] \rVert_{\mathcal{C}^{\frac{1}{2}  -\delta}} \nonumber \\
\lesssim& \lVert b^{
\scalebox{0.16}{\begin{tikzpicture}
\draw[black, thick] (-0.7,0.9) -- (0,0);
\draw[black, thick] (0.7,0.9) -- (0,0);
\draw[black, thick] (0.7,0) -- (0,-0.9);
\draw[snake=zigzag](0,0) -- (0,-0.9);
\draw[snake=zigzag](0,-0.9) -- (0,-1.8);
\filldraw[gray] (-0.7,0.9) circle (7pt); 
\filldraw[gray] (0.7,0.9) circle (7pt); 
\filldraw[gray] (0.7,0) circle (7pt); 
\end{tikzpicture}
}}_{i_{1}} + b^{F}_{i_{1}} \rVert_{\mathcal{C}^{\frac{1}{2} - \delta_{0}}} \lVert K_{j_{1}}^{b} \rVert_{\mathcal{C}^{\frac{3}{2} - \delta}} + \lVert b^{
\scalebox{0.16}{\begin{tikzpicture}
\draw[black, thick] (-0.7,0.9) -- (0,0);
\draw[black, thick] (0.7,0.9) -- (0,0);
\draw[black, thick] (0.7,0) -- (0,-0.9);
\draw[snake=zigzag](0,0) -- (0,-0.9);
\draw[snake=zigzag](0,-0.9) -- (0,-1.8);
\filldraw[gray] (-0.7,0.9) circle (7pt); 
\filldraw[gray] (0.7,0.9) circle (7pt); 
\filldraw[gray] (0.7,0) circle (7pt); 
\end{tikzpicture}
}}_{j_{1}} + b^{F}_{j_{1}} \rVert_{\mathcal{C}^{\frac{1}{2} - \delta_{0}}}\lVert K_{i_{1}}^{b} \rVert_{\mathcal{C}^{\frac{3}{2} - \delta}}, 
\end{align}
\end{subequations} 
by Lemma \ref{Lemma 2.7}, Lemma \ref{Lemma 1.1} (1), and (\ref{2}). Similar estimates may be deduced for 
\begin{align*}
&\lVert  \mathcal{P}_{ ii_{1}}  \partial_{x_{j_{1}}}  [- \pi_{<} ( u^{
\scalebox{0.16}{\begin{tikzpicture}
\draw[black, thick] (-0.7,0.9) -- (0,0);
\draw[black, thick] (0.7,0.9) -- (0,0);
\draw[black, thick] (0.7,0) -- (0,-0.9);
\draw[snake=zigzag](0,0) -- (0,-0.9);
\draw[snake=zigzag](0,-0.9) -- (0,-1.8);
\filldraw[pink] (-0.7,0.9) circle (7pt); 
\filldraw[pink] (0.7,0.9) circle (7pt); 
\filldraw[pink] (0.7,0) circle (7pt); 
\end{tikzpicture}
}}_{i_{1}} + u^{F}_{i_{1}}, K_{j_{1}}^{b}) + \pi_{<} ( u^{
\scalebox{0.16}{\begin{tikzpicture}
\draw[black, thick] (-0.7,0.9) -- (0,0);
\draw[black, thick] (0.7,0.9) -- (0,0);
\draw[black, thick] (0.7,0) -- (0,-0.9);
\draw[snake=zigzag](0,0) -- (0,-0.9);
\draw[snake=zigzag](0,-0.9) -- (0,-1.8);
\filldraw[pink] (-0.7,0.9) circle (7pt); 
\filldraw[pink] (0.7,0.9) circle (7pt); 
\filldraw[pink] (0.7,0) circle (7pt); 
\end{tikzpicture}
}}_{j_{1}} + u^{F}_{j_{1}}, K_{i_{1}}^{b})] \rVert_{\mathcal{C}^{\frac{1}{2} - \delta}},\\
&\lVert \mathcal{P}_{ii_{1}}  \partial_{x_{j_{1}}}  [ \pi_{<} (b^{
\scalebox{0.16}{\begin{tikzpicture}
\draw[black, thick] (-0.7,0.9) -- (0,0);
\draw[black, thick] (0.7,0.9) -- (0,0);
\draw[black, thick] (0.7,0) -- (0,-0.9);
\draw[snake=zigzag](0,0) -- (0,-0.9);
\draw[snake=zigzag](0,-0.9) -- (0,-1.8);
\filldraw[gray] (-0.7,0.9) circle (7pt); 
\filldraw[gray] (0.7,0.9) circle (7pt); 
\filldraw[gray] (0.7,0) circle (7pt); 
\end{tikzpicture}
}}_{i_{1}} + b^{F}_{i_{1}}, K_{j_{1}}^{u}) - \pi_{<} ( b^{
\scalebox{0.16}{\begin{tikzpicture}
\draw[black, thick] (-0.7,0.9) -- (0,0);
\draw[black, thick] (0.7,0.9) -- (0,0);
\draw[black, thick] (0.7,0) -- (0,-0.9);
\draw[snake=zigzag](0,0) -- (0,-0.9);
\draw[snake=zigzag](0,-0.9) -- (0,-1.8);
\filldraw[gray] (-0.7,0.9) circle (7pt); 
\filldraw[gray] (0.7,0.9) circle (7pt); 
\filldraw[gray] (0.7,0) circle (7pt); 
\end{tikzpicture}
}}_{j_{1}} + b^{F}_{j_{1}}, K_{i_{1}}^{u})] \rVert_{\mathcal{C}^{\frac{1}{2}  -\delta}}. 
\end{align*}
Moreover, we have $\mathcal{C}^{\frac{1}{2} + \beta} \hookrightarrow \mathcal{C}^{\frac{1}{2} - \delta}$ by (\ref{44}). Therefore, we obtain 
\begin{align}\label{49} 
& \lVert u^{F}_{i} \rVert_{\mathcal{C}^{\frac{1}{2} - \delta}} + \lVert b^{F}_{i} \rVert_{\mathcal{C}^{\frac{1}{2} - \delta}} \lesssim \lVert (u^{\sharp}_{i}, b^{\sharp}_{i} ) \rVert_{\mathcal{C}^{\frac{1}{2} + \beta}} \\
& \hspace{5mm} \sum_{i_{1}, j_{1}=1}^{3} \lVert (u^{
\scalebox{0.16}{

}}_{j_{1}} + b^{F}_{j_{1}} ) \rVert_{\mathcal{C}^{\frac{1}{2} - \delta_{0}}} \lVert (K_{i_{1}}^{u}, K_{j_{1}}^{b}) \rVert_{\mathcal{C}^{\frac{3}{2} - \delta}}. \nonumber 
\end{align} 
Now we obtain from (\ref{29}) 
\begin{align*} 
L u^{\sharp}_{i} 
=& - \frac{1}{2} \sum_{i_{1}, j=1}^{3} \mathcal{P}_{ii_{1}} \partial_{x_{j}} [ \pi_{<}(u^{
\scalebox{0.16}{%
}}_{j} + b^{F}_{j}), \nabla K_{i_{1}}^{b}) ]) 
\end{align*}
where we used (\ref{26}), that $L = \partial_{t} - \Delta$,  (\ref{20}), (\ref{23a})-(\ref{23d}). We make a crucial observation that we can cancel out 
\begin{align*}
\pi_{<} ( u^{
\scalebox{0.16}{%
}}_{j} + b^{F}_{j}), \nabla K_{i_{1}}^{b}) ]) \triangleq \phi_{i}^{\sharp, u}. 
\end{align}
Similarly we can compute 
\begin{align*} 
L b^{\sharp}_{i} 
=& -\frac{1}{2} \sum_{i_{1}, j=1}^{3} \mathcal{P}_{ii_{1}} \partial_{x_{j}} [ \pi_{<} ( u^{
\scalebox{0.16}{%
}}_{j} + b^{F}_{j}), \nabla K_{i_{1}}^{u})] 
\end{align*}
by (\ref{32}), that $L = \partial_{t} - \Delta$, (\ref{26}), (\ref{21}), (\ref{23e})-(\ref{23h}). Again we cancel out 
\begin{align*}
\pi_{<} ( u^{
\scalebox{0.16}{%
}}_{j} + b^{F}_{j}), \nabla K_{i_{1}}^{u})] \triangleq \phi_{i}^{\sharp, b}. 
\end{align} 
\subsection{Renormalizations}
In contrast to the NSE, we not only have to define $\pi_{0}(u^{F}_{i}, u^{\scalebox{0.6}{%
}}_{j} ) \nonumber
\end{align} 
by (\ref{29}) and Leibniz rule. Similarly, 
\begin{align}\label{53} 
\pi_{0} (b^{F}_{i}, b^{
\scalebox{0.6}{%
}}_{j})\nonumber 
\end{align} 
by (\ref{32}) and Leibniz rule. We can define $\pi_{0} (u^{F}_{i}, b^{
\scalebox{0.6}{%
}}_{j} )$ similarly. We only consider the first four terms in $\pi_{0} (u^{F}_{i}, u^{\scalebox{0.6}{%
}}_{j})$ of (\ref{53}) as other terms are similar. For the first term in $\pi_{0}(u^{F}_{i}, u^{\scalebox{0.6}{%
}}_{j} ), 
\end{align} 
for the second term in $\pi_{0} (u^{F}_{i}, u^{\scalebox{0.6}{%
}}_{j} ),
\end{align} 
for the third term in $\pi_{0} (u^{F}_{i}, u^{\scalebox{0.6}{%
}}_{j} ), 
\end{align} 
and for the fourth term in $\pi_{0} (u^{F}_{i}, u^{\scalebox{0.6}{%
}}_{j} ). 
\end{align} 
Similarly we can write the first four terms of $\pi_{0}(b^{F}_{i}, b^{
\scalebox{0.6}{%
}}_{j})$. For the convergence of $\pi_{0} ( \mathcal{P}_{ii_{1}}  \partial_{x_{j_{1}}}  K_{j_{1}}^{u}, u^{\scalebox{0.6}{%
}}_{j} )$ as $\epsilon \to 0$, we need to do renormalization. We now estimate
\begin{align}\label{58} 
& \lVert \pi_{0, \diamond} ( \mathcal{P}_{ii_{1}} \pi_{<} ( u^{
\scalebox{0.16}{%
}}_{j} ) \rVert_{\mathcal{C}^{-\delta}}
\end{align} 
by (\ref{28b}). For 
\begin{equation}\label{59} 
\delta \leq \delta_{0} < \frac{1}{2} - \frac{3\delta}{2},  
\end{equation}
we may firstly estimate 
\begin{align}\label{60} 
& \lVert \pi_{0} (\mathcal{P}_{ii_{1}} \pi_{<} ( u^{
\scalebox{0.16}{%
}}_{j}  \rVert_{\mathcal{C}^{-\frac{1}{2} - \frac{\delta}{2}}}  
\end{align} 
by linearity of $\pi_{0}(f, \cdot)$, that $- \delta \leq \frac{1}{2} - \frac{3 \delta}{2} - \delta_{0}$ due to (\ref{59}),  (\ref{2}), Lemma \ref{Lemma 1.1} (3) and Lemma \ref{Lemma 2.5}. Second, we may estimate 
\begin{align}\label{61} 
& \lVert \pi_{0} ( \pi_{<} ( u^{
\scalebox{0.16}{%
}}_{j}  \rVert_{\mathcal{C}^{-\frac{1}{2} - \frac{\delta}{2}}} 
\end{align} 
where we used that $- \delta \leq \frac{1}{2} - \frac{3\delta}{2} - \delta_{0}$ due to (\ref{59}), Lemmas \ref{Lemma 2.4} and \ref{Lemma 2.7}. 

\begin{remark}\label{Remark 3.3}
Let us emphasize that this estimate (\ref{61}) seems very difficult, if not impossible, without relying on the commutator estimate Lemma \ref{Lemma 2.4}, e.g.,  by utilizing only Lemma \ref{Lemma 1.1}.
\end{remark}
Third, we also estimate 
\begin{align}\label{62} 
&\lVert ( u^{
\scalebox{0.16}{%
}}_{j}  ) \rVert_{\mathcal{C}^{-\delta}}. 
\end{align} 
Similarly we can deduce 
\begin{align}\label{64} 
& \lVert \pi_{0, \diamond} ( \mathcal{P}_{ii_{1}} \pi_{<} ( b^{
\scalebox{0.16}{%
}}_{j}  ) \rVert_{\mathcal{C}^{-\delta}}, 
\end{align}
as well as 
\begin{subequations}
\begin{align}
& \lVert \pi_{0, \diamond} ( \mathcal{P}_{ii_{1}} \pi_{<} ( u^{
\scalebox{0.16}{%
}}_{j}) \rVert_{\mathcal{C}^{-\delta}}. \label{65b}
\end{align}
\end{subequations} 
This leads to 
\begin{align}\label{66}
\lVert \pi_{0, \diamond}(u^{F}_{i}, u^{\scalebox{0.6}{%
}}_{j} ) \rVert_{\mathcal{C}^{-\delta}}  \nonumber
\end{align} 
by (\ref{27}), (\ref{63}) and (\ref{64}). We may further estimate firstly within (\ref{66}), 
\begin{align}\label{67} 
& \lVert \pi_{0} ( \mathcal{P}_{ii_{1}} \pi_{<}( \partial_{x_{j_{1}}} (u^{
\scalebox{0.16}{%
}}_{j}  \rVert_{\mathcal{C}^{-\frac{1}{2} - \frac{\delta}{2}}} \lesssim C_{\xi}^{3} + ( \lVert u^{F} \rVert_{\mathcal{C}^{\frac{1}{2} - \delta_{0}}} + \lVert b^{F} \rVert_{\mathcal{C}^{\frac{1}{2} - \delta_{0}}}) C_{\xi}^{2} 
\end{align} 
by Lemma \ref{Lemma 1.1} (3) as $\frac{1}{2} - \delta_{0} - \frac{3\delta}{2} > 0$ due to (\ref{59}), Lemma \ref{Lemma 2.7}, Lemma \ref{Lemma 1.1} (2), (\ref{35}), (\ref{2}), (\ref{33}), (\ref{37}) and (\ref{39}). Second, within (\ref{66}) we may estimate 
\begin{equation}\label{68} 
\lVert \pi_{0} ( u^{\sharp}_{i}, u^{\scalebox{0.6}{

}}_{j}  \rVert_{\mathcal{C}^{-\frac{1}{2} - \frac{\delta}{2}}}
\lesssim \lVert u^{\sharp}_{i} \rVert_{\mathcal{C}^{\frac{1}{2} + \beta}} C_{\xi} 
\end{equation} 
as $\beta > \frac{\delta}{2}$ due to (\ref{44}), Lemma \ref{Lemma 1.1} (3) and (\ref{37}). Third, within (\ref{66}) we may estimate 
\begin{align}\label{69} 
& \lVert u^{
\scalebox{0.16}{%
}}_{j} ) \rVert_{\mathcal{C}^{-\delta}}  \lesssim C_{\xi}^{3} + 1 + \lVert y^{F} \rVert_{\mathcal{C}^{\frac{1}{2} - \delta_{0}}}  (C_{\xi}^{2} + 1)  \nonumber 
\end{align} 
by (\ref{59}), (\ref{37}) and (\ref{39}). Therefore, by applying (\ref{67})-(\ref{70}) in (\ref{66}) we obtain 
\begin{equation}\label{71} 
\lVert \pi_{0, \diamond}(u^{F}_{i}, u^{\scalebox{0.6}{%
}}_{j}) \rVert_{\mathcal{C}^{-\delta}} \nonumber
\end{align} 
by (\ref{31}), (\ref{65a}) and (\ref{65b}),  where tracing previous inequalities (\ref{67})-(\ref{70}), we see that 
\begin{subequations}
\begin{align}
&  \lVert \pi_{0} ( \mathcal{P}_{ii_{1}} \pi_{<} ( \partial_{x_{j_{1}}}  (u^{
\scalebox{0.16}{%
}}_{j}) \rVert_{\mathcal{C}^{-\delta}} 
\lesssim C_{\xi}^{3} + 1 +  \lVert y^{F} \rVert_{\mathcal{C}^{\frac{1}{2} - \delta_{0}}} (C_{\xi}^{2} + 1).  \label{76} 
\end{align} 
\end{subequations}
Thus, by applying (\ref{73})-(\ref{76}) to (\ref{72}) we obtain 
\begin{equation}\label{77} 
\lVert \pi_{0, \diamond} (b^{F}_{i}, b^{
\scalebox{0.6}{%
}}_{j}) \rVert_{\mathcal{C}^{-\delta}} \lesssim C_{\xi}^{3} + \lVert y^{F} \rVert_{\mathcal{C}^{\frac{1}{2} - \delta_{0}}}(C_{\xi}^{2} + 1) + \lVert b^{\sharp} \rVert_{\mathcal{C}^{\frac{1}{2} + \beta}} C_{\xi} + 1
\end{equation} 
and similar estimates for $\lVert \pi_{0, \diamond} ( u^{F}_{i}, b^{
\scalebox{0.6}{%
}}_{j} ) \rVert_{\mathcal{C}^{ -\delta}}$ follow. 

Next, by (\ref{18})-(\ref{21}), (\ref{23a})-(\ref{23h}), we see that 
\begin{align}\label{78} 
&\lVert L (u^{
\scalebox{0.16}{%
}}_{j} + b^{F}_{j}) \rVert_{\mathcal{C}^{-\frac{3}{2} - \frac{\delta}{2}}}. \nonumber 
\end{align} 
First, within (\ref{78})-(\ref{79}) we may estimate 
\begin{align}\label{80}
&\lVert \mathcal{P}_{ii_{1}}  \partial_{x_{j}}  ( u^{\scalebox{0.6}{%
}}_{j} ) \rVert_{\mathcal{C}^{-\frac{3}{2} - \frac{\delta}{2}}} \lesssim C_{\xi}  
\end{align} 
by Lemma \ref{Lemma 2.7} and (\ref{37}). Second, within (\ref{78})-(\ref{79}) we may estimate 
\begin{align}\label{81}
& \lVert \mathcal{P}_{ii_{1}}  \partial_{x_{j}}  [ \pi_{<} ( u^{
\scalebox{0.16}{%
}}_{j}+ u^{F}_{j} ) \rVert_{\mathcal{C}^{\frac{1}{2} - \delta_{0}}} \nonumber\\
\lesssim& C_{\xi}^{3} + 1 + (1+ C_{\xi}^{2}) \lVert y^{F} \rVert_{\mathcal{C}^{\frac{1}{2} - \delta_{0}}}
\end{align} 
due to Lemma \ref{Lemma 2.7}, that $-\frac{1}{2} - \frac{\delta}{2} \leq - \frac{\delta}{2} - \delta_{0}$, Lemma \ref{Lemma 1.1} (1), Lemma \ref{Lemma 1.1} (2), (\ref{37}), (\ref{35}) and (\ref{39}). Third, within (\ref{78})-(\ref{79}) we may estimate 
\begin{align}\label{82}
& \lVert \mathcal{P}_{ii_{1}}  \partial_{x_{j}}  [ \pi_{0, \diamond} (u^{
\scalebox{0.16}{\begin{tikzpicture}
\draw[black, thick] (-0.7,0.9) -- (0,0);
\draw[black, thick] (0.7,0.9) -- (0,0);
\draw[black, thick] (0.7,0) -- (0,-0.9);
\draw[snake=zigzag](0,0) -- (0,-0.9);
\draw[snake=zigzag](0,-0.9) -- (0,-1.8);
\filldraw[pink] (-0.7,0.9) circle (7pt); 
\filldraw[pink] (0.7,0.9) circle (7pt); 
\filldraw[pink] (0.7,0) circle (7pt); 
\end{tikzpicture}
}}_{j}, u^{\scalebox{0.6}{\begin{tikzpicture}
\draw[black, thick] (0,0.5) -- (0,0);
\filldraw[red] (0,0.5) circle (2pt); 
\end{tikzpicture}
}}_{i_{1}} ) + \pi_{0, \diamond} (u^{F}_{j}, u^{\scalebox{0.6}{\begin{tikzpicture}
\draw[black, thick] (0,0.5) -- (0,0);
\filldraw[red] (0,0.5) circle (2pt); 
\end{tikzpicture}
}}_{i_{1}} ) + \pi_{0, \diamond}(u^{
\scalebox{0.16}{\begin{tikzpicture}
\draw[black, thick] (-0.7,0.9) -- (0,0);
\draw[black, thick] (0.7,0.9) -- (0,0);
\draw[black, thick] (0.7,0) -- (0,-0.9);
\draw[snake=zigzag](0,0) -- (0,-0.9);
\draw[snake=zigzag](0,-0.9) -- (0,-1.8);
\filldraw[pink] (-0.7,0.9) circle (7pt); 
\filldraw[pink] (0.7,0.9) circle (7pt); 
\filldraw[pink] (0.7,0) circle (7pt); 
\end{tikzpicture}
}}_{i_{1}}, u^{\scalebox{0.6}{\begin{tikzpicture}
\draw[black, thick] (0,0.5) -- (0,0);
\filldraw[red] (0,0.5) circle (2pt); 
\end{tikzpicture}
}}_{j} ) + \pi_{0, \diamond}(u^{F}_{i_{1}}, u^{\scalebox{0.6}{\begin{tikzpicture}
\draw[black, thick] (0,0.5) -- (0,0);
\filldraw[red] (0,0.5) circle (2pt); 
\end{tikzpicture}
}}_{j} ) \nonumber\\
&\hspace{5mm} - \pi_{0, \diamond} (b^{
\scalebox{0.16}{\begin{tikzpicture}
\draw[black, thick] (-0.7,0.9) -- (0,0);
\draw[black, thick] (0.7,0.9) -- (0,0);
\draw[black, thick] (0.7,0) -- (0,-0.9);
\draw[snake=zigzag](0,0) -- (0,-0.9);
\draw[snake=zigzag](0,-0.9) -- (0,-1.8);
\filldraw[gray] (-0.7,0.9) circle (7pt); 
\filldraw[gray] (0.7,0.9) circle (7pt); 
\filldraw[gray] (0.7,0) circle (7pt); 
\end{tikzpicture}
}}_{j}, b^{
\scalebox{0.6}{\begin{tikzpicture}
\draw[black, thick] (0,0.5) -- (0,0);
\filldraw[blue] (0,0.5) circle (2pt); 
\end{tikzpicture}
}}_{i_{1}}) - \pi_{0, \diamond} (b^{F}_{j}, b^{
\scalebox{0.6}{\begin{tikzpicture}
\draw[black, thick] (0,0.5) -- (0,0);
\filldraw[blue] (0,0.5) circle (2pt); 
\end{tikzpicture}
}}_{i_{1}}) - \pi_{0, \diamond}(b^{
\scalebox{0.16}{\begin{tikzpicture}
\draw[black, thick] (-0.7,0.9) -- (0,0);
\draw[black, thick] (0.7,0.9) -- (0,0);
\draw[black, thick] (0.7,0) -- (0,-0.9);
\draw[snake=zigzag](0,0) -- (0,-0.9);
\draw[snake=zigzag](0,-0.9) -- (0,-1.8);
\filldraw[gray] (-0.7,0.9) circle (7pt); 
\filldraw[gray] (0.7,0.9) circle (7pt); 
\filldraw[gray] (0.7,0) circle (7pt); 
\end{tikzpicture}
}}_{i_{1}}, b^{
\scalebox{0.6}{\begin{tikzpicture}
\draw[black, thick] (0,0.5) -- (0,0);
\filldraw[blue] (0,0.5) circle (2pt); 
\end{tikzpicture}
}}_{j}) - \pi_{0, \diamond} (b^{F}_{i_{1}}, b^{
\scalebox{0.6}{\begin{tikzpicture}
\draw[black, thick] (0,0.5) -- (0,0);
\filldraw[blue] (0,0.5) circle (2pt); 
\end{tikzpicture}
}}_{j}) ] \rVert_{\mathcal{C}^{-\frac{3}{2} - \frac{\delta}{2}}}\nonumber\\
& + \lVert \mathcal{P}_{ii_{1}}  \partial_{x_{j}}  [ \pi_{0, \diamond} (u^{
\scalebox{0.16}{\begin{tikzpicture}
\draw[black, thick] (-0.7,0.9) -- (0,0);
\draw[black, thick] (0.7,0.9) -- (0,0);
\draw[black, thick] (0.7,0) -- (0,-0.9);
\draw[snake=zigzag](0,0) -- (0,-0.9);
\draw[snake=zigzag](0,-0.9) -- (0,-1.8);
\filldraw[pink] (-0.7,0.9) circle (7pt); 
\filldraw[pink] (0.7,0.9) circle (7pt); 
\filldraw[pink] (0.7,0) circle (7pt); 
\end{tikzpicture}
}}_{j}, b^{
\scalebox{0.6}{\begin{tikzpicture}
\draw[black, thick] (0,0.5) -- (0,0);
\filldraw[blue] (0,0.5) circle (2pt); 
\end{tikzpicture}
}}_{i_{1}}) + \pi_{0, \diamond} (u^{F}_{j}, b^{
\scalebox{0.6}{\begin{tikzpicture}
\draw[black, thick] (0,0.5) -- (0,0);
\filldraw[blue] (0,0.5) circle (2pt); 
\end{tikzpicture}
}}_{i_{1}}) + \pi_{0, \diamond}(b^{
\scalebox{0.16}{\begin{tikzpicture}
\draw[black, thick] (-0.7,0.9) -- (0,0);
\draw[black, thick] (0.7,0.9) -- (0,0);
\draw[black, thick] (0.7,0) -- (0,-0.9);
\draw[snake=zigzag](0,0) -- (0,-0.9);
\draw[snake=zigzag](0,-0.9) -- (0,-1.8);
\filldraw[gray] (-0.7,0.9) circle (7pt); 
\filldraw[gray] (0.7,0.9) circle (7pt); 
\filldraw[gray] (0.7,0) circle (7pt); 
\end{tikzpicture}
}}_{i_{1}}, u^{\scalebox{0.6}{\begin{tikzpicture}
\draw[black, thick] (0,0.5) -- (0,0);
\filldraw[red] (0,0.5) circle (2pt); 
\end{tikzpicture}
}}_{j} ) + \pi_{0, \diamond}(b^{F}_{i_{1}}, u^{\scalebox{0.6}{\begin{tikzpicture}
\draw[black, thick] (0,0.5) -- (0,0);
\filldraw[red] (0,0.5) circle (2pt); 
\end{tikzpicture}
}}_{j} ) \nonumber\\
&\hspace{5mm} - \pi_{0, \diamond} (b^{
\scalebox{0.16}{\begin{tikzpicture}
\draw[black, thick] (-0.7,0.9) -- (0,0);
\draw[black, thick] (0.7,0.9) -- (0,0);
\draw[black, thick] (0.7,0) -- (0,-0.9);
\draw[snake=zigzag](0,0) -- (0,-0.9);
\draw[snake=zigzag](0,-0.9) -- (0,-1.8);
\filldraw[gray] (-0.7,0.9) circle (7pt); 
\filldraw[gray] (0.7,0.9) circle (7pt); 
\filldraw[gray] (0.7,0) circle (7pt); 
\end{tikzpicture}
}}_{j}, u^{\scalebox{0.6}{\begin{tikzpicture}
\draw[black, thick] (0,0.5) -- (0,0);
\filldraw[red] (0,0.5) circle (2pt); 
\end{tikzpicture}
}}_{i_{1}} ) - \pi_{0, \diamond} (b^{F}_{j}, u^{\scalebox{0.6}{\begin{tikzpicture}
\draw[black, thick] (0,0.5) -- (0,0);
\filldraw[red] (0,0.5) circle (2pt); 
\end{tikzpicture}
}}_{i_{1}} ) - \pi_{0, \diamond}(u^{
\scalebox{0.16}{\begin{tikzpicture}
\draw[black, thick] (-0.7,0.9) -- (0,0);
\draw[black, thick] (0.7,0.9) -- (0,0);
\draw[black, thick] (0.7,0) -- (0,-0.9);
\draw[snake=zigzag](0,0) -- (0,-0.9);
\draw[snake=zigzag](0,-0.9) -- (0,-1.8);
\filldraw[pink] (-0.7,0.9) circle (7pt); 
\filldraw[pink] (0.7,0.9) circle (7pt); 
\filldraw[pink] (0.7,0) circle (7pt); 
\end{tikzpicture}
}}_{i_{1}}, b^{
\scalebox{0.6}{\begin{tikzpicture}
\draw[black, thick] (0,0.5) -- (0,0);
\filldraw[blue] (0,0.5) circle (2pt); 
\end{tikzpicture}
}}_{j}) - \pi_{0, \diamond} (u^{F}_{i_{1}}, b^{
\scalebox{0.6}{\begin{tikzpicture}
\draw[black, thick] (0,0.5) -- (0,0);
\filldraw[blue] (0,0.5) circle (2pt); 
\end{tikzpicture}
}}_{j}) ] \rVert_{\mathcal{C}^{-\frac{3}{2} - \frac{\delta}{2}}}\nonumber\\
\lesssim& C_{\xi}^{3} + 1 + (1+ C_{\xi}^{2}) \lVert y^{F} \rVert_{\mathcal{C}^{\frac{1}{2} - \delta_{0}}} + C_{\xi}\lVert (u^{\sharp}, b^{\sharp}) \rVert_{\mathcal{C}^{\frac{1}{2} + \beta}} 
\end{align} 
by Lemma \ref{Lemma 2.7}, that $- \frac{1}{2}  - \frac{\delta}{2} \leq- \delta$, (\ref{35}), (\ref{77}), (\ref{37}) and (\ref{71}). Fourth, within (\ref{78})-(\ref{79}) we may estimate 
\begin{align}\label{83}
& \lVert \mathcal{P}_{ii_{1}}  \partial_{x_{j}}  (u^{
\scalebox{0.18}{

}}_{j} ) \rVert_{\mathcal{C}^{-\delta}} \lesssim C_{\xi}
\end{align} 
by Lemma \ref{Lemma 2.7}, that $- \delta \geq - \frac{1}{2} - \frac{\delta}{2}$, (\ref{35}) and (\ref{37}). Fifth, within (\ref{78})-(\ref{79}) we may estimate 
\begin{align}\label{84}
& \lVert \mathcal{P}_{ii_{1}}  \partial_{x_{j}}  ( u^{
\scalebox{0.18}{%
}}_{j} + b^{F}_{j}) \rVert_{\mathcal{C}^{\delta}}^{2} \lesssim C_{\xi}^{2} + (C_{\xi}^{2} + 1) \lVert y^{F} \rVert_{\mathcal{C}^{\frac{1}{2} - \delta_{0}}} + \lVert y^{F}\rVert_{\mathcal{C}^{\delta}}^{2} \nonumber 
\end{align} 
where we used Lemma \ref{Lemma 2.7}, that $- \frac{1}{2} - \frac{\delta}{2} \leq- \delta$, (\ref{59}), Lemma \ref{Lemma 1.1} (4) and (\ref{39}).  Applying (\ref{80})-(\ref{84}) to (\ref{78}) and (\ref{79}) shows that 
\begin{align}\label{85}
&\lVert L(u^{
\scalebox{0.16}{%
}}_{i} + b^{F}_{i}) \rVert_{\mathcal{C}^{-\frac{3}{2} - \frac{\delta}{2}}}) \lVert (K_{j}^{u}, K_{j}^{b}) \rVert_{\mathcal{C}^{\frac{3}{2} - \delta}}\nonumber\\
\lesssim& [C_{\xi}^{3} + 1+ (1+ C_{\xi}^{2}) \lVert y^{F} \rVert_{\mathcal{C}^{\frac{1}{2} - \delta_{0}}} + C_{\xi} \lVert (u^{\sharp}, b^{\sharp}) \rVert_{\mathcal{C}^{\frac{1}{2} + \beta}} + \lVert y^{F} \rVert_{\mathcal{C}^{\delta}}^{2}] \lVert (K_{j}^{u}, K_{j}^{b}) \rVert_{\mathcal{C}^{\frac{3}{2} - \delta}}  
\end{align} 
by Lemma \ref{Lemma 1.1} (2) and (\ref{85}). Next, we estimate 
\begin{align}\label{87}
& \lVert ( \pi_{<} ( \nabla (u^{
\scalebox{0.16}{

}}_{j_{1}} + b^{F}_{j_{1}} ) \rVert_{\mathcal{C}^{\frac{1}{2} - \delta_{0}}} \lVert (K^{u}, K^{b}) \rVert_{\mathcal{C}^{\frac{3}{2} - \delta}}) \nonumber
\end{align} 
by Lemma \ref{Lemma 1.1} (2), (\ref{33}), (\ref{37}) and (\ref{49}). 
\subsection{Estimates of $\phi^{\sharp,u}$ and $\phi^{\sharp, b}$}
We have 
\begin{align}\label{88}
& \lVert \phi_{i}^{\sharp, u} \rVert_{\mathcal{C}^{-1 - 2 \delta}} \\
=& \lVert - \frac{1}{2} \sum_{i_{1}, j=1}^{3} \mathcal{P}_{ii_{1}}  \partial_{x_{j}}  [ \pi_{>} (u^{
\scalebox{0.16}{%
}}_{j} + b^{F}_{j}), \nabla K_{i_{1}}^{u})] \rVert_{\mathcal{C}^{-1-2\delta}} \nonumber
\end{align} 
by (\ref{51}). First, we may bound within (\ref{88})-(\ref{89}), 
\begin{align}\label{90}
& \lVert \mathcal{P}_{ii_{1}}  \partial_{x_{j}}  [ u^{
\scalebox{0.18}{%
}}_{j} + b^{F}_{j}) \rVert_{\mathcal{C}^{\delta}} \nonumber\\
\lesssim& (1+ C_{\xi}^{4}) [ 1 + \lVert y^{F} \rVert_{\mathcal{C}^{\frac{1}{2} - \delta_{0}}} + \lVert y^{F} \rVert_{\mathcal{C}^{\delta}}^{2}] 
\end{align} 
by Lemma \ref{Lemma 2.7}, that $- 2 \delta \leq - \delta$, Lemma \ref{Lemma 1.1} (4), (\ref{59}) and (\ref{39}). Similar computations show that 
\begin{align}\label{91}
& \lVert \mathcal{P}_{ii_{1}}  \partial_{x_{j}}  [ b^{
\scalebox{0.18}{%
}}_{j} + b^{F}_{j})] \rVert_{\mathcal{C}^{-1 - 2\delta}} \nonumber\\
\lesssim& ( 1+ C_{\xi}^{4}) [ 1+ \lVert y^{F} \rVert_{\mathcal{C}^{\frac{1}{2} - \delta_{0}}} + \lVert y^{F} \rVert_{\mathcal{C}^{\delta}}^{2}]. 
\end{align} 
Second, we bound within (\ref{88})-(\ref{89})
\begin{equation}\label{92}
 \lVert \mathcal{P}_{ii_{1}}  \partial_{x_{j}}  (u^{
\scalebox{0.18}{%
}}_{j} ) \rVert_{\mathcal{C}^{-2\delta}} \lesssim C_{\xi}. 
\end{equation}
Third, we bound within (\ref{88})-(\ref{89})
\begin{align}\label{94}
& \lVert \mathcal{P}_{ii_{1}}  \partial_{x_{j}}  [( \pi_{>}(u^{
\scalebox{0.16}{%
}}_{j_{1}}+ b^{F}_{j_{1}}) \rVert_{\mathcal{C}^{\frac{1}{2} - \delta_{0}}} \lVert (K^{u}, K^{b}) \rVert_{\mathcal{C}^{\frac{3}{2} - \delta}})\nonumber\\
&+ [ C_{\xi}^{3} + 1 + (1+ C_{\xi}^{2}) \lVert y^{F} \rVert_{\mathcal{C}^{\frac{1}{2} - \delta_{0}}} + C_{\xi} \lVert (u^{\sharp}, b^{\sharp}) \rVert_{\mathcal{C}^{\frac{1}{2} + \beta}} + \lVert y^{F} \rVert_{\mathcal{C}^{\delta}}^{2}]  \times \lVert (K^{u}, K^{b}) \rVert_{\mathcal{C}^{\frac{3}{2} - \delta}}\nonumber\\
\lesssim& (1+ C_{\xi}^{4}) [ 1+ \lVert (u^{\sharp}, b^{\sharp}) \rVert_{\mathcal{C}^{\frac{1}{2} + \beta}} + \lVert y^{F} \rVert_{\mathcal{C}^{\frac{1}{2} - \delta_{0}}} + \lVert y^{F} \rVert_{\mathcal{C}^{\delta}}^{2}] 
\end{align} 
by Lemma \ref{Lemma 2.7}, (\ref{2}), (\ref{87}), (\ref{86}), (\ref{33}), (\ref{37}) and (\ref{39}). Similarly we bound 
\begin{align}\label{95}
& \lVert \mathcal{P}_{ii_{1}} \partial_{x_{j}}  [ (\pi_{>} ( u^{
\scalebox{0.16}{\begin{tikzpicture}
\draw[black, thick] (-0.7,0.9) -- (0,0);
\draw[black, thick] (0.7,0.9) -- (0,0);
\draw[black, thick] (0.7,0) -- (0,-0.9);
\draw[snake=zigzag](0,0) -- (0,-0.9);
\draw[snake=zigzag](0,-0.9) -- (0,-1.8);
\filldraw[pink] (-0.7,0.9) circle (7pt); 
\filldraw[pink] (0.7,0.9) circle (7pt); 
\filldraw[pink] (0.7,0) circle (7pt); 
\end{tikzpicture}
}}_{j} + u^{F}_{j}, b^{
\scalebox{0.6}{\begin{tikzpicture}
\draw[black, thick] (0,0.5) -- (0,0);
\filldraw[blue] (0,0.5) circle (2pt); 
\end{tikzpicture}
}}_{i_{1}}), \pi_{>}(b^{
\scalebox{0.16}{\begin{tikzpicture}
\draw[black, thick] (-0.7,0.9) -- (0,0);
\draw[black, thick] (0.7,0.9) -- (0,0);
\draw[black, thick] (0.7,0) -- (0,-0.9);
\draw[snake=zigzag](0,0) -- (0,-0.9);
\draw[snake=zigzag](0,-0.9) -- (0,-1.8);
\filldraw[gray] (-0.7,0.9) circle (7pt); 
\filldraw[gray] (0.7,0.9) circle (7pt); 
\filldraw[gray] (0.7,0) circle (7pt); 
\end{tikzpicture}
}}_{i_{1}} + b^{F}_{i_{1}}, u^{\scalebox{0.6}{\begin{tikzpicture}
\draw[black, thick] (0,0.5) -- (0,0);
\filldraw[red] (0,0.5) circle (2pt); 
\end{tikzpicture}
}}_{j} ), \pi_{>}(b^{
\scalebox{0.16}{\begin{tikzpicture}
\draw[black, thick] (-0.7,0.9) -- (0,0);
\draw[black, thick] (0.7,0.9) -- (0,0);
\draw[black, thick] (0.7,0) -- (0,-0.9);
\draw[snake=zigzag](0,0) -- (0,-0.9);
\draw[snake=zigzag](0,-0.9) -- (0,-1.8);
\filldraw[gray] (-0.7,0.9) circle (7pt); 
\filldraw[gray] (0.7,0.9) circle (7pt); 
\filldraw[gray] (0.7,0) circle (7pt); 
\end{tikzpicture}
}}_{j} + b^{F}_{j}, u^{\scalebox{0.6}{\begin{tikzpicture}
\draw[black, thick] (0,0.5) -- (0,0);
\filldraw[red] (0,0.5) circle (2pt); 
\end{tikzpicture}
}}_{i_{1}} ), \nonumber\\
& \hspace{5mm} \pi_{>}(u^{
\scalebox{0.16}{\begin{tikzpicture}
\draw[black, thick] (-0.7,0.9) -- (0,0);
\draw[black, thick] (0.7,0.9) -- (0,0);
\draw[black, thick] (0.7,0) -- (0,-0.9);
\draw[snake=zigzag](0,0) -- (0,-0.9);
\draw[snake=zigzag](0,-0.9) -- (0,-1.8);
\filldraw[pink] (-0.7,0.9) circle (7pt); 
\filldraw[pink] (0.7,0.9) circle (7pt); 
\filldraw[pink] (0.7,0) circle (7pt); 
\end{tikzpicture}
}}_{i_{1}} + u^{F}_{i_{1}}, b^{
\scalebox{0.6}{\begin{tikzpicture}
\draw[black, thick] (0,0.5) -- (0,0);
\filldraw[blue] (0,0.5) circle (2pt); 
\end{tikzpicture}
}}_{j}),  \pi_{<}(L(u^{
\scalebox{0.16}{\begin{tikzpicture}
\draw[black, thick] (-0.7,0.9) -- (0,0);
\draw[black, thick] (0.7,0.9) -- (0,0);
\draw[black, thick] (0.7,0) -- (0,-0.9);
\draw[snake=zigzag](0,0) -- (0,-0.9);
\draw[snake=zigzag](0,-0.9) -- (0,-1.8);
\filldraw[pink] (-0.7,0.9) circle (7pt); 
\filldraw[pink] (0.7,0.9) circle (7pt); 
\filldraw[pink] (0.7,0) circle (7pt); 
\end{tikzpicture}
}}_{i_{1}} + u^{F}_{i_{1}}), K_{j}^{b}), \pi_{<}(L(u^{
\scalebox{0.16}{\begin{tikzpicture}
\draw[black, thick] (-0.7,0.9) -- (0,0);
\draw[black, thick] (0.7,0.9) -- (0,0);
\draw[black, thick] (0.7,0) -- (0,-0.9);
\draw[snake=zigzag](0,0) -- (0,-0.9);
\draw[snake=zigzag](0,-0.9) -- (0,-1.8);
\filldraw[pink] (-0.7,0.9) circle (7pt); 
\filldraw[pink] (0.7,0.9) circle (7pt); 
\filldraw[pink] (0.7,0) circle (7pt); 
\end{tikzpicture}
}}_{j} + u^{F}_{j}), K_{i_{1}}^{b}), \nonumber\\
& \hspace{5mm} \pi_{<} (L(b^{
\scalebox{0.16}{\begin{tikzpicture}
\draw[black, thick] (-0.7,0.9) -- (0,0);
\draw[black, thick] (0.7,0.9) -- (0,0);
\draw[black, thick] (0.7,0) -- (0,-0.9);
\draw[snake=zigzag](0,0) -- (0,-0.9);
\draw[snake=zigzag](0,-0.9) -- (0,-1.8);
\filldraw[gray] (-0.7,0.9) circle (7pt); 
\filldraw[gray] (0.7,0.9) circle (7pt); 
\filldraw[gray] (0.7,0) circle (7pt); 
\end{tikzpicture}
}}_{i_{1}} + b^{F}_{i_{1}}), K_{j}^{u}), \pi_{<} (L(b^{
\scalebox{0.16}{\begin{tikzpicture}
\draw[black, thick] (-0.7,0.9) -- (0,0);
\draw[black, thick] (0.7,0.9) -- (0,0);
\draw[black, thick] (0.7,0) -- (0,-0.9);
\draw[snake=zigzag](0,0) -- (0,-0.9);
\draw[snake=zigzag](0,-0.9) -- (0,-1.8);
\filldraw[gray] (-0.7,0.9) circle (7pt); 
\filldraw[gray] (0.7,0.9) circle (7pt); 
\filldraw[gray] (0.7,0) circle (7pt); 
\end{tikzpicture}
}}_{j} + b^{F}_{j}), K_{i_{1}}^{u}), \nonumber\\
& \hspace{5mm} \pi_{<}(\nabla (u^{
\scalebox{0.16}{\begin{tikzpicture}
\draw[black, thick] (-0.7,0.9) -- (0,0);
\draw[black, thick] (0.7,0.9) -- (0,0);
\draw[black, thick] (0.7,0) -- (0,-0.9);
\draw[snake=zigzag](0,0) -- (0,-0.9);
\draw[snake=zigzag](0,-0.9) -- (0,-1.8);
\filldraw[pink] (-0.7,0.9) circle (7pt); 
\filldraw[pink] (0.7,0.9) circle (7pt); 
\filldraw[pink] (0.7,0) circle (7pt); 
\end{tikzpicture}
}}_{i_{1}} + u^{F}_{i_{1}}), \nabla K_{j}^{b}), \pi_{<} (\nabla (u^{
\scalebox{0.16}{\begin{tikzpicture}
\draw[black, thick] (-0.7,0.9) -- (0,0);
\draw[black, thick] (0.7,0.9) -- (0,0);
\draw[black, thick] (0.7,0) -- (0,-0.9);
\draw[snake=zigzag](0,0) -- (0,-0.9);
\draw[snake=zigzag](0,-0.9) -- (0,-1.8);
\filldraw[pink] (-0.7,0.9) circle (7pt); 
\filldraw[pink] (0.7,0.9) circle (7pt); 
\filldraw[pink] (0.7,0) circle (7pt); 
\end{tikzpicture}
}}_{j} + u^{F}_{j}), \nabla K_{i_{1}}^{b}),\nonumber\\
& \hspace{5mm} \pi_{<}(\nabla (b^{
\scalebox{0.16}{\begin{tikzpicture}
\draw[black, thick] (-0.7,0.9) -- (0,0);
\draw[black, thick] (0.7,0.9) -- (0,0);
\draw[black, thick] (0.7,0) -- (0,-0.9);
\draw[snake=zigzag](0,0) -- (0,-0.9);
\draw[snake=zigzag](0,-0.9) -- (0,-1.8);
\filldraw[gray] (-0.7,0.9) circle (7pt); 
\filldraw[gray] (0.7,0.9) circle (7pt); 
\filldraw[gray] (0.7,0) circle (7pt); 
\end{tikzpicture}
}}_{i_{1}} + b^{F}_{i_{1}}), \nabla K_{j}^{u}), \pi_{<}(\nabla (b^{
\scalebox{0.16}{\begin{tikzpicture}
\draw[black, thick] (-0.7,0.9) -- (0,0);
\draw[black, thick] (0.7,0.9) -- (0,0);
\draw[black, thick] (0.7,0) -- (0,-0.9);
\draw[snake=zigzag](0,0) -- (0,-0.9);
\draw[snake=zigzag](0,-0.9) -- (0,-1.8);
\filldraw[gray] (-0.7,0.9) circle (7pt); 
\filldraw[gray] (0.7,0.9) circle (7pt); 
\filldraw[gray] (0.7,0) circle (7pt); 
\end{tikzpicture}
}}_{j} + b^{F}_{j}), \nabla K_{i_{1}}^{u}))] \rVert_{\mathcal{C}^{-1 - 2\delta}} \nonumber\\
\lesssim& (1+ C_{\xi}^{4})[ 1+ \lVert (u^{\sharp}, b^{\sharp}) \rVert_{\mathcal{C}^{\frac{1}{2} + \beta}} + \lVert y^{F}\rVert_{\mathcal{C}^{\frac{1}{2} - \delta_{0}}} + \lVert y^{F}\rVert_{\mathcal{C}^{\delta}}^{2}]. 
\end{align}  
Fourth, we bound within (\ref{88})-(\ref{89}) 
\begin{align}\label{96}
& \lVert \mathcal{P}_{ii_{1}}  \partial_{x_{j}} (\pi_{0,\diamond} (u^{
\scalebox{0.16}{\begin{tikzpicture}
\draw[black, thick] (-0.7,0.9) -- (0,0);
\draw[black, thick] (0.7,0.9) -- (0,0);
\draw[black, thick] (0.7,0) -- (0,-0.9);
\draw[snake=zigzag](0,0) -- (0,-0.9);
\draw[snake=zigzag](0,-0.9) -- (0,-1.8);
\filldraw[pink] (-0.7,0.9) circle (7pt); 
\filldraw[pink] (0.7,0.9) circle (7pt); 
\filldraw[pink] (0.7,0) circle (7pt); 
\end{tikzpicture}
}}_{j}, u^{\scalebox{0.6}{\begin{tikzpicture}
\draw[black, thick] (0,0.5) -- (0,0);
\filldraw[red] (0,0.5) circle (2pt); 
\end{tikzpicture}
}}_{i_{1}} ), \pi_{0,\diamond}(u^{F}_{j}, u^{\scalebox{0.6}{\begin{tikzpicture}
\draw[black, thick] (0,0.5) -- (0,0);
\filldraw[red] (0,0.5) circle (2pt); 
\end{tikzpicture}
}}_{i_{1}} ), \pi_{0,\diamond} (u^{
\scalebox{0.16}{\begin{tikzpicture}
\draw[black, thick] (-0.7,0.9) -- (0,0);
\draw[black, thick] (0.7,0.9) -- (0,0);
\draw[black, thick] (0.7,0) -- (0,-0.9);
\draw[snake=zigzag](0,0) -- (0,-0.9);
\draw[snake=zigzag](0,-0.9) -- (0,-1.8);
\filldraw[pink] (-0.7,0.9) circle (7pt); 
\filldraw[pink] (0.7,0.9) circle (7pt); 
\filldraw[pink] (0.7,0) circle (7pt); 
\end{tikzpicture}
}}_{i_{1}}, u^{\scalebox{0.6}{\begin{tikzpicture}
\draw[black, thick] (0,0.5) -- (0,0);
\filldraw[red] (0,0.5) circle (2pt); 
\end{tikzpicture}
}}_{j} ), \pi_{0, \diamond}(u^{F}_{i_{1}}, u^{\scalebox{0.6}{\begin{tikzpicture}
\draw[black, thick] (0,0.5) -- (0,0);
\filldraw[red] (0,0.5) circle (2pt); 
\end{tikzpicture}
}}_{j} ),\nonumber\\
& \hspace{5mm} \pi_{0,\diamond}(b^{
\scalebox{0.16}{\begin{tikzpicture}
\draw[black, thick] (-0.7,0.9) -- (0,0);
\draw[black, thick] (0.7,0.9) -- (0,0);
\draw[black, thick] (0.7,0) -- (0,-0.9);
\draw[snake=zigzag](0,0) -- (0,-0.9);
\draw[snake=zigzag](0,-0.9) -- (0,-1.8);
\filldraw[gray] (-0.7,0.9) circle (7pt); 
\filldraw[gray] (0.7,0.9) circle (7pt); 
\filldraw[gray] (0.7,0) circle (7pt); 
\end{tikzpicture}
}}_{j}, b^{
\scalebox{0.6}{\begin{tikzpicture}
\draw[black, thick] (0,0.5) -- (0,0);
\filldraw[blue] (0,0.5) circle (2pt); 
\end{tikzpicture}
}}_{i_{1}}), \pi_{0,\diamond}(b^{F}_{j}, b^{
\scalebox{0.6}{\begin{tikzpicture}
\draw[black, thick] (0,0.5) -- (0,0);
\filldraw[blue] (0,0.5) circle (2pt); 
\end{tikzpicture}
}}_{i_{1}}), \pi_{0,\diamond}(b^{
\scalebox{0.16}{\begin{tikzpicture}
\draw[black, thick] (-0.7,0.9) -- (0,0);
\draw[black, thick] (0.7,0.9) -- (0,0);
\draw[black, thick] (0.7,0) -- (0,-0.9);
\draw[snake=zigzag](0,0) -- (0,-0.9);
\draw[snake=zigzag](0,-0.9) -- (0,-1.8);
\filldraw[gray] (-0.7,0.9) circle (7pt); 
\filldraw[gray] (0.7,0.9) circle (7pt); 
\filldraw[gray] (0.7,0) circle (7pt); 
\end{tikzpicture}
}}_{i_{1}}, b^{
\scalebox{0.6}{\begin{tikzpicture}
\draw[black, thick] (0,0.5) -- (0,0);
\filldraw[blue] (0,0.5) circle (2pt); 
\end{tikzpicture}
}}_{j}), \pi_{0,\diamond}(b^{F}_{i_{1}}, b^{
\scalebox{0.6}{\begin{tikzpicture}
\draw[black, thick] (0,0.5) -- (0,0);
\filldraw[blue] (0,0.5) circle (2pt); 
\end{tikzpicture}
}}_{j})\nonumber\\
& \hspace{5mm} \pi_{0,\diamond}(u^{
\scalebox{0.16}{\begin{tikzpicture}
\draw[black, thick] (-0.7,0.9) -- (0,0);
\draw[black, thick] (0.7,0.9) -- (0,0);
\draw[black, thick] (0.7,0) -- (0,-0.9);
\draw[snake=zigzag](0,0) -- (0,-0.9);
\draw[snake=zigzag](0,-0.9) -- (0,-1.8);
\filldraw[pink] (-0.7,0.9) circle (7pt); 
\filldraw[pink] (0.7,0.9) circle (7pt); 
\filldraw[pink] (0.7,0) circle (7pt); 
\end{tikzpicture}
}}_{j}, b^{
\scalebox{0.6}{\begin{tikzpicture}
\draw[black, thick] (0,0.5) -- (0,0);
\filldraw[blue] (0,0.5) circle (2pt); 
\end{tikzpicture}
}}_{i_{1}}), \pi_{0,\diamond} (u^{F}_{j}, b^{
\scalebox{0.6}{\begin{tikzpicture}
\draw[black, thick] (0,0.5) -- (0,0);
\filldraw[blue] (0,0.5) circle (2pt); 
\end{tikzpicture}
}}_{i_{1}}), \pi_{0,\diamond}(b^{
\scalebox{0.16}{\begin{tikzpicture}
\draw[black, thick] (-0.7,0.9) -- (0,0);
\draw[black, thick] (0.7,0.9) -- (0,0);
\draw[black, thick] (0.7,0) -- (0,-0.9);
\draw[snake=zigzag](0,0) -- (0,-0.9);
\draw[snake=zigzag](0,-0.9) -- (0,-1.8);
\filldraw[gray] (-0.7,0.9) circle (7pt); 
\filldraw[gray] (0.7,0.9) circle (7pt); 
\filldraw[gray] (0.7,0) circle (7pt); 
\end{tikzpicture}
}}_{i_{1}}, u^{\scalebox{0.6}{\begin{tikzpicture}
\draw[black, thick] (0,0.5) -- (0,0);
\filldraw[red] (0,0.5) circle (2pt); 
\end{tikzpicture}
}}_{j} ), \pi_{0,\diamond}(b^{F}_{i_{1}}, u^{\scalebox{0.6}{\begin{tikzpicture}
\draw[black, thick] (0,0.5) -- (0,0);
\filldraw[red] (0,0.5) circle (2pt); 
\end{tikzpicture}
}}_{j} ), \nonumber\\
& \hspace{5mm} \pi_{0,\diamond} (b^{
\scalebox{0.16}{\begin{tikzpicture}
\draw[black, thick] (-0.7,0.9) -- (0,0);
\draw[black, thick] (0.7,0.9) -- (0,0);
\draw[black, thick] (0.7,0) -- (0,-0.9);
\draw[snake=zigzag](0,0) -- (0,-0.9);
\draw[snake=zigzag](0,-0.9) -- (0,-1.8);
\filldraw[gray] (-0.7,0.9) circle (7pt); 
\filldraw[gray] (0.7,0.9) circle (7pt); 
\filldraw[gray] (0.7,0) circle (7pt); 
\end{tikzpicture}
}}_{j}, u^{\scalebox{0.6}{\begin{tikzpicture}
\draw[black, thick] (0,0.5) -- (0,0);
\filldraw[red] (0,0.5) circle (2pt); 
\end{tikzpicture}
}}_{i_{1}} ), \pi_{0,\diamond} (b^{F}_{j}, u^{\scalebox{0.6}{\begin{tikzpicture}
\draw[black, thick] (0,0.5) -- (0,0);
\filldraw[red] (0,0.5) circle (2pt); 
\end{tikzpicture}
}}_{i_{1}} ), \pi_{0,\diamond}(u^{
\scalebox{0.16}{\begin{tikzpicture}
\draw[black, thick] (-0.7,0.9) -- (0,0);
\draw[black, thick] (0.7,0.9) -- (0,0);
\draw[black, thick] (0.7,0) -- (0,-0.9);
\draw[snake=zigzag](0,0) -- (0,-0.9);
\draw[snake=zigzag](0,-0.9) -- (0,-1.8);
\filldraw[pink] (-0.7,0.9) circle (7pt); 
\filldraw[pink] (0.7,0.9) circle (7pt); 
\filldraw[pink] (0.7,0) circle (7pt); 
\end{tikzpicture}
}}_{i_{1}}, b^{
\scalebox{0.6}{\begin{tikzpicture}
\draw[black, thick] (0,0.5) -- (0,0);
\filldraw[blue] (0,0.5) circle (2pt); 
\end{tikzpicture}
}}_{j}), \pi_{0,\diamond}(u^{F}_{i_{1}}, b^{
\scalebox{0.6}{\begin{tikzpicture}
\draw[black, thick] (0,0.5) -- (0,0);
\filldraw[blue] (0,0.5) circle (2pt); 
\end{tikzpicture}
}}_{j})) \rVert_{\mathcal{C}^{-1-2\delta}} \nonumber\\
\lesssim& \lVert (\pi_{0,\diamond} (u^{
\scalebox{0.16}{\begin{tikzpicture}
\draw[black, thick] (-0.7,0.9) -- (0,0);
\draw[black, thick] (0.7,0.9) -- (0,0);
\draw[black, thick] (0.7,0) -- (0,-0.9);
\draw[snake=zigzag](0,0) -- (0,-0.9);
\draw[snake=zigzag](0,-0.9) -- (0,-1.8);
\filldraw[pink] (-0.7,0.9) circle (7pt); 
\filldraw[pink] (0.7,0.9) circle (7pt); 
\filldraw[pink] (0.7,0) circle (7pt); 
\end{tikzpicture}
}}_{j}, u^{\scalebox{0.6}{\begin{tikzpicture}
\draw[black, thick] (0,0.5) -- (0,0);
\filldraw[red] (0,0.5) circle (2pt); 
\end{tikzpicture}
}}_{i_{1}} ), \pi_{0,\diamond}(u^{
\scalebox{0.16}{\begin{tikzpicture}
\draw[black, thick] (-0.7,0.9) -- (0,0);
\draw[black, thick] (0.7,0.9) -- (0,0);
\draw[black, thick] (0.7,0) -- (0,-0.9);
\draw[snake=zigzag](0,0) -- (0,-0.9);
\draw[snake=zigzag](0,-0.9) -- (0,-1.8);
\filldraw[pink] (-0.7,0.9) circle (7pt); 
\filldraw[pink] (0.7,0.9) circle (7pt); 
\filldraw[pink] (0.7,0) circle (7pt); 
\end{tikzpicture}
}}_{i_{1}}, u^{\scalebox{0.6}{\begin{tikzpicture}
\draw[black, thick] (0,0.5) -- (0,0);
\filldraw[red] (0,0.5) circle (2pt); 
\end{tikzpicture}
}}_{j} ), \pi_{0,\diamond}(b^{
\scalebox{0.16}{\begin{tikzpicture}
\draw[black, thick] (-0.7,0.9) -- (0,0);
\draw[black, thick] (0.7,0.9) -- (0,0);
\draw[black, thick] (0.7,0) -- (0,-0.9);
\draw[snake=zigzag](0,0) -- (0,-0.9);
\draw[snake=zigzag](0,-0.9) -- (0,-1.8);
\filldraw[gray] (-0.7,0.9) circle (7pt); 
\filldraw[gray] (0.7,0.9) circle (7pt); 
\filldraw[gray] (0.7,0) circle (7pt); 
\end{tikzpicture}
}}_{j}, b^{
\scalebox{0.6}{\begin{tikzpicture}
\draw[black, thick] (0,0.5) -- (0,0);
\filldraw[blue] (0,0.5) circle (2pt); 
\end{tikzpicture}
}}_{i_{1}}), \pi_{0,\diamond} (b^{
\scalebox{0.16}{\begin{tikzpicture}
\draw[black, thick] (-0.7,0.9) -- (0,0);
\draw[black, thick] (0.7,0.9) -- (0,0);
\draw[black, thick] (0.7,0) -- (0,-0.9);
\draw[snake=zigzag](0,0) -- (0,-0.9);
\draw[snake=zigzag](0,-0.9) -- (0,-1.8);
\filldraw[gray] (-0.7,0.9) circle (7pt); 
\filldraw[gray] (0.7,0.9) circle (7pt); 
\filldraw[gray] (0.7,0) circle (7pt); 
\end{tikzpicture}
}}_{i_{1}}, b^{
\scalebox{0.6}{\begin{tikzpicture}
\draw[black, thick] (0,0.5) -- (0,0);
\filldraw[blue] (0,0.5) circle (2pt); 
\end{tikzpicture}
}}_{j}), \nonumber\\
& \hspace{5mm} \pi_{0,\diamond}(u^{
\scalebox{0.16}{\begin{tikzpicture}
\draw[black, thick] (-0.7,0.9) -- (0,0);
\draw[black, thick] (0.7,0.9) -- (0,0);
\draw[black, thick] (0.7,0) -- (0,-0.9);
\draw[snake=zigzag](0,0) -- (0,-0.9);
\draw[snake=zigzag](0,-0.9) -- (0,-1.8);
\filldraw[pink] (-0.7,0.9) circle (7pt); 
\filldraw[pink] (0.7,0.9) circle (7pt); 
\filldraw[pink] (0.7,0) circle (7pt); 
\end{tikzpicture}
}}_{j}, b^{
\scalebox{0.6}{\begin{tikzpicture}
\draw[black, thick] (0,0.5) -- (0,0);
\filldraw[blue] (0,0.5) circle (2pt); 
\end{tikzpicture}
}}_{i_{1}}), \pi_{0,\diamond}(b^{
\scalebox{0.16}{\begin{tikzpicture}
\draw[black, thick] (-0.7,0.9) -- (0,0);
\draw[black, thick] (0.7,0.9) -- (0,0);
\draw[black, thick] (0.7,0) -- (0,-0.9);
\draw[snake=zigzag](0,0) -- (0,-0.9);
\draw[snake=zigzag](0,-0.9) -- (0,-1.8);
\filldraw[gray] (-0.7,0.9) circle (7pt); 
\filldraw[gray] (0.7,0.9) circle (7pt); 
\filldraw[gray] (0.7,0) circle (7pt); 
\end{tikzpicture}
}}_{i_{1}}, u^{\scalebox{0.6}{\begin{tikzpicture}
\draw[black, thick] (0,0.5) -- (0,0);
\filldraw[red] (0,0.5) circle (2pt); 
\end{tikzpicture}
}}_{j} ), \pi_{0,\diamond}(b^{
\scalebox{0.16}{\begin{tikzpicture}
\draw[black, thick] (-0.7,0.9) -- (0,0);
\draw[black, thick] (0.7,0.9) -- (0,0);
\draw[black, thick] (0.7,0) -- (0,-0.9);
\draw[snake=zigzag](0,0) -- (0,-0.9);
\draw[snake=zigzag](0,-0.9) -- (0,-1.8);
\filldraw[gray] (-0.7,0.9) circle (7pt); 
\filldraw[gray] (0.7,0.9) circle (7pt); 
\filldraw[gray] (0.7,0) circle (7pt); 
\end{tikzpicture}
}}_{j}, u^{\scalebox{0.6}{\begin{tikzpicture}
\draw[black, thick] (0,0.5) -- (0,0);
\filldraw[red] (0,0.5) circle (2pt); 
\end{tikzpicture}
}}_{i_{1}} ), \pi_{0,\diamond}(u^{
\scalebox{0.16}{\begin{tikzpicture}
\draw[black, thick] (-0.7,0.9) -- (0,0);
\draw[black, thick] (0.7,0.9) -- (0,0);
\draw[black, thick] (0.7,0) -- (0,-0.9);
\draw[snake=zigzag](0,0) -- (0,-0.9);
\draw[snake=zigzag](0,-0.9) -- (0,-1.8);
\filldraw[pink] (-0.7,0.9) circle (7pt); 
\filldraw[pink] (0.7,0.9) circle (7pt); 
\filldraw[pink] (0.7,0) circle (7pt); 
\end{tikzpicture}
}}_{i_{1}}, b^{
\scalebox{0.6}{\begin{tikzpicture}
\draw[black, thick] (0,0.5) -- (0,0);
\filldraw[blue] (0,0.5) circle (2pt); 
\end{tikzpicture}
}}_{j})) \rVert_{\mathcal{C}^{-2\delta}} \nonumber\\
&+ C_{\xi}^{3} + \lVert y^{F} \rVert_{\mathcal{C}^{\frac{1}{2} - \delta_{0}}} (C_{\xi}^{2} + 1) + \lVert (u^{\sharp}, b^{\sharp}) \rVert_{\mathcal{C}^{\frac{1}{2} + \beta}} C_{\xi} + 1 \nonumber\\
\lesssim& C_{\xi}^{3} + \lVert y^{F} \rVert_{\mathcal{C}^{\frac{1}{2} - \delta_{0}}}(C_{\xi}^{2} + 1) + \lVert (u^{\sharp}, b^{\sharp}) \rVert_{\mathcal{C}^{\frac{1}{2} + \beta}} C_{\xi} + 1
\end{align} 
by Lemma \ref{Lemma 2.7}, (\ref{71}), (\ref{77}), (\ref{37}). Therefore, inserting (\ref{90})-(\ref{96}) in (\ref{88}) and (\ref{89}) gives 
\begin{equation}\label{97}
\lVert (\phi_{i}^{\sharp, u}, \phi_{i}^{\sharp, b}) (t) \rVert_{\mathcal{C}^{-1-2\delta}} \lesssim (1+ C_{\xi}^{4}) [ 1+ \lVert (u^{\sharp}, b^{\sharp}) \rVert_{\mathcal{C}^{\frac{1}{2} + \beta}} + \lVert y^{F} \rVert_{\mathcal{C}^{\frac{1}{2} - \delta_{0}}} + \lVert y^{F} \rVert_{\mathcal{C}^{\delta}}^{2}].
\end{equation}  
\subsection{Construction of the solution}
From the paracontrolled ansatz (\ref{29}) and (\ref{32}), for any $t \in [0, \overline{T}], \overline{T} > 0$ depending only on $C_{\xi}$, we can obtain 
\begin{align}\label{98}
& \lVert (u^{F}_{i}, b^{F}_{i})(t) \rVert_{\mathcal{C}^{\frac{1}{2} - \delta_{0}}}\nonumber\\
\lesssim& \sum_{i_{1}, j=1}^{3} \lVert (u^{
\scalebox{0.16}{\begin{tikzpicture}
\draw[black, thick] (-0.7,0.9) -- (0,0);
\draw[black, thick] (0.7,0.9) -- (0,0);
\draw[black, thick] (0.7,0) -- (0,-0.9);
\draw[snake=zigzag](0,0) -- (0,-0.9);
\draw[snake=zigzag](0,-0.9) -- (0,-1.8);
\filldraw[pink] (-0.7,0.9) circle (7pt); 
\filldraw[pink] (0.7,0.9) circle (7pt); 
\filldraw[pink] (0.7,0) circle (7pt); 
\end{tikzpicture}
}}_{i_{1}} + u^{F}_{i_{1}}, b^{
\scalebox{0.16}{\begin{tikzpicture}
\draw[black, thick] (-0.7,0.9) -- (0,0);
\draw[black, thick] (0.7,0.9) -- (0,0);
\draw[black, thick] (0.7,0) -- (0,-0.9);
\draw[snake=zigzag](0,0) -- (0,-0.9);
\draw[snake=zigzag](0,-0.9) -- (0,-1.8);
\filldraw[gray] (-0.7,0.9) circle (7pt); 
\filldraw[gray] (0.7,0.9) circle (7pt); 
\filldraw[gray] (0.7,0) circle (7pt); 
\end{tikzpicture}
}}_{i_{1}} + b^{F}_{i_{1}}) \rVert_{L^{\infty}} \lVert (K_{j}^{u}, K_{j}^{b}) \rVert_{\mathcal{C}^{\frac{3}{2} - \delta}} + \lVert (u^{\sharp}_{i}, b^{\sharp}_{i}) \rVert_{\mathcal{C}^{\frac{1}{2} - \delta_{0}}} \nonumber\\
\leq& C \sum_{i_{1}, j=1}^{3} ( \lVert (u^{
\scalebox{0.16}{\begin{tikzpicture}
\draw[black, thick] (-0.7,0.9) -- (0,0);
\draw[black, thick] (0.7,0.9) -- (0,0);
\draw[black, thick] (0.7,0) -- (0,-0.9);
\draw[snake=zigzag](0,0) -- (0,-0.9);
\draw[snake=zigzag](0,-0.9) -- (0,-1.8);
\filldraw[pink] (-0.7,0.9) circle (7pt); 
\filldraw[pink] (0.7,0.9) circle (7pt); 
\filldraw[pink] (0.7,0) circle (7pt); 
\end{tikzpicture}
}}_{i_{1}}, b^{
\scalebox{0.16}{\begin{tikzpicture}
\draw[black, thick] (-0.7,0.9) -- (0,0);
\draw[black, thick] (0.7,0.9) -- (0,0);
\draw[black, thick] (0.7,0) -- (0,-0.9);
\draw[snake=zigzag](0,0) -- (0,-0.9);
\draw[snake=zigzag](0,-0.9) -- (0,-1.8);
\filldraw[gray] (-0.7,0.9) circle (7pt); 
\filldraw[gray] (0.7,0.9) circle (7pt); 
\filldraw[gray] (0.7,0) circle (7pt); 
\end{tikzpicture}
}}_{i_{1}}) \rVert_{\mathcal{C}^{\frac{1}{2} - \delta}} + \lVert (u^{F}_{i_{1}}, b^{F}_{i_{1}}) \rVert_{\mathcal{C}^{\frac{1}{2} - \delta_{0}}}) t^{\frac{\delta}{4}} C_{\xi} + C \lVert (u^{\sharp}_{i}, b^{\sharp}_{i}) \rVert_{\mathcal{C}^{\frac{1}{2} - \delta_{0}}} 
\end{align} 
for some $C \geq 0$ by Lemma \ref{Lemma 1.1} (1), (\ref{59}), (\ref{33}) and (\ref{37}). Therefore, for $t \in [0, ( \frac{1}{C C_{\xi}})^{\frac{4}{\delta}})$ 
\begin{equation}\label{99}
\sum_{i=1}^{3} \lVert (u^{F}_{i}, b^{F}_{i})(t) \rVert_{\mathcal{C}^{\frac{1}{2} - \delta_{0}}} 
\lesssim C_{\xi}^{2} + \sum_{i=1}^{3} \lVert (u^{\sharp}_{i}, b^{\sharp}_{i}) \rVert_{\mathcal{C}^{\frac{1}{2} - \delta_{0}}} 
\end{equation} 
due to (\ref{39}). Similarly for any $t \in [0, \overline{T}], \overline{T} > 0$ depending only on $C_{\xi}$, 
\begin{align}\label{100}
\lVert (u^{F}_{i}, b^{F}_{i}(t) \rVert_{\mathcal{C}^{\delta}} 
\lesssim& \sum_{i_{1}, j=1}^{3} \lVert (u^{
\scalebox{0.16}{\begin{tikzpicture}
\draw[black, thick] (-0.7,0.9) -- (0,0);
\draw[black, thick] (0.7,0.9) -- (0,0);
\draw[black, thick] (0.7,0) -- (0,-0.9);
\draw[snake=zigzag](0,0) -- (0,-0.9);
\draw[snake=zigzag](0,-0.9) -- (0,-1.8);
\filldraw[pink] (-0.7,0.9) circle (7pt); 
\filldraw[pink] (0.7,0.9) circle (7pt); 
\filldraw[pink] (0.7,0) circle (7pt); 
\end{tikzpicture}
}}_{i_{1}} + u^{F}_{i_{1}}, b^{
\scalebox{0.16}{\begin{tikzpicture}
\draw[black, thick] (-0.7,0.9) -- (0,0);
\draw[black, thick] (0.7,0.9) -- (0,0);
\draw[black, thick] (0.7,0) -- (0,-0.9);
\draw[snake=zigzag](0,0) -- (0,-0.9);
\draw[snake=zigzag](0,-0.9) -- (0,-1.8);
\filldraw[gray] (-0.7,0.9) circle (7pt); 
\filldraw[gray] (0.7,0.9) circle (7pt); 
\filldraw[gray] (0.7,0) circle (7pt); 
\end{tikzpicture}
}}_{i_{1}} + b^{F}_{i_{1}}) \rVert_{\mathcal{C}^{2\delta - \frac{1}{2}}} \lVert (K_{j}^{u}, K_{j}^{b}) \rVert_{\mathcal{C}^{\frac{3}{2} - \delta}} + \lVert (u^{\sharp}_{i}, b^{\sharp}_{i}) \rVert_{\mathcal{C}^{\delta}} \nonumber\\
\leq& C( \lVert y^{
\scalebox{0.16}{\begin{tikzpicture}
\draw[black, thick] (-0.7,0.9) -- (0,0);
\draw[black, thick] (0.7,0.9) -- (0,0);
\draw[black, thick] (0.7,0) -- (0,-0.9);
\draw[snake=zigzag](0,0) -- (0,-0.9);
\draw[snake=zigzag](0,-0.9) -- (0,-1.8);
\filldraw[black] (-0.7,0.9) circle (7pt); 
\filldraw[black] (0.7,0.9) circle (7pt); 
\filldraw[black] (0.7,0) circle (7pt); 
\end{tikzpicture}
}}
 \rVert_{\mathcal{C}^{\frac{1}{2} - \delta}} + \lVert y^{F} \rVert_{\mathcal{C}^{\delta}}) t^{\frac{\delta}{4}} C_{\xi} + C\sum_{i=1}^{3} \lVert (u^{\sharp}_{i}, b^{\sharp}_{i})\rVert_{\mathcal{C}^{\delta}} 
\end{align} 
by (\ref{29}), (\ref{32}), Lemma \ref{Lemma 1.1} (2) (\ref{59}), (\ref{33}) and (\ref{37}). This gives for $t \in [0, (\frac{1}{CC_{\xi}})^{\frac{4}{\delta}})$
\begin{equation}\label{101}
\sum_{i=1}^{3}\lVert (u^{F}_{i}, b^{F}_{i}) (t) \rVert_{\mathcal{C}^{\delta}} \lesssim C_{\xi}^{2} + \sum_{i=1}^{3} \lVert (u^{\sharp}_{i}, b^{\sharp}_{i}) \rVert_{\mathcal{C}^{ \delta}}
\end{equation} 
due to (\ref{39}). Now, due to (\ref{29}), (\ref{32}), (\ref{19}) and (\ref{26}) we see that  
\begin{equation}
u^{\sharp}_{i}(\cdot, 0) = \sum_{i_{1} = 1}^{3} \mathcal{P}_{ii_{1}}u_{i_{1}}^{\text{in}}(\cdot) - 
u^{\scalebox{0.6}{\begin{tikzpicture}
\draw[black, thick] (0,0.5) -- (0,0);
\filldraw[red] (0,0.5) circle (2pt); 
\end{tikzpicture}
}}_{i} (\cdot, 0) \text{ and }  b^{\sharp}_{i}(\cdot, 0) = \sum_{i_{1} =1}^{3} \mathcal{P}_{ii_{1}} b_{i_{1}}^{\text{in}}(\cdot) - 
b^{\scalebox{0.6}{\begin{tikzpicture}
\draw[black, thick] (0,0.5) -- (0,0);
\filldraw[blue] (0,0.5) circle (2pt); 
\end{tikzpicture}
}}_{i}(\cdot, 0)
\end{equation} 
which, together with (\ref{50}) and (\ref{51}), leads to 
\begin{subequations} 
\begin{align}
& u^{\sharp}_{i} (t) = P_{t} ( \sum_{i_{1} =1}^{3} \mathcal{P}_{ii_{1}} u^{\text{in}}_{i_{1}}- u^{\scalebox{0.6}{\begin{tikzpicture}
\draw[black, thick] (0,0.5) -- (0,0);
\filldraw[red] (0,0.5) circle (2pt); 
\end{tikzpicture}
}}_{i} (0) ) + \int_{0}^{t} P_{t-s} \phi_{i}^{\sharp, u}(s) ds, \label{102a}\\
& b^{\sharp}_{i}(t) = P_{t} ( \sum_{i_{1} =1}^{3} \mathcal{P}_{ii_{1}} b_{i_{1}}^{\text{in}} - b^{
\scalebox{0.6}{\begin{tikzpicture}
\draw[black, thick] (0,0.5) -- (0,0);
\filldraw[blue] (0,0.5) circle (2pt); 
\end{tikzpicture}
}}_{i}(0)) + \int_{0}^{t} P_{t-s} \phi_{i}^{\sharp, b}(s) ds. \label{102b}
\end{align}
\end{subequations} 
Then we obtain 
\begin{align}
& t^{\delta + z} \lVert (u^{\sharp}, b^{\sharp} ) (t) \rVert_{\mathcal{C}^{\frac{1}{2} + \beta}} \label{103}\\
\lesssim& t^{\delta + z} \lVert P_{t} ( \mathcal{P} y^{\text{in}} - y^{\scalebox{0.6}{\begin{tikzpicture}
\draw[black, thick] (0,0.5) -- (0,0);
\filldraw[black] (0,0.5) circle (2pt); 
\end{tikzpicture}
}}(0) ) \rVert_{\mathcal{C}^{\frac{1}{2}+ \beta}} + t^{\delta  + z} \int_{0}^{t} \lVert P_{t-s} ( \phi^{\sharp,u}, \phi^{\sharp,b} )(s) \rVert_{\mathcal{C}^{\frac{1}{2} + \beta}} ds \nonumber \\ 
\lesssim& \lVert \mathcal{P} y^{\text{in}} -y^{\scalebox{0.6}{\begin{tikzpicture}
\draw[black, thick] (0,0.5) -- (0,0);
\filldraw[black] (0,0.5) circle (2pt); 
\end{tikzpicture}
}}(0) \rVert_{\mathcal{C}^{-z}} + t^{\delta + z}  \int_{0}^{t} (t-s)^{- \frac{3}{4} - \frac{\beta}{2} - \delta} \lVert (\phi^{\sharp, u}, \phi^{\sharp, b})(s) \rVert_{\mathcal{C}^{-1 - 2 \delta}} ds \nonumber 
\end{align}  
by (\ref{102a}), (\ref{102b}), Lemma \ref{Lemma 2.6} and (\ref{44}). We are also able to estimate 
\begin{align}
&t^{\delta + z} \lVert  (u^{\sharp},  b^{\sharp})(t) \rVert_{\mathcal{C}^{\delta}}^{2} \label{104} \\
\lesssim& t^{\delta + z} [ \lVert P_{t} ( \mathcal{P} y^{\text{in}} - y^{\scalebox{0.6}{\begin{tikzpicture}
\draw[black, thick] (0,0.5) -- (0,0);
\filldraw[black] (0,0.5) circle (2pt); 
\end{tikzpicture}
}}(0)) \rVert_{\mathcal{C}^{\delta}}^{2} + \left( \int_{0}^{t} \lVert P_{t-s} (\phi^{\sharp, u}, \phi^{\sharp, b})(s) \rVert_{\mathcal{C}^{\delta}} ds \right)^{2} ] \nonumber\\
\lesssim& \lVert \mathcal{P} y^{\text{in}} - y^{\scalebox{0.6}{\begin{tikzpicture}
\draw[black, thick] (0,0.5) -- (0,0);
\filldraw[black] (0,0.5) circle (2pt); 
\end{tikzpicture}
}}(0) \rVert_{\mathcal{C}^{-z}}^{2}  + t^{\frac{1}{2} - \frac{3\delta}{2}} \int_{0}^{t} (t-s)^{ - \frac{(3\delta + 1)}{2}} s^{- (\delta + z)} (s^{\delta + z} \lVert (\phi^{\sharp, u}, \phi^{\sharp, b})(s) \rVert_{\mathcal{C}^{-1 - 2 \delta}})^{2} ds \nonumber
\end{align} 
by (\ref{102a}), (\ref{102b}), Lemma \ref{Lemma 2.6}, H$\ddot{\mathrm{o}}$lder's inequality, (\ref{59}) and (\ref{44}). Thus, 
\begin{align}\label{105}
 t^{\delta + z} \lVert (\phi^{\sharp, u}, \phi^{\sharp, b})(t) &\rVert_{\mathcal{C}^{-1 - 2\delta}} 
\lesssim t^{\delta + z} (1+ C_{\xi}^{4})[ 1+ \lVert (u^{\sharp}, b^{\sharp}) (t) \rVert_{\mathcal{C}^{\frac{1}{2} + \beta}} + C_{\xi}^{4} + \lVert (u^{\sharp}, b^{\sharp} )(t) \rVert_{\mathcal{C}^{\delta}}^{2}]\nonumber\\
&\lesssim (1+ C_{\xi}^{8}) + (1+ C_{\xi}^{4})[ \lVert \mathcal{P} y^{\text{in}} - y^{\scalebox{0.6}{\begin{tikzpicture}
\draw[black, thick] (0,0.5) -- (0,0);
\filldraw[black] (0,0.5) circle (2pt); 
\end{tikzpicture}
}}(0) \rVert_{\mathcal{C}^{-z}}^{2} \nonumber\\
&+ t^{\delta + z} \int_{0}^{t} (t-s)^{-\frac{3}{4} - \frac{\beta}{2} - \delta} s^{-(\delta + z)} (s^{\delta + z} \lVert (\phi^{\sharp, u}, \phi^{\sharp, b})(s) \rVert_{\mathcal{C}^{-1 - 2\delta}} ) ds \nonumber\\
&+ t^{\frac{1}{2} - \frac{3\delta}{2}} \int_{0}^{t}(t-s)^{- \frac{(3\delta + 1)}{2}} s^{-(\delta + z)} (s^{\delta + z} \lVert (\phi^{\sharp, u}, \phi^{\sharp, b})(s) \rVert_{\mathcal{C}^{-1 - 2\delta}})^{2} ds]
\end{align} 
by (\ref{97}), (\ref{99}), (\ref{101}), (\ref{103}) and (\ref{104}). By Bihari's inequality and Remark \ref{Remark 3.1}, this implies that for $\delta < \frac{1-z}{4}$, there exists some $T_{0} \in (0, \overline{T}]$ which is independent of $\epsilon \in (0,1)$ such that 
\begin{equation}\label{106}
\sup_{t \in [0, T_{0}]} t^{\delta + z} \lVert (\phi^{\sharp, u}, \phi^{\sharp, b}) (t) \rVert_{\mathcal{C}^{-1 - 2\delta}} \lesssim C(T_{0}, C_{\xi}, \lVert y^{\text{in}} \rVert_{\mathcal{C}^{-z}} , \lVert y^{\scalebox{0.6}{\begin{tikzpicture}
\draw[black, thick] (0,0.5) -- (0,0);
\filldraw[black] (0,0.5) circle (2pt); 
\end{tikzpicture}
}}(0) \rVert_{\mathcal{C}^{-z}}).
\end{equation} 
Thus, if $C_{\xi}^{\epsilon}$ is uniformly bounded over $\epsilon \in (0,1)$, then (\ref{106}) holds for all $\epsilon \in (0,1)$. Next, we estimate 
\begin{align}\label{107}
&t^{ \frac{ \frac{1}{2} - \delta_{0} + z}{2}} \lVert (u^{\sharp}, b^{\sharp})(t) \rVert_{\mathcal{C}^{\frac{1}{2} - \delta_{0}}} \\
\lesssim& t^{ \frac{ \frac{1}{2} - \delta_{0} + z}{2}}( \lVert P_{t}(\mathcal{P}y^{\text{in}} - y^{\scalebox{0.6}{\begin{tikzpicture}
\draw[black, thick] (0,0.5) -- (0,0);
\filldraw[black] (0,0.5) circle (2pt); 
\end{tikzpicture}
}}(0) ) \rVert_{\mathcal{C}^{\frac{1}{2} - \delta_{0}}} + \int_{0}^{t} \lVert P_{t-s} ( \phi^{\sharp, u}, \phi^{\sharp, b} )(s) \rVert_{\mathcal{C}^{\frac{1}{2} - \delta_{0}}} ds ) \nonumber\\
\lesssim&  \lVert \mathcal{P} y^{\text{in}} - y^{\scalebox{0.6}{\begin{tikzpicture}
\draw[black, thick] (0,0.5) -- (0,0);
\filldraw[black] (0,0.5) circle (2pt); 
\end{tikzpicture}
}}(0) \rVert_{\mathcal{C}^{-z}} + t^{\frac{1}{2} - 2\delta - \frac{z}{2}} ( \sup_{s \in [0,t]} s^{\delta + z} \lVert (\phi^{\sharp, u}, \phi^{\sharp, b})(s) \rVert_{\mathcal{C}^{-1 - 2\delta}} )  \nonumber 
\end{align} 
by (\ref{102a}), (\ref{102b}) and Lemma \ref{Lemma 2.6}. Thus, 
\begin{align}\label{108}
\sup_{t \in [0, T_{0}]} t^{\frac{ \frac{1}{2} - \delta_{0} + z}{2}} & \lVert y^{F}(t) \rVert_{\mathcal{C}^{\frac{1}{2} - \delta_{0}}}
\lesssim \sup_{t \in [0,T_{0}]} t^{ \frac{ \frac{1}{2} - \delta_{0} + z}{2}}[ C_{\xi}^{2} + \lVert (u^{\sharp}, b^{\sharp}) \rVert_{\mathcal{C}^{\frac{1}{2} - \delta_{0}}}]\nonumber\\
\lesssim& C_{\xi}^{2} + \sup_{t \in [0,T_{0}]} \lVert \mathcal{P}y^{\text{in}} - y^{\scalebox{0.6}{\begin{tikzpicture}
\draw[black, thick] (0,0.5) -- (0,0);
\filldraw[black] (0,0.5) circle (2pt); 
\end{tikzpicture}
}}(0) \rVert_{\mathcal{C}^{-z}} + t^{\frac{1}{2} - 2\delta - \frac{z}{2}} ( \sup_{s \in [0,t]} s^{\delta + z} \lVert (\phi^{\sharp, u}, \phi^{\sharp, b}) (s) \rVert_{\mathcal{C}^{-1 - 2\delta}}) \nonumber\\
\lesssim& C_{\xi}^{2} + C(T_{0}, C_{\xi}, \lVert y^{\text{in}} \rVert_{\mathcal{C}^{-z}}, \lVert y^{\scalebox{0.6}{\begin{tikzpicture}
\draw[black, thick] (0,0.5) -- (0,0);
\filldraw[black] (0,0.5) circle (2pt); 
\end{tikzpicture}
}}(0) \rVert_{\mathcal{C}^{-z}}) 
\end{align} 
by (\ref{99}), (\ref{107}), (\ref{106}), (\ref{44}) and Remark \ref{Remark 3.1}. By (\ref{43}) and (\ref{108}) we conclude that $T_{\epsilon} \geq T_{0}$. Finally, 
\begin{align}\label{109}
 \lVert y^{F}(t) \rVert_{\mathcal{C}^{-z}} \lesssim& \lVert (\pi_{<} (u^{
\scalebox{0.16}{\begin{tikzpicture}
\draw[black, thick] (-0.7,0.9) -- (0,0);
\draw[black, thick] (0.7,0.9) -- (0,0);
\draw[black, thick] (0.7,0) -- (0,-0.9);
\draw[snake=zigzag](0,0) -- (0,-0.9);
\draw[snake=zigzag](0,-0.9) -- (0,-1.8);
\filldraw[pink] (-0.7,0.9) circle (7pt); 
\filldraw[pink] (0.7,0.9) circle (7pt); 
\filldraw[pink] (0.7,0) circle (7pt); 
\end{tikzpicture}
}} + u^{F}, K^{u}), \pi_{<} (b^{
\scalebox{0.16}{\begin{tikzpicture}
\draw[black, thick] (-0.7,0.9) -- (0,0);
\draw[black, thick] (0.7,0.9) -- (0,0);
\draw[black, thick] (0.7,0) -- (0,-0.9);
\draw[snake=zigzag](0,0) -- (0,-0.9);
\draw[snake=zigzag](0,-0.9) -- (0,-1.8);
\filldraw[gray] (-0.7,0.9) circle (7pt); 
\filldraw[gray] (0.7,0.9) circle (7pt); 
\filldraw[gray] (0.7,0) circle (7pt); 
\end{tikzpicture}
}} + b^{F}, K^{u}) ) \rVert_{\mathcal{C}^{1-z}} \nonumber\\
&+ \lVert (\pi_{<}(u^{
\scalebox{0.16}{\begin{tikzpicture}
\draw[black, thick] (-0.7,0.9) -- (0,0);
\draw[black, thick] (0.7,0.9) -- (0,0);
\draw[black, thick] (0.7,0) -- (0,-0.9);
\draw[snake=zigzag](0,0) -- (0,-0.9);
\draw[snake=zigzag](0,-0.9) -- (0,-1.8);
\filldraw[pink] (-0.7,0.9) circle (7pt); 
\filldraw[pink] (0.7,0.9) circle (7pt); 
\filldraw[pink] (0.7,0) circle (7pt); 
\end{tikzpicture}
}} + u^{F}, K^{b}), \pi_{<}(b^{
\scalebox{0.16}{\begin{tikzpicture}
\draw[black, thick] (-0.7,0.9) -- (0,0);
\draw[black, thick] (0.7,0.9) -- (0,0);
\draw[black, thick] (0.7,0) -- (0,-0.9);
\draw[snake=zigzag](0,0) -- (0,-0.9);
\draw[snake=zigzag](0,-0.9) -- (0,-1.8);
\filldraw[gray] (-0.7,0.9) circle (7pt); 
\filldraw[gray] (0.7,0.9) circle (7pt); 
\filldraw[gray] (0.7,0) circle (7pt); 
\end{tikzpicture}
}} + b^{F}, K^{b})) \rVert_{\mathcal{C}^{1-z}} + \lVert (u^{\sharp}, b^{\sharp}) \rVert_{\mathcal{C}^{-z}} \nonumber\\
\lesssim& \lVert (u^{
\scalebox{0.16}{\begin{tikzpicture}
\draw[black, thick] (-0.7,0.9) -- (0,0);
\draw[black, thick] (0.7,0.9) -- (0,0);
\draw[black, thick] (0.7,0) -- (0,-0.9);
\draw[snake=zigzag](0,0) -- (0,-0.9);
\draw[snake=zigzag](0,-0.9) -- (0,-1.8);
\filldraw[pink] (-0.7,0.9) circle (7pt); 
\filldraw[pink] (0.7,0.9) circle (7pt); 
\filldraw[pink] (0.7,0) circle (7pt); 
\end{tikzpicture}
}} + u^{F}, b^{
\scalebox{0.16}{\begin{tikzpicture}
\draw[black, thick] (-0.7,0.9) -- (0,0);
\draw[black, thick] (0.7,0.9) -- (0,0);
\draw[black, thick] (0.7,0) -- (0,-0.9);
\draw[snake=zigzag](0,0) -- (0,-0.9);
\draw[snake=zigzag](0,-0.9) -- (0,-1.8);
\filldraw[gray] (-0.7,0.9) circle (7pt); 
\filldraw[gray] (0.7,0.9) circle (7pt); 
\filldraw[gray] (0.7,0) circle (7pt); 
\end{tikzpicture}
}} + b^{F})(t) \rVert_{\mathcal{C}^{-z}} \lVert (K^{u}, K^{b}) \rVert_{\mathcal{C}^{\frac{3}{2} - \delta}} + \lVert (u^{\sharp}, b^{\sharp}) \rVert_{\mathcal{C}^{-z}} \nonumber\\
\leq& C [ (t^{\frac{\delta}{4}}C_{\xi} + \lVert y^{F}(t) \rVert_{\mathcal{C}^{-z}} ) t^{\frac{\delta}{4}} C_{\xi} + \lVert (u^{\sharp}, b^{\sharp}) \rVert_{\mathcal{C}^{-z}}]
\end{align} 
for some constant $C \geq 0$ by (\ref{29}), (\ref{32}), Lemma \ref{Lemma 1.1} (2), (\ref{35}), (\ref{39}), (\ref{33}) and (\ref{37}). Thus, for $t \in [0, ( \frac{1}{C C_{\xi}})^{\frac{4}{\delta}})$ we have 
\begin{align}\label{110}
\lVert y^{F}(t) & \rVert_{\mathcal{C}^{-z}} \leq \frac{C}{1- C C_{\xi} t^{\frac{\delta}{4}}} [ C_{\xi}^{2} t^{\frac{\delta}{4}} + \lVert (u^{\sharp}, b^{\sharp})(t) \rVert_{\mathcal{C}^{-z}}]\nonumber\\
\lesssim& C_{\xi}^{2} + \lVert y^{\text{in}}\rVert_{\mathcal{C}^{-z}} + \lVert y^{\scalebox{0.6}{\begin{tikzpicture}
\draw[black, thick] (0,0.5) -- (0,0);
\filldraw[black] (0,0.5) circle (2pt); 
\end{tikzpicture}
}}(0)\rVert_{\mathcal{C}^{-z}} + \sup_{s \in [0,T]} s^{\delta + z} \lVert \phi^{\sharp} (s) \rVert_{\mathcal{C}^{-1 - 2\delta}} \int_{0}^{t}(t-r)^{- \frac{(1+ 2 \delta - z)}{2}} r^{-(\delta +z)} dr \nonumber\\
\lesssim& C(T, C_{\xi}, \lVert y^{\text{in}} \rVert_{\mathcal{C}^{-z}}, \lVert y^{\scalebox{0.6}{\begin{tikzpicture}
\draw[black, thick] (0,0.5) -- (0,0);
\filldraw[black] (0,0.5) circle (2pt); 
\end{tikzpicture}
}}(0)\rVert_{\mathcal{C}^{-z}}) 
\end{align} 
by (\ref{109}), (\ref{102a}), (\ref{102b}), Lemma \ref{Lemma 2.6} and (\ref{106}). Based on (\ref{37}) we now define  
\begin{align}
\mathbb{Z}( \xi^{\epsilon}) \triangleq& ( u^{\scalebox{0.6}{\begin{tikzpicture}
\draw[black, thick] (0,0.5) -- (0,0);
\filldraw[red] (0,0.5) circle (2pt); 
\end{tikzpicture}
}\epsilon}, b^{
\scalebox{0.6}{\begin{tikzpicture}
\draw[black, thick] (0,0.5) -- (0,0);
\filldraw[blue] (0,0.5) circle (2pt); 
\end{tikzpicture}
}\epsilon}, u^{\scalebox{0.6}{\begin{tikzpicture}
\draw[black, thick] (0,0.5) -- (0,0);
\filldraw[red] (0,0.5) circle (2pt); 
\end{tikzpicture}
}\epsilon} \diamond u^{\scalebox{0.6}{\begin{tikzpicture}
\draw[black, thick] (0,0.5) -- (0,0);
\filldraw[red] (0,0.5) circle (2pt); 
\end{tikzpicture}
}\epsilon}, b^{
\scalebox{0.6}{\begin{tikzpicture}
\draw[black, thick] (0,0.5) -- (0,0);
\filldraw[blue] (0,0.5) circle (2pt); 
\end{tikzpicture}
}\epsilon}\diamond b^{
\scalebox{0.6}{\begin{tikzpicture}
\draw[black, thick] (0,0.5) -- (0,0);
\filldraw[blue] (0,0.5) circle (2pt); 
\end{tikzpicture}
}\epsilon}, u^{\scalebox{0.6}{\begin{tikzpicture}
\draw[black, thick] (0,0.5) -- (0,0);
\filldraw[red] (0,0.5) circle (2pt); 
\end{tikzpicture}
}\epsilon} \diamond b^{
\scalebox{0.6}{\begin{tikzpicture}
\draw[black, thick] (0,0.5) -- (0,0);
\filldraw[blue] (0,0.5) circle (2pt); 
\end{tikzpicture}
}\epsilon}, b^{
\scalebox{0.6}{\begin{tikzpicture}
\draw[black, thick] (0,0.5) -- (0,0);
\filldraw[blue] (0,0.5) circle (2pt); 
\end{tikzpicture}
}\epsilon} \diamond u^{\scalebox{0.6}{\begin{tikzpicture}
\draw[black, thick] (0,0.5) -- (0,0);
\filldraw[red] (0,0.5) circle (2pt); 
\end{tikzpicture}
}\epsilon}, \label{111}\\
&  u^{\scalebox{0.6}{\begin{tikzpicture}
\draw[black, thick] (0,0.5) -- (0,0);
\filldraw[red] (0,0.5) circle (2pt); 
\end{tikzpicture}
}\epsilon} \diamond u^{
\scalebox{0.18}{\begin{tikzpicture}
\draw[black, thick] (-0.7,0.9) -- (0,0);
\draw[black, thick] (0.7,0.9) -- (0,0);
\draw[snake=zigzag](0,0) -- (0,-0.9);
\filldraw[green] (-0.7,0.9) circle (6pt); 
\filldraw[green] (0.7,0.9) circle (6pt); 
\end{tikzpicture}
}\epsilon}, b^{
\scalebox{0.6}{\begin{tikzpicture}
\draw[black, thick] (0,0.5) -- (0,0);
\filldraw[blue] (0,0.5) circle (2pt); 
\end{tikzpicture}
}\epsilon} \diamond b^{
\scalebox{0.18}{\begin{tikzpicture}
\draw[black, thick] (-0.7,0.9) -- (0,0);
\draw[black, thick] (0.7,0.9) -- (0,0);
\draw[snake=zigzag](0,0) -- (0,-0.9);
\filldraw[violet] (-0.7,0.9) circle (6pt); 
\filldraw[violet] (0.7,0.9) circle (6pt); 
\end{tikzpicture}
}\epsilon}, b^{
\scalebox{0.6}{\begin{tikzpicture}
\draw[black, thick] (0,0.5) -- (0,0);
\filldraw[blue] (0,0.5) circle (2pt); 
\end{tikzpicture}
}\epsilon} \diamond u^{
\scalebox{0.18}{\begin{tikzpicture}
\draw[black, thick] (-0.7,0.9) -- (0,0);
\draw[black, thick] (0.7,0.9) -- (0,0);
\draw[snake=zigzag](0,0) -- (0,-0.9);
\filldraw[green] (-0.7,0.9) circle (6pt); 
\filldraw[green] (0.7,0.9) circle (6pt); 
\end{tikzpicture}
}\epsilon}, b^{
\scalebox{0.18}{\begin{tikzpicture}
\draw[black, thick] (-0.7,0.9) -- (0,0);
\draw[black, thick] (0.7,0.9) -- (0,0);
\draw[snake=zigzag](0,0) -- (0,-0.9);
\filldraw[violet] (-0.7,0.9) circle (6pt); 
\filldraw[violet] (0.7,0.9) circle (6pt); 
\end{tikzpicture}
}\epsilon} \diamond u^{\scalebox{0.6}{\begin{tikzpicture}
\draw[black, thick] (0,0.5) -- (0,0);
\filldraw[red] (0,0.5) circle (2pt); 
\end{tikzpicture}
}\epsilon},  \nonumber\\
&u^{
\scalebox{0.18}{\begin{tikzpicture}
\draw[black, thick] (-0.7,0.9) -- (0,0);
\draw[black, thick] (0.7,0.9) -- (0,0);
\draw[snake=zigzag](0,0) -- (0,-0.9);
\filldraw[green] (-0.7,0.9) circle (6pt); 
\filldraw[green] (0.7,0.9) circle (6pt); 
\end{tikzpicture}
}\epsilon} \diamond u^{
\scalebox{0.18}{\begin{tikzpicture}
\draw[black, thick] (-0.7,0.9) -- (0,0);
\draw[black, thick] (0.7,0.9) -- (0,0);
\draw[snake=zigzag](0,0) -- (0,-0.9);
\filldraw[green] (-0.7,0.9) circle (6pt); 
\filldraw[green] (0.7,0.9) circle (6pt); 
\end{tikzpicture}
}\epsilon},  b^{
\scalebox{0.18}{\begin{tikzpicture}
\draw[black, thick] (-0.7,0.9) -- (0,0);
\draw[black, thick] (0.7,0.9) -- (0,0);
\draw[snake=zigzag](0,0) -- (0,-0.9);
\filldraw[violet] (-0.7,0.9) circle (6pt); 
\filldraw[violet] (0.7,0.9) circle (6pt); 
\end{tikzpicture}
}\epsilon} \diamond b^{
\scalebox{0.18}{\begin{tikzpicture}
\draw[black, thick] (-0.7,0.9) -- (0,0);
\draw[black, thick] (0.7,0.9) -- (0,0);
\draw[snake=zigzag](0,0) -- (0,-0.9);
\filldraw[violet] (-0.7,0.9) circle (6pt); 
\filldraw[violet] (0.7,0.9) circle (6pt); 
\end{tikzpicture}
}\epsilon}, b^{
\scalebox{0.18}{\begin{tikzpicture}
\draw[black, thick] (-0.7,0.9) -- (0,0);
\draw[black, thick] (0.7,0.9) -- (0,0);
\draw[snake=zigzag](0,0) -- (0,-0.9);
\filldraw[violet] (-0.7,0.9) circle (6pt); 
\filldraw[violet] (0.7,0.9) circle (6pt); 
\end{tikzpicture}
}\epsilon} \diamond u^{
\scalebox{0.18}{\begin{tikzpicture}
\draw[black, thick] (-0.7,0.9) -- (0,0);
\draw[black, thick] (0.7,0.9) -- (0,0);
\draw[snake=zigzag](0,0) -- (0,-0.9);
\filldraw[green] (-0.7,0.9) circle (6pt); 
\filldraw[green] (0.7,0.9) circle (6pt); 
\end{tikzpicture}
}\epsilon}, \nonumber\\
&\pi_{0,\diamond} (u^{
\scalebox{0.16}{\begin{tikzpicture}
\draw[black, thick] (-0.7,0.9) -- (0,0);
\draw[black, thick] (0.7,0.9) -- (0,0);
\draw[black, thick] (0.7,0) -- (0,-0.9);
\draw[snake=zigzag](0,0) -- (0,-0.9);
\draw[snake=zigzag](0,-0.9) -- (0,-1.8);
\filldraw[pink] (-0.7,0.9) circle (7pt); 
\filldraw[pink] (0.7,0.9) circle (7pt); 
\filldraw[pink] (0.7,0) circle (7pt); 
\end{tikzpicture}
}\epsilon}, u^{\scalebox{0.6}{\begin{tikzpicture}
\draw[black, thick] (0,0.5) -- (0,0);
\filldraw[red] (0,0.5) circle (2pt); 
\end{tikzpicture}
}\epsilon}), \pi_{0, \diamond} (b^{
\scalebox{0.16}{\begin{tikzpicture}
\draw[black, thick] (-0.7,0.9) -- (0,0);
\draw[black, thick] (0.7,0.9) -- (0,0);
\draw[black, thick] (0.7,0) -- (0,-0.9);
\draw[snake=zigzag](0,0) -- (0,-0.9);
\draw[snake=zigzag](0,-0.9) -- (0,-1.8);
\filldraw[gray] (-0.7,0.9) circle (7pt); 
\filldraw[gray] (0.7,0.9) circle (7pt); 
\filldraw[gray] (0.7,0) circle (7pt); 
\end{tikzpicture}
}\epsilon}, b^{
\scalebox{0.6}{\begin{tikzpicture}
\draw[black, thick] (0,0.5) -- (0,0);
\filldraw[blue] (0,0.5) circle (2pt); 
\end{tikzpicture}
}\epsilon}), \pi_{0,\diamond}(u^{
\scalebox{0.16}{\begin{tikzpicture}
\draw[black, thick] (-0.7,0.9) -- (0,0);
\draw[black, thick] (0.7,0.9) -- (0,0);
\draw[black, thick] (0.7,0) -- (0,-0.9);
\draw[snake=zigzag](0,0) -- (0,-0.9);
\draw[snake=zigzag](0,-0.9) -- (0,-1.8);
\filldraw[pink] (-0.7,0.9) circle (7pt); 
\filldraw[pink] (0.7,0.9) circle (7pt); 
\filldraw[pink] (0.7,0) circle (7pt); 
\end{tikzpicture}
}\epsilon},  b^{
\scalebox{0.6}{\begin{tikzpicture}
\draw[black, thick] (0,0.5) -- (0,0);
\filldraw[blue] (0,0.5) circle (2pt); 
\end{tikzpicture}
}\epsilon}), \pi_{0,\diamond} (u^{
\scalebox{0.16}{\begin{tikzpicture}
\draw[black, thick] (-0.7,0.9) -- (0,0);
\draw[black, thick] (0.7,0.9) -- (0,0);
\draw[black, thick] (0.7,0) -- (0,-0.9);
\draw[snake=zigzag](0,0) -- (0,-0.9);
\draw[snake=zigzag](0,-0.9) -- (0,-1.8);
\filldraw[gray] (-0.7,0.9) circle (7pt); 
\filldraw[gray] (0.7,0.9) circle (7pt); 
\filldraw[gray] (0.7,0) circle (7pt); 
\end{tikzpicture}
}\epsilon}, u^{\scalebox{0.6}{\begin{tikzpicture}
\draw[black, thick] (0,0.5) -- (0,0);
\filldraw[red] (0,0.5) circle (2pt); 
\end{tikzpicture}
}\epsilon}), \nonumber\\
& \pi_{0,\diamond} (\mathcal{P} D K^{u,\epsilon}, u^{\scalebox{0.6}{\begin{tikzpicture}
\draw[black, thick] (0,0.5) -- (0,0);
\filldraw[red] (0,0.5) circle (2pt); 
\end{tikzpicture}
}\epsilon}), \pi_{0,\diamond} (\mathcal{P} D K^{b,\epsilon}, u^{\scalebox{0.6}{\begin{tikzpicture}
\draw[black, thick] (0,0.5) -- (0,0);
\filldraw[red] (0,0.5) circle (2pt); 
\end{tikzpicture}
}\epsilon}), \pi_{0,\diamond} (\mathcal{P} D K^{u,\epsilon}, b^{
\scalebox{0.6}{\begin{tikzpicture}
\draw[black, thick] (0,0.5) -- (0,0);
\filldraw[blue] (0,0.5) circle (2pt); 
\end{tikzpicture}
}\epsilon}), \pi_{0,\diamond} (\mathcal{P} D K^{b,\epsilon}, b^{
\scalebox{0.6}{\begin{tikzpicture}
\draw[black, thick] (0,0.5) -- (0,0);
\filldraw[blue] (0,0.5) circle (2pt); 
\end{tikzpicture}
}\epsilon}))\nonumber\\
\in& \mathbb{X} \triangleq C([0,T]; \mathcal{C}^{-\frac{1}{2} - \frac{\delta}{2}})^{2} \times C([0,T]; \mathcal{C}^{-1- \frac{\delta}{2}})^{4} C([0,T]; \mathcal{C}^{-\frac{1}{2} - \frac{\delta}{2}})^{4} \times C([0,T]; \mathcal{C}^{-\delta})^{11},  \nonumber 
\end{align} 
equipped with product topology. Then we may show via similar arguments that for all $a > 0$, there exists $T_{0} > 0$ sufficiently small such that the mapping $(y^{\text{in}}, \mathbb{Z}(\xi^{\epsilon})) \mapsto (u^{F}, b^{F})$ is Lipschitz in a norm of  $C([0, T_{0}]; \mathcal{C}^{-z})$ on the set $\{ (y^{\text{in}}, \mathbb{Z}(\xi^{\epsilon})): \max \{ \lVert y^{\text{in}} \rVert_{\mathcal{C}^{-z}}, C_{\xi} \} \leq a \}$. This implies the following result. 
\begin{proposition}\label{Proposition 3.1}
Let $\delta_{0} \in (0, \frac{1}{2})$, $z \in (\frac{1}{2}, \frac{1}{2} + \delta_{0})$ and $(\xi^{\epsilon})_{\epsilon > 0}$ be a family of smooth functions converging to $\xi$ as $\epsilon \to 0$. Suppose that for any $\epsilon > 0, y^{\text{in}} \in \mathcal{C}^{-z}$ given, $y^{\epsilon}$ is the unique maximal solution to 
\begin{subequations}\label{112}
\begin{align}
Lu_{i} = \sum_{i_{1} = 1}^{3} \mathcal{P}_{ii_{1}} \xi_{i_{1}}^{u} - \frac{1}{2} \sum_{i_{1}, j = 1}^{3} \mathcal{P}_{ii_{1}}  \partial_{x_{j}} (u_{i} u_{j}) + \frac{1}{2} \sum_{i_{1}, j=1}^{3} \mathcal{P}_{ii_{1}}  \partial_{x_{j}}  (b_{i} b_{j}), \\
Lb_{i} = \sum_{i_{1} =1}^{3} \mathcal{P}_{ii_{1}} \xi_{i_{1}}^{b} - \frac{1}{2} \sum_{i_{1}, j=1}^{3} \mathcal{P}_{ii_{1}}  \partial_{x_{j}} (b_{i} u_{j}) + \frac{1}{2} \sum_{i_{1}, j=1}^{3} \mathcal{P}_{ii_{1}}  \partial_{x_{j}}  (u_{i} b_{j}),\\
y^{\epsilon} (\cdot, 0) = \mathcal{P} y^{\text{in}}(\cdot)
\end{align}
\end{subequations} 
such that $y^{F,\epsilon} = (u^{F,\epsilon}, b^{F, \epsilon}) \in (C((0, T_{\epsilon}); \mathcal{C}^{\frac{1}{2} - \delta_{0}}))^{2}$. Suppose that $\mathbb{Z}( \xi^{\epsilon})$ converges in $\mathbb{X}$ so that for $i, i_{1}, j, j_{1} \in \{1,2,3\}$, there exist families 
\begin{align*}
&v^{\scalebox{0.6}{

}\epsilon}_{j_{1}}),
\end{align} 
with $\{C_{0,k}^{\epsilon, ij}\}_{\epsilon> 0}, \{C_{2,k}^{\epsilon, ij} \}_{\epsilon > 0}, \{C_{1,k}^{\epsilon, ij} \}_{\epsilon > 0} \subset \mathbb{R}$ for $k \in \{1,2,3, 4\}$ to be specified subsequently, e.g.,  $C_{0,1}^{\epsilon, ij}$, $C_{2,3}^{\epsilon, ij}$ and $C_{1,3}^{\epsilon, ij}$ in (\ref{117}), (\ref{162}) and (\ref{195}), respectively. Then there exists a unique $y \in C([0, T]; \mathcal{C}^{-z})^{2}$ where $T = T(y^{\text{in}}, v^{\scalebox{0.6}{%
}}_{1}, \hdots, v_{21}) > 0$ such that $\lim_{\epsilon \to 0} \lVert y^{\epsilon} - y \rVert_{C([0, T]; \mathcal{C}^{-z})} = 0$, and $y$ depends only on $(y^{\text{in}}, v^{\scalebox{0.6}{%
}}_{1}, \hdots, v_{21})$, and not on the approximating family. 
\end{proposition} 
Details of the proof of Proposition \ref{Proposition 3.1} can be found in \cite[Theorem 3.11, Proposition 3.12, and Corollary 3.13]{CC18} (see also \cite[Remark 3.9]{ZZ15} in the case of the NSE). This concludes the fixed point procedure of the proof of Theorem \ref{Theorem 1.2}. 

\section{Proof of Theorem \ref{Theorem 1.2}: Renormalization} 
Hereafter let us write $X_{t}^{u} \triangleq u^{\scalebox{0.6}{%
}}(t)$ where $y = (u_{1}, u_{2}, u_{3}, b_{1}, b_{2}, b_{3})$ and following \cite[Notation 4.1]{CC18}, for $k_{1}, \hdots, k_{n} \in \mathbb{Z}^{3}$, we also write $k_{1, \hdots, n} \triangleq \sum_{i=1}^{n} k_{i}$. Since $X_{t,i}^{u} = u^{\scalebox{0.6}{%
}}_{i}(t)$, we have 
\begin{equation}\label{115}
X_{t,i}^{u} = \sum_{k \neq 0} \hat{X}_{t,i}^{u}(k) e_{k}, \hspace{3mm} X_{t,i}^{b} = \sum_{k \neq 0} \hat{X}_{t,i}^{b}(k) e_{k}, \hspace{3mm} e_{k} \triangleq (2\pi)^{-\frac{3}{2}} e^{ix\cdot k} 
\end{equation} 
where $\hat{X}_{t}^{u}(0) = 0, \hat{X}_{t}^{b}(0) = 0$ due to mean-zero property of $\xi^{u}$ and $\xi^{b}$ and  
\begin{subequations}\label{116}
\begin{align}
& \mathbb{E} [ \hat{X}_{t,i}^{u} (k) \hat{X}_{s,j}^{u}(k') ] = 1_{k + k'= 0} \sum_{i_{1} =1}^{3} \frac{e^{ - \lvert k \rvert^{2} \lvert t-s \rvert}}{2 \lvert k \rvert^{2}} \hat{\mathcal{P}}_{ii_{1}}(k) \hat{\mathcal{P}}_{j i_{1}}(k),\\
& \mathbb{E} [ \hat{X}_{t,i}^{b} (k) \hat{X}_{s,j}^{b}(k') ] = 1_{k + k'= 0} \sum_{i_{1} =1}^{3} \frac{e^{ - \lvert k \rvert^{2} \lvert t-s \rvert}}{2 \lvert k \rvert^{2}}\hat{\mathcal{P}}_{ii_{1}}(k) \hat{\mathcal{P}}_{j i_{1}}(k),\\
& \mathbb{E} [ \hat{X}_{t,i}^{u} (k) \hat{X}_{s,j}^{b}(k') ] = 1_{k + k'= 0} \sum_{i_{1} =1}^{3} \frac{e^{ - \lvert k \rvert^{2} \lvert t-s \rvert}}{2 \lvert k \rvert^{2}}\hat{\mathcal{P}}_{ii_{1}}(k) \hat{\mathcal{P}}_{j i_{1}}(k),\\
& \mathbb{E} [ \hat{X}_{t,i}^{b} (k) \hat{X}_{s,j}^{u}(k') ] = 1_{k + k'= 0} \sum_{i_{1} =1}^{3} \frac{e^{ - \lvert k \rvert^{2} \lvert t-s \rvert}}{2 \lvert k \rvert^{2}}\hat{\mathcal{P}}_{ii_{1}}(k) \hat{\mathcal{P}}_{j i_{1}}(k),
\end{align}
\end{subequations} 
for $k \in \mathbb{Z}^{3} \setminus \{0 \}$ due to (\ref{16}). We regularize $\xi$ by $\xi^{\epsilon} \triangleq \sum_{k} f(\epsilon k) \hat{\xi}(k) e_{k}$ where $f$ is a smooth radial cut-off function with compact support such that $f(0) = 1$ so that 
\begin{align*}
X_{t,i}^{u, \epsilon} = \int_{-\infty}^{t} \sum_{i_{1} = 1}^{3} \mathcal{P}_{ii_{1}} P_{t-s} \sum_{k \neq 0} f(\epsilon k) \hat{\xi}^{u, \epsilon}_{i_{1}} (k,s) ds,  X_{t,i}^{b, \epsilon} = \int_{-\infty}^{t} \sum_{i_{1} = 1}^{3} \mathcal{P}_{ii_{1}} P_{t-s} \sum_{k \neq 0} f(\epsilon k) \hat{\xi}^{b, \epsilon}_{i_{1}} (k,s) ds, 
\end{align*}
and the covariance of $X_{t,i}^{u, \epsilon}, X_{t,i}^{b, \epsilon}$ follow from (\ref{116}), only multiplied by $f(\epsilon k)^{2}$. 

We now devote ourselves to convergence and renormalizations. First, the existence of $v^{\scalebox{0.6}{\begin{tikzpicture}
\draw[black, thick] (0,0.5) -- (0,0);
\filldraw[red] (0,0.5) circle (2pt); 
\end{tikzpicture}
}}_{1}, v^{
\scalebox{0.6}{\begin{tikzpicture}
\draw[black, thick] (0,0.5) -- (0,0);
\filldraw[blue] (0,0.5) circle (2pt); 
\end{tikzpicture}
}}_{2}$ such that $u^{\scalebox{0.6}{\begin{tikzpicture}
\draw[black, thick] (0,0.5) -- (0,0);
\filldraw[red] (0,0.5) circle (2pt); 
\end{tikzpicture}
}\epsilon} \to v^{\scalebox{0.6}{\begin{tikzpicture}
\draw[black, thick] (0,0.5) -- (0,0);
\filldraw[red] (0,0.5) circle (2pt); 
\end{tikzpicture}
}}_{1}, b^{
\scalebox{0.6}{\begin{tikzpicture}
\draw[black, thick] (0,0.5) -- (0,0);
\filldraw[blue] (0,0.5) circle (2pt); 
\end{tikzpicture}
}\epsilon} \to v^{
\scalebox{0.6}{\begin{tikzpicture}
\draw[black, thick] (0,0.5) -- (0,0);
\filldraw[blue] (0,0.5) circle (2pt); 
\end{tikzpicture}
}}_{2}$ in $L^{p}(\Omega; C([0, T]; \mathcal{C}^{-\frac{1}{2} - \frac{\delta}{2}}))$ for all $p \geq 1$ as $\epsilon \to 0$ is immediate from (\ref{16}). Second, the convergence issues of 
\begin{align*}
& u^{\scalebox{0.6}{\begin{tikzpicture}
\draw[black, thick] (0,0.5) -- (0,0);
\filldraw[red] (0,0.5) circle (2pt); 
\end{tikzpicture}
}\epsilon}_{i} \diamond u^{\scalebox{0.6}{\begin{tikzpicture}
\draw[black, thick] (0,0.5) -- (0,0);
\filldraw[red] (0,0.5) circle (2pt); 
\end{tikzpicture}
}\epsilon}_{j} = u^{\scalebox{0.6}{\begin{tikzpicture}
\draw[black, thick] (0,0.5) -- (0,0);
\filldraw[red] (0,0.5) circle (2pt); 
\end{tikzpicture}
}\epsilon}_{i} u^{\scalebox{0.6}{\begin{tikzpicture}
\draw[black, thick] (0,0.5) -- (0,0);
\filldraw[red] (0,0.5) circle (2pt); 
\end{tikzpicture}
}\epsilon}_{j} - C_{0,1}^{\epsilon, ij} \to v^{
\scalebox{0.6}{\begin{tikzpicture}
\draw[black, thick] (-0.2,0.5) -- (0,0);
\draw[black, thick] (0.2,0.5) -- (0,0);
\filldraw[red] (-0.2,0.5) circle (2pt); 
\filldraw[red] (0.2,0.5) circle (2pt); 
\end{tikzpicture}
}}_{3,ij}, \hspace{2mm} b^{
\scalebox{0.6}{\begin{tikzpicture}
\draw[black, thick] (0,0.5) -- (0,0);
\filldraw[blue] (0,0.5) circle (2pt); 
\end{tikzpicture}
}\epsilon}_{i} \diamond b^{
\scalebox{0.6}{\begin{tikzpicture}
\draw[black, thick] (0,0.5) -- (0,0);
\filldraw[blue] (0,0.5) circle (2pt); 
\end{tikzpicture}
}\epsilon}_{j} = b^{
\scalebox{0.6}{\begin{tikzpicture}
\draw[black, thick] (0,0.5) -- (0,0);
\filldraw[blue] (0,0.5) circle (2pt); 
\end{tikzpicture}
}\epsilon}_{i} b^{
\scalebox{0.6}{\begin{tikzpicture}
\draw[black, thick] (0,0.5) -- (0,0);
\filldraw[blue] (0,0.5) circle (2pt); 
\end{tikzpicture}
}\epsilon}_{j} - C_{0,2}^{\epsilon, ij} \to v^{
\scalebox{0.6}{\begin{tikzpicture}
\draw[black, thick] (-0.2,0.5) -- (0,0);
\draw[black, thick] (0.2,0.5) -- (0,0);
\filldraw[blue] (-0.2,0.5) circle (2pt); 
\filldraw[blue] (0.2,0.5) circle (2pt); 
\end{tikzpicture}
}}_{4,ij}, \\
&u^{\scalebox{0.6}{\begin{tikzpicture}
\draw[black, thick] (0,0.5) -- (0,0);
\filldraw[red] (0,0.5) circle (2pt); 
\end{tikzpicture}
}\epsilon}_{i} \diamond b^{
\scalebox{0.6}{\begin{tikzpicture}
\draw[black, thick] (0,0.5) -- (0,0);
\filldraw[blue] (0,0.5) circle (2pt); 
\end{tikzpicture}
}\epsilon}_{j} = u^{\scalebox{0.6}{\begin{tikzpicture}
\draw[black, thick] (0,0.5) -- (0,0);
\filldraw[red] (0,0.5) circle (2pt); 
\end{tikzpicture}
}\epsilon}_{i} b^{
\scalebox{0.6}{\begin{tikzpicture}
\draw[black, thick] (0,0.5) -- (0,0);
\filldraw[blue] (0,0.5) circle (2pt); 
\end{tikzpicture}
}\epsilon}_{j} - C_{0,3}^{\epsilon, ij} \to v^{
\scalebox{0.6}{\begin{tikzpicture}
\draw[black, thick] (-0.2,0.5) -- (0,0);
\draw[black, thick] (0.2,0.5) -- (0,0);
\filldraw[red] (-0.2,0.5) circle (2pt); 
\filldraw[blue] (0.2,0.5) circle (2pt); 
\end{tikzpicture}
}}_{5,ij}, \hspace{2mm} b^{
\scalebox{0.6}{\begin{tikzpicture}
\draw[black, thick] (0,0.5) -- (0,0);
\filldraw[blue] (0,0.5) circle (2pt); 
\end{tikzpicture}
}\epsilon}_{i} \diamond u^{\scalebox{0.6}{\begin{tikzpicture}
\draw[black, thick] (0,0.5) -- (0,0);
\filldraw[red] (0,0.5) circle (2pt); 
\end{tikzpicture}
}\epsilon}_{j} = b^{
\scalebox{0.6}{\begin{tikzpicture}
\draw[black, thick] (0,0.5) -- (0,0);
\filldraw[blue] (0,0.5) circle (2pt); 
\end{tikzpicture}
}\epsilon}_{i} u^{\scalebox{0.6}{\begin{tikzpicture}
\draw[black, thick] (0,0.5) -- (0,0);
\filldraw[red] (0,0.5) circle (2pt); 
\end{tikzpicture}
}\epsilon}_{j} - C_{0,4}^{\epsilon, ij} \to v^{
\scalebox{0.6}{\begin{tikzpicture}
\draw[black, thick] (-0.2,0.5) -- (0,0);
\draw[black, thick] (0.2,0.5) -- (0,0);
\filldraw[blue] (-0.2,0.5) circle (2pt); 
\filldraw[red] (0.2,0.5) circle (2pt); 
\end{tikzpicture}
}}_{6,ij} 
\end{align*}
by (\ref{114}) in $L^{p}(\Omega; C([0,T]; \mathcal{C}^{-1-\frac{\delta}{2}}))$ for all $p \geq 1$ as $\epsilon \to 0$ are clear because $: \xi_{1} \xi_{2} : = \xi_{1} \xi_{2} - \mathbb{E} [ \xi_{1} \xi_{2}]$ (see \cite{J97}) so that e.g.,  
\begin{equation}\label{117}
C_{0, 1}^{\epsilon, ij} = \mathbb{E} [u^{\scalebox{0.6}{\begin{tikzpicture}
\draw[black, thick] (0,0.5) -- (0,0);
\filldraw[red] (0,0.5) circle (2pt); 
\end{tikzpicture}
}\epsilon}_{i}(t) u^{\scalebox{0.6}{\begin{tikzpicture}
\draw[black, thick] (0,0.5) -- (0,0);
\filldraw[red] (0,0.5) circle (2pt); 
\end{tikzpicture}
}\epsilon}_{j}(t)] = (2\pi)^{-\frac{3}{2}} \sum_{k_{1} \neq 0} \sum_{i_{1} =1}^{3} \frac{f(\epsilon k_{1})^{2}}{2 \lvert k_{1} \rvert^{2}} \hat{\mathcal{P}}_{ii_{1}}(k_{1}) \hat{\mathcal{P}}_{ji_{1}}(k_{1})
\end{equation} 
by (\ref{115}) and (\ref{116}). It follows that $C_{0, 1}^{\epsilon, ij} \to \infty$ as $\epsilon \searrow 0$. 

We need to perform renormalizations on the following groups in (\ref{114}); 
\begin{enumerate}
\item a first group of $u^{\scalebox{0.6}{\begin{tikzpicture}
\draw[black, thick] (0,0.5) -- (0,0);
\filldraw[red] (0,0.5) circle (2pt); 
\end{tikzpicture}
}\epsilon}_{i} \diamond u^{
\scalebox{0.18}{\begin{tikzpicture}
\draw[black, thick] (-0.7,0.9) -- (0,0);
\draw[black, thick] (0.7,0.9) -- (0,0);
\draw[snake=zigzag](0,0) -- (0,-0.9);
\filldraw[green] (-0.7,0.9) circle (6pt); 
\filldraw[green] (0.7,0.9) circle (6pt); 
\end{tikzpicture}
}\epsilon}_{j}$, $b^{
\scalebox{0.6}{\begin{tikzpicture}
\draw[black, thick] (0,0.5) -- (0,0);
\filldraw[blue] (0,0.5) circle (2pt); 
\end{tikzpicture}
}\epsilon}_{i} \diamond b^{
\scalebox{0.18}{\begin{tikzpicture}
\draw[black, thick] (-0.7,0.9) -- (0,0);
\draw[black, thick] (0.7,0.9) -- (0,0);
\draw[snake=zigzag](0,0) -- (0,-0.9);
\filldraw[violet] (-0.7,0.9) circle (6pt); 
\filldraw[violet] (0.7,0.9) circle (6pt); 
\end{tikzpicture}
}\epsilon}_{j}$, $u^{\scalebox{0.6}{\begin{tikzpicture}
\draw[black, thick] (0,0.5) -- (0,0);
\filldraw[red] (0,0.5) circle (2pt); 
\end{tikzpicture}
}\epsilon}_{i} \diamond b^{
\scalebox{0.18}{\begin{tikzpicture}
\draw[black, thick] (-0.7,0.9) -- (0,0);
\draw[black, thick] (0.7,0.9) -- (0,0);
\draw[snake=zigzag](0,0) -- (0,-0.9);
\filldraw[violet] (-0.7,0.9) circle (6pt); 
\filldraw[violet] (0.7,0.9) circle (6pt); 
\end{tikzpicture}
}\epsilon}_{j}$ and $b^{
\scalebox{0.6}{\begin{tikzpicture}
\draw[black, thick] (0,0.5) -- (0,0);
\filldraw[blue] (0,0.5) circle (2pt); 
\end{tikzpicture}
}\epsilon}_{i} \diamond u^{
\scalebox{0.18}{\begin{tikzpicture}
\draw[black, thick] (-0.7,0.9) -- (0,0);
\draw[black, thick] (0.7,0.9) -- (0,0);
\draw[snake=zigzag](0,0) -- (0,-0.9);
\filldraw[green] (-0.7,0.9) circle (6pt); 
\filldraw[green] (0.7,0.9) circle (6pt); 
\end{tikzpicture}
}\epsilon}_{j}$, 
\item a second group of $u^{
\scalebox{0.18}{\begin{tikzpicture}
\draw[black, thick] (-0.7,0.9) -- (0,0);
\draw[black, thick] (0.7,0.9) -- (0,0);
\draw[snake=zigzag](0,0) -- (0,-0.9);
\filldraw[green] (-0.7,0.9) circle (6pt); 
\filldraw[green] (0.7,0.9) circle (6pt); 
\end{tikzpicture}
}\epsilon}_{i} \diamond u^{
\scalebox{0.18}{\begin{tikzpicture}
\draw[black, thick] (-0.7,0.9) -- (0,0);
\draw[black, thick] (0.7,0.9) -- (0,0);
\draw[snake=zigzag](0,0) -- (0,-0.9);
\filldraw[green] (-0.7,0.9) circle (6pt); 
\filldraw[green] (0.7,0.9) circle (6pt); 
\end{tikzpicture}
}\epsilon}_{j}$, $b^{
\scalebox{0.18}{\begin{tikzpicture}
\draw[black, thick] (-0.7,0.9) -- (0,0);
\draw[black, thick] (0.7,0.9) -- (0,0);
\draw[snake=zigzag](0,0) -- (0,-0.9);
\filldraw[violet] (-0.7,0.9) circle (6pt); 
\filldraw[violet] (0.7,0.9) circle (6pt); 
\end{tikzpicture}
}\epsilon}_{i} \diamond b^{
\scalebox{0.18}{\begin{tikzpicture}
\draw[black, thick] (-0.7,0.9) -- (0,0);
\draw[black, thick] (0.7,0.9) -- (0,0);
\draw[snake=zigzag](0,0) -- (0,-0.9);
\filldraw[violet] (-0.7,0.9) circle (6pt); 
\filldraw[violet] (0.7,0.9) circle (6pt); 
\end{tikzpicture}
}\epsilon}_{j}$ and $b^{
\scalebox{0.18}{\begin{tikzpicture}
\draw[black, thick] (-0.7,0.9) -- (0,0);
\draw[black, thick] (0.7,0.9) -- (0,0);
\draw[snake=zigzag](0,0) -- (0,-0.9);
\filldraw[violet] (-0.7,0.9) circle (6pt); 
\filldraw[violet] (0.7,0.9) circle (6pt); 
\end{tikzpicture}
}\epsilon}_{i} \diamond u^{
\scalebox{0.18}{\begin{tikzpicture}
\draw[black, thick] (-0.7,0.9) -- (0,0);
\draw[black, thick] (0.7,0.9) -- (0,0);
\draw[snake=zigzag](0,0) -- (0,-0.9);
\filldraw[green] (-0.7,0.9) circle (6pt); 
\filldraw[green] (0.7,0.9) circle (6pt); 
\end{tikzpicture}
}\epsilon}_{j}$, 
\item a third group of $\pi_{0, \diamond} (u^{
\scalebox{0.16}{\begin{tikzpicture}
\draw[black, thick] (-0.7,0.9) -- (0,0);
\draw[black, thick] (0.7,0.9) -- (0,0);
\draw[black, thick] (0.7,0) -- (0,-0.9);
\draw[snake=zigzag](0,0) -- (0,-0.9);
\draw[snake=zigzag](0,-0.9) -- (0,-1.8);
\filldraw[pink] (-0.7,0.9) circle (7pt); 
\filldraw[pink] (0.7,0.9) circle (7pt); 
\filldraw[pink] (0.7,0) circle (7pt); 
\end{tikzpicture}
}\epsilon}_{i}, u^{\scalebox{0.6}{\begin{tikzpicture}
\draw[black, thick] (0,0.5) -- (0,0);
\filldraw[red] (0,0.5) circle (2pt); 
\end{tikzpicture}
}\epsilon}_{j})$, $\pi_{0, \diamond} (u^{
\scalebox{0.16}{\begin{tikzpicture}
\draw[black, thick] (-0.7,0.9) -- (0,0);
\draw[black, thick] (0.7,0.9) -- (0,0);
\draw[black, thick] (0.7,0) -- (0,-0.9);
\draw[snake=zigzag](0,0) -- (0,-0.9);
\draw[snake=zigzag](0,-0.9) -- (0,-1.8);
\filldraw[gray] (-0.7,0.9) circle (7pt); 
\filldraw[gray] (0.7,0.9) circle (7pt); 
\filldraw[gray] (0.7,0) circle (7pt); 
\end{tikzpicture}
}\epsilon}_{i}, b^{
\scalebox{0.6}{\begin{tikzpicture}
\draw[black, thick] (0,0.5) -- (0,0);
\filldraw[blue] (0,0.5) circle (2pt); 
\end{tikzpicture}
}\epsilon}_{j})$, $\pi_{0,\diamond}(u^{
\scalebox{0.16}{\begin{tikzpicture}
\draw[black, thick] (-0.7,0.9) -- (0,0);
\draw[black, thick] (0.7,0.9) -- (0,0);
\draw[black, thick] (0.7,0) -- (0,-0.9);
\draw[snake=zigzag](0,0) -- (0,-0.9);
\draw[snake=zigzag](0,-0.9) -- (0,-1.8);
\filldraw[pink] (-0.7,0.9) circle (7pt); 
\filldraw[pink] (0.7,0.9) circle (7pt); 
\filldraw[pink] (0.7,0) circle (7pt); 
\end{tikzpicture}
}\epsilon}_{i}, b^{
\scalebox{0.6}{\begin{tikzpicture}
\draw[black, thick] (0,0.5) -- (0,0);
\filldraw[blue] (0,0.5) circle (2pt); 
\end{tikzpicture}
}\epsilon}_{j})$ and $\pi_{0, \diamond} (u^{
\scalebox{0.16}{\begin{tikzpicture}
\draw[black, thick] (-0.7,0.9) -- (0,0);
\draw[black, thick] (0.7,0.9) -- (0,0);
\draw[black, thick] (0.7,0) -- (0,-0.9);
\draw[snake=zigzag](0,0) -- (0,-0.9);
\draw[snake=zigzag](0,-0.9) -- (0,-1.8);
\filldraw[gray] (-0.7,0.9) circle (7pt); 
\filldraw[gray] (0.7,0.9) circle (7pt); 
\filldraw[gray] (0.7,0) circle (7pt); 
\end{tikzpicture}
}\epsilon}_{i}, u^{\scalebox{0.6}{\begin{tikzpicture}
\draw[black, thick] (0,0.5) -- (0,0);
\filldraw[red] (0,0.5) circle (2pt); 
\end{tikzpicture}
}\epsilon}_{j})$, 
\item a fourth group of $\pi_{0, \diamond} (\mathcal{P}_{ii_{1}} \partial_{x_{j}}  K_{j}^{u, \epsilon}, u^{\scalebox{0.6}{\begin{tikzpicture}
\draw[black, thick] (0,0.5) -- (0,0);
\filldraw[red] (0,0.5) circle (2pt); 
\end{tikzpicture}
}\epsilon}_{j_{1}})$, $\pi_{0,\diamond}(\mathcal{P}_{ii_{1}}  \partial_{x_{j}}  K_{j}^{b, \epsilon}, u^{\scalebox{0.6}{\begin{tikzpicture}
\draw[black, thick] (0,0.5) -- (0,0);
\filldraw[red] (0,0.5) circle (2pt); 
\end{tikzpicture}
}\epsilon}_{j_{1}})$, $\pi_{0, \diamond}(\mathcal{P}_{ii_{1}}  \partial_{x_{j}}  K_{j}^{u, \epsilon}, b^{
\scalebox{0.6}{\begin{tikzpicture}
\draw[black, thick] (0,0.5) -- (0,0);
\filldraw[blue] (0,0.5) circle (2pt); 
\end{tikzpicture}
}\epsilon}_{j_{1}})$ and $\pi_{0,\diamond}(\mathcal{P}_{ii_{1}}  \partial_{x_{j}}  K_{j}^{b, \epsilon}, b^{
\scalebox{0.6}{\begin{tikzpicture}
\draw[black, thick] (0,0.5) -- (0,0);
\filldraw[blue] (0,0.5) circle (2pt); 
\end{tikzpicture}
}\epsilon}_{j_{1}})$. 
\end{enumerate}

\subsection{Group 1}
Within the group 1 of (\ref{114}), specifically $u^{\scalebox{0.6}{\begin{tikzpicture}
\draw[black, thick] (0,0.5) -- (0,0);
\filldraw[red] (0,0.5) circle (2pt); 
\end{tikzpicture}
}\epsilon}_{i} \diamond u^{
\scalebox{0.18}{\begin{tikzpicture}
\draw[black, thick] (-0.7,0.9) -- (0,0);
\draw[black, thick] (0.7,0.9) -- (0,0);
\draw[snake=zigzag](0,0) -- (0,-0.9);
\filldraw[green] (-0.7,0.9) circle (6pt); 
\filldraw[green] (0.7,0.9) circle (6pt); 
\end{tikzpicture}
}\epsilon}_{j}, b^{
\scalebox{0.6}{\begin{tikzpicture}
\draw[black, thick] (0,0.5) -- (0,0);
\filldraw[blue] (0,0.5) circle (2pt); 
\end{tikzpicture}
}\epsilon}_{i} \diamond b^{
\scalebox{0.18}{\begin{tikzpicture}
\draw[black, thick] (-0.7,0.9) -- (0,0);
\draw[black, thick] (0.7,0.9) -- (0,0);
\draw[snake=zigzag](0,0) -- (0,-0.9);
\filldraw[violet] (-0.7,0.9) circle (6pt); 
\filldraw[violet] (0.7,0.9) circle (6pt); 
\end{tikzpicture}
}\epsilon}_{j}, u^{\scalebox{0.6}{\begin{tikzpicture}
\draw[black, thick] (0,0.5) -- (0,0);
\filldraw[red] (0,0.5) circle (2pt); 
\end{tikzpicture}
}\epsilon}_{i} \diamond b^{
\scalebox{0.18}{\begin{tikzpicture}
\draw[black, thick] (-0.7,0.9) -- (0,0);
\draw[black, thick] (0.7,0.9) -- (0,0);
\draw[snake=zigzag](0,0) -- (0,-0.9);
\filldraw[violet] (-0.7,0.9) circle (6pt); 
\filldraw[violet] (0.7,0.9) circle (6pt); 
\end{tikzpicture}
}\epsilon}_{j}$, and $b^{
\scalebox{0.6}{\begin{tikzpicture}
\draw[black, thick] (0,0.5) -- (0,0);
\filldraw[blue] (0,0.5) circle (2pt); 
\end{tikzpicture}
}\epsilon}_{i} \diamond u^{
\scalebox{0.18}{\begin{tikzpicture}
\draw[black, thick] (-0.7,0.9) -- (0,0);
\draw[black, thick] (0.7,0.9) -- (0,0);
\draw[snake=zigzag](0,0) -- (0,-0.9);
\filldraw[green] (-0.7,0.9) circle (6pt); 
\filldraw[green] (0.7,0.9) circle (6pt); 
\end{tikzpicture}
}\epsilon}_{j}$, we focus on $b^{
\scalebox{0.6}{\begin{tikzpicture}
\draw[black, thick] (0,0.5) -- (0,0);
\filldraw[blue] (0,0.5) circle (2pt); 
\end{tikzpicture}
}\epsilon}_{i} \diamond u^{
\scalebox{0.18}{\begin{tikzpicture}
\draw[black, thick] (-0.7,0.9) -- (0,0);
\draw[black, thick] (0.7,0.9) -- (0,0);
\draw[snake=zigzag](0,0) -- (0,-0.9);
\filldraw[green] (-0.7,0.9) circle (6pt); 
\filldraw[green] (0.7,0.9) circle (6pt); 
\end{tikzpicture}
}\epsilon}_{j}$ and prove the existence of $v^{
\scalebox{0.2}{\begin{tikzpicture}
\draw[black, thick] (-0.7,0.9) -- (0,0);
\draw[black, thick] (0.7,0.9) -- (0,0);
\draw[snake=zigzag](0,0) -- (0,-0.9);
\draw[black, thick] (-0.7,0.0) -- (0,-0.9);
\filldraw[blue] (-0.7,0.0) circle (6pt); 
\filldraw[green] (-0.7,0.9) circle (6pt); 
\filldraw[green] (0.7,0.9) circle (6pt); 
\end{tikzpicture}
}}_{9,ij} \in C([0, T]; \mathcal{C}^{- \frac{1}{2} - \frac{\delta}{2}})$ such that $b^{
\scalebox{0.6}{\begin{tikzpicture}
\draw[black, thick] (0,0.5) -- (0,0);
\filldraw[blue] (0,0.5) circle (2pt); 
\end{tikzpicture}
}\epsilon}_{i} \diamond u^{
\scalebox{0.18}{\begin{tikzpicture}
\draw[black, thick] (-0.7,0.9) -- (0,0);
\draw[black, thick] (0.7,0.9) -- (0,0);
\draw[snake=zigzag](0,0) -- (0,-0.9);
\filldraw[green] (-0.7,0.9) circle (6pt); 
\filldraw[green] (0.7,0.9) circle (6pt); 
\end{tikzpicture}
}\epsilon}_{j} \to v^{
\scalebox{0.2}{\begin{tikzpicture}
\draw[black, thick] (-0.7,0.9) -- (0,0);
\draw[black, thick] (0.7,0.9) -- (0,0);
\draw[snake=zigzag](0,0) -- (0,-0.9);
\draw[black, thick] (-0.7,0.0) -- (0,-0.9);
\filldraw[blue] (-0.7,0.0) circle (6pt); 
\filldraw[green] (-0.7,0.9) circle (6pt); 
\filldraw[green] (0.7,0.9) circle (6pt); 
\end{tikzpicture}
}}_{9,ij}$ as $\epsilon \to 0$.  For simplicity of notations we write $b^{
\scalebox{0.6}{\begin{tikzpicture}
\draw[black, thick] (0,0.5) -- (0,0);
\filldraw[blue] (0,0.5) circle (2pt); 
\end{tikzpicture}
}\epsilon}_{j} u^{
\scalebox{0.18}{\begin{tikzpicture}
\draw[black, thick] (-0.7,0.9) -- (0,0);
\draw[black, thick] (0.7,0.9) -- (0,0);
\draw[snake=zigzag](0,0) -- (0,-0.9);
\filldraw[green] (-0.7,0.9) circle (6pt); 
\filldraw[green] (0.7,0.9) circle (6pt); 
\end{tikzpicture}
}\epsilon}_{i}$. First, from (\ref{17}), (\ref{114}) and (\ref{115}), we obtain 
\begin{align}\label{118}
b^{
\scalebox{0.6}{\begin{tikzpicture}
\draw[black, thick] (0,0.5) -- (0,0);
\filldraw[blue] (0,0.5) circle (2pt); 
\end{tikzpicture}
}\epsilon}_{j}(t) u^{
\scalebox{0.18}{\begin{tikzpicture}
\draw[black, thick] (-0.7,0.9) -- (0,0);
\draw[black, thick] (0.7,0.9) -- (0,0);
\draw[snake=zigzag](0,0) -- (0,-0.9);
\filldraw[green] (-0.7,0.9) circle (6pt); 
\filldraw[green] (0.7,0.9) circle (6pt); 
\end{tikzpicture}
}\epsilon}_{i}(t) 
=& -\frac{1}{2(2\pi)^{3}} \sum_{k} \sum_{i_{1}, i_{2} = 1}^{3} \sum_{k_{1}, k_{2}, k_{3}: k_{123} = k} \int_{0}^{t} e^{- \lvert k_{12} \rvert^{2} \lvert t-s \rvert} \hat{\mathcal{P}}_{ii_{1}} (k_{12}) ik_{12}^{i_{2}} \nonumber\\
& \times [ \hat{X}_{t, j}^{b, \epsilon} (k_{3}) \hat{X}_{s,i_{1}}^{u, \epsilon}(k_{1}) \hat{X}_{s,i_{2}}^{u, \epsilon}(k_{2}) - \hat{X}_{t, j}^{b, \epsilon} (k_{3}) \hat{X}_{s,i_{1}}^{b, \epsilon}(k_{1}) \hat{X}_{s,i_{2}}^{b, \epsilon}(k_{2})] ds e_{k}.
\end{align}  
We rely on $:\xi_{1} \xi_{2} \xi_{3}: = \xi_{1} \xi_{2} \xi_{3} - \mathbb{E} [ \xi_{2} \xi_{3}] \xi_{1} - \mathbb{E} [\xi_{1} \xi_{3}] \xi_{2} - \mathbb{E} [ \xi_{1} \xi_{2}] \xi_{3}$ (see \cite{J97}) and (\ref{116}) to deduce 
\begin{align}\label{119}
& \hat{X}_{t, j}^{b, \epsilon} (k_{3}) \hat{X}_{s,i_{1}}^{u, \epsilon}(k_{1}) \hat{X}_{s,i_{2}}^{u, \epsilon}(k_{2}) - \hat{X}_{t, j}^{b, \epsilon} (k_{3}) \hat{X}_{s,i_{1}}^{b, \epsilon}(k_{1}) \hat{X}_{s,i_{2}}^{b, \epsilon}(k_{2})\nonumber\\
=& : \hat{X}_{t, j}^{b, \epsilon} (k_{3}) \hat{X}_{s,i_{1}}^{u, \epsilon}(k_{1}) \hat{X}_{s,i_{2}}^{u, \epsilon}(k_{2}): \nonumber\\
&+ 1_{k_{23} = 0, k_{2} \neq 0} \sum_{i_{3}  =1}^{3} \frac{e^{ - \lvert k_{2} \rvert^{2} \lvert t-s \rvert}}{2 \lvert k_{2} \rvert^{2}} f(\epsilon k_{2})^{2} \hat{\mathcal{P}}_{j i_{3}} (k_{2}) \hat{\mathcal{P}}_{i_{2} i_{3}}(k_{2}) \hat{X}_{s,i_{1}}^{u, \epsilon}(k_{1}) \nonumber\\
&+ 1_{k_{13} = 0, k_{1} \neq 0} \sum_{i_{3}  =1}^{3} \frac{e^{ - \lvert k_{1} \rvert^{2} \lvert t-s \rvert}}{2 \lvert k_{1} \rvert^{2}} f(\epsilon k_{1})^{2} \hat{\mathcal{P}}_{j i_{3}}(k_{1}) \hat{\mathcal{P}}_{i_{1}i_{3}}(k_{1}) \hat{X}_{s,i_{2}}^{u, \epsilon}(k_{2}) \nonumber\\
& - : \hat{X}_{t, j}^{b, \epsilon} (k_{3}) \hat{X}_{s,i_{1}}^{b, \epsilon}(k_{1}) \hat{X}_{s,i_{2}}^{b, \epsilon}(k_{2}): \nonumber\\
& - 1_{k_{23}  =0, k_{2} \neq 0} \sum_{i_{3}  =1}^{3} \frac{e^{- \lvert k_{2} \rvert^{2} \lvert t-s \rvert}}{2 \lvert k_{2} \rvert^{2}} f(\epsilon k_{2})^{2} \hat{\mathcal{P}}_{j i_{3}} (k_{2}) \hat{\mathcal{P}}_{i_{2} i_{3}} (k_{2}) \hat{X}_{s,i_{1}}^{b, \epsilon}(k_{1})\nonumber\\
& - 1_{k_{13}  =0, k_{1} \neq 0} \sum_{i_{3}  =1}^{3} \frac{e^{- \lvert k_{1} \rvert^{2} \lvert t-s \rvert}}{2 \lvert k_{1} \rvert^{2}} f(\epsilon k_{1})^{2} \hat{\mathcal{P}}_{j i_{3}}(k_{1}) \hat{\mathcal{P}}_{i_{1} i_{3}}(k_{1}) \hat{X}_{s,i_{2}}^{b, \epsilon}(k_{2}). 
\end{align} 
Applying (\ref{119}) to (\ref{118}) gives 
\begin{align}\label{120}
& b^{
\scalebox{0.6}{\begin{tikzpicture}
\draw[black, thick] (0,0.5) -- (0,0);
\filldraw[blue] (0,0.5) circle (2pt); 
\end{tikzpicture}
}\epsilon}_{j}(t) u^{
\scalebox{0.18}{\begin{tikzpicture}
\draw[black, thick] (-0.7,0.9) -- (0,0);
\draw[black, thick] (0.7,0.9) -- (0,0);
\draw[snake=zigzag](0,0) -- (0,-0.9);
\filldraw[green] (-0.7,0.9) circle (6pt); 
\filldraw[green] (0.7,0.9) circle (6pt); 
\end{tikzpicture}
}\epsilon}_{i}(t) \\
=& - \frac{1}{2(2\pi)^{3}} \sum_{k} \sum_{i_{1}, i_{2}  =1}^{3} \sum_{k_{1}, k_{2}, k_{3}: k_{123} = k}  \int_{0}^{t} e^{- \lvert k_{12} \rvert^{2} \lvert t-s \rvert}\hat{\mathcal{P}}_{ii_{1}} (k_{12}) \nonumber\\
& \hspace{20mm} \times i k_{12}^{i_{2}}: \hat{X}_{t, j}^{b, \epsilon} (k_{3}) \hat{X}_{s,i_{1}}^{u, \epsilon}(k_{1}) \hat{X}_{s,i_{2}}^{u, \epsilon}(k_{2}): ds e_{k}\nonumber \\
& - \frac{1}{2(2\pi)^{3}} \sum_{k} \sum_{i_{1}, i_{2}, i_{3} =1}^{3} \sum_{k_{1}, k_{2}, k_{3}: k_{123} = k}  \int_{0}^{t} e^{- \lvert k_{12} \rvert^{2} \lvert t-s \rvert} \hat{\mathcal{P}}_{ii_{1}} (k_{12}) \nonumber\\
& \hspace{20mm}  \times i k_{12}^{i_{2}} 1_{k_{23} = 0, k_{2} \neq 0}  \frac{e^{ - \lvert k_{2} \rvert^{2} \lvert t-s \rvert}}{2 \lvert k_{2} \rvert^{2}} f(\epsilon k_{2})^{2} \hat{\mathcal{P}}_{j i_{3}}(k_{2}) \hat{\mathcal{P}}_{i_{2} i_{3}}(k_{2}) \hat{X}_{s,i_{1}}^{u, \epsilon}(k_{1}) ds e_{k} \nonumber\\
& - \frac{1}{2(2\pi)^{3}} \sum_{k} \sum_{i_{1}, i_{2}, i_{3} =1}^{3} \sum_{k_{1}, k_{2}, k_{3}: k_{123} = k}  \int_{0}^{t} e^{- \lvert k_{12} \rvert^{2} \lvert t-s \rvert} \hat{\mathcal{P}}_{ii_{1}} (k_{12})  \nonumber\\
&\hspace{20mm}  \times i k_{12}^{i_{2}} 1_{k_{13} = 0, k_{1} \neq 0}  \frac{e^{ - \lvert k_{1} \rvert^{2} \lvert t-s \rvert}}{2 \lvert k_{1} \rvert^{2}} f(\epsilon k_{1})^{2} \hat{\mathcal{P}}_{j i_{3}}(k_{1}) \hat{\mathcal{P}}_{i_{1} i_{3}} (k_{1}) \hat{X}_{s,i_{2}}^{u, \epsilon}(k_{2}) ds e_{k} \nonumber\\
&+ \frac{1}{2(2\pi)^{3}} \sum_{k} \sum_{i_{1}, i_{2}  =1}^{3} \sum_{k_{1}, k_{2}, k_{3}: k_{123} = k}  \int_{0}^{t} e^{- \lvert k_{12} \rvert^{2} \lvert t-s \rvert}\hat{\mathcal{P}}_{ii_{1}} (k_{12}) \nonumber\\
& \hspace{20mm} \times  i k_{12}^{i_{2}}: \hat{X}_{t, j}^{b, \epsilon} (k_{3}) \hat{X}_{s,i_{1}}^{b, \epsilon}(k_{1}) \hat{X}_{s,i_{2}}^{b, \epsilon}(k_{2}): ds e_{k}\nonumber \\
&+ \frac{1}{2(2\pi)^{3}} \sum_{k} \sum_{i_{1}, i_{2}, i_{3}  =1}^{3} \sum_{k_{1}, k_{2}, k_{3}: k_{123} = k}  \int_{0}^{t} e^{- \lvert k_{12} \rvert^{2} \lvert t-s \rvert} \hat{\mathcal{P}}_{ii_{1}} (k_{12}) \nonumber\\
&\hspace{20mm}  \times  i k_{12}^{i_{2}} 1_{k_{23} = 0, k_{2} \neq 0}  \frac{e^{ - \lvert k_{2} \rvert^{2} \lvert t-s \rvert}}{2 \lvert k_{2} \rvert^{2}} f(\epsilon k_{2})^{2} \hat{\mathcal{P}}_{j i_{3}}(k_{2}) \hat{\mathcal{P}}_{i_{2} i_{3} }(k_{2}) \hat{X}_{s,i_{1}}^{b, \epsilon}(k_{1}) ds e_{k} \nonumber\\
& + \frac{1}{2(2\pi)^{3}} \sum_{k} \sum_{i_{1}, i_{2}, i_{3}=1}^{3} \sum_{k_{1}, k_{2}, k_{3}: k_{123} = k}  \int_{0}^{t} e^{- \lvert k_{12} \rvert^{2} \lvert t-s \rvert} \hat{\mathcal{P}}_{ii_{1}} (k_{12}) \nonumber\\
& \hspace{20mm} \times  i k_{12}^{i_{2}} 1_{k_{13} = 0, k_{1} \neq 0}  \frac{e^{ - \lvert k_{1} \rvert^{2} \lvert t-s \rvert}}{2 \lvert k_{1} \rvert^{2}} f(\epsilon k_{1})^{2} \hat{\mathcal{P}}_{j i_{3}}(k_{1}) \hat{\mathcal{P}}_{i_{1} i_{3} }(k_{1}) \hat{X}_{s,i_{2}}^{b, \epsilon}(k_{2}) ds e_{k} 
\triangleq \sum_{l=1}^{6} \RomanII_{t,\epsilon}^{l},  \nonumber
\end{align} 
where $\RomanII_{t,\epsilon}^{1}, \RomanII_{t,\epsilon}^{4}$ are the terms in the third chaos while $\RomanII_{t,\epsilon}^{2}, \RomanII_{t,\epsilon}^{3}, \RomanII_{t,\epsilon}^{5}, \RomanII_{t,\epsilon}^{6}$ are in the first chaos. 

\subsubsection{Terms in the first chaos}
Let us work on $\RomanII_{t, \epsilon}^{5}$ of (\ref{120}). We first rewrite 
\begin{align}\label{121}
\RomanII_{t,\epsilon}^{5} =& \frac{1}{2(2\pi)^{3}} \sum_{i_{1}, i_{2}, i_{3}  =1}^{3} \sum_{k_{1}, k_{2} \neq 0} \int_{0}^{t} e^{-\lvert k_{12} \rvert^{2} \lvert t-s \rvert} \nonumber\\
& \times ik_{12}^{i_{2}} \hat{X}_{s,i_{1}}^{b, \epsilon}(k_{1}) \frac{ e^{ - \lvert k_{2} \rvert^{2} \lvert t-s \rvert} f(\epsilon k_{2})^{2}}{2 \lvert k_{2} \rvert^{2}} \hat{\mathcal{P}}_{ii_{1}} (k_{12}) \hat{\mathcal{P}}_{i_{2} i_{3}} (k_{2}) \hat{\mathcal{P}}_{j i_{3}}(k_{2}) ds e_{k_{1}}
\end{align} 
and write 
\begin{equation}\label{122}
\RomanII_{t,\epsilon}^{5} = \RomanII_{t,\epsilon}^{5}  - \widetilde{\RomanII}_{t,\epsilon}^{5} + \widetilde{\RomanII}_{t,\epsilon}^{5} - \sum_{i_{1} = 1}^{3} X_{t,i_{1}}^{b, \epsilon} C_{t}^{\epsilon, i_{1}}
\end{equation} 
where 
\begin{subequations}\label{123}
\begin{align}
 \widetilde{\RomanII}_{t,\epsilon}^{5} \triangleq \frac{1}{2(2\pi)^{3}} & \sum_{i_{1}, i_{2}, i_{3} =1}^{3} \sum_{k_{1}, k_{2} \neq 0} \int_{0}^{t} e^{- \lvert k_{12} \rvert^{2} \lvert t-s \rvert} \nonumber\\
 & \times ik_{12}^{i_{2}} \hat{X}_{t,i_{1}}^{b, \epsilon} (k_{1}) \frac{e^{- \lvert k_{2} \rvert^{2} \lvert t-s \rvert} f(\epsilon k_{2})^{2}}{2 \lvert k_{2} \rvert^{2}} \hat{\mathcal{P}}_{ii_{1}}(k_{12}) \hat{\mathcal{P}}_{i_{2} i_{3}} (k_{2}) \hat{\mathcal{P}}_{j i_{3}} (k_{2}) ds e_{k_{1}}, \\
 C_{t}^{\epsilon, i_{1}} \triangleq  \frac{1}{2(2\pi)^{3}} & \sum_{i_{2}, i_{3} = 1}^{3} \sum_{k_{2} \neq 0} \int_{0}^{t} e^{-2 \lvert k_{2} \rvert^{2} \lvert t-s \rvert} \nonumber \\
 & \times  ik_{2}^{i_{2}} \frac{f(\epsilon k_{2})^{2}}{2\lvert k_{2} \rvert^{2}} \hat{\mathcal{P}}_{ii_{1}}(k_{2}) \hat{\mathcal{P}}_{i_{2} i_{3}}(k_{2}) \hat{\mathcal{P}}_{j i_{3}} (k_{2}) ds = 0. 
\end{align}
\end{subequations} 
We compute within (\ref{122}), 
 \begin{align}\label{124}
& \mathbb{E} [ \lvert \Delta_{q} ( \RomanII_{t,\epsilon}^{5} - \widetilde{\RomanII}_{t,\epsilon}^{5}) \rvert^{2}]\nonumber\\
\approx& \mathbb{E}[ \lvert \sum_{k_{1} \neq 0} \theta(2^{-q} k_{1}) \sum_{i_{1}, i_{2}, i_{3} =1}^{3} \sum_{k_{2} \neq 0} \int_{0}^{t} e^{- \lvert k_{12} \rvert^{2} \lvert t-s \rvert}\nonumber\\
& \times k_{12}^{i_{2}} ( \hat{X}_{s,i_{1}}^{b, \epsilon}(k_{1}) - \hat{X}_{t,i_{1}}^{b, \epsilon}(k_{1})) \frac{e^{- \lvert k_{2} \rvert^{2} \lvert t-s \rvert} f(\epsilon k_{2})^{2}}{2 \lvert k_{2} \rvert^{2}} \hat{\mathcal{P}}_{ii_{1}}(k_{12}) \hat{\mathcal{P}}_{i_{2} i_{3}}(k_{2}) \hat{\mathcal{P}}_{j i_{3}}(k_{2}) ds e_{k_{1}} \rvert^{2}]\nonumber\\
\lesssim& \mathbb{E} [ \lvert \sum_{i_{1},i_{2},i_{3} = 1}^{3} \int_{0}^{t} \sum_{k_{1} \neq 0} \theta(2^{-q} k_{1}) e_{k_{1}} \sum_{k_{2} \neq 0} e^{- \lvert k_{12} \rvert^{2} \lvert t-s \rvert} \nonumber\\
& \times k_{12}^{i_{2}} \frac{ e^{- \lvert k_{2} \rvert^{2} \lvert t-s \rvert} f(\epsilon k_{2})^{2}}{\lvert k_{2} \rvert^{2}} \hat{\mathcal{P}}_{ii_{1}}(k_{12}) \hat{\mathcal{P}}_{i_{2} i_{3}}(k_{2}) \hat{\mathcal{P}}_{j i_{3}}(k_{2}) ( \hat{X}_{s,i_{1}}^{b, \epsilon}(k_{1}) - \hat{X}_{t,i_{1}}^{b, \epsilon}(k_{1})) ds \rvert^{2} ]\nonumber\\
\lesssim& \sum_{i_{1}, i_{2}, i_{3}, i_{1}', i_{2}', i_{3}' = 1}^{3} \int_{[0,t]^{2}} \sum_{k_{1}, k_{1}' \neq 0} \theta(2^{-q} k_{1}) \theta(2^{-q} k_{1}') \lvert a_{k_{1}}^{i_{1}i_{2}i_{3}} (t-s) a_{k_{1}'}^{i_{1}'i_{2}'i_{3}'} (t- \overline{s}) \rvert \nonumber\\
& \times \mathbb{E} [ \lvert ( \hat{X}_{s,i_{1}}^{b, \epsilon}(k_{1}) - \hat{X}_{t,i_{1}}^{b, \epsilon}(k_{1})) \overline{(\hat{X}_{\overline{s}, i_{1}'}^{b, \epsilon} (k_{1}') - \hat{X}_{t,i_{1}'}^{b, \epsilon} (k_{1}'))}\rvert] ds d \overline{s} 
\end{align}  
where we denoted 
\begin{align}\label{125}
a_{k_{1}}^{i_{1}i_{2}i_{3}} (t-s) \triangleq & \sum_{k_{2} \neq 0} e^{- \lvert k_{12} \rvert^{2} \lvert t-s \rvert} k_{12}^{i_{2}} \frac{e^{-\lvert k_{2} \rvert^{2} \lvert t-s \rvert} f(\epsilon k_{2})^{2}}{\lvert k_{2} \rvert^{2}} \hat{\mathcal{P}}_{ii_{1}} (k_{12}) \hat{\mathcal{P}}_{i_{2} i_{3}}(k_{2}) \hat{\mathcal{P}}_{j i_{3}}(k_{2}).
\end{align} 
We may further estimate for $k_{1} \neq 0$, for any $\eta \in (0,1)$, 
\begin{align}\label{126}
&\mathbb{E} [ \lvert (\hat{X}_{s,i_{1}}^{b, \epsilon}(k_{1}) - \hat{X}_{t,i_{1}}^{b, \epsilon}(k_{1})) \overline{( \hat{X}_{\overline{s}, i_{1}'}^{b, \epsilon}(k_{1}') - \hat{X}_{t,i_{1}'}^{b, \epsilon} (k_{1}'))} \rvert]  \nonumber\\
\lesssim& 1_{k_{1} + k_{1} ' = 0} \frac{f(\epsilon k_{1})^{2}}{ \lvert k_{1} \rvert^{2}} \lvert k_{1} \rvert^{2\eta} \lvert t-s \rvert^{\frac{\eta}{2}} \lvert t- \overline{s} \rvert^{\frac{\eta}{2}}
\end{align} 
by H$\ddot{\mathrm{o}}$lder's inequality, (\ref{116}), (\ref{16}) and mean value theorem. Applying (\ref{126}) to (\ref{124}) gives 
\begin{align}\label{127}
\mathbb{E} &[ \lvert \Delta_{q} ( \RomanII_{t,\epsilon}^{5} - \widetilde{\RomanII}_{t,\epsilon}^{5} ) \rvert^{2} ] \lesssim \sum_{i_{1}, i_{2}, i_{3}, i_{1}', i_{2}', i_{3}' = 1}^{3} \int_{[0,t]^{2}} \sum_{k_{1} \neq 0} \theta( 2^{-q} k)^{2}\nonumber \\
& \times \lvert a_{k_{1}}^{i_{1}, i_{2}, i_{3}} (t-s) a_{k_{1}}^{i_{1}', i_{2}', i_{3}'} (t- \overline{s}) \rvert \frac{f(\epsilon k_{1})^{2}}{ \lvert k_{1} \rvert^{2}} \lvert k_{1} \rvert^{2\eta} \lvert t-s \rvert^{\frac{\eta}{2}} \lvert t- \overline{s} \rvert^{\frac{\eta}{2}}. 
\end{align}  
Moreover, 
\begin{equation*}
\lvert a_{k_{1}}^{i_{1}i_{2}i_{3}}(t-s) \rvert \lesssim  \sum_{k_{2} \neq 0} \frac{ e^{- \lvert k_{2} \rvert^{2}(t-s)}}{\lvert k_{2} \rvert^{2}} \lesssim \frac{1}{(t-s)^{1+ \frac{\epsilon}{2}}}  
\end{equation*}  
by (\ref{125}) and (\ref{14}).  This gives 
\begin{align}\label{128}
&\sum_{i_{1}, i_{2}, i_{3}, i_{1}', i_{2}', i_{3}' = 1}^{3} \int_{[0,t]^{2}} \lvert a_{k_{1}}^{i_{1}, i_{2}, i_{3}} (t-s) a_{k_{1}}^{i_{1}', i_{2}', i_{3}'} (t- \overline{s}) \rvert \lvert t-s \rvert^{\frac{\eta}{2}} \lvert t- \overline{s} \rvert^{\frac{\eta}{2}} ds d \overline{s}\nonumber\\
& \hspace{20mm} \lesssim \int_{[0,t]^{2}} (t-s)^{\frac{\eta}{2} -1 - \frac{\epsilon}{2}} (t- \overline{s})^{\frac{\eta}{2} - 1 - \frac{\epsilon}{2}} ds d \overline{s} \lesssim t^{\eta - \epsilon}.
\end{align} 
Thus, applying (\ref{128}) to (\ref{127}) gives 
\begin{equation}\label{129}
\mathbb{E} [ \lvert \Delta_{q} ( \RomanII_{t,\epsilon}^{5} - \widetilde{\RomanII}_{t, \epsilon}^{5} ) \rvert^{2} ]\lesssim \sum_{k_{1} \neq 0} \theta(2^{-q} k_{1})^{2} t^{\eta - \epsilon} \lvert k_{1} \rvert^{2 \eta -2} \approx t^{\eta - \epsilon} 2^{q( 1+ 2 \eta)}.
\end{equation} 
Next, for any $\eta \in (0,1)$, we estimate within (\ref{122}), 
\begin{align}\label{130}
& \mathbb{E} [ \lvert \Delta_{q} ( \widetilde{\RomanII}_{t,\epsilon}^{5} - \sum_{i_{1} =1}^{3} X_{t,i_{1}}^{b, \epsilon} C_{t}^{\epsilon, i_{1}}) \rvert^{2}] \nonumber\\
\lesssim& \sum_{k_{1}} \mathbb{E} [ \lvert \hat{X}_{t,i_{1}}^{b, \epsilon} (k_{1}) \rvert^{2}] \theta( 2^{-q} k_{1})^{2} [ \sum_{i_{1}, i_{2}, i_{3}= 1}^{3} \sum_{k_{2} \neq 0} \int_{0}^{t} \frac{ e^{- \lvert k_{2} \rvert^{2} (t-s)}f(\epsilon k_{2})^{2}}{ \lvert k_{2} \rvert^{2}}\nonumber\\
\times& (e ^{- \lvert k_{12} \rvert^{2} (t-s)} k_{12}^{i_{2}} \hat{\mathcal{P}}_{ii_{1}} (k_{12}) - e^{- \lvert k_{2} \rvert^{2} (t-s)} k_{2}^{i_{2}} \hat{\mathcal{P}}_{ii_{1}} (k_{2})) \hat{\mathcal{P}}_{i_{2} i_{3}}(k_{2}) \hat{\mathcal{P}}_{j i_{3}}(k_{2}) ds]^{2} \nonumber\\
\lesssim& \sum_{k_{1} \neq 0} \frac{f ( \epsilon k_{1})^{2}}{\lvert k_{1} \rvert^{2}} \theta( 2^{-q} k_{1})^{2} ( \sum_{k_{2} \neq 0} \int_{0}^{t} \frac{ e^{- \lvert k_{2} \rvert^{2} (t-s)} f(\epsilon k_{2})^{2}}{ \lvert k_{2} \rvert^{2}} \lvert k_{1} \rvert^{\eta} (t-s)^{ - \frac{(1-\eta)}{2}} ds )^{2} 
\end{align} 
by (\ref{123}), Lemma \ref{Lemma 2.8} and (\ref{116}). We furthermore estimate for $\epsilon \in (0,\eta)$, 
\begin{equation}\label{131}
\left( \sum_{k_{2} \neq 0} \int_{0}^{t} \frac{ e^{- \lvert k_{2} \rvert^{2} (t-s)} f(\epsilon k_{2})^{2}}{ \lvert k_{2} \rvert^{2}} (t-s)^{ - \frac{(1-\eta)}{2}} ds \right)^{2}  \lesssim t^{\eta - \epsilon} 
\end{equation}
by (\ref{14}). We also estimate 
\begin{align*} 
\sum_{k_{1} \neq 0} \frac{ f(\epsilon k_{1})^{2}}{ \lvert k_{1} \rvert^{2- 2 \eta}} \theta( 2 ^{-q} k_{1})^{2} \lesssim \sum_{k_{1} \neq 0} \frac{1}{\lvert k_{1} \rvert^{3}} 2^{q( 1+ 2 \eta)}\theta ( 2^{-q} k_{1}) \lesssim 2^{q(1+ 2\eta)};
\end{align*}
applying this and (\ref{131}) to (\ref{130}) leads to, together with(\ref{129}), 
\begin{equation}\label{132}
\mathbb{E} [ \lvert \Delta_{q} \RomanII_{t,\epsilon}^{5} \rvert^{2}] \lesssim t^{\eta - \epsilon} 2^{q(1+ 2\eta)}. 
\end{equation} 
Similarly we can show $\sum_{k = 2, 3, 6} \mathbb{E} [ \lvert \Delta_{q} \RomanII_{t,\epsilon}^{k} \rvert^{2}] \lesssim t^{\eta - \epsilon} 2^{q(1+ 2\eta)}$. 
 
\subsubsection{Terms in the third chaos} 
We work on $\RomanII_{t,\epsilon}^{1}$ of (\ref{120}) as follows: 
\begin{align}\label{133}
 \mathbb{E} [ \lvert \Delta_{q} \RomanII_{t,\epsilon}^{1} \rvert^{2}] \approx &  \sum_{k} \sum_{i_{1}, i_{2}, i_{1}', i_{2}' = 1}^{3} \sum_{k_{1}, k_{2}, k_{3}: k_{123} = k, k_{1}', k_{2}', k_{3}': k_{123}' = k} \theta(2^{-q} k)^{2} \\
& \times \int_{[0,t]^{2}} \mathbb{E} [ : \hat{X}_{s,i_{1}}^{u, \epsilon}(k_{1}) \hat{X}_{s,i_{2}}^{u, \epsilon}(k_{2}) \hat{X}_{t, j}^{b, \epsilon} (k_{3}): : \hat{X}_{\overline{s}, i_{1}'}^{u, \epsilon} (k_{1}) \hat{X}_{\overline{s}, i_{2}'}^{u, \epsilon} (k_{2}) \hat{X}_{t, j}^{b, \epsilon} (k_{3}):] \nonumber \\
& \hspace{60mm} \times b_{k_{12}}^{i_{1}, i_{2}}(t-s) b_{k_{12}}^{i_{1}', i_{2}'} (t- \overline{s}) ds d \overline{s}  \nonumber 
\end{align}  
due to (\ref{120}) and the fact that $:\xi_{1}\xi_{2}\xi_{3}: = :\xi_{3}\xi_{1}\xi_{2}:$ (see \cite{J97}), where we also defined 
\begin{equation*}
b_{k_{12}}^{i_{1}, i_{2}}(t-s) \triangleq e^{- \lvert k_{12} \rvert^{2}(t-s)} k_{12}^{i_{2}} \hat{\mathcal{P}}_{i i_{1}}(k_{12}). 
\end{equation*}
We can now apply Lemma \ref{Lemma 5.7} (2) with ``$Y_{1}$'' $= : \hat{X}_{s,i_{1}}^{u, \epsilon}(k_{1}) \hat{X}_{s,i_{2}}^{u, \epsilon}(k_{2}) \hat{X}_{t, j}^{b, \epsilon}(k_{3}): $ and ``$Y_{2}$'' $= : \hat{X}_{\overline{s}, i_{1}'}^{u, \epsilon}(k_{1}') \hat{X}_{\overline{s}, i_{2}'}^{u, \epsilon}(k_{2}') \hat{X}_{t, j}^{b, \epsilon}(k_{3}'):$ and explicitly compute $\mathbb{E} [Y_{1}Y_{2}] = \sum_{\gamma} v(\gamma)$ (see \cite[Example 2.2]{Y20}), where the sum consists of six terms with 
\begin{align*}
&[ 1_{k_{1} + k_{1}' = 0, k_{1} \neq 0}\sum_{l=1}^{3} \frac{ e^{- \lvert k_{1} \rvert^{2} \lvert s- \overline{s} \rvert}}{ 2 \lvert k_{1} \rvert^{2}} f( \epsilon k_{1})^{2} \hat{\mathcal{P}}_{ i_{1}l} (k_{1}) \hat{\mathcal{P}}_{ i_{1}' l} (k_{1})]\\
& \times [1_{k_{2} + k_{2}' = 0, k_{2} \neq 0} \sum_{l=1}^{3} \frac{ e^{- \lvert k_{2} \rvert^{2} \lvert s - \overline{s} \rvert}}{2 \lvert k_{2} \rvert^{2}} f(\epsilon k_{2})^{2} \hat{\mathcal{P}}_{ i_{2} l} (k_{2}) \hat{\mathcal{P}}_{ i_{2}' l} (k_{2})]  [1_{k_{3} + k_{3}' = 0, k_{3} \neq 0} \sum_{l=1}^{3} \frac{ f(\epsilon k_{3})^{2}}{2 \lvert k_{3} \rvert^{2}} \lvert \hat{\mathcal{P}}_{ jl} (k_{3}) \rvert^{2}]
\end{align*}
as one representative, and this can readily be bounded by a constant multiple of $\prod_{i=1}^{3} \frac{f(\epsilon k_{i})^{2}}{ \lvert k_{i} \rvert^{2}} e^{- ( \lvert k_{1} \rvert^{2} + \lvert k_{2} \rvert^{2}) \lvert s - \overline{s} \rvert}$. The other five terms may be computed and bounded similarly (see \cite[Section 9.2]{GP17}) so that we are led to an estimate of 
\begin{align}\label{134}
& \mathbb{E} [ \lvert \Delta_{q} \RomanII_{t,\epsilon}^{1} \rvert^{2}] \lesssim \sum_{k} \sum_{i_{1}, i_{2}, i_{1}', i_{2}' = 1}^{3} \sum_{k_{1}, k_{2},k_{3} \neq 0: k_{123} = k} \theta (2^{-q} k)^{2} \int_{[0,t]^{2}} \prod_{i=1}^{3} \frac{f(\epsilon k_{i})^{2}}{\lvert k_{i} \rvert^{2}}\\
& \hspace{5mm} \times [ e^{- ( \lvert k_{1} \rvert^{2} + \lvert k_{2} \rvert^{2}) \lvert s - \overline{s} \rvert} \lvert b_{k_{12}}^{i_{1}, i_{2}} (t-s) b_{k_{12}}^{i_{1}', i_{2}'} (t- \overline{s} ) \rvert \nonumber\\
& \hspace{5mm}  + e^{- \lvert k_{1} \rvert^{2} \lvert t-s \rvert - \lvert k_{2} \rvert^{2} \lvert s - \overline{s} \rvert - \lvert k_{3} \rvert^{2} \lvert t-  \overline{s} \rvert} \lvert b_{k_{12}}^{i_{1}, i_{2}}(t-s) b_{k_{12}}^{i_{1}', i_{2}'} (t- \overline{s}) \rvert] ds d \overline{s} \triangleq \RomanII_{t,\epsilon}^{1,1} + \RomanII_{t,\epsilon}^{1,2}. \nonumber
\end{align} 
We may further estimate for any $\eta \in (0,1)$, 
\begin{equation}\label{135}
\lvert b_{k_{12}}^{i_{1}, i_{2}}(t-s) \rvert \lesssim \frac{1}{ \lvert k_{12} \rvert^{1-\eta} (t-s)^{1- \frac{\eta}{2}}}
\end{equation} 
by (\ref{14}). Applying (\ref{135}) to (\ref{134}) shows that 
\begin{align}\label{136}
\RomanII_{t,\epsilon}^{1,1} \approx & \sum_{k} \sum_{i_{1}, i_{2}, i_{1}', i_{2}' = 1}^{3} \sum_{k_{1}, k_{2},k_{3} \neq 0: k_{123} = k} \theta (2^{-q} k)^{2} \int_{[0,t]^{2}} \prod_{i=1}^{3} \frac{f(\epsilon k_{i})^{2}}{\lvert k_{i} \rvert^{2}}\\
& \times e^{- ( \lvert k_{1} \rvert^{2} + \lvert k_{2} \rvert^{2}) \lvert s - \overline{s} \rvert} \lvert b_{k_{12}}^{i_{1}, i_{2}} (t-s) b_{k_{12}}^{i_{1}', i_{2}'} (t- \overline{s} ) \rvert ds d \overline{s} \nonumber\\
\lesssim& \sum_{k} \theta (2^{-q} k) \sum_{k_{1}, k_{2}, k_{3} \neq 0: k_{123} = k} \prod_{i=1}^{3} \frac{1}{\lvert k_{i} \rvert^{2}} \frac{ t^{\eta}}{ \lvert k_{12} \rvert^{2-2\eta}}\lesssim \sum_{k} \theta(2^{-q }k) \frac{t^{\eta}}{ \lvert k \rvert^{2- 2\eta}}  \lesssim t^{\eta} 2^{q(1+ 2 \eta)}\nonumber 
\end{align}  
where we used Lemma \ref{Lemma 2.9}. Next, 
\begin{align}\label{137}
\RomanII_{t,\epsilon}^{1,2} \approx& \sum_{k} \sum_{i_{1}, i_{2}, i_{1}', i_{2}' = 1}^{3} \sum_{k_{1}, k_{2},k_{3} \neq 0: k_{123} = k} \theta (2^{-q} k)^{2} \int_{[0,t]^{2}} \prod_{i=1}^{3} \frac{f(\epsilon k_{i})^{2}}{\lvert k_{i} \rvert^{2}}\nonumber\\
& \times e^{- \lvert k_{1} \rvert^{2} \lvert t-s \rvert - \lvert k_{2} \rvert^{2} \lvert s - \overline{s} \rvert - \lvert k_{3} \rvert^{2} \lvert t-  \overline{s} \rvert} \lvert b_{k_{12}}^{i_{1}, i_{2}}(t-s) b_{k_{12}}^{i_{1}', i_{2}'} (t- \overline{s}) \rvert ds d \overline{s}\nonumber\\
\lesssim& \sum_{k} \theta(2^{-q} k) \sum_{k_{1}, k_{2}, k_{3} \neq 0: k_{123} = k} \prod_{i=1}^{3} \frac{1}{\lvert k_{i} \rvert^{2}} \frac{1}{\lvert k_{12} \rvert^{2-2\eta}} t^{\eta}  
\end{align} 
due to (\ref{134}) and (\ref{135}). At this point, this is identical to the estimate of $\RomanII_{t,\epsilon}^{1,1}$ in (\ref{136}); thus, it may be bounded by the same bound on $\RomanII_{t,\epsilon}^{1,1}$ in (\ref{136}). Therefore, we now conclude from (\ref{134}), (\ref{132}) and (\ref{120})  that 
\begin{equation}\label{138}
\mathbb{E} [ \lvert \Delta_{q} b^{
\scalebox{0.6}{\begin{tikzpicture}
\draw[black, thick] (0,0.5) -- (0,0);
\filldraw[blue] (0,0.5) circle (2pt); 
\end{tikzpicture}
}\epsilon}_{j} (t) u^{
\scalebox{0.18}{\begin{tikzpicture}
\draw[black, thick] (-0.7,0.9) -- (0,0);
\draw[black, thick] (0.7,0.9) -- (0,0);
\draw[snake=zigzag](0,0) -- (0,-0.9);
\filldraw[green] (-0.7,0.9) circle (6pt); 
\filldraw[green] (0.7,0.9) circle (6pt); 
\end{tikzpicture}
}\epsilon}_{i} (t) \rvert^{2}] \lesssim t^{\eta - \epsilon} 2^{q(1+ 2\eta)}
\end{equation} 
for any $t \in (0,1)$. 

Let us now first assume that for $t_{1} < t_{2}$, 
\begin{align}\label{139}
& \mathbb{E} [ \lvert \Delta_{q} ( b^{
\scalebox{0.6}{\begin{tikzpicture}
\draw[black, thick] (0,0.5) -- (0,0);
\filldraw[blue] (0,0.5) circle (2pt); 
\end{tikzpicture}
}\epsilon_{1}}_{j}u^{
\scalebox{0.18}{\begin{tikzpicture}
\draw[black, thick] (-0.7,0.9) -- (0,0);
\draw[black, thick] (0.7,0.9) -- (0,0);
\draw[snake=zigzag](0,0) -- (0,-0.9);
\filldraw[green] (-0.7,0.9) circle (6pt); 
\filldraw[green] (0.7,0.9) circle (6pt); 
\end{tikzpicture}
}\epsilon_{1}}_{i} (t_{1}) - b^{
\scalebox{0.6}{\begin{tikzpicture}
\draw[black, thick] (0,0.5) -- (0,0);
\filldraw[blue] (0,0.5) circle (2pt); 
\end{tikzpicture}
}\epsilon_{1}}_{j} u^{
\scalebox{0.18}{\begin{tikzpicture}
\draw[black, thick] (-0.7,0.9) -- (0,0);
\draw[black, thick] (0.7,0.9) -- (0,0);
\draw[snake=zigzag](0,0) -- (0,-0.9);
\filldraw[green] (-0.7,0.9) circle (6pt); 
\filldraw[green] (0.7,0.9) circle (6pt); 
\end{tikzpicture}
}\epsilon_{1}}_{i} (t_{2}) - b^{
\scalebox{0.6}{\begin{tikzpicture}
\draw[black, thick] (0,0.5) -- (0,0);
\filldraw[blue] (0,0.5) circle (2pt); 
\end{tikzpicture}
}\epsilon_{2}}_{j}  u^{
\scalebox{0.18}{\begin{tikzpicture}
\draw[black, thick] (-0.7,0.9) -- (0,0);
\draw[black, thick] (0.7,0.9) -- (0,0);
\draw[snake=zigzag](0,0) -- (0,-0.9);
\filldraw[green] (-0.7,0.9) circle (6pt); 
\filldraw[green] (0.7,0.9) circle (6pt); 
\end{tikzpicture}
}\epsilon_{2}}_{i} (t_{1}) + b^{
\scalebox{0.6}{\begin{tikzpicture}
\draw[black, thick] (0,0.5) -- (0,0);
\filldraw[blue] (0,0.5) circle (2pt); 
\end{tikzpicture}
}\epsilon_{2}}_{j} u^{
\scalebox{0.18}{\begin{tikzpicture}
\draw[black, thick] (-0.7,0.9) -- (0,0);
\draw[black, thick] (0.7,0.9) -- (0,0);
\draw[snake=zigzag](0,0) -- (0,-0.9);
\filldraw[green] (-0.7,0.9) circle (6pt); 
\filldraw[green] (0.7,0.9) circle (6pt); 
\end{tikzpicture}
}\epsilon_{2}}_{i}(t_{2})) \rvert^{2}] \nonumber\\
\lesssim& (\epsilon_{1}^{2\gamma} + \epsilon_{2}^{2\gamma}) \lvert t_{1} - t_{2} \rvert^{\eta\beta_{0}} 2^{q (1+ 2 \eta(1+ \beta_{0}))} 
\end{align} 
for $\epsilon_{1}, \epsilon_{2} \in (0,\eta)$, $\gamma > 0$ and $\beta_{0} \in (0, \frac{1}{4})$ sufficiently small. Now it is clear that 
\begin{equation}\label{140}
\lVert f \rVert_{\mathcal{C}^{ - \frac{1}{2} - \eta(1+ \beta_{0}) - \epsilon - \frac{3}{p}}} \lesssim \lVert f \rVert_{B_{p,\infty}^{-\frac{1}{2} - \eta(1+ \beta_{0}) - \epsilon}} 
\lesssim \lVert f \rVert_{B_{p,p}^{- \frac{1}{2} - \eta(1+ \beta_{0}) - \epsilon}} 
\end{equation} 
by Besov embedding (e.g., \cite{BCD11}). Therefore, 
\begin{align}\label{141}
& \mathbb{E} [ \lVert (b^{
\scalebox{0.6}{\begin{tikzpicture}
\draw[black, thick] (0,0.5) -- (0,0);
\filldraw[blue] (0,0.5) circle (2pt); 
\end{tikzpicture}
}\epsilon_{1}}_{j} u^{
\scalebox{0.18}{\begin{tikzpicture}
\draw[black, thick] (-0.7,0.9) -- (0,0);
\draw[black, thick] (0.7,0.9) -- (0,0);
\draw[snake=zigzag](0,0) -- (0,-0.9);
\filldraw[green] (-0.7,0.9) circle (6pt); 
\filldraw[green] (0.7,0.9) circle (6pt); 
\end{tikzpicture}
}\epsilon_{1}}_{i} (t_{1}) - b^{
\scalebox{0.6}{\begin{tikzpicture}
\draw[black, thick] (0,0.5) -- (0,0);
\filldraw[blue] (0,0.5) circle (2pt); 
\end{tikzpicture}
}\epsilon_{1}}_{j} u^{
\scalebox{0.18}{\begin{tikzpicture}
\draw[black, thick] (-0.7,0.9) -- (0,0);
\draw[black, thick] (0.7,0.9) -- (0,0);
\draw[snake=zigzag](0,0) -- (0,-0.9);
\filldraw[green] (-0.7,0.9) circle (6pt); 
\filldraw[green] (0.7,0.9) circle (6pt); 
\end{tikzpicture}
}\epsilon_{1}}_{i}(t_{2})   - b^{
\scalebox{0.6}{\begin{tikzpicture}
\draw[black, thick] (0,0.5) -- (0,0);
\filldraw[blue] (0,0.5) circle (2pt); 
\end{tikzpicture}
}\epsilon_{2}}_{j} u^{
\scalebox{0.18}{\begin{tikzpicture}
\draw[black, thick] (-0.7,0.9) -- (0,0);
\draw[black, thick] (0.7,0.9) -- (0,0);
\draw[snake=zigzag](0,0) -- (0,-0.9);
\filldraw[green] (-0.7,0.9) circle (6pt); 
\filldraw[green] (0.7,0.9) circle (6pt); 
\end{tikzpicture}
}\epsilon_{2}}_{j} (t_{1}) + b^{
\scalebox{0.6}{\begin{tikzpicture}
\draw[black, thick] (0,0.5) -- (0,0);
\filldraw[blue] (0,0.5) circle (2pt); 
\end{tikzpicture}
}\epsilon_{2}}_{j} u^{
\scalebox{0.18}{\begin{tikzpicture}
\draw[black, thick] (-0.7,0.9) -- (0,0);
\draw[black, thick] (0.7,0.9) -- (0,0);
\draw[snake=zigzag](0,0) -- (0,-0.9);
\filldraw[green] (-0.7,0.9) circle (6pt); 
\filldraw[green] (0.7,0.9) circle (6pt); 
\end{tikzpicture}
}\epsilon_{2}}_{i} (t_{2})) \rVert_{\mathcal{C}^{-\frac{1}{2} - \eta(1+ \beta_{0}) - \epsilon - \frac{3}{p}}}^{p} ] \nonumber\\
\lesssim& \mathbb{E} [ \sum_{q \geq -1} 2^{qp(- \frac{1}{2} - \eta(1+ \beta_{0}) - \epsilon)}  \lVert \lvert \Delta_{q} ( b^{
\scalebox{0.6}{\begin{tikzpicture}
\draw[black, thick] (0,0.5) -- (0,0);
\filldraw[blue] (0,0.5) circle (2pt); 
\end{tikzpicture}
}\epsilon_{1}}_{j} u^{
\scalebox{0.18}{\begin{tikzpicture}
\draw[black, thick] (-0.7,0.9) -- (0,0);
\draw[black, thick] (0.7,0.9) -- (0,0);
\draw[snake=zigzag](0,0) -- (0,-0.9);
\filldraw[green] (-0.7,0.9) circle (6pt); 
\filldraw[green] (0.7,0.9) circle (6pt); 
\end{tikzpicture}
}\epsilon_{1}}_{i} (t_{1}) - b^{
\scalebox{0.6}{\begin{tikzpicture}
\draw[black, thick] (0,0.5) -- (0,0);
\filldraw[blue] (0,0.5) circle (2pt); 
\end{tikzpicture}
}\epsilon_{1}}_{j} u^{
\scalebox{0.18}{\begin{tikzpicture}
\draw[black, thick] (-0.7,0.9) -- (0,0);
\draw[black, thick] (0.7,0.9) -- (0,0);
\draw[snake=zigzag](0,0) -- (0,-0.9);
\filldraw[green] (-0.7,0.9) circle (6pt); 
\filldraw[green] (0.7,0.9) circle (6pt); 
\end{tikzpicture}
}\epsilon_{1}}_{i} (t_{2}) \nonumber\\
&\hspace{16mm}   - b^{
\scalebox{0.6}{\begin{tikzpicture}
\draw[black, thick] (0,0.5) -- (0,0);
\filldraw[blue] (0,0.5) circle (2pt); 
\end{tikzpicture}
}\epsilon_{2}}_{j} u^{
\scalebox{0.18}{\begin{tikzpicture}
\draw[black, thick] (-0.7,0.9) -- (0,0);
\draw[black, thick] (0.7,0.9) -- (0,0);
\draw[snake=zigzag](0,0) -- (0,-0.9);
\filldraw[green] (-0.7,0.9) circle (6pt); 
\filldraw[green] (0.7,0.9) circle (6pt); 
\end{tikzpicture}
}\epsilon_{2}}_{i} (t_{1}) + b^{
\scalebox{0.6}{\begin{tikzpicture}
\draw[black, thick] (0,0.5) -- (0,0);
\filldraw[blue] (0,0.5) circle (2pt); 
\end{tikzpicture}
}\epsilon_{2}}_{j} u^{
\scalebox{0.18}{\begin{tikzpicture}
\draw[black, thick] (-0.7,0.9) -- (0,0);
\draw[black, thick] (0.7,0.9) -- (0,0);
\draw[snake=zigzag](0,0) -- (0,-0.9);
\filldraw[green] (-0.7,0.9) circle (6pt); 
\filldraw[green] (0.7,0.9) circle (6pt); 
\end{tikzpicture}
}\epsilon_{2}}_{i} (t_{2})) \rvert^{2} \rVert_{L^{\frac{p}{2}}}^{\frac{p}{2}}] \nonumber\\
\lesssim& \sum_{q \geq -1} 2^{qp(- \frac{1}{2} - \eta(1+ \beta_{0}) - \epsilon)}   \lVert \mathbb{E} [ \lvert \Delta_{q} ( b^{
\scalebox{0.6}{\begin{tikzpicture}
\draw[black, thick] (0,0.5) -- (0,0);
\filldraw[blue] (0,0.5) circle (2pt); 
\end{tikzpicture}
}\epsilon_{1}}_{j} u^{
\scalebox{0.18}{\begin{tikzpicture}
\draw[black, thick] (-0.7,0.9) -- (0,0);
\draw[black, thick] (0.7,0.9) -- (0,0);
\draw[snake=zigzag](0,0) -- (0,-0.9);
\filldraw[green] (-0.7,0.9) circle (6pt); 
\filldraw[green] (0.7,0.9) circle (6pt); 
\end{tikzpicture}
}\epsilon_{1}}_{i} (t_{1}) - b^{
\scalebox{0.6}{\begin{tikzpicture}
\draw[black, thick] (0,0.5) -- (0,0);
\filldraw[blue] (0,0.5) circle (2pt); 
\end{tikzpicture}
}\epsilon_{1}}_{j} u^{
\scalebox{0.18}{\begin{tikzpicture}
\draw[black, thick] (-0.7,0.9) -- (0,0);
\draw[black, thick] (0.7,0.9) -- (0,0);
\draw[snake=zigzag](0,0) -- (0,-0.9);
\filldraw[green] (-0.7,0.9) circle (6pt); 
\filldraw[green] (0.7,0.9) circle (6pt); 
\end{tikzpicture}
}\epsilon_{1}}_{i} (t_{2})\nonumber\\
& \hspace{16mm}  - b^{
\scalebox{0.6}{\begin{tikzpicture}
\draw[black, thick] (0,0.5) -- (0,0);
\filldraw[blue] (0,0.5) circle (2pt); 
\end{tikzpicture}
}\epsilon_{2}}_{j} u^{
\scalebox{0.18}{\begin{tikzpicture}
\draw[black, thick] (-0.7,0.9) -- (0,0);
\draw[black, thick] (0.7,0.9) -- (0,0);
\draw[snake=zigzag](0,0) -- (0,-0.9);
\filldraw[green] (-0.7,0.9) circle (6pt); 
\filldraw[green] (0.7,0.9) circle (6pt); 
\end{tikzpicture}
}\epsilon_{2}}_{i} (t_{1}) + b^{
\scalebox{0.6}{\begin{tikzpicture}
\draw[black, thick] (0,0.5) -- (0,0);
\filldraw[blue] (0,0.5) circle (2pt); 
\end{tikzpicture}
}\epsilon_{2}}_{j} u^{
\scalebox{0.18}{\begin{tikzpicture}
\draw[black, thick] (-0.7,0.9) -- (0,0);
\draw[black, thick] (0.7,0.9) -- (0,0);
\draw[snake=zigzag](0,0) -- (0,-0.9);
\filldraw[green] (-0.7,0.9) circle (6pt); 
\filldraw[green] (0.7,0.9) circle (6pt); 
\end{tikzpicture}
}\epsilon_{2}}_{i} (t_{2})) \rvert^{2}] \rVert_{L^{\frac{p}{2}}}^{\frac{p}{2}}  \lesssim (\epsilon_{1}^{p\gamma} + \epsilon_{2}^{p\gamma}) \lvert t_{1} - t_{2} \rvert^{\frac{ p\eta\beta_{0}}{2}} 
\end{align} 
by (\ref{140}), Gaussian hypercontractivity \cite[Theorem 3.50]{J97} and (\ref{139}). Thus, for every $i, j \in \{1,2,3 \}$, there exists $v^{
\scalebox{0.2}{\begin{tikzpicture}
\draw[black, thick] (-0.7,0.9) -- (0,0);
\draw[black, thick] (0.7,0.9) -- (0,0);
\draw[snake=zigzag](0,0) -- (0,-0.9);
\draw[black, thick] (-0.7,0.0) -- (0,-0.9);
\filldraw[blue] (-0.7,0.0) circle (6pt); 
\filldraw[green] (-0.7,0.9) circle (6pt); 
\filldraw[green] (0.7,0.9) circle (6pt); 
\end{tikzpicture}
}}_{9,ij}$ such that $b^{
\scalebox{0.6}{\begin{tikzpicture}
\draw[black, thick] (0,0.5) -- (0,0);
\filldraw[blue] (0,0.5) circle (2pt); 
\end{tikzpicture}
}\epsilon}_{i} \diamond u^{
\scalebox{0.18}{\begin{tikzpicture}
\draw[black, thick] (-0.7,0.9) -- (0,0);
\draw[black, thick] (0.7,0.9) -- (0,0);
\draw[snake=zigzag](0,0) -- (0,-0.9);
\filldraw[green] (-0.7,0.9) circle (6pt); 
\filldraw[green] (0.7,0.9) circle (6pt); 
\end{tikzpicture}
}\epsilon}_{j} \to v^{
\scalebox{0.2}{\begin{tikzpicture}
\draw[black, thick] (-0.7,0.9) -- (0,0);
\draw[black, thick] (0.7,0.9) -- (0,0);
\draw[snake=zigzag](0,0) -- (0,-0.9);
\draw[black, thick] (-0.7,0.0) -- (0,-0.9);
\filldraw[blue] (-0.7,0.0) circle (6pt); 
\filldraw[green] (-0.7,0.9) circle (6pt); 
\filldraw[green] (0.7,0.9) circle (6pt); 
\end{tikzpicture}
}}_{9,ij}$ as $\epsilon \to 0$ in $C([0,T]; \mathcal{C}^{-\frac{1}{2} - \frac{\delta}{2}})$ as desired in (\ref{113}) if $\eta(1+ \beta_{0}) + \epsilon + \frac{3}{p} \leq \frac{\delta}{2}$; therefore, by taking $p$ sufficiently large and $\eta, \epsilon, \beta_{0} > 0$ sufficiently small, we may assume that $\delta > 0$ is arbitrary small. Now to prove (\ref{139}), we may use that $b^{
\scalebox{0.6}{\begin{tikzpicture}
\draw[black, thick] (0,0.5) -- (0,0);
\filldraw[blue] (0,0.5) circle (2pt); 
\end{tikzpicture}
}\epsilon}_{j} (t) u^{
\scalebox{0.18}{\begin{tikzpicture}
\draw[black, thick] (-0.7,0.9) -- (0,0);
\draw[black, thick] (0.7,0.9) -- (0,0);
\draw[snake=zigzag](0,0) -- (0,-0.9);
\filldraw[green] (-0.7,0.9) circle (6pt); 
\filldraw[green] (0.7,0.9) circle (6pt); 
\end{tikzpicture}
}\epsilon}_{i}(t) = \sum_{l=1}^{6} \RomanII_{t,\epsilon}^{l}$ from (\ref{120}) so that 
\begin{align}\label{142}
&b^{
\scalebox{0.6}{\begin{tikzpicture}
\draw[black, thick] (0,0.5) -- (0,0);
\filldraw[blue] (0,0.5) circle (2pt); 
\end{tikzpicture}
}\epsilon_{1}}_{j}u^{
\scalebox{0.18}{\begin{tikzpicture}
\draw[black, thick] (-0.7,0.9) -- (0,0);
\draw[black, thick] (0.7,0.9) -- (0,0);
\draw[snake=zigzag](0,0) -- (0,-0.9);
\filldraw[green] (-0.7,0.9) circle (6pt); 
\filldraw[green] (0.7,0.9) circle (6pt); 
\end{tikzpicture}
}\epsilon_{1}}_{i}(t_{1}) - b^{
\scalebox{0.6}{\begin{tikzpicture}
\draw[black, thick] (0,0.5) -- (0,0);
\filldraw[blue] (0,0.5) circle (2pt); 
\end{tikzpicture}
}\epsilon_{1}}_{j} u^{
\scalebox{0.18}{\begin{tikzpicture}
\draw[black, thick] (-0.7,0.9) -- (0,0);
\draw[black, thick] (0.7,0.9) -- (0,0);
\draw[snake=zigzag](0,0) -- (0,-0.9);
\filldraw[green] (-0.7,0.9) circle (6pt); 
\filldraw[green] (0.7,0.9) circle (6pt); 
\end{tikzpicture}
}\epsilon_{1}}_{i}(t_{2}) - b^{
\scalebox{0.6}{\begin{tikzpicture}
\draw[black, thick] (0,0.5) -- (0,0);
\filldraw[blue] (0,0.5) circle (2pt); 
\end{tikzpicture}
}\epsilon_{2}}_{j} u^{
\scalebox{0.18}{\begin{tikzpicture}
\draw[black, thick] (-0.7,0.9) -- (0,0);
\draw[black, thick] (0.7,0.9) -- (0,0);
\draw[snake=zigzag](0,0) -- (0,-0.9);
\filldraw[green] (-0.7,0.9) circle (6pt); 
\filldraw[green] (0.7,0.9) circle (6pt); 
\end{tikzpicture}
}\epsilon_{2}}_{j}(t_{1}) + b^{
\scalebox{0.6}{\begin{tikzpicture}
\draw[black, thick] (0,0.5) -- (0,0);
\filldraw[blue] (0,0.5) circle (2pt); 
\end{tikzpicture}
}\epsilon_{2}}_{j} u^{
\scalebox{0.18}{\begin{tikzpicture}
\draw[black, thick] (-0.7,0.9) -- (0,0);
\draw[black, thick] (0.7,0.9) -- (0,0);
\draw[snake=zigzag](0,0) -- (0,-0.9);
\filldraw[green] (-0.7,0.9) circle (6pt); 
\filldraw[green] (0.7,0.9) circle (6pt); 
\end{tikzpicture}
}\epsilon_{2}}_{i}(t_{2})\nonumber\\
=& \left( \sum_{l=1}^{6} \RomanII_{t_{1}, \epsilon_{1}}^{l} \right) - \left( \sum_{l=1}^{6} \RomanII_{t_{2}, \epsilon_{1}}^{l} \right) - \left( \sum_{l=1}^{6} \RomanII_{t_{1}, \epsilon_{2}}^{l} \right) + \left( \sum_{l=1}^{6} \RomanII_{t_{2}, \epsilon_{2}}^{l} \right). 
\end{align} 
For brevity we only consider when $l= 5$, and rewrite 
\begin{align}\label{143}
\RomanII_{t_{1}, \epsilon_{1}}^{5} - \RomanII_{t_{2},\epsilon_{1}}^{5} - \RomanII_{t_{1}, \epsilon_{2}}^{5} + \RomanII_{t_{2},\epsilon_{2}}^{5} 
=& [\RomanII_{t_{1}, \epsilon_{1}}^{5} - \widetilde{\RomanII}_{t_{1}, \epsilon_{1}}^{5} + \widetilde{\RomanII}_{t_{1}, \epsilon_{1}}^{5} - \sum_{i_{1} = 1}^{3} X_{t_{1}, i_{1}}^{b, \epsilon_{1}} C_{t_{1}}^{\epsilon_{1}, i_{1}}]  \\
& - [\RomanII_{t_{2}, \epsilon_{1}}^{5} - \widetilde{\RomanII}_{t_{2}, \epsilon_{1}}^{5} + \widetilde{\RomanII}_{t_{2}, \epsilon_{1}}^{5} - \sum_{i_{1} = 1}^{3} X_{t_{2}, i_{1}}^{b, \epsilon_{1}} C_{t_{2}}^{\epsilon_{1}, i_{1}}]\nonumber\\
-& [\RomanII_{t_{1}, \epsilon_{2}}^{5} - \widetilde{\RomanII}_{t_{1}, \epsilon_{2}}^{5} + \widetilde{\RomanII}_{t_{1}, \epsilon_{2}}^{5} - \sum_{i_{1} = 1}^{3} X_{t_{1}, i_{1}}^{b, \epsilon_{2}} C_{t_{1}}^{\epsilon_{2}, i_{1}}] \nonumber\\
&+ [\RomanII_{t_{2}, \epsilon_{2}}^{5} - \widetilde{\RomanII}_{t_{2}, \epsilon_{2}}^{5} + \widetilde{\RomanII}_{t_{2}, \epsilon_{2}}^{5} - \sum_{i_{1} = 1}^{3} X_{t_{2}, i_{1}}^{b, \epsilon_{2}} C_{t_{2}}^{\epsilon_{2}, i_{1}}] = \sum_{i=1}^{16} \RomanIV^{i} \nonumber
\end{align} 
as we did in (\ref{122}) and (\ref{123}). For brevity we only consider $\RomanIV^{3} + \RomanIV^{4} + \RomanIV^{7} + \RomanIV^{8}$; i.e. $( \tilde{I}_{t_{1}, \epsilon_{1}}^{5} - \sum_{i_{1} = 1}^{3} X_{t_{1}, i_{1}}^{b, \epsilon_{1}} C_{t_{1}}^{\epsilon_{1}, i_{1}}) - ( \tilde{I}_{t_{2}, \epsilon_{1}}^{5} - \sum_{i_{1} = 1}^{3} X_{t_{2}, i_{1}}^{b, \epsilon_{1}} C_{t_{2}}^{\epsilon_{1}, i_{1}})$. We first compute 
\begin{align}\label{144}
& \mathbb{E} [ \lvert \Delta_{q} ( \tilde{I}_{t_{1}, \epsilon_{1}}^{5} - \sum_{i_{1} =1}^{3} X_{t_{1}, i_{1}}^{b, \epsilon_{1}} C_{t_{1}}^{\epsilon_{1}, i_{1}} - \tilde{I}_{t_{2}, \epsilon_{1}}^{5} + \sum_{i_{1} = 1}^{3} X_{t_{2}, i_{1}}^{b, \epsilon_{1}} C_{t_{2}}^{\epsilon_{1}, i_{1}}) \rvert^{2} ]\nonumber\\
\lesssim& \mathbb{E} [ \lvert \sum_{i_{1}, i_{2}, i_{3} = 1}^{3} \sum_{k_{1}} \hat{X}_{t_{1}, i_{1}}^{b, \epsilon_{1}} (k_{1}) \theta( 2^{-q}k_{1}) e_{k_{1}} \nonumber\\
& \times [ \sum_{k_{2} \neq  0} \int_{0}^{t_{1}} \frac{e^{ - \lvert k_{2} \rvert^{2} \lvert t_{1} - s \rvert} f(\epsilon_{1} k_{2})^{2}}{2 \lvert k_{2} \rvert^{2}} \hat{\mathcal{P}}_{i_{2} i_{3}}(k_{2}) \hat{\mathcal{P}}_{j i_{3}}(k_{2})\nonumber\\
& \hspace{5mm} \times \left( e^{- \lvert k_{12} \rvert^{2} \lvert t_{1} -s \rvert} k_{12}^{i_{2}} \hat{\mathcal{P}}_{ ii_{1}} (k_{12})  - e^{- \lvert k_{2} \rvert^{2} \lvert t_{1} - s \rvert} k_{2}^{i_{2}} \hat{\mathcal{P}}_{ii_{1}} (k_{2})  \right) ds \nonumber\\
&  \hspace{5mm}- \sum_{k_{2} \neq 0} \int_{0}^{t_{2}} \frac{ e^{- \lvert k_{2} \rvert^{2} \lvert t_{2} - s \rvert} f( \epsilon_{1} k_{2})^{2}}{ 2 \lvert k_{2} \rvert^{2}} \hat{\mathcal{P}}_{i_{2} i_{3}}(k_{2}) \hat{\mathcal{P}}_{ j i_{3}}(k_{2})\nonumber\\
& \hspace{5mm} \times \left( e^{- \lvert k_{12} \rvert^{2} \lvert t_{2} -s \rvert} k_{12}^{i_{2}} \hat{\mathcal{P}}_{ ii_{1}} (k_{12}) - e^{- \lvert k_{2} \rvert^{2} \lvert t_{2} - s \rvert} k_{2}^{i_{2}} \hat{\mathcal{P}}_{ii_{1}} (k_{2})  \right) ds] \rvert^{2} ]\nonumber\\
&+ \mathbb{E} [ \lvert \sum_{i_{1}, i_{2}, i_{3} = 1}^{3} \sum_{k_{1}} \left( \hat{X}_{t_{1}, i_{1}}^{b, \epsilon_{1}} (k_{1}) - \hat{X}_{t_{2}, i_{1}}^{b, \epsilon_{1}}(k_{1})\right) \nonumber\\
& \times \theta (2^{-q}  k_{1}) e_{k_{1}} \hat{\mathcal{P}}_{i_{2} i_{3}}(k_{2}) \hat{\mathcal{P}}_{j i_{3}}(k_{2})  [ \sum_{k_{2} \neq 0} \int_{0}^{t_{2}} \frac{ e^{- \lvert k_{2} \rvert^{2} \lvert t_{2} - s \rvert} f(\epsilon_{1} k_{2})^{2}}{ 2 \lvert k_{2} \rvert^{2}} \nonumber\\
& \hspace{5mm} \times \left( e^{- \lvert k_{12} \rvert^{2}  \lvert t_{2} -s \rvert} k_{12}^{i_{2}} \hat{\mathcal{P}}_{ ii_{1}} (k_{12}) - e^{- \lvert k_{2} \rvert^{2} \lvert t_{2} -s \rvert } k_{2}^{i_{2}} \hat{\mathcal{P}}_{ ii_{1}} (k_{2}) \right) ds ] \rvert^{2} ]
\end{align} 
by (\ref{123}). Now we have two expectations in (\ref{144}). For the first expectation in (\ref{144}), we can simply rewrite it for $0 \leq t_{1} < t_{2} \leq T$ as 
\begin{align}\label{145}
&\mathbb{E} [ \lvert \sum_{i_{1}, i_{2}, i_{3} = 1}^{3} \sum_{k_{1}} \hat{X}_{t_{1}, i_{1}}^{b, \epsilon_{1}} (k_{1}) \theta( 2^{-q}k_{1}) e_{k_{1}}\\
& \times [ \sum_{k_{2} \neq 0} \int_{0}^{t_{1}} \frac{e^{ - \lvert k_{2} \rvert^{2} \lvert t_{1} - s \rvert} f(\epsilon_{1} k_{2})^{2}}{2 \lvert k_{2} \rvert^{2}} \hat{\mathcal{P}}_{i_{2} i_{3}}(k_{2}) \hat{\mathcal{P}}_{j i_{3}}(k_{2})\nonumber\\
& \hspace{5mm} \times \left( e^{- \lvert k_{12} \rvert^{2} \lvert t_{1} -s \rvert} k_{12}^{i_{2}} \hat{\mathcal{P}}_{ ii_{1}} (k_{12}) - e^{- \lvert k_{2} \rvert^{2} \lvert t_{1} - s \rvert} k_{2}^{i_{2}} \hat{\mathcal{P}}_{ii_{1}} (k_{2})\right) ds \nonumber\\
&  \hspace{5mm} - \sum_{k_{2} \neq 0} \int_{0}^{t_{2}} \frac{ e^{- \lvert k_{2} \rvert^{2} \lvert t_{2} - s \rvert} f( \epsilon_{1} k_{2})^{2}}{ 2 \lvert k_{2} \rvert^{2}} \hat{\mathcal{P}}_{i_{2} i_{3}}(k_{2}) \hat{\mathcal{P}}_{j i_{3}}(k_{2}) \nonumber\\
& \hspace{5mm} \times \left( e^{- \lvert k_{12} \rvert^{2} \lvert t_{2} -s \rvert} k_{12}^{i_{2}} \hat{\mathcal{P}}_{ ii_{1}} (k_{12})  - e^{- \lvert k_{2} \rvert^{2} \lvert t_{2} - s \rvert} k_{2}^{i_{2}} \hat{\mathcal{P}}_{ii_{1}} (k_{2}) \right) ds] \rvert^{2} ] \lesssim V_{t_{1}}^{1} + V_{t_{1}}^{2} + V_{t_{1}, t_{2}}^{3} \nonumber
\end{align} 
where 
\begin{subequations}
\begin{align}
V_{t_{1}}^{1} \triangleq& \sum_{k_{1} \neq 0} \sum_{i_{1}, i_{2} =1}^{3} \frac{1}{\lvert k_{1} \rvert^{2}} \theta (2^{-q} k_{1})^{2} [ \sum_{k_{2} \neq 0}  \int_{0}^{t_{1}} \frac{ e^{- \lvert k_{2} \rvert^{2}(t_{1} -s)} (1- e^{- \lvert k_{2} \rvert^{2}(t_{2} - t_{1})})}{\lvert k_{2} \rvert^{2}} \nonumber\\
& \hspace{5mm} \times \left( e^{- \lvert k_{12} \rvert^{2}(t_{1} - s)} k_{12}^{i_{2}} \hat{\mathcal{P}}_{ii_{1}}(k_{12}) - e^{- \lvert k_{2} \rvert^{2}(t_{1} -s)} k_{2}^{i_{2}} \hat{\mathcal{P}}_{ii_{1}} (k_{2}) \right) ds]^{2}, \label{146a}\\
V_{t_{1}}^{2} \triangleq& \sum_{k_{1} \neq 0} \sum_{i_{1}, i_{2} =1}^{3} \frac{1}{\lvert k_{1} \rvert^{2}} \theta (2^{-q} k_{1})^{2} [ \sum_{k_{2} \neq 0}  \int_{0}^{t_{1}} \frac{ e^{- \lvert k_{2} \rvert^{2}(t_{2} -s)}}{\lvert k_{2} \rvert^{2}} \nonumber\\
& \times ( e^{- \lvert k_{12} \rvert^{2}(t_{1} - s)} k_{12}^{i_{2}} \hat{\mathcal{P}}_{ii_{1}}(k_{12}) - e^{- \lvert k_{2} \rvert^{2}(t_{1} -s)} k_{2}^{i_{2}} \hat{\mathcal{P}}_{ii_{1}} (k_{2}) \nonumber\\
& \hspace{5mm} -e^{- \lvert k_{12} \rvert^{2}(t_{2} - s)} k_{12}^{i_{2}} \hat{\mathcal{P}}_{ii_{1}}(k_{12}) + e^{- \lvert k_{2} \rvert^{2}(t_{2} -s)} k_{2}^{i_{2}} \hat{\mathcal{P}}_{ii_{1}} (k_{2})) ds]^{2}, \label{146b}\\
V_{t_{1}, t_{2}}^{3} \triangleq& \sum_{k_{1} \neq 0} \sum_{i_{1}, i_{2} = 1}^{3} \frac{1}{\lvert k_{1} \rvert^{2}} \theta(2^{-q} k_{1})^{2}[ \sum_{k_{2} \neq 0} \int_{t_{1}}^{t_{2}} \frac{ e^{- \lvert k_{2} \rvert^{2}(t_{2} -s)} }{\lvert k_{2} \rvert^{2}} \nonumber\\
& \hspace{5mm} \times \left( e^{- \lvert k_{12} \rvert^{2}(t_{2} -s)} k_{12}^{i_{2}} \hat{\mathcal{P}}_{ii_{1}}(k_{12}) - e^{- \lvert k_{2} \rvert^{2}(t_{2} -s)} k_{2}^{i_{2}} \hat{\mathcal{P}}_{ii_{1}} (k_{2}) \right) ds]^{2}\label{146c}  
\end{align}
\end{subequations} 
due to (\ref{116}). On the other hand, the second expectation in (\ref{144}) may be bounded clearly as follows: 
\begin{align}\label{147}
&\mathbb{E} [ \lvert \sum_{i_{1}, i_{2}, i_{3} = 1}^{3} \sum_{k_{1}} \left( \hat{X}_{t_{1}, i_{1}}^{b, \epsilon_{1}} (k_{1}) - \hat{X}_{t_{2}, i_{1}}^{b, \epsilon_{1}}(k_{1})\right) \theta (2^{-q}  k_{1}) e_{k_{1}}\hat{\mathcal{P}}_{i_{2} i_{3}}(k_{2}) \hat{\mathcal{P}}_{ji_{3}}(k_{2}) \nonumber\\
& \times [ \sum_{k_{2} \neq 0} \int_{0}^{t_{2}} \frac{ e^{- \lvert k_{2} \rvert^{2} \lvert t_{2} - s \rvert} f(\epsilon_{1} k_{2})^{2}}{ 2 \lvert k_{2} \rvert^{2}} ( e^{- \lvert k_{12} \rvert^{2} \lvert t_{2} -s \rvert} k_{12}^{i_{2}} \hat{\mathcal{P}}_{ ii_{1}} (k_{12})\nonumber\\
& \hspace{20mm}  - e^{- \lvert k_{2} \rvert^{2} \lvert t_{2} -s \rvert } k_{2}^{i_{2}} \hat{\mathcal{P}}_{ ii_{1}} (k_{2}) ) ds ] \rvert^{2} ]\nonumber\\
\lesssim& \sum_{i_{1}, i_{2} = 1}^{3} \sum_{k_{1}, k_{2} \neq 0} \mathbb{E} [ \lvert ( \hat{X}_{t_{1}, i_{1}}^{b, \epsilon_{1}} (k_{1}) - \hat{X}_{t_{2}, i_{1}}^{b, \epsilon_{1}}(k_{1})) \theta (2^{-q} k_{1})  \int_{0}^{t_{2}} \frac{e^{- \lvert k_{2} \rvert^{2}(t_{2} - s)}}{ \lvert k_{2} \rvert^{2}}  \nonumber\\
& \times\left( e^{- \lvert k_{12} \rvert^{2} (t_{2} -s)} k_{12}^{i_{2}} \hat{\mathcal{P}}_{ii_{1}}(k_{12}) - e^{- \lvert k_{2} \rvert^{2}(t_{2} - s)} k_{2}^{i_{2}} \hat{\mathcal{P}}_{ ii_{}}(k_{2}) \right) ds \rvert^{2}] \triangleq V_{t_{2}}^{4}
\end{align} 
where we used that $\hat{X}_{t_{1}, i_{1}}^{b, \epsilon_{1}}(0) - \hat{X}_{t_{2}, i_{1}}^{b, \epsilon_{1}}(0) = 0$. 
Now on $V_{t_{1}}^{2}$, we may bound 
\begin{align}\label{148}
&  \lvert e^{ - \lvert k_{12} \rvert^{2} (t_{1} -s)} k_{12}^{i_{2}} \hat{\mathcal{P}}_{ii_{1}}(k_{12}) - e^{ - \lvert k_{2} \rvert^{2} (t_{1} -s)} k_{2}^{i_{2}} \hat{\mathcal{P}}_{ii_{1}}(k_{2}) \nonumber\\
& - e^{- \lvert k_{12} \rvert^{2}(t_{2} -s)} k_{12}^{i_{2}} \hat{\mathcal{P}}_{ii_{1}}(k_{12})+  e^{- \lvert k_{2} \rvert^{2} (t_{2} -s)} k_{2}^{i_{2}} \hat{\mathcal{P}}_{ ii_{1}}(k_{2}) \rvert \nonumber\\
\leq& \lvert e^{- \lvert k_{12} \rvert^{2}(t_{1} -s)} k_{12}^{i_{2}} \hat{\mathcal{P}}_{ii_{1}}(k_{12}) - e^{- \lvert k_{2} \rvert^{2}(t_{1} -s)} k_{2}^{i_{2}} \hat{\mathcal{P}}_{ii_{1}}(k_{2}) \rvert \nonumber\\
&+ \lvert e^{- \lvert k_{12} \rvert^{2}(t_{2} -s)} k_{12}^{i_{2}} \hat{\mathcal{P}}_{ ii_{1}}(k_{12}) - e^{- \lvert k_{2} \rvert^{2}(t_{2} -s)} k_{2}^{i_{2}} \hat{\mathcal{P}}_{ ii_{1}}(k_{2}) \rvert 
\end{align} 
or we may bound it instead by 
\begin{align}\label{149}
& \lvert e^{- \lvert k_{12} \rvert^{2} (t_{1} -s)} k_{12}^{i_{2}} \hat{\mathcal{P}}_{ ii_{1}} (k_{12}) - e^{- \lvert k_{12} \rvert^{2}(t_{2} -s)} k_{12}^{i_{2}} \hat{\mathcal{P}}_{ ii_{1}}(k_{12}) \rvert \nonumber\\
&+ \lvert e^{- \lvert k_{2} \rvert^{2}(t_{1} -s)} k_{2}^{i_{2}} \hat{\mathcal{P}}_{ ii_{1}}(k_{2}) - e^{ - \lvert k_{2} \rvert^{2}(t_{2} -s)} k_{2}^{i_{2}} \hat{\mathcal{P}}_{ ii_{1}}(k_{2}) \rvert.  
\end{align} 
In the first case of (\ref{148}) we may bound by 
\begin{equation}\label{150}
 \lvert k_{1} \rvert^{\eta} \lvert t_{1} -s \rvert^{ - \frac{(1-\eta)}{2}} + \lvert k_{1} \rvert^{\eta} \lvert t_{2} -s \rvert^{- \frac{(1-\eta)}{2}} \lesssim \lvert k_{1} \rvert^{\eta} \lvert t_{1} -s \rvert^{- \frac{(1-\eta)}{2}} 
\end{equation} 
for $\eta \in (0,1)$ due to Lemma \ref{Lemma 2.8}. In the second case of (\ref{149}) we may bound by
\begin{align}\label{151}
& \lvert k_{12} \rvert \lvert [ e^{- \lvert k_{12} \rvert^{2} (t_{1} -s)} - e^{- \lvert k_{12} \rvert^{2} (t_{2} -s)}]  \hat{\mathcal{P}}_{ ii_{1}}(k_{12}) \rvert \\
& \hspace{2mm} + \lvert k_{2} \rvert \lvert [ e^{- \lvert k_{2} \rvert^{2}(t_{1} -s)} - e^{- \lvert k_{2} \rvert^{2}(t_{2} -s)}] \hat{\mathcal{P}}_{ ii_{1}}(k_{2}) \rvert  
\lesssim ( \lvert k_{12} \rvert^{2\eta} + \lvert k_{2} \rvert^{2\eta}) \lvert t_{2} - t_{1} \rvert^{\frac{\eta}{2}} (t_{1} -s)^{ - (\frac{1-\eta}{2})} \nonumber
\end{align} 
due to mean value theorem and (\ref{14}). Applying (\ref{148})-(\ref{151}) to (\ref{146b}) gives for any $\beta_{0} \in (0,1)$, 
\begin{align}\label{152}
V_{t_{1}}^{2} \lesssim& \sum_{k_{1} \neq 0} \frac{ \lvert k_{1} \rvert^{2\eta (1-\beta_{0})}}{ \lvert k_{1} \rvert^{2}} \theta(2^{-q} k_{1})^{2} \lvert t_{2} - t_{1} \rvert^{\eta \beta_{0}} \nonumber\\
\times& \left( \sum_{k_{2} \neq 0} \frac{1}{\lvert k_{2} \rvert^{2}} ( \lvert k_{12} \rvert^{2\eta \beta_{0}} + \lvert k_{2} \rvert^{2\eta \beta_{0}}) \int_{0}^{t_{1}} e^{- \lvert k_{2} \rvert^{2}(t_{2} -s)} (t_{1} -s)^{- \frac{(1-\eta)}{2}} ds \right)^{2}. 
\end{align}  
Furthermore, we can compute 
\begin{equation*}
\int_{0}^{t_{1}} e^{- \lvert k_{2} \rvert^{2}(t_{2} -s)} (t_{1} -s)^{- \frac{(1-\eta)}{2}} ds \lesssim \int_{0}^{t_{1}} e^{- \lvert k_{2} \rvert^{2}(t_{1} -s)} (t_{1} -s)^{- \frac{(1-\eta)}{2}} ds  \lesssim \lvert k_{2} \rvert^{- (1+ \frac{\eta}{2})}
\end{equation*} 
by (\ref{14}). Therefore, we may estimate from (\ref{152}) 
\begin{equation}\label{153}
V_{t_{1}}^{2} \lesssim \lvert t_{2} - t_{1} \rvert^{\eta \beta_{0}} 2^{q(1+ 2 \eta(1+ \beta_{0}))} \sum_{k_{1} \neq 0} \frac{ \theta (2^{-q} k_{1})}{\lvert k_{1} \rvert^{3}} \lesssim \lvert t_{2} - t_{1} \rvert^{\eta \beta_{0}} 2^{q(1+ 2 \eta(1+ \beta_{0}))} 
\end{equation}
if we choose $\beta_{0} < \frac{1}{4}$. Similar estimates may be obtained for $V_{t_{1}}^{1}, V_{t_{1}}^{3}$ and $V_{t_{1}}^{4}$ so that applying these estimates in (\ref{144}) and (\ref{145}) lead to 
\begin{equation}\label{265}
\mathbb{E} [ \lvert \Delta_{q} ( \RomanIV^{3} + \RomanIV^{4} + \RomanIV^{7} + \RomanIV^{8}) \rvert^{2}] \lesssim \lvert t_{2} - t_{1} \rvert^{\eta \beta_{0}} 2^{q(1+ 2\eta ( 1+ \beta_{0}))}. 
\end{equation}
Through (\ref{143}) and (\ref{142}), this finally leads to (\ref{139}). 

\begin{remark}\label{new remark}
Our estimate in \eqref{139} is slightly different from the analogous bound, specifically ``$(\epsilon_{1}^{p\gamma} + \epsilon_{2}^{p\gamma}) \lvert t_{1} - t_{2} \rvert^{p( \eta - \epsilon)/2}$,'' in \cite[Equation (A.2)]{ZZ15}. Moreover, our estimate in \eqref{265} also differs from the analogous bound of ``$\lvert t_{1} - t_{2} \rvert^{\frac{n \beta_{0}}{2}} 2^{q(1+ 2 \eta (1+ \beta_{0}))}$'' on \cite[p. 4504]{ZZ15}. 
\end{remark}

\subsection{Group 2}
Within the Group 2 of (\ref{114}), specifically $u^{
\scalebox{0.18}{

}\epsilon}_{j}$, we focus on $b^{
\scalebox{0.18}{%
}\epsilon}_{j}$ and prove that $b^{
\scalebox{0.18}{%
}}_{13,ij}$ as $\epsilon \to 0$ in $C([0,T]; \mathcal{C}^{-\delta})$. Due to similarity to the estimates for Group 1, we leave this in the Appendix. 

\subsection{Group 3}
Within $\pi_{0, \diamond} (u^{
\scalebox{0.16}{%
}\epsilon}_{j})$ from the Group 3 of (\ref{114}), we focus on $\pi_{0, \diamond} (u^{
\scalebox{0.16}{%
}\epsilon}_{j})$. Considering (\ref{18}), (\ref{24a})-(\ref{24b}) and (\ref{17}) we see that we may rewrite $L u^{
\scalebox{0.16}{%
}\epsilon}_{j_{0}})
\end{equation} 
due to linearity. Now by necessity, as we will see, we shall actually work on $\pi_{0} (u_{1, i_{0}}^{
\scalebox{0.16}{%
}\epsilon}_{j_{0}})$. Without loss of generality we work on the last one, elaborating on the computations of $u_{8, i_{0}}^{
\scalebox{0.16}{%
}\epsilon}$ first. First, we see from (\ref{181}) that 
\begin{align}\label{183}
 \pi_{0} ( u_{8, i_{0}}^{
\scalebox{0.16}{%
}\epsilon}_{j_{0}})(t) =& - \frac{1}{4 (2\pi)^{\frac{9}{2}}} \sum_{k} \sum_{ \lvert i-j \rvert \leq 1} \sum_{k_{1}, k_{2}, k_{3}, k_{4}: k_{1234} = k} \sum_{i_{1}, i_{2}, i_{3}, j_{1} =1}^{3} \theta (2^{-i} k_{123}) \theta (2^{-j} k_{4})\nonumber\\
& \times  \int_{0}^{t} e^{- \lvert k_{123} \rvert^{2}(t-s)}  \int_{0}^{s} \hat{X}_{\sigma, i_{2}}^{u, \epsilon}(k_{1}) \hat{X}_{\sigma, i_{3}}^{b, \epsilon}(k_{2}) \hat{X}_{s, j_{1}}^{b, \epsilon}(k_{3}) \hat{X}_{t, j_{0}}^{b, \epsilon}(k_{4})\nonumber\\
& \times  e^{- \lvert k_{12} \rvert^{2} (s-\sigma)} d\sigma ds k_{12}^{i_{3}} k_{123}^{j_{1}} \hat{\mathcal{P}}_{i_{1} i_{2}}(k_{12}) \hat{\mathcal{P}}_{i_{0} i_{1}} (k_{123}) e_{k}.
\end{align} 
By using the well-known expression of $:\xi_{1}\xi_{2}\xi_{3}\xi_{4}:$ (\cite{J97} and \cite[Example 2.2]{Y20}) we can rewrite 
\begin{align}\label{184}
&\pi_{0} ( u_{8, i_{0}}^{
\scalebox{0.16}{\begin{tikzpicture}
\draw[black, thick] (-0.7,0.9) -- (0,0);
\draw[black, thick] (0.7,0.9) -- (0,0);
\draw[black, thick] (0.7,0) -- (0,-0.9);
\draw[snake=zigzag](0,0) -- (0,-0.9);
\draw[snake=zigzag](0,-0.9) -- (0,-1.8);
\filldraw[pink] (-0.7,0.9) circle (7pt); 
\filldraw[pink] (0.7,0.9) circle (7pt); 
\filldraw[pink] (0.7,0) circle (7pt); 
\end{tikzpicture}
}\epsilon}, b^{
\scalebox{0.6}{\begin{tikzpicture}
\draw[black, thick] (0,0.5) -- (0,0);
\filldraw[blue] (0,0.5) circle (2pt); 
\end{tikzpicture}
}\epsilon}_{j_{0}})(t) \\
=&  \frac{-1}{4(2\pi)^{\frac{9}{2}}} \sum_{k} \sum_{\lvert i-j \rvert \leq 1} \sum_{k_{1}, k_{2}, k_{3}, k_{4}: k_{1234} = k} \sum_{i_{1}, i_{2},  i_{3}, j_{1} =1}^{3} \theta (2^{-i} k_{123}) \theta (2^{-j} k_{4}) \int_{0}^{t} e^{- \lvert k_{123} \rvert^{2}(t-s)}  \nonumber\\
& \times  \int_{0}^{s} [: \hat{X}_{\sigma, i_{2}}^{u, \epsilon}(k_{1}) \hat{X}_{\sigma, i_{3}}^{b, \epsilon}(k_{2}) \hat{X}_{s, j_{1}}^{b, \epsilon}(k_{3}) \hat{X}_{t, j_{0}}^{b, \epsilon}(k_{4}): \nonumber\\
& + \mathbb{E} [ \hat{X}_{\sigma, i_{2}}^{u, \epsilon}(k_{1}) \hat{X}_{\sigma, i_{3}}^{b, \epsilon}(k_{2})]: \hat{X}_{s, j_{1}}^{b, \epsilon}(k_{3}) \hat{X}_{t, j_{0}}^{b, \epsilon}(k_{4}): + \mathbb{E} [ \hat{X}_{\sigma, i_{2}}^{u, \epsilon}(k_{1}) \hat{X}_{s, j_{1}}^{b, \epsilon}(k_{3})]: \hat{X}_{\sigma, i_{3}}^{b, \epsilon}(k_{2}) \hat{X}_{t, j_{0}}^{b, \epsilon}(k_{4}):  \nonumber\\
&+ \mathbb{E} [ \hat{X}_{\sigma, i_{2}}^{u, \epsilon}(k_{1}) \hat{X}_{t, j_{0}}^{b, \epsilon}(k_{4})]: \hat{X}_{\sigma, i_{3}}^{b, \epsilon}(k_{2}) \hat{X}_{s, j_{1}}^{b, \epsilon}(k_{3}): + \mathbb{E} [ \hat{X}_{\sigma, i_{3}}^{b, \epsilon}(k_{2}) \hat{X}_{s, j_{1}}^{b, \epsilon}(k_{3})]: \hat{X}_{\sigma, i_{2}}^{u, \epsilon}(k_{1}) \hat{X}_{t, j_{0}}^{b, \epsilon}(k_{4}):\nonumber\\
& + \mathbb{E} [ \hat{X}_{\sigma, i_{3}}^{b, \epsilon}(k_{2}) \hat{X}_{t, j_{0}}^{b, \epsilon}(k_{4})]: \hat{X}_{\sigma, i_{2}}^{u, \epsilon}(k_{1}) \hat{X}_{s, j_{1}}^{b, \epsilon}(k_{3}): + \mathbb{E} [ \hat{X}_{s, j_{1}}^{b, \epsilon}(k_{3}) \hat{X}_{t, j_{0}}^{b, \epsilon}(k_{4})]: \hat{X}_{\sigma, i_{2}}^{u, \epsilon}(k_{1}) \hat{X}_{\sigma, i_{3}}^{b, \epsilon}(k_{2}): \nonumber\\
&+ \mathbb{E} [ \hat{X}_{\sigma, i_{3}}^{b, \epsilon}(k_{2}) \hat{X}_{s, j_{1}}^{b, \epsilon}(k_{3})] \mathbb{E} [ \hat{X}_{\sigma, i_{2}}^{u, \epsilon}(k_{1}) \hat{X}_{t, j_{0}}^{b, \epsilon}(k_{4})]+  \mathbb{E} [ \hat{X}_{\sigma, i_{3}}^{b, \epsilon}(k_{2}) \hat{X}_{t, j_{0}}^{b, \epsilon}(k_{4})] \mathbb{E} [ \hat{X}_{\sigma, i_{2}}^{u, \epsilon}(k_{1}) \hat{X}_{s, j_{1}}^{b, \epsilon}(k_{3})]  \nonumber\\
&+  \mathbb{E} [ \hat{X}_{s, j_{1}}^{b, \epsilon}(k_{3}) \hat{X}_{t, j_{0}}^{b, \epsilon}(k_{4})] \mathbb{E} [ \hat{X}_{\sigma, i_{2}}^{u, \epsilon}(k_{1}) \hat{X}_{\sigma, i_{3}}^{b, \epsilon}(k_{2})] ] \nonumber\\
 &  \times e^{- \lvert k_{12} \rvert^{2} (s- \sigma)} d \sigma ds k_{12}^{i_{3}} k_{123}^{j_{1}} \hat{\mathcal{P}}_{ i_{1} i_{2}}(k_{12}) \hat{\mathcal{P}}_{i_{0} i_{1}}(k_{123}) e_{k} \triangleq \RomanIX_{t}^{8,1} + \sum_{j=1}^{9} \RomanVIII_{t}^{8, j}  \nonumber
\end{align} 
where $\RomanVIII_{t}^{8,1}$ and $\RomanVIII_{t}^{8,9}$ vanish due to $1_{k_{12} = 0}$ and $k_{12}^{i_{3}}$ within the integrand. Using (\ref{116}) we may compute  
\begin{align}\label{185}
\RomanVIII_{t}^{8,2} 
=& \frac{-1}{4(2\pi)^{\frac{9}{2}}}  \sum_{k} \sum_{\lvert i-j \rvert \leq 1} \sum_{k_{1}, k_{4}: k_{14} = k, k_{2} \neq 0} \sum_{i_{1}, i_{2}, i_{3}, i_{4}, j_{1} =1}^{3} \theta(2^{-i} k_{1}) \theta (2^{-j} k_{4}) \\
& \times \int_{0}^{t} e^{- \lvert k_{1} \rvert^{2}(t-s)}  \int_{0}^{s} : \hat{X}_{\sigma, i_{5}}^{b, \epsilon}(k_{1}) \hat{X}_{t, j_{0}}^{b, \epsilon}(k_{4}): \frac{ e^{- \lvert k_{2} \rvert^{2} (s- \sigma)} f(\epsilon k_{2})^{2}}{2 \lvert k_{2} \rvert^{2}}\nonumber\\
& \times \hat{\mathcal{P}}_{i_{6} i_{4}} (k_{2}) \hat{\mathcal{P}}_{j_{1} i_{4}}(k_{2}) e^{- \lvert k_{12} \rvert^{2}(s-\sigma)} d \sigma ds k_{12}^{i_{3}} k_{1}^{j_{1}} \hat{\mathcal{P}}_{i_{1} i_{2}} (k_{12}) \hat{\mathcal{P}}_{i_{0} i_{1}}(k_{1}) e_{k}1_{i_{5} = i_{3}, i_{6} = i_{2}} \triangleq \RomanIX_{t}^{8,6} \nonumber
\end{align} 
by switching variables $k_{1}$ and $k_{2}$. Next, we similarly compute using (\ref{116}),  
\begin{align}\label{186}
\RomanVIII_{t}^{8,3} =& - \frac{1}{4(2\pi)^{\frac{9}{2}}} \sum_{k} \sum_{\lvert i - j \rvert \leq 1} \sum_{k_{2}, k_{3}: k_{23} = k, k_{1}\neq 0} \sum_{i_{1}, i_{2}, i_{3}, i_{4}, j_{1} =1}^{3}\nonumber \\
& \times \theta(2^{-i} k_{123}) \theta(2^{-j} k_{1}) \int_{0}^{t} e^{- \lvert k_{123} \rvert^{2}(t-s)}  \int_{0}^{s} : \hat{X}_{\sigma, i_{5}}^{b, \epsilon} (k_{2}) \hat{X}_{s, j_{1}}^{b, \epsilon}(k_{3}): \nonumber \\
& \times \frac{ e^{- \lvert k_{1} \rvert^{2} (t-\sigma)} f(\epsilon k_{1})^{2}}{2 \lvert k_{1} \rvert^{2}} \hat{\mathcal{P}}_{i_{6} i_{4}}(k_{1}) \hat{\mathcal{P}}_{j_{0} i_{4}}(k_{1}) e^{- \lvert k_{12} \rvert^{2}(s-\sigma)} d \sigma ds\nonumber \\
& \times k_{12}^{i_{3}} k_{123}^{j_{1}} \hat{\mathcal{P}}_{i_{1} i_{2}}(k_{12}) \hat{\mathcal{P}}_{i_{0} i_{1}}(k_{123}) e_{k}1_{i_{5} = i_{3}, i_{6} = i_{2}} \triangleq \RomanIX_{t}^{8,2}, 
\end{align} 
\begin{align}\label{187}
\RomanVIII_{t}^{8,4} =& - \frac{1}{4(2\pi)^{\frac{9}{2}}} \sum_{k} \sum_{\lvert i-j \rvert \leq 1} \sum_{k_{1}, k_{4}: k_{14} = k, k_{2} \neq 0} \sum_{i_{1}, i_{2}, i_{3}, i_{4}, j_{1} =1}^{3} \theta(2^{-i} k_{1})  \nonumber \\
& \times \theta(2^{-j} k_{4})\int_{0}^{t} e^{- \lvert k_{1} \rvert^{2}(t-s)}  \int_{0}^{s} \frac{ e^{- \lvert k_{2} \rvert^{2} (s-\sigma)} f( \epsilon k_{2})^{2}}{2 \lvert k_{2} \rvert^{2}} \hat{\mathcal{P}}_{i_{6} i_{4}}(k_{2}) \nonumber \\
& \times \hat{\mathcal{P}}_{j_{1} i_{4}}(k_{2})  : \hat{X}_{\sigma, i_{5}}^{u, \epsilon}(k_{1}) \hat{X}_{t, j_{0}}^{b, \epsilon}(k_{4}):  e^{- \lvert k_{12} \rvert^{2} (s-\sigma)} d \sigma ds k_{12}^{i_{3}} k_{1}^{j_{1}} \nonumber \\
& \hspace{30mm} \times \hat{\mathcal{P}}_{i_{1} i_{2}}(k_{12}) \hat{\mathcal{P}}_{i_{0} i_{1}}(k_{1}) e_{k} 1_{i_{5} = i_{2}, i_{6} = i_{3}} \triangleq \RomanIX_{t}^{8,5},    
\end{align} 
\begin{align}\label{188}
\RomanVIII_{t}^{8,5} =& \frac{-1}{4(2\pi)^{\frac{9}{2}}} \sum_{k} \sum_{\lvert i-j \rvert \leq 1} \sum_{k_{2}, k_{3}: k_{23} = k, k_{1} \neq 0} \sum_{i_{1}, i_{2}, i_{3}, i_{4}, j_{1} =1}^{3} \theta(2^{-i} k_{123}) \theta(2^{-j} k_{1})\nonumber \\
& \times  \int_{0}^{t} e^{- \lvert k_{123} \rvert^{2}(t-s)}  \frac{ e^{- \lvert k_{1} \rvert^{2}(t-\sigma)} f( \epsilon k_{1})^{2}}{2 \lvert k_{1} \rvert^{2}} \hat{\mathcal{P}}_{i_{6} i_{4}}(k_{1}) \hat{\mathcal{P}}_{j_{0} i_{4}}(k_{1})\nonumber \\
& \times : \hat{X}_{\sigma, i_{5}}^{u, \epsilon}(k_{2}) \hat{X}_{s, j_{1}}^{b, \epsilon}(k_{3}): e^{- \lvert k_{12} \rvert^{2} (s-\sigma)} d \sigma ds\nonumber  \\
& \times k_{12}^{i_{3}} k_{123}^{j_{1}} \hat{\mathcal{P}}_{i_{1} i_{2}}(k_{12}) \hat{\mathcal{P}}_{i_{0} i_{1}}(k_{123}) e_{k}1_{i_{5} = i_{2}, i_{6} = i_{3}} \triangleq \RomanIX_{t}^{8,3},  
\end{align} 
\begin{align}\label{189}
\RomanVIII_{t}^{8,6} =&  - \frac{1}{4(2\pi)^{\frac{9}{2}}} \sum_{k} \sum_{\lvert i-j \rvert \leq 1} \sum_{k_{1}, k_{2}: k_{12} = k, k_{3} \neq 0} \sum_{i_{1}, i_{2}, i_{3}, i_{4}, j_{1} =1}^{3} \theta(2^{-i} k_{123}) \nonumber \\
& \times \theta(2^{-j} k_{3}) \int_{0}^{t} e^{- \lvert k_{123} \rvert^{2}(t-s)} \int_{0}^{s} \frac{ e^{- \lvert k_{3} \rvert^{2}(t-s)} f( \epsilon k_{3})^{2}}{2 \lvert k_{3} \rvert^{2}} \nonumber \\
& \times \hat{\mathcal{P}}_{j_{1} i_{4}}(k_{3}) \hat{\mathcal{P}}_{j_{0} i_{4}}(k_{3}) : \hat{X}_{\sigma, i_{2}}^{u, \epsilon}(k_{1}) \hat{X}_{\sigma, i_{3}}^{b, \epsilon}(k_{2}):  e^{- \lvert k_{12} \rvert^{2} (s-\sigma)} d \sigma ds \nonumber \\
& \times k_{12}^{i_{3}} k_{123}^{j_{1}} \hat{\mathcal{P}}_{i_{1} i_{2}}(k_{12}) \hat{\mathcal{P}}_{i_{0} i_{1}}(k_{123}) e_{k} \triangleq \RomanIX_{t}^{8,4},  
\end{align} 
\begin{align}\label{190}
\RomanVIII_{t}^{8,7} =& - \frac{1}{4(2\pi)^{\frac{9}{2}}} \sum_{ \lvert i-j \rvert \leq 1} \sum_{k_{1}, k_{2} \neq 0} \sum_{i_{1}, i_{2}, i_{3}, i_{4}, i_{5}, j_{1} =1}^{3} \theta(2^{-i} k_{2}) \theta (2^{-j} k_{2})   \\
& \times \int_{0}^{t} e^{- \lvert k_{2} \rvert^{2}(t-s)} \int_{0}^{s} \frac{ f( \epsilon k_{1})^{2} f(\epsilon k_{2})^{2}}{4 \lvert k_{1} \rvert^{2} \lvert k_{2} \rvert^{2}}  \hat{\mathcal{P}}_{i_{3} i_{4}}(k_{1}) \hat{\mathcal{P}}_{j_{1} i_{4}}(k_{1}) \nonumber \\
& \times \hat{\mathcal{P}}_{i_{2} i_{5}}(k_{2}) \hat{\mathcal{P}}_{j_{0} i_{5}}(k_{2}) e^{ - \lvert k_{12} \rvert^{2}(s-\sigma) - \lvert k_{1} \rvert^{2} (s-\sigma) - \lvert k_{2} \rvert^{2}(t-\sigma)} d \sigma ds k_{12}^{i_{3}} k_{2}^{j_{1}}  \hat{\mathcal{P}}_{i_{1} i_{2}}(k_{12}) \hat{\mathcal{P}}_{i_{0} i_{1}}(k_{2}),  \nonumber 
\end{align} 
and 
\begin{align}\label{191}
\RomanVIII_{t}^{8,8} =& - \frac{1}{4(2\pi)^{\frac{9}{2}}} \sum_{\lvert i-j \rvert \leq 1} \sum_{k_{1}, k_{2} \neq 0} \sum_{i_{1}, i_{2}, i_{3}, i_{4}, i_{5}, j_{1} =1}^{3} \theta( 2^{-i} k_{2}) \theta (2^{-j} k_{2})   \\
& \times \int_{0}^{t} e^{- \lvert k_{2} \rvert^{2} (t-s)} \int_{0}^{s} \frac{ e^{ - \lvert k_{2} \rvert^{2} (t-\sigma)} e^{- \lvert k_{1} \rvert^{2}(s-\sigma)} f( \epsilon k_{1})^{2} f( \epsilon k_{2})^{2}}{ 4 \lvert k_{1} \rvert^{2} \lvert k_{2} \rvert^{2}}\nonumber \\
& \times \hat{\mathcal{P}}_{i_{2} i_{5}}(k_{1}) \hat{\mathcal{P}}_{j_{1} i_{5}}(k_{1}) \hat{\mathcal{P}}_{i_{3} i_{4}}(k_{2}) \hat{\mathcal{P}}_{j_{0} i_{4}}(k_{2}) e^{ - \lvert k_{12} \rvert^{2} (s- \sigma)} d \sigma  ds k_{12}^{i_{3}} k_{2}^{j_{1}} \hat{\mathcal{P}}_{i_{1} i_{2}}(k_{12}) \hat{\mathcal{P}}_{i_{0} i_{1}}(k_{2}). \nonumber 
\end{align} 
We define the sum of right hand side of $\RomanVIII_{t}^{8,7},\RomanVIII_{t}^{8,8}$ in (\ref{190})-(\ref{191}) to be $\RomanIX_{t}^{8,7}$; i.e. 
\begin{align}\label{192}
\RomanIX_{t}^{8,7}&\triangleq - \frac{1}{4(2\pi)^{\frac{9}{2}}} \sum_{ \lvert i-j \rvert \leq 1} \sum_{k_{1}, k_{2} \neq 0} \sum_{i_{1}, i_{2}, i_{3}, i_{4}, i_{5}, j_{1} =1}^{3} \theta( 2^{-i} k_{2}) \theta (2^{-j} k_{2}) \nonumber \\
& \times \int_{0}^{t} e^{- \lvert k_{2} \rvert^{2}(t-s)}  \int_{0}^{s} \frac{ f( \epsilon k_{1})^{2} f( \epsilon k_{2})^{2}}{4 \lvert k_{1} \rvert^{2} \lvert k_{2} \rvert^{2}} e^{ - \lvert k_{12} \rvert^{2}(s-\sigma) - \lvert k_{1} \rvert^{2}(s-\sigma) - \lvert k_{2} \rvert^{2}(t-\sigma)} \nonumber \\
& \times [ \hat{\mathcal{P}}_{ i_{3} i_{4}}(k_{1}) \hat{\mathcal{P}}_{j_{1} i_{4}}(k_{1}) \hat{\mathcal{P}}_{i_{2} i_{5}}(k_{2}) \hat{\mathcal{P}}_{j_{0} i_{5}}(k_{2})\nonumber \\
&  + \hat{\mathcal{P}}_{i_{2} i_{5}}(k_{1}) \hat{\mathcal{P}}_{j_{1} i_{5}}(k_{1}) \hat{\mathcal{P}}_{i_{3} i_{4}}(k_{2}) \hat{\mathcal{P}}_{j_{0} i_{4}}(k_{2})] k_{12}^{i_{3}} k_{2}^{j_{1}} \hat{\mathcal{P}}_{i_{1} i_{2}}(k_{12}) \hat{\mathcal{P}}_{i_{0} i_{1}}(k_{2}) 
\end{align} 
and formally define
\begin{equation}\label{193}
\RomanIX_{t}^{8,7} \triangleq C_{1,3,8}^{\epsilon, i_{0}j_{0}}
\end{equation} 
where we observe that $\lim_{\epsilon \to 0} C_{1,3,8}^{\epsilon, i_{0}j_{0}} = \infty$. Due to (\ref{185})-(\ref{192}) applied to (\ref{184}), we see that 
\begin{equation}\label{194}
\pi_{0} (u_{8, i_{0}}^{
\scalebox{0.16}{\begin{tikzpicture}
\draw[black, thick] (-0.7,0.9) -- (0,0);
\draw[black, thick] (0.7,0.9) -- (0,0);
\draw[black, thick] (0.7,0) -- (0,-0.9);
\draw[snake=zigzag](0,0) -- (0,-0.9);
\draw[snake=zigzag](0,-0.9) -- (0,-1.8);
\filldraw[pink] (-0.7,0.9) circle (7pt); 
\filldraw[pink] (0.7,0.9) circle (7pt); 
\filldraw[pink] (0.7,0) circle (7pt); 
\end{tikzpicture}
}\epsilon}, b^{
\scalebox{0.6}{\begin{tikzpicture}
\draw[black, thick] (0,0.5) -- (0,0);
\filldraw[blue] (0,0.5) circle (2pt); 
\end{tikzpicture}
}\epsilon}_{j_{0}})(t) = \sum_{k=1}^{7} \RomanIX_{t}^{8,k} = \sum_{k=1}^{6} \RomanIX_{t}^{8,k} + C_{1,3,8}^{\epsilon, i_{0}j_{0}}. 
\end{equation}
Repeating similar procedure for $\pi_{0} ( u_{k, i_{0}}^{
\scalebox{0.16}{\begin{tikzpicture}
\draw[black, thick] (-0.7,0.9) -- (0,0);
\draw[black, thick] (0.7,0.9) -- (0,0);
\draw[black, thick] (0.7,0) -- (0,-0.9);
\draw[snake=zigzag](0,0) -- (0,-0.9);
\draw[snake=zigzag](0,-0.9) -- (0,-1.8);
\filldraw[pink] (-0.7,0.9) circle (7pt); 
\filldraw[pink] (0.7,0.9) circle (7pt); 
\filldraw[pink] (0.7,0) circle (7pt); 
\end{tikzpicture}
}\epsilon}, b^{
\scalebox{0.6}{\begin{tikzpicture}
\draw[black, thick] (0,0.5) -- (0,0);
\filldraw[blue] (0,0.5) circle (2pt); 
\end{tikzpicture}
}\epsilon}_{j_{0}})(t)$ for $k \in \{1, \hdots, 7\}$ within (\ref{182}), we can similarly define $C_{1,3,k}^{\epsilon, i_{0}j_{0}}$ for $k \in \{1, \hdots, 7 \}$. Thereafter we shall define 
\begin{equation}\label{195}
C_{1,3}^{\epsilon, i_{0} j_{0}} = \sum_{k=1}^{8} C_{1,3, k}^{\epsilon, i_{0} j_{0}}. 
\end{equation} 

\subsubsection{Terms in the second chaos}
Within (\ref{194}) we see that $\RomanIX_{t}^{8,1}$ is a term in the fourth chaos while $\RomanIX_{t}^{8,k}$ for $k \in \{2, \hdots, 6 \}$ are in the second chaos. Let us first work on $\RomanIX_{t}^{8,2}$ as follows: 
\begin{align}\label{196}
& \mathbb{E} [ \lvert \Delta_{q} \RomanIX_{t}^{8,2} \rvert^{2}] \nonumber \\
\approx& \lvert \sum_{k, k' } \sum_{\lvert i-j \rvert \leq 1, \lvert i' - j' \rvert \leq 1} \sum_{k_{2}, k_{3}: k_{23} = k, k_{1} \neq 0, k_{2}', k_{3}': k_{23}' = k', k_{1}' \neq 0} \sum_{i_{1}, i_{2}, i_{3}, i_{4}, j_{1}, i_{1}', i_{2}', i_{3}', i_{4}', j_{1}' = 1}^{3}\nonumber \\
& \times \theta(2^{-q} k)^{2} \theta(2^{-i} k_{123}) \theta (2^{-i'} k_{123}') \theta(2^{-j} k_{1}) \theta (2^{-j'} k_{1}')  \int_{[0,t]^{2}} e^{- \lvert k_{123} \rvert^{2}(t-s)} e^{- \lvert k_{123} ' \rvert^{2} (t- \overline{s})}\nonumber \\
& \times  \int_{0}^{s} \int_{0}^{\overline{s}} \mathbb{E} [ : \hat{X}_{\sigma, i_{3}}^{b, \epsilon}(k_{2}) \hat{X}_{s, j_{1}}^{b, \epsilon}(k_{3}): : \hat{X}_{\overline{\sigma}, i_{3}'}^{b, \epsilon} (k_{2}') \hat{X}_{\overline{s}, j_{1}'}^{b, \epsilon} (k_{3}'): ]\frac{e^{ - \lvert k_{1} \rvert^{2}(t- \sigma)} f( \epsilon k_{1})^{2}}{2 \lvert k_{1} \rvert^{2}} \frac{ e^{- \lvert k_{1}' \rvert^{2}(t - \overline{\sigma})} f( \epsilon k_{1}')^{2}}{2 \lvert k_{1}' \rvert^{2}}\nonumber \\
& \times \hat{\mathcal{P}}_{i_{2} i_{4}}(k_{1}) \hat{\mathcal{P}}_{i_{2}' i_{4}'} (k_{1}') \hat{\mathcal{P}}_{j_{0} i_{4}}(k_{1}) \hat{\mathcal{P}}_{j_{0} i_{4}'} (k_{1}') e^{- \lvert k_{12} \rvert^{2}(s- \sigma)} e^{- \lvert k_{12}' \rvert^{2} ( \overline{s} - \overline{\sigma})} d \sigma d \overline{\sigma} ds d \overline{s}\nonumber \\
& \times k_{12}^{i_{3}} (k_{12}')^{i_{3}'} k_{123}^{j_{1}} (k_{123}')^{j_{1}'} \hat{\mathcal{P}}_{i_{1} i_{2}}(k_{12}) \hat{\mathcal{P}}_{i_{1}' i_{2}'} (k_{12}') \hat{\mathcal{P}}_{i_{0} i_{1}}(k_{123}) \hat{\mathcal{P}}_{i_{0} i_{1}'} (k_{123}') e_{k} e_{k}' \rvert  
\end{align} 
due to (\ref{186}). By $\mathbb{E} [ : \xi_{11} \xi_{12}: : \xi_{21} \xi_{22}:] = \mathbb{E} [ \xi_{11} \xi_{21}] \mathbb{E} [ \xi_{12} \xi_{22}] + \mathbb{E} [ \xi_{11} \xi_{22}]\mathbb{E}[\xi_{12} \xi_{21}]$ (see \cite{J97}) we can compute $ \mathbb{E} [ : \hat{X}_{\sigma, i_{3}}^{b, \epsilon} (k_{2}) \hat{X}_{s, j_{1}}^{b, \epsilon}(k_{3}): : \hat{X}_{\overline{\sigma}, i_{3}'}^{b, \epsilon} (k_{2}') \hat{X}_{\overline{s}, j_{1}'}^{b, \epsilon} (k_{3}'): ]$ using \eqref{116}, and rely on \cite[Section 9.2]{GP17} to deduce 
\begin{align}\label{198}
&\mathbb{E} [ \lvert \Delta_{q} \RomanIX_{t}^{8, 2} \rvert^{2}] \\
\lesssim& \sum_{k} \sum_{\lvert i-j \rvert \leq 1, \lvert i' - j' \rvert \leq 1} \sum_{k_{2}, k_{3} \neq 0: k_{23} = k, k_{1}, k_{4} \neq 0} \theta(2^{-i} k_{123}) \theta (2^{-i'} k_{234}) \theta (2^{-j} k_{1})  \nonumber \\
& \times  \theta (2^{-j'} k_{4}) \theta (2^{-q} k)^{2} \prod_{i=1}^{4} \frac{ f( \epsilon k_{i})^{2}}{\lvert k_{i} \rvert^{2}} \int_{[0,t]^{2}} e^{- \lvert k_{123} \rvert^{2} (t-s) - \lvert k_{234} \rvert^{2}(t - \overline{s})} \nonumber \\
& \times  \int_{0}^{s} \int_{0}^{\overline{s}} e^{ - \lvert k_{12} \rvert^{2} (s - \sigma) - \lvert k_{24} \rvert^{2} (\overline{s} - \overline{\sigma})} \lvert k_{12} \rvert \lvert k_{24} \rvert \lvert k_{123} \rvert \lvert k_{234} \rvert e^{- \lvert k_{1} \rvert^{2} (t-\sigma) - \lvert k_{4} \rvert^{2} (t- \overline{\sigma})} d \sigma d \overline{\sigma} ds d \overline{s}. \nonumber 
\end{align}  
Within (\ref{198}), we may estimate furthermore for $k_{1}, k_{2}, k_{3}, k_{4} \neq 0$, 
\begin{align}\label{199}
& \prod_{i=1}^{4} \frac{ f( \epsilon k_{i})^{2}}{ \lvert k_{i} \rvert^{2}} \int_{[0,t]^{2}} e^{- \lvert k_{123} \rvert^{2}(t-s) - \lvert k_{234} \rvert^{2}(t - \overline{s})}\nonumber \\
& \times \int_{0}^{s} \int_{0}^{\overline{s}}  e^{- \lvert k_{12} \rvert^{2}(s- \sigma) - \lvert k_{24} \rvert^{2} ( \overline{s} - \overline{\sigma})} \lvert k_{12} \rvert \lvert k_{24} \rvert \lvert k_{123} \rvert \lvert k_{234} \rvert  e^{- \lvert k_{1} \rvert^{2} (t-\sigma) - \lvert k_{4} \rvert^{2} (t - \overline{\sigma})} d \sigma d \overline{\sigma} ds d \overline{s} \nonumber \\
\lesssim& \prod_{i=1}^{4} \frac{1}{\lvert k_{i} \rvert^{2}} e^{( - \lvert k_{123} \rvert^{2} - \lvert k_{234} \rvert^{2} - \lvert k_{1} \rvert^{2} - \lvert k_{4} \rvert^{2}) t} \int_{[0,t]^{2}} e^{ \lvert k_{123} \rvert^{2} s + \lvert k_{234} \rvert^{2} \overline{s}} \lvert k_{12} \rvert \lvert k_{123} \rvert \lvert k_{234} \rvert \nonumber \\
& \times  \lvert k_{24} \rvert  e^{ \lvert k_{1} \rvert^{2} s} \frac{ (1- e^{ - ( \lvert k_{12} \rvert^{2} + \lvert k_{1} \rvert^{2} ) s})}{\lvert k_{12} \rvert^{2} + \lvert k_{1} \rvert^{2}} e^{ \lvert k_{4} \rvert^{2} \overline{s}} \frac{ ( 1 - e^{- ( \lvert k_{24} \rvert^{2} + \lvert k_{4} \rvert^{2} ) \overline{s}})}{\lvert k_{24} \rvert^{2} + \lvert k_{4} \rvert^{2}} 1_{k_{12}, k_{24} \neq 0} ds d \overline{s} \nonumber \\
\lesssim& t^{\eta} \frac{1}{ \lvert k_{2} \rvert^{2} \lvert k_{3} \rvert^{2} \lvert k_{1} \rvert^{4-\eta} \lvert k_{4} \rvert^{4-\eta}} 
\end{align} 
by mean value theorem. Thus, applying (\ref{199}) to (\ref{198}) leads to 
\begin{align}\label{200}
&\mathbb{E} [ \lvert \Delta_{q} \RomanIX_{t}^{8,2} \rvert^{2}] \lesssim \sum_{k} \sum_{\lvert i-j \rvert \leq 1, \lvert i' - j' \rvert \leq 1} \sum_{k_{2}, k_{3} \neq 0: k_{23} = k, k_{1}, k_{4} \neq 0}  \\
& \hspace{10mm} \times \theta(2^{-i} k_{123}) \theta(2^{-j} k_{1}) \theta(2^{-i'} k_{234}) \theta(2^{-j'} k_{4}) \theta(2^{-q} k)^{2}  \frac{t^{\eta}}{ \lvert k_{2} \rvert^{2} \lvert k_{3} \rvert^{2} \lvert k_{1} \rvert^{4-\eta} \lvert k_{4} \rvert^{4-\eta}}. \nonumber 
\end{align} 
Now $2^{q} \approx \lvert k \rvert = \lvert k_{2} + k_{3} \rvert \lesssim \lvert k_{123} \rvert + \lvert k_{1} \rvert \approx 2^{i}$ as $\lvert i-j \rvert \leq 1$ so that $q \lesssim i$. Similarly $2^{q} \lesssim 2^{i'}$ as $\lvert i' - j' \rvert \leq 1$ so that $q \lesssim i'$. Thus for $\epsilon \in (0, 1-\eta)$ sufficiently small we estimate from (\ref{200}), 
\begin{align}\label{201}
 \mathbb{E} [ \lvert \Delta_{q} \RomanIX_{t}^{8,2} \rvert^{2}] 
\lesssim \sum_{k} \sum_{k_{2}, k_{3} \neq 0: k_{23} = k} \sum_{q \lesssim j, q \lesssim j'} \frac{ t^{\eta} \theta (2^{-q} k)^{2}}{ \lvert k_{2} \rvert^{2} \lvert k_{3} \rvert^{2} 2^{j (1- \eta - \frac{\epsilon}{4})} 2^{j' (1- \eta - \frac{\epsilon}{4})}}\lesssim t^{\eta} 2^{q (2 \eta + \epsilon)}  
\end{align} 
by Lemma \ref{Lemma 2.9}. The estimate of $\RomanIX_{t}^{8,3}$ may be achieved very similarly to $\RomanIX_{t}^{8,2}$.  

We now consider $\RomanIX_{t}^{8,4}$ of (\ref{194}). Let us make an important remark here. 
\begin{remark}\label{Remark 3.5}
In particular, this is the renormalization on which we must diverge from the previous study of a single equation (stochastic quantization \cite{CC18} or NSE \cite{ZZ15}) instead of a system of coupled non-linear PDEs such as the MHD system. For example, if we write 
\begin{equation}\label{202}
\RomanIX_{t}^{8,4} = \RomanIX_{t}^{8,4} - \widetilde{\RomanIX}_{t}^{8,4} + \widetilde{\RomanIX}_{t}^{8,4} - \sum_{i_{1} =1}^{3} u^{
\scalebox{0.18}{\begin{tikzpicture}
\draw[black, thick] (-0.7,0.9) -- (0,0);
\draw[black, thick] (0.7,0.9) -- (0,0);
\draw[snake=zigzag](0,0) -- (0,-0.9);
\filldraw[green] (-0.7,0.9) circle (6pt); 
\filldraw[green] (0.7,0.9) circle (6pt); 
\end{tikzpicture}
}}_{i_{1}}(t) C_{3}^{\epsilon, i_{1}}(t)
\end{equation} 
where 
\begin{align}\label{203}
& \widetilde{\RomanIX}_{t}^{8,4} \triangleq (2\pi)^{- \frac{9}{2}} \sum_{k\neq 0} \sum_{\lvert i-j \rvert \leq 1} \sum_{k_{12} = k, k_{3} \neq 0} \sum_{i_{1}, i_{2}, i_{3}, i_{4} j_{1} = 1}^{3} \theta(2^{-i} k_{123}) \theta(2^{-j} k_{3}) \nonumber \\
& \times \int_{0}^{t}: \hat{X}_{\sigma, i_{2}}^{u, \epsilon} (k_{1}) \hat{X}_{\sigma, i_{3}}^{u, \epsilon}(k_{2}): e^{- \lvert k_{12} \rvert^{2} (t-\sigma)} i k_{12}^{i_{3}} \hat{\mathcal{P}}_{i_{1} i_{2}}(k_{12}) e_{k} d\sigma \nonumber  \\
& \times \int_{0}^{t} e^{- \lvert k_{123} \rvert^{2}(t-s)} \frac{ e^{- \lvert k_{3} \rvert^{2}(t-s)} f( \epsilon k_{3})^{2}}{2 \lvert k_{3} \rvert^{2}}  \hat{\mathcal{P}}_{j_{i} i_{4}}(k_{3}) \hat{\mathcal{P}}_{ j_{0} i_{4}}(k_{3}) i k_{123}^{j_{1}} \hat{\mathcal{P}}_{i_{0} i_{1}}(k_{123}) dx
\end{align} 
and
\begin{align}\label{204}
&C_{3}^{\epsilon, i_{1}}(t) \triangleq (2\pi)^{-\frac{9}{2}} \sum_{\lvert i-j \rvert \leq 1} \sum_{k_{3}} \sum_{j_{1} = 1}^{3} \theta (2^{-i} k_{3}) \theta(2^{-j} k_{3}) \int_{0}^{t}  \frac{ e^{ -2 \lvert k_{3} \rvert^{2}(t-s)} f(\epsilon k_{3})^{2}}{\lvert k_{3} \rvert^{2}}\nonumber \\
& \times \sum_{i_{4}} \hat{\mathcal{P}}_{j_{1} i_{4}}(k_{3}) \hat{\mathcal{P}}_{j_{0} i_{4}}(k_{3}) ik_{3}^{j_{1}} \hat{\mathcal{P}}_{i_{0} i_{1}}(k_{3}) = 0 
\end{align} 
as Zhu and Zhu did for the NSE (see \cite[pg. 4489]{ZZ15}), then the necessary estimate of $\widetilde{\RomanIX}_{t}^{4} - \sum_{i_{1} =1}^{3} u^{
\scalebox{0.18}{\begin{tikzpicture}
\draw[black, thick] (-0.7,0.9) -- (0,0);
\draw[black, thick] (0.7,0.9) -- (0,0);
\draw[snake=zigzag](0,0) -- (0,-0.9);
\filldraw[green] (-0.7,0.9) circle (6pt); 
\filldraw[green] (0.7,0.9) circle (6pt); 
\end{tikzpicture}
}}_{i_{1}}(t) C_{3}^{\epsilon, i_{1}}(t)$ on \cite[pg. 4491]{ZZ15} works well because 
\begin{equation*}
L u^{
\scalebox{0.18}{\begin{tikzpicture}
\draw[black, thick] (-0.7,0.9) -- (0,0);
\draw[black, thick] (0.7,0.9) -- (0,0);
\draw[snake=zigzag](0,0) -- (0,-0.9);
\filldraw[green] (-0.7,0.9) circle (6pt); 
\filldraw[green] (0.7,0.9) circle (6pt); 
\end{tikzpicture}
}}_{i} = - \frac{1}{2} \sum_{i_{1} =1}^{2} \mathcal{P}_{ii_{1}} ( \sum_{j=1}^{3}  \partial_{x_{j}}  (u^{\scalebox{0.6}{\begin{tikzpicture}
\draw[black, thick] (0,0.5) -- (0,0);
\filldraw[red] (0,0.5) circle (2pt); 
\end{tikzpicture}
}}_{i_{1}}  \diamond u^{\scalebox{0.6}{\begin{tikzpicture}
\draw[black, thick] (0,0.5) -- (0,0);
\filldraw[red] (0,0.5) circle (2pt); 
\end{tikzpicture}
}}_{j} ))
\end{equation*} 
(see \cite[pg. 4476]{ZZ15}) in the case of the NSE. However, $Lu^{
\scalebox{0.18}{\begin{tikzpicture}
\draw[black, thick] (-0.7,0.9) -- (0,0);
\draw[black, thick] (0.7,0.9) -- (0,0);
\draw[snake=zigzag](0,0) -- (0,-0.9);
\filldraw[green] (-0.7,0.9) circle (6pt); 
\filldraw[green] (0.7,0.9) circle (6pt); 
\end{tikzpicture}
}}_{i}$ in the case of the MHD system does not work due to the additional term of $b^{
\scalebox{0.6}{\begin{tikzpicture}
\draw[black, thick] (0,0.5) -- (0,0);
\filldraw[blue] (0,0.5) circle (2pt); 
\end{tikzpicture}
}}_{i_{1}} \diamond b^{
\scalebox{0.6}{\begin{tikzpicture}
\draw[black, thick] (0,0.5) -- (0,0);
\filldraw[blue] (0,0.5) circle (2pt); 
\end{tikzpicture}
}}_{j}$ in (\ref{17}):
\begin{equation*}
Lu^{
\scalebox{0.18}{\begin{tikzpicture}
\draw[black, thick] (-0.7,0.9) -- (0,0);
\draw[black, thick] (0.7,0.9) -- (0,0);
\draw[snake=zigzag](0,0) -- (0,-0.9);
\filldraw[green] (-0.7,0.9) circle (6pt); 
\filldraw[green] (0.7,0.9) circle (6pt); 
\end{tikzpicture}
}}_{i} = -\frac{1}{2} \sum_{i_{1}, j=1}^{3} \mathcal{P}_{ i i_{1}} \partial_{x_{j}} (u^{\scalebox{0.6}{\begin{tikzpicture}
\draw[black, thick] (0,0.5) -- (0,0);
\filldraw[red] (0,0.5) circle (2pt); 
\end{tikzpicture}
}}_{i_{1}}  \diamond u^{\scalebox{0.6}{\begin{tikzpicture}
\draw[black, thick] (0,0.5) -- (0,0);
\filldraw[red] (0,0.5) circle (2pt); 
\end{tikzpicture}
}}_{j}  - b^{
\scalebox{0.6}{\begin{tikzpicture}
\draw[black, thick] (0,0.5) -- (0,0);
\filldraw[blue] (0,0.5) circle (2pt); 
\end{tikzpicture}
}}_{i_{1}} \diamond b^{
\scalebox{0.6}{\begin{tikzpicture}
\draw[black, thick] (0,0.5) -- (0,0);
\filldraw[blue] (0,0.5) circle (2pt); 
\end{tikzpicture}
}}_{j} ).
\end{equation*}
This creates a huge obstacle. 

We can actually overcome this difficulty remarkably by considering the sum of $u_{i_{0}, 8}^{
\scalebox{0.16}{\begin{tikzpicture}
\draw[black, thick] (-0.7,0.9) -- (0,0);
\draw[black, thick] (0.7,0.9) -- (0,0);
\draw[black, thick] (0.7,0) -- (0,-0.9);
\draw[snake=zigzag](0,0) -- (0,-0.9);
\draw[snake=zigzag](0,-0.9) -- (0,-1.8);
\filldraw[pink] (-0.7,0.9) circle (7pt); 
\filldraw[pink] (0.7,0.9) circle (7pt); 
\filldraw[pink] (0.7,0) circle (7pt); 
\end{tikzpicture}
}}$ with $u_{i_{0}, 7}^{
\scalebox{0.16}{\begin{tikzpicture}
\draw[black, thick] (-0.7,0.9) -- (0,0);
\draw[black, thick] (0.7,0.9) -- (0,0);
\draw[black, thick] (0.7,0) -- (0,-0.9);
\draw[snake=zigzag](0,0) -- (0,-0.9);
\draw[snake=zigzag](0,-0.9) -- (0,-1.8);
\filldraw[pink] (-0.7,0.9) circle (7pt); 
\filldraw[pink] (0.7,0.9) circle (7pt); 
\filldraw[pink] (0.7,0) circle (7pt); 
\end{tikzpicture}
}}$ in (\ref{181}). This technique of strategically coupling certain renormalizations is very reminiscent of the basic energy identity (\ref{7}) and (\ref{8}) actually. We emphasize that it must be $u_{i_{0}, 7}^{
\scalebox{0.16}{\begin{tikzpicture}
\draw[black, thick] (-0.7,0.9) -- (0,0);
\draw[black, thick] (0.7,0.9) -- (0,0);
\draw[black, thick] (0.7,0) -- (0,-0.9);
\draw[snake=zigzag](0,0) -- (0,-0.9);
\draw[snake=zigzag](0,-0.9) -- (0,-1.8);
\filldraw[pink] (-0.7,0.9) circle (7pt); 
\filldraw[pink] (0.7,0.9) circle (7pt); 
\filldraw[pink] (0.7,0) circle (7pt); 
\end{tikzpicture}
}}$ that we couple with $u_{i_{0}, 8}^{
\scalebox{0.16}{\begin{tikzpicture}
\draw[black, thick] (-0.7,0.9) -- (0,0);
\draw[black, thick] (0.7,0.9) -- (0,0);
\draw[black, thick] (0.7,0) -- (0,-0.9);
\draw[snake=zigzag](0,0) -- (0,-0.9);
\draw[snake=zigzag](0,-0.9) -- (0,-1.8);
\filldraw[pink] (-0.7,0.9) circle (7pt); 
\filldraw[pink] (0.7,0.9) circle (7pt); 
\filldraw[pink] (0.7,0) circle (7pt); 
\end{tikzpicture}
}}$, not any other $u_{i_{0}, k}^{
\scalebox{0.16}{\begin{tikzpicture}
\draw[black, thick] (-0.7,0.9) -- (0,0);
\draw[black, thick] (0.7,0.9) -- (0,0);
\draw[black, thick] (0.7,0) -- (0,-0.9);
\draw[snake=zigzag](0,0) -- (0,-0.9);
\draw[snake=zigzag](0,-0.9) -- (0,-1.8);
\filldraw[pink] (-0.7,0.9) circle (7pt); 
\filldraw[pink] (0.7,0.9) circle (7pt); 
\filldraw[pink] (0.7,0) circle (7pt); 
\end{tikzpicture}
}}$ for $k \in \{1, \hdots, 6\}$ in (\ref{181}).  
\end{remark}

Now recalling (\ref{189}), we see that the only differences between $L u_{i_{0}, 7}^{
\scalebox{0.16}{\begin{tikzpicture}
\draw[black, thick] (-0.7,0.9) -- (0,0);
\draw[black, thick] (0.7,0.9) -- (0,0);
\draw[black, thick] (0.7,0) -- (0,-0.9);
\draw[snake=zigzag](0,0) -- (0,-0.9);
\draw[snake=zigzag](0,-0.9) -- (0,-1.8);
\filldraw[pink] (-0.7,0.9) circle (7pt); 
\filldraw[pink] (0.7,0.9) circle (7pt); 
\filldraw[pink] (0.7,0) circle (7pt); 
\end{tikzpicture}
}}$ and $Lu_{i_{0}, 8}^{
\scalebox{0.16}{\begin{tikzpicture}
\draw[black, thick] (-0.7,0.9) -- (0,0);
\draw[black, thick] (0.7,0.9) -- (0,0);
\draw[black, thick] (0.7,0) -- (0,-0.9);
\draw[snake=zigzag](0,0) -- (0,-0.9);
\draw[snake=zigzag](0,-0.9) -- (0,-1.8);
\filldraw[pink] (-0.7,0.9) circle (7pt); 
\filldraw[pink] (0.7,0.9) circle (7pt); 
\filldraw[pink] (0.7,0) circle (7pt); 
\end{tikzpicture}
}}$ in (\ref{181}) consist of the sign and $b^{
\scalebox{0.6}{\begin{tikzpicture}
\draw[black, thick] (0,0.5) -- (0,0);
\filldraw[blue] (0,0.5) circle (2pt); 
\end{tikzpicture}
}}_{i_{2}}u^{\scalebox{0.6}{\begin{tikzpicture}
\draw[black, thick] (0,0.5) -- (0,0);
\filldraw[red] (0,0.5) circle (2pt); 
\end{tikzpicture}
}}_{i_{3}}$ replaced by $u^{\scalebox{0.6}{\begin{tikzpicture}
\draw[black, thick] (0,0.5) -- (0,0);
\filldraw[red] (0,0.5) circle (2pt); 
\end{tikzpicture}
}}_{i_{2}}b^{
\scalebox{0.6}{\begin{tikzpicture}
\draw[black, thick] (0,0.5) -- (0,0);
\filldraw[blue] (0,0.5) circle (2pt); 
\end{tikzpicture}
}}_{i_{3}}$ so that we have
\begin{align}\label{205}
\RomanIX_{t}^{7,4} =&  \frac{1}{4(2\pi)^{\frac{9}{2}}} \sum_{k} \sum_{\lvert i-j \rvert \leq 1} \sum_{k_{1}, k_{2}: k_{12} = k, k_{3} \neq 0} \sum_{i_{1}, i_{2}, i_{3}, i_{4}, j_{1} =1}^{3} \theta(2^{-i} k_{123}) \theta(2^{-j} k_{3})  \\
& \times \int_{0}^{t} e^{- \lvert k_{123} \rvert^{2}(t-s)}  \int_{0}^{s} \frac{ e^{- \lvert k_{3} \rvert^{2}(t-s)} f( \epsilon k_{3})^{2}}{2 \lvert k_{3} \rvert^{2}} \hat{\mathcal{P}}_{j_{1} i_{4}}(k_{3}) \hat{\mathcal{P}}_{j_{0} i_{4}}(k_{3}) \nonumber\\
& \times : \hat{X}_{\sigma, i_{2}}^{b, \epsilon}(k_{1}) \hat{X}_{\sigma, i_{3}}^{u, \epsilon}(k_{2}):  e^{- \lvert k_{12} \rvert^{2} (s-\sigma)} d \sigma ds  k_{12}^{i_{3}} k_{123}^{j_{1}} \hat{\mathcal{P}}_{i_{1} i_{2}}(k_{12}) \hat{\mathcal{P}}_{i_{0} i_{1}}(k_{123}) e_{k}. \nonumber
\end{align} 
In sum of (\ref{189}) and (\ref{205}) we obtain 
\begin{align}\label{206}
\RomanIX_{t}^{7,4} + \RomanIX_{t}^{8,4} &= \frac{1}{4(2\pi)^{\frac{9}{2}}}  \sum_{k} \sum_{\lvert i-j \rvert \leq 1} \sum_{k_{1}, k_{2}: k_{12} = k, k_{3} \neq 0} \sum_{i_{1}, i_{2}, i_{3}, i_{4}, j_{1} =1}^{3}\nonumber \\
& \times  \theta(2^{-i} k_{123}) \theta(2^{-j} k_{3})  \int_{0}^{t} e^{- \lvert k_{123} \rvert^{2} (t-s)}\int_{0}^{s} \frac{ e^{- \lvert k_{3} \rvert^{2}(t-s)} f( \epsilon k_{3})^{2}}{2 \lvert k_{3} \rvert^{2}}\nonumber \\
& \times \hat{\mathcal{P}}_{j_{1} i_{4}}(k_{3}) \hat{\mathcal{P}}_{j_{0} i_{4}}(k_{3})  [ : \hat{X}_{\sigma, i_{2}}^{b, \epsilon}(k_{1}) \hat{X}_{\sigma, i_{3}}^{u, \epsilon}(k_{2}): - : \hat{X}_{\sigma, i_{2}}^{u, \epsilon}(k_{1}) \hat{X}_{\sigma, i_{3}}^{b, \epsilon}(k_{2}):]\nonumber \\
& \times  e^{- \lvert k_{12} \rvert^{2}(s-\sigma)} d \sigma ds k_{12}^{i_{3}} k_{123}^{j_{1}} \hat{\mathcal{P}}_{i_{1} i_{2}} (k_{12}) \hat{\mathcal{P}}_{i_{0} i_{1}}(k_{123}) e_{k}. 
\end{align} 
We define now 
\begin{align}\label{207}
\widetilde{\RomanIX}_{t}^{7,8,4} \triangleq& \frac{1}{4(2\pi)^{\frac{9}{2}}} \sum_{k} \sum_{\lvert i-j \rvert \leq 1} \sum_{k_{1}, k_{2}: k_{12} = k, k_{3} \neq 0} \sum_{i_{1}, i_{2}, i_{3}, i_{4}, j_{1} =1}^{3} \nonumber \\
& \times \theta(2^{-i} k_{123}) \theta(2^{-j} k_{3})  \int_{0}^{t} [: \hat{X}_{\sigma, i_{2}}^{b, \epsilon}(k_{1}) \hat{X}_{\sigma, i_{3}}^{u, \epsilon}(k_{2}): - : \hat{X}_{\sigma, i_{2}}^{u, \epsilon}(k_{1}) \hat{X}_{\sigma, i_{3}}^{b, \epsilon}(k_{2}):] \nonumber \\
& \times e^{- \lvert k_{12} \rvert^{2}(t-\sigma)} k_{12}^{i_{3}} \hat{\mathcal{P}}_{i_{1} i_{2}}(k_{12}) \hat{\mathcal{P}}_{i_{0} i_{1}}(k_{123}) d \sigma\nonumber \\
& \times \int_{0}^{t} e^{- \lvert k_{123} \rvert^{2} (t-s)} \frac{ e^{- \lvert k_{3} \rvert^{2}(t-s)} f(\epsilon k_{3})^{2}}{2 \lvert k_{3} \rvert^{2}}  \hat{\mathcal{P}}_{j_{1} i_{4}}(k_{3}) \hat{\mathcal{P}}_{j_{0} i_{4}}(k_{3}) k_{123}^{j_{1}} ds e_{k}, 
\end{align}
and 
\begin{align}\label{208}
C_{3}^{7, 8, \epsilon, i_{1}}(t) \triangleq& \frac{1}{2(2\pi)^{3}} \sum_{\lvert i-j \rvert \leq 1} \sum_{k_{3} \neq 0} \sum_{i_{4}, j_{1} =1}^{3} \theta(2^{-i} k_{3}) \theta(2^{-j} k_{3}) \nonumber \\
& \times \int_{0}^{t} \frac{ e^{-2 \lvert k_{3} \rvert^{2}(t-s)} f(\epsilon k_{3})^{2}}{2 \lvert k_{3} \rvert^{2}} \hat{\mathcal{P}}_{j_{1} i_{4}}(k_{3}) \hat{\mathcal{P}}_{j_{0} i_{4}}(k_{3}) i k_{3}^{j_{1}} \hat{\mathcal{P}}_{i_{0} i_{1}}(k_{3}) ds
\end{align} 
where it can be readily confirmed that $C_{3}^{7, 8, \epsilon, i_{1}}(t) = 0$. Now we split 
\begin{equation}\label{209}
\RomanIX_{t}^{7,4} + \RomanIX_{t}^{8,4} = (\RomanIX_{t}^{7,4} + \RomanIX_{t}^{8,4}) - \widetilde{\RomanIX}_{t}^{7,8,4} + \widetilde{\RomanIX}_{t}^{7,8,4} - \sum_{i_{1} =1}^{3} b^{
\scalebox{0.18}{\begin{tikzpicture}
\draw[black, thick] (-0.7,0.9) -- (0,0);
\draw[black, thick] (0.7,0.9) -- (0,0);
\draw[snake=zigzag](0,0) -- (0,-0.9);
\filldraw[violet] (-0.7,0.9) circle (6pt); 
\filldraw[violet] (0.7,0.9) circle (6pt); 
\end{tikzpicture}
}\epsilon}_{i_{1}}(t) C_{3}^{7, 8, \epsilon, i_{1}}(t).
\end{equation} 
Within (\ref{209}) we first work on 
\begin{align}\label{210}
& (\RomanIX_{t}^{7,4} + \RomanIX_{t}^{8,4}) - \widetilde{\RomanIX}_{t}^{7,8,4} \nonumber \\
=& \frac{1}{4(2\pi)^{\frac{9}{2}}} \sum_{k} \sum_{ \lvert i-j \rvert \leq 1} \sum_{k_{1}, k_{2}: k_{12} =k, k_{3} \neq 0} \sum_{i_{1}, i_{2}, i_{3}, j_{1} =1}^{3} \theta(2^{-i} k_{123}) \theta(2^{-j} k_{3}) \nonumber \\
& \times \hat{\mathcal{P}}_{j_{1} i_{4}}(k_{3}) \hat{\mathcal{P}}_{j_{0} i_{4}}(k_{3}) k_{123}^{j_{1}} k_{12}^{i_{3}} \hat{\mathcal{P}}_{i_{1} i_{2}}(k_{12}) \hat{\mathcal{P}}_{i_{0}i_{1}}(k_{123}) e_{k} \int_{0}^{t} e^{- \lvert k_{123} \rvert^{2}(t-s)} \frac{ e^{- \lvert k_{3} \rvert^{2}(t-s)} f(\epsilon k_{3})^{2}}{2 \lvert k_{3} \rvert^{2}}\nonumber \\
& \times [ \int_{0}^{s} [ : \hat{X}_{\sigma, i_{2}}^{b, \epsilon}(k_{1}) \hat{X}_{\sigma, i_{3}}^{u, \epsilon}(k_{2}): - : \hat{X}_{\sigma, i_{2}}^{u, \epsilon}(k_{1}) \hat{X}_{\sigma, i_{3}}^{b, \epsilon}(k_{2}):] e^{- \lvert k_{12} \rvert^{2}(s-\sigma)} d\sigma \nonumber \\
& - \int_{0}^{t} [: \hat{X}_{\sigma, i_{2}}^{b, \epsilon}(k_{1}) \hat{X}_{\sigma, i_{3}}^{u, \epsilon}(k_{2}): - : \hat{X}_{\sigma, i_{2}}^{u, \epsilon}(k_{1}) \hat{X}_{\sigma, i_{3}}^{b, \epsilon}(k_{2}):] e^{- \lvert k_{12} \rvert^{2} (t-\sigma)} d \sigma] ds 
\end{align} 
where we relied on (\ref{206}) and (\ref{207}). Within (\ref{210}) we first focus on 
\begin{align}\label{211}
&\int_{0}^{s} [ : \hat{X}_{\sigma, i_{2}}^{b, \epsilon}(k_{1}) \hat{X}_{\sigma, i_{3}}^{u, \epsilon}(k_{2}): - : \hat{X}_{\sigma, i_{2}}^{u, \epsilon}(k_{1}) \hat{X}_{\sigma, i_{3}}^{b, \epsilon}(k_{2}):] e^{- \lvert k_{12} \rvert^{2}(s-\sigma)} d\sigma \nonumber \\
& - \int_{0}^{t} [: \hat{X}_{\sigma, i_{2}}^{b, \epsilon}(k_{1}) \hat{X}_{\sigma, i_{3}}^{u, \epsilon}(k_{2}): - : \hat{X}_{\sigma, i_{2}}^{u, \epsilon}(k_{1}) \hat{X}_{\sigma, i_{3}}^{b, \epsilon}(k_{2}):] e^{- \lvert k_{12} \rvert^{2} (t-\sigma)} d \sigma\nonumber \\
=& \int_{0}^{s} [: \hat{X}_{\sigma, i_{2}}^{b, \epsilon} (k_{1}) \hat{X}_{\sigma, i_{3}}^{u, \epsilon}(k_{2}): - : \hat{X}_{\sigma, i_{2}}^{u, \epsilon}(k_{1}) \hat{X}_{\sigma, i_{3}}^{b, \epsilon}(k_{2}):] (e^{- \lvert k_{12} \rvert^{2} (s-\sigma)} - e^{-\lvert k_{12} \rvert^{2} (t-\sigma)}) d\sigma \nonumber \\
&- \int_{s}^{t} [: \hat{X}_{\sigma, i_{2}}^{b, \epsilon} (k_{1}) \hat{X}_{\sigma, i_{3}}^{u, \epsilon}(k_{2}): - : \hat{X}_{\sigma, i_{2}}^{u, \epsilon}(k_{1}) \hat{X}_{\sigma, i_{3}}^{b, \epsilon}(k_{2}):] e^{- \lvert k_{12} \rvert^{2} (t-\sigma)} d\sigma. 
\end{align} 
We also define for $k_{3} \neq 0$, 
\begin{equation}\label{212}
C_{k_{123}, k_{3}}^{j_{1}} (t-s) \triangleq \sum_{i_{1} =1}^{3} e^{- \lvert k_{123} \rvert^{2}(t-s)} \frac{ e^{- \lvert k_{3} \rvert^{2}(t-s)} f( \epsilon k_{3})^{2}}{\lvert k_{3} \rvert^{2}} \lvert k_{123}^{j_{1}} \hat{\mathcal{P}}_{i_{0} i_{1}}(k_{123}) \rvert 
\end{equation}
so that we can now estimate 
\begin{align}\label{213}
& \mathbb{E} [ \lvert \Delta_{q} ((\RomanIX_{t}^{7,4} + \RomanIX_{t}^{8,4}) - \widetilde{\RomanIX}_{t}^{7,8,4}) \rvert^{2}]\\
\lesssim& \sum_{k, k' } \theta(2^{-q} k) \theta(2^{-q} k') \sum_{\lvert i-j\rvert \leq 1, \lvert i'-j'\rvert \leq 1} \sum_{k_{1}, k_{2}: k_{12} = k, k_{3} \neq 0, k_{1}', k_{2}': k_{12}' = k', k_{3}' \neq 0} \nonumber \\
& \times \sum_{i_{2}, i_{3}, j_{1}, i_{2}', i_{3}', j_{1} ' = 1}   \theta(2^{-i} k_{123}) \theta(2^{-i'} k_{123}') \theta(2^{-j} k_{3}) \theta(2^{-j} k_{3}') \int_{[0,t]^{2}} \nonumber \\
& \times C_{k_{123}, k_{3}}^{j_{1}}(t-s) C_{k_{123}', k_{3}'}^{j_{1}'} (t- \overline{s}) \lvert k_{12} \rvert \lvert k_{12} ' \rvert  \nonumber \\
& \times [\int_{0}^{s} \int_{0}^{\overline{s}} \mathbb{E} [ : \hat{X}_{\sigma, i_{2}}^{b, \epsilon}(k_{1}) \hat{X}_{\sigma, i_{3}}^{u, \epsilon}(k_{2}): : \hat{X}_{\overline{\sigma}, i_{2}'}^{b, \epsilon} (k_{1}') \hat{X}_{\overline{\sigma}, i_{3}'}^{u, \epsilon} (k_{2}') : ]\nonumber \\
& \times (e^{- \lvert k_{12} \rvert^{2}(s-\sigma)} - e^{-\lvert k_{12} \rvert^{2}(t-\sigma)}) (e^{- \lvert k_{12}' \rvert^{2} (\overline{s} - \overline{\sigma})} - e^{-\lvert k_{12} ' \rvert^{2}(t- \overline{\sigma})}) d \sigma d \overline{\sigma} \nonumber \\
&+ \int_{s}^{t} \int_{\overline{s}}^{t} \mathbb{E} [ : \hat{X}_{\sigma, i_{2}}^{b, \epsilon}(k_{1}) \hat{X}_{\sigma, i_{3}}^{u, \epsilon}(k_{2}): : \hat{X}_{\overline{\sigma}, i_{2}'}^{b, \epsilon} (k_{1}') \hat{X}_{\overline{\sigma}, i_{3}'}^{u, \epsilon}(k_{2}'): ]  e^{- \lvert k_{12} \rvert^{2} [(t- \sigma) + (t - \overline{\sigma})]} d \sigma d \overline{\sigma} \nonumber \\
&+ \int_{0}^{s} \int_{0}^{\overline{s}} \mathbb{E} [: \hat{X}_{\sigma, i_{2}}^{u, \epsilon}(k_{1}) \hat{X}_{\sigma, i_{3}}^{b, \epsilon}(k_{2}):: \hat{X}_{\overline{\sigma}, i_{2}'}^{u, \epsilon} (k_{1}') \hat{X}_{\overline{\sigma}, i_{3}'}^{b, \epsilon}(k_{2}'):] \nonumber \\
& \times (e^{- \lvert k_{12} \rvert^{2}(s-\sigma)} - e^{- \lvert k_{12} \rvert^{2} (t-\sigma)}) (e^{-\lvert k_{12}' \rvert^{2}(\overline{s} - \overline{\sigma})} - e^{-\lvert k_{12}' \rvert^{2}(t- \overline{\sigma})}) d\sigma d \overline{\sigma} \nonumber \\
&+ \int_{s}^{t} \int_{\overline{s}}^{t} \mathbb{E} [: \hat{X}_{\sigma, i_{2}}^{u, \epsilon}(k_{1}) \hat{X}_{\sigma, i_{3}}^{b, \epsilon}(k_{2}): \hat{X}_{\overline{\sigma}, i_{2}'}^{u, \epsilon}(k_{1}') \hat{X}_{\overline{\sigma}, i_{3}'}^{b, \epsilon}(k_{2}'):]e^{- \lvert k_{12} \rvert^{2} [(t - \sigma) + (t- \overline{\sigma})]} d \sigma d \overline{\sigma} ] ds d \overline{s} \triangleq \sum_{i=1}^{4} X^{i}\nonumber
\end{align} 
by (\ref{210})-(\ref{212}) and Young's inequality. Among the four terms on the right side of (\ref{213}), it suffices to work on the first two terms, namely $X^{1} + X^{2}$. First, due to $\mathbb{E} [ : \xi_{11} \xi_{12}: : \xi_{21} \xi_{22}:] = \mathbb{E} [ \xi_{11} \xi_{21}] \mathbb{E} [ \xi_{12} \xi_{22}] + \mathbb{E} [ \xi_{11} \xi_{22}]\mathbb{E}[\xi_{12} \xi_{21}]$ (see \cite{J97}) we can compute $\mathbb{E} [ : \hat{X}_{\sigma, i_{2}}^{b, \epsilon}(k_{1}) \hat{X}_{\sigma, i_{3}}^{u, \epsilon}(k_{2}): : \hat{X}_{\overline{\sigma}, i_{2}'}^{b, \epsilon}(k_{1}') \hat{X}_{\overline{\sigma}, i_{3}'}^{u, \epsilon}(k_{2}'): ]$ using \eqref{116} and deduce from \eqref{213} 
\begin{align}\label{215}
X^{1} + X^{2} 
\lesssim&  \sum_{k} \theta(2^{-q} k)^{2} \sum_{ \lvert i-j \rvert \leq 1, \lvert i' - j' \rvert \leq 1} \sum_{k_{1}, k_{2} \neq 0: k_{12} = k, k_{3}, k_{4} \neq 0} \sum_{j_{1}, j_{1}' = 1}^{3} \lvert k_{12} \rvert^{2} \nonumber \\
& \times \int_{[0,t]^{2}} \theta(2^{-i} k_{123}) \theta(2^{-i'} k_{124}) \theta(2^{-j} k_{3}) \theta(2^{-j'} k_{4})  \nonumber \\
& \times C_{k_{123}, k_{3}}^{j_{1}}(t-s) C_{k_{124}, k_{4}}^{j_{1}'} (t- \overline{s})  [\int_{0}^{s} \int_{0}^{\overline{s}} \frac{ e^{- (\lvert k_{1} \rvert^{2} + \lvert k_{2} \rvert^{2}) \lvert \sigma - \overline{\sigma} \rvert}}{\lvert k_{1} \rvert^{2} \lvert k_{2} \rvert^{2}} \nonumber \\
& \hspace{5mm} \times (e^{- \lvert k_{12} \rvert^{2} (s- \sigma)} - e^{- \lvert k_{12} \rvert^{2} (t- \sigma)})(e^{- \lvert k_{12} \rvert^{2} (\overline{s} - \overline{\sigma})} - e^{- \lvert k_{12} \rvert^{2} (t- \overline{\sigma})}) d \sigma d \overline{\sigma} \nonumber \\
&+ \int_{s}^{t} \int_{\overline{s}}^{t} \frac{ e^{- ( \lvert k_{1} \rvert^{2} + \lvert k_{2} \rvert^{2}) \lvert \sigma - \overline{\sigma} \rvert}}{\lvert k_{1} \rvert^{2} \lvert k_{2} \rvert^{2}} e^{- \lvert k_{12} \rvert^{2} (t- \sigma + t - \overline{\sigma})} d \sigma d \overline{\sigma}] ds d \overline{s}
\end{align} 
where we used a change of variable of $k_{3} '$ with $-k_{4}$. Within (\ref{215}) we may further estimate for $k_{12} \neq 0$,  
\begin{align}\label{216}
& \int_{0}^{s} \int_{0}^{\overline{s}}  e^{- (\lvert k_{1} \rvert^{2} + \lvert k_{2} \rvert^{2}) \lvert \sigma - \overline{\sigma} \rvert} (e^{- \lvert k_{12} \rvert^{2} (s- \sigma)} - e^{- \lvert k_{12} \rvert^{2} (t- \sigma)}) (e^{- \lvert k_{12} \rvert^{2} (\overline{s} - \overline{\sigma})} - e^{- \lvert k_{12} \rvert^{2} (t- \overline{\sigma})}) d \sigma d \overline{\sigma} \nonumber \\
&+ \int_{s}^{t} \int_{\overline{s}}^{t}  e^{- ( \lvert k_{1} \rvert^{2} + \lvert k_{2} \rvert^{2}) \lvert \sigma - \overline{\sigma} \rvert} e^{- \lvert k_{12} \rvert^{2} (t- \sigma + t - \overline{\sigma})} d \sigma d \overline{\sigma} \lesssim \frac{1}{\lvert k_{12} \rvert^{3}} \lvert t-s \rvert^{\frac{1}{4}} \lvert t- \overline{s} \rvert^{\frac{1}{4}}
\end{align} 
due to mean value theorem and (\ref{14}). Therefore, applying (\ref{216}) to (\ref{215}) gives 
\begin{align}\label{217}
X^{1} + X^{2} \lesssim&  \sum_{k \neq 0} \theta(2^{-q} k)^{2} \sum_{ \lvert i-j \rvert \leq 1, \lvert i' - j' \rvert \leq 1} \sum_{k_{1}, k_{2} \neq 0: k_{12} = k, k_{3}, k_{4} \neq 0} \sum_{j_{1}, j_{1}' = 1}^{3} \nonumber \\
& \times \theta(2^{-i} k_{123}) \theta(2^{-i'} k_{124}) \theta(2^{-j} k_{3}) \theta(2^{-j'} k_{4})  \int_{[0,t]^{2}} C_{k_{123}, k_{3}}^{j_{1}}(t-s)\nonumber \\
& \times  C_{k_{124}, k_{4}}^{j_{1}'} (t- \overline{s}) 
 \frac{1}{\lvert k_{12} \rvert \lvert k_{1} \rvert^{2} \lvert k_{2} \rvert^{2}} (t- s )^{\frac{1}{4}} (t - \overline{s})^{\frac{1}{4}} ds d \overline{s}.
\end{align} 
Moreover, for $k_{3}, k_{4} \neq 0$, 
\begin{align}\label{218}
& \int_{[0,t]^{2}} (t-s)^{\frac{1}{4}} (t- \overline{s})^{\frac{1}{4}} C_{k_{123}, k_{3}}^{j_{1}}(t-s) C_{k_{124}, k_{4}}^{j_{1}'} (t- \overline{s}) ds d \overline{s} \nonumber \\
\lesssim& \frac{ \lvert k_{123} \rvert^{\frac{1}{2}} \lvert k_{124} \rvert^{\frac{1}{2}}}{\lvert k_{3} \rvert^{2} \lvert k_{4} \rvert^{2}} [ \frac{1 - e^{- \frac{1}{2} (\lvert k_{123} \rvert^{2} + \lvert k_{3} \rvert^{2})t}}{\lvert k_{123} \rvert^{2} + \lvert k_{3} \rvert^{2}} ] [ \frac{1- e^{ -\frac{1}{2}(\lvert k_{124} \rvert^{2} + \lvert k_{4} \rvert^{2} )t}}{\lvert k_{124} \rvert^{2} + \lvert k_{4} \rvert^{2}}]\nonumber \\
\lesssim& \frac{t^{ 2(\frac{\eta}{3} + \frac{\epsilon}{6})}}{\lvert k_{3} \rvert^{2} \lvert k_{4} \rvert^{2} (\lvert k_{123} \rvert^{2} + \lvert k_{3} \rvert^{2} )^{\frac{3}{4} - (\frac{\eta}{3} + \frac{\epsilon}{6})} (\lvert k_{124} \rvert^{2} + \lvert k_{4} \rvert^{2})^{\frac{3}{4} - (\frac{\eta}{3} + \frac{\epsilon}{6})}} 
\end{align} 
by (\ref{212}) and (\ref{14}). Applying (\ref{218}) to (\ref{217}) leads to 
\begin{align}\label{219}
X^{1} + X^{2} \lesssim& \sum_{k \neq 0} \theta(2^{-q} k)^{2} \sum_{k_{1}, k_{2} \neq 0: k_{12} = k} \frac{t^{2(\frac{\eta}{3} + \frac{\epsilon}{6})}}{\lvert k_{12} \rvert \lvert k_{1} \rvert^{2} \lvert k_{2} \rvert^{2}} \sum_{q \lesssim i, q \lesssim i'} \frac{1}{2^{i ( \frac{1}{2} - 3 (\frac{\eta}{3} + \frac{\epsilon}{6}))}} \frac{1}{2^{i' ( \frac{1}{2} - 3 (\frac{\eta}{3} + \frac{\epsilon}{6}))}}\nonumber \\
\lesssim& t^{2(\frac{\eta}{3} + \frac{\epsilon}{6})} 2^{2q (3(\frac{\eta}{3} + \frac{\epsilon}{6}))} \sum_{2^{q-1} \lesssim \lvert k \rvert \lesssim 2^{q+1}} \frac{1}{\lvert k \rvert^{3}} \lesssim t^{2(\frac{\eta}{3} + \frac{\epsilon}{6})} 2^{2q ( \eta + \frac{\epsilon}{2})}  
\end{align} 
where we used that $2^{q} \lesssim 2^{i}, 2^{q} \lesssim 2^{i'}$, and Lemma \ref{Lemma 2.9}. Similar estimates may be obtained for $X^{3}$ and $X^{4}$. Therefore, we conclude by applying (\ref{219}) to (\ref{213}) that 
\begin{equation}\label{220}
\mathbb{E} [ \lvert \Delta_{q} (( \RomanIX_{t}^{7,4} + \RomanIX_{t}^{8,4}) - \widetilde{\RomanIX}_{t}^{7,8,4}) \rvert^{2} ] \lesssim t^{2( \frac{\eta}{3} + \frac{\epsilon}{6})} 2^{2q (\eta + \frac{\epsilon}{2})}. 
\end{equation}
Next, within (\ref{209}) we work on $\mathbb{E} [ \lvert \Delta_{q} ( \widetilde{\RomanIX}_{t}^{7, 8, 4} - \sum_{i_{1} = 1}^{3} b^{
\scalebox{0.18}{\begin{tikzpicture}
\draw[black, thick] (-0.7,0.9) -- (0,0);
\draw[black, thick] (0.7,0.9) -- (0,0);
\draw[snake=zigzag](0,0) -- (0,-0.9);
\filldraw[violet] (-0.7,0.9) circle (6pt); 
\filldraw[violet] (0.7,0.9) circle (6pt); 
\end{tikzpicture}
}\epsilon}_{i_{1}}(t) C_{3}^{7, 8, \epsilon, i_{1}}(t)) \rvert^{2} ]$ where we may write 
\begin{align}\label{221}
& \sum_{i_{1} =1}^{3} b^{
\scalebox{0.18}{\begin{tikzpicture}
\draw[black, thick] (-0.7,0.9) -- (0,0);
\draw[black, thick] (0.7,0.9) -- (0,0);
\draw[snake=zigzag](0,0) -- (0,-0.9);
\filldraw[violet] (-0.7,0.9) circle (6pt); 
\filldraw[violet] (0.7,0.9) circle (6pt); 
\end{tikzpicture}
}\epsilon}_{i_{1}}(t) C_{3}^{7, 8, \epsilon, i_{1}}(t)  = \frac{1}{4(2\pi)^{\frac{9}{2}}} \sum_{k \neq 0} \sum_{\lvert i-j \rvert \leq 1} \sum_{k_{1}, k_{2}: k_{12} =k, k_{3} \neq 0} \sum_{i_{1}, i_{2}, i_{3}, i_{4}, j_{1} =1}^{3}  \\
& \times \int_{0}^{t} e^{- \lvert k_{12} \rvert^{2}(t-\sigma)} \hat{\mathcal{P}}_{i_{1} i_{2}}(k_{12}) k_{12}^{i_{3}}[ \hat{X}_{\sigma, i_{2}}^{b, \epsilon}(k_{1}) \hat{X}_{\sigma, i_{3}}^{u, \epsilon}(k_{2}) - \hat{X}_{\sigma, i_{2}}^{u, \epsilon}(k_{1}) \hat{X}_{\sigma, i_{3}}^{b, \epsilon}(k_{2})] d \sigma \nonumber \\
& \hspace{10mm}  \times \theta(2^{-i} k_{3}) \theta(2^{-j} k_{3}) \int_{0}^{t} \frac{ e^{-2 \lvert k_{3} \rvert^{2}(t-s)} f(\epsilon k_{3})^{2}}{2 \lvert k_{3} \rvert^{2}}  \hat{\mathcal{P}}_{j_{1} i_{4}}(k_{3}) \hat{\mathcal{P}}_{j_{0} i_{4}}(k_{3}) k_{3}^{j_{1}} \hat{\mathcal{P}}_{i_{0} i_{1}}(k_{3}) ds e_{k} \nonumber
\end{align} 
by (\ref{17}), (\ref{24b}) and (\ref{24c}). Thus, by (\ref{207}) and (\ref{221}) we obtain 
\begin{align}\label{222}
& \widetilde{\RomanIX}_{t}^{7,8,4} - \sum_{i_{1} =1}^{3} b^{
\scalebox{0.18}{\begin{tikzpicture}
\draw[black, thick] (-0.7,0.9) -- (0,0);
\draw[black, thick] (0.7,0.9) -- (0,0);
\draw[snake=zigzag](0,0) -- (0,-0.9);
\filldraw[violet] (-0.7,0.9) circle (6pt); 
\filldraw[violet] (0.7,0.9) circle (6pt); 
\end{tikzpicture}
}\epsilon}_{i_{1}}(t) C_{3}^{7, 8, \epsilon, i_{1}}(t) \nonumber \\
=& \frac{1}{4(2\pi)^{\frac{9}{2}}} \sum_{k \neq 0} \sum_{\lvert i-j \rvert \leq 1} \sum_{k_{1}, k_{2}: k_{12} = k, k_{3} \neq 0} \sum_{i_{1}, i_{2}, i_{3}, i_{4}, j_{1} =1}^{3} \int_{[0,t]^{2}} e^{- \lvert k_{12} \rvert^{2}(t- \sigma)} k_{12}^{i_{3}} \hat{\mathcal{P}}_{i_{1} i_{2}}(k_{12}) e_{k} \theta(2^{-j} k_{3}) \nonumber \\
& \times [ ( : \hat{X}_{\sigma, i_{2}}^{b, \epsilon}(k_{1}) \hat{X}_{\sigma, i_{3}}^{u, \epsilon}(k_{2}): e^{- \lvert k_{123} \rvert^{2} (t-s)} \theta(2^{-i} k_{123}) \hat{\mathcal{P}}_{i_{0} i_{1}}(k_{123}) k_{123}^{j_{1}} \nonumber \\
& - \hat{X}_{\sigma, i_{2}}^{b, \epsilon}(k_{1}) \hat{X}_{\sigma, i_{3}}^{u, \epsilon}(k_{2}) e^{- \lvert k_{3} \rvert^{2} (t-s)} \theta(2^{-i} k_{3}) \hat{\mathcal{P}}_{i_{0} i_{1}}(k_{3}) k_{3}^{j_{1}}) \nonumber \\
& - ( : \hat{X}_{\sigma, i_{2}}^{u, \epsilon}(k_{1}) \hat{X}_{\sigma, i_{3}}^{b, \epsilon}(k_{2}): e^{- \lvert k_{123} \rvert^{2}(t-s)} \theta(2^{-i} k_{123}) \hat{\mathcal{P}}_{i_{0} i_{1}}(k_{123}) k_{123}^{j_{1}}\nonumber \\
& \hspace{3mm}  - \hat{X}_{\sigma, i_{2}}^{u, \epsilon}(k_{1}) \hat{X}_{\sigma, i_{3}}^{b, \epsilon}(k_{2}) e^{- \lvert k_{3} \rvert^{2} (t-s)} \theta(2^{-i} k_{3}) \hat{\mathcal{P}}_{i_{0} i_{1}}(k_{3}) k_{3}^{j_{1}})]\nonumber \\
& \times \frac{ e^{- \lvert k_{3} \rvert^{2}(t-s)} f(\epsilon k_{3})^{2}}{2 \lvert k_{3} \rvert^{2}} \hat{\mathcal{P}}_{j_{1} i_{4}}(k_{3}) \hat{\mathcal{P}}_{j_{0} i_{4}}(k_{3}) d\sigma ds \triangleq \sum_{i=1}^{2} \RomanXI_{t}^{i}. 
\end{align} 
Due to similarity, let us work only on $\RomanXI_{t}^{1}$, to which we use $:\xi_{1}\xi_{2}: = \xi_{1}\xi_{2} - \mathbb{E}[\xi_{1}\xi_{2}]$ (see \cite{J97}) to deduce 
\begin{align}\label{223}
\RomanXI_{t}^{1} =& \frac{1}{4(2\pi)^{\frac{9}{2}}} \sum_{k \neq 0} \sum_{\lvert i-j \rvert \leq 1} \sum_{k_{1}, k_{2}: k_{12} = k, k_{3} \neq 0} \sum_{i_{1}, i_{2}, i_{3}, i_{4}, j_{1} =1}^{3}\nonumber \\
& \times \int_{[0,t]^{2}} e^{- \lvert k_{12} \rvert^{2}(t-\sigma)} k_{12}^{i_{3}} \hat{\mathcal{P}}_{i_{1} i_{2}} (k_{12}) e_{k} \theta(2^{-j} k_{3}) \hat{X}_{\sigma, i_{1}}^{b, \epsilon}(k_{1}) \hat{X}_{\sigma, i_{3}}^{u, \epsilon}(k_{2})\nonumber \\
& \times [ e^{- \lvert k_{123} \rvert^{2}(t-s)} \theta(2^{-i} k_{123}) \hat{\mathcal{P}}_{i_{0} i_{1}}(k_{123}) k_{123}^{j_{1}}   - e^{- \lvert k_{3} \rvert^{2}(t-s)} \theta(2^{-i} k_{3}) \hat{\mathcal{P}}_{i_{0} i_{1}}(k_{3}) k_{3}^{j_{1}}]\nonumber \\
& \times  \frac{e^{ - \lvert k_{3} \rvert^{2} (t-s)} f( \epsilon k_{3})^{2}}{2 \lvert k_{3} \rvert^{2}} \hat{\mathcal{P}}_{j_{1} i_{4}}(k_{3}) \hat{\mathcal{P}}_{j_{0} i_{4}}(k_{3}) d\sigma ds. 
\end{align} 
Now upon computing 
\begin{align*}
\mathbb{E} [ \lvert \Delta_{q} ( \widetilde{\RomanIX}_{t}^{7,8,4} - \sum_{i_{1} =1}^{3} b^{
\scalebox{0.18}{\begin{tikzpicture}
\draw[black, thick] (-0.7,0.9) -- (0,0);
\draw[black, thick] (0.7,0.9) -- (0,0);
\draw[snake=zigzag](0,0) -- (0,-0.9);
\filldraw[violet] (-0.7,0.9) circle (6pt); 
\filldraw[violet] (0.7,0.9) circle (6pt); 
\end{tikzpicture}
}\epsilon}_{i_{1}}(t) C_{3}^{7, 8, \epsilon, i_{1}}(t) ) \rvert^{2}], 
\end{align*}
we need to compute $\mathbb{E} [ \lvert \Delta_{q} \RomanXI_{t}^{1} \rvert^{2}]$. In its endeavor, we rely on the identity of $\mathbb{E} [ \xi_{1} \xi_{2} \xi_{3} \xi_{4}] = \mathbb{E} [\xi_{2} \xi_{3}] \mathbb{E} [\xi_{1} \xi_{4}] + \mathbb{E} [\xi_{2} \xi_{4}] \mathbb{E} [\xi_{1} \xi_{3}] + \mathbb{E} [\xi_{3} \xi_{4}] \mathbb{E} [\xi_{1} \xi_{2}]$ (\cite{J97} and \cite[Example 2.2]{Y20}) and (\ref{116}) to compute $ \mathbb{E} [ \hat{X}_{\sigma, i_{1}}^{b, \epsilon}(k_{1}) \hat{X}_{\sigma, i_{3}}^{u, \epsilon}(k_{2}) \hat{X}_{\overline{\sigma}, i_{1}'}^{b, \epsilon}(k_{1}') \hat{X}_{\overline{\sigma}, i_{3}'}^{u, \epsilon}(k_{2}')]$ and deduce  
\begin{align*} 
& \mathbb{E} [ \lvert \Delta_{q} \RomanXI_{t}^{1} \rvert^{2} ]\lesssim \sum_{k, k' \neq 0} \sum_{k_{1}, k_{2} \neq 0: k_{12} = k, k_{1}', k_{2}':  k_{12}' = k'} 1_{k_{2} + k_{1} ' = 0, k_{1} + k_{2} ' = 0}  \\
& \hspace{8mm} \times \int_{[0,t]^{2}} e^{- \lvert k_{12} \rvert^{2} (t-\sigma) - \lvert k_{12}' \rvert^{2} ( t - \overline{\sigma})} \lvert k_{12} \rvert \lvert k_{12}' \rvert  \theta (2^{-q} k) \theta(2^{-q} k')  \frac{ e^{ - ( \lvert k_{1} \rvert^{2} + \lvert k_{2} \rvert^{2}) \lvert \sigma - \overline{\sigma} \rvert}}{ \lvert k_{1} \rvert^{2} \lvert k_{2} \rvert^{2}} d \sigma d \overline{\sigma} \\
& \hspace{8mm} \times [ \sum_{\lvert i-j \rvert \leq 1} \sum_{i_{1}, j_{1} = 1}^{3} \sum_{k_{3} \neq 0} \theta(2^{-j} k_{3}) \int_{0}^{t} \frac{ e^{ - \lvert k_{3} \rvert^{2} ( t-s)} f( \epsilon k_{3})^{2}}{\lvert k_{3} \rvert^{2}}\\
&\hspace{8mm}  \times ( e^{- \lvert k_{123} \rvert^{2} (t-s)} \theta(2^{-i} k_{123}) \hat{\mathcal{P}}_{i_{0} i_{1}}(k_{123}) k_{123}^{j_{1}} - e^{- \lvert k_{3} \rvert^{2}(t-s)} \theta (2^{-i} k_{3}) \hat{\mathcal{P}}_{i_{0} i_{1}}(k_{3}) k_{3}^{j_{1}}) ds ]^{2}, 
\end{align*} 
where we observe that $\lvert k_{12} ' \rvert = \lvert k_{12} \rvert$ due to $1_{k_{2} + k_{1} ' = 0, k_{1} + k_{2} ' = 0}$ so that we may estimate
\begin{align*} 
 \int_{[0,t]^{2}} e^{- \lvert k_{12} \rvert^{2} (t-\sigma) - \lvert k_{12}' \rvert^{2} (t- \overline{\sigma})} \lvert k_{12} \rvert \lvert k_{12} ' \rvert e^{- ( \lvert k_{1} \rvert^{2} + \lvert k_{2} \rvert^{2}) \lvert \sigma - \overline{\sigma} \rvert} d \sigma d \overline{\sigma}  \lesssim \frac{1}{\lvert k_{12} \rvert^{2}}
\end{align*}
for $k_{12} \neq 0$. Therefore, (\ref{216}) gives for any $\eta \in (0,1)$, 
\begin{align}\label{225}
&\mathbb{E} [ \lvert \Delta_{q} \RomanXI_{t}^{1} \rvert^{2}] \nonumber\\
\lesssim& \sum_{k\neq 0} \sum_{k_{1}, k_{2} \neq 0: k_{12} = k} \frac{ \theta(2^{-q} k)^{2}}{\lvert k_{1} \rvert^{2} \lvert k_{2} \rvert^{2} \lvert k_{12} \rvert^{2}} [ \sum_{\lvert i-j \rvert \leq 1} \sum_{i_{1}, j_{1} =1}^{3} \sum_{k_{3} \neq 0} \theta(2^{-j} k_{3}) \int_{0}^{t} \frac{ e^{- \lvert k_{3} \rvert^{2}(t-s)} f( \epsilon k_{3})^{2}}{\lvert k_{3} \rvert^{2}}\nonumber \\
\times& ( e^{- \lvert k_{123} \rvert^{2}(t-s)} \theta( 2^{-i} k_{123}) \hat{\mathcal{P}}_{i_{0} i_{1}}(k_{123}) k_{123}^{j_{1}}   - e^{- \lvert k_{3} \rvert^{2}(t-s)} \theta( 2^{-i} k_{3}) \hat{\mathcal{P}}_{i_{0} i_{1}}(k_{3}) k_{3}^{j_{1}}) ds ]^{2} \nonumber \\
\lesssim& \sum_{k \neq 0} \frac{ \theta(2^{-q} k)^{2}}{\lvert k \rvert^{3}} 2^{q( 2\eta)} [\sum_{k_{3} \neq 0} \frac{1}{\lvert k_{3} \rvert^{3+ \epsilon}}]^{2} t^{\eta - \epsilon}  \lesssim 2^{q(2\eta)} t^{\eta - \epsilon}
\end{align} 
due to a straight-forward extension of Lemma \ref{Lemma 2.8}, Lemma \ref{Lemma 2.9} and (\ref{14}). We obtain similar estimates for $\mathbb{E} [ \lvert \Delta_{q} \RomanXI_{t}^{2} \rvert^{2} ]$ in (\ref{222}). Together with (\ref{220}), this concludes our estimate of 
\begin{equation}\label{226}
\mathbb{E} [ \lvert \Delta_{q}(\RomanIX_{t}^{7,4} + \RomanIX_{t}^{8,4} ) \rvert^{2}] \lesssim 2^{2q(\eta + \frac{\epsilon}{2})} t^{2( \frac{\eta}{3} + \frac{\epsilon}{6})} 
\end{equation}
if we choose $\epsilon, \eta > 0$ such that $\epsilon \leq \frac{\eta}{4}$. 

For $\RomanIX_{t}^{8,k}, k \in \{1, \hdots, 6\}$, in (\ref{194}), we obtained estimates of $\RomanIX_{t}^{8,2}$ in (\ref{201}) and $\RomanIX_{t}^{7,4} + \RomanIX_{t}^{8,4}$ in (\ref{226}). Next, within (\ref{194}) let us work on 
\begin{equation}\label{227}
\RomanIX_{t}^{8,5} = \RomanIX_{t}^{8,5} - \widetilde{\RomanIX}_{t}^{8,5} + \widetilde{\RomanIX}_{t}^{8,5} -\overline{\RomanIX}_{t}^{8,5}
\end{equation}
where 
\begin{align}\label{228}
\widetilde{\RomanIX}_{t}^{8,5} 
&\triangleq -\frac{1}{4(2\pi)^{\frac{9}{2}}} \sum_{k} \sum_{\lvert i-j \rvert \leq 1} \sum_{k_{1}, k_{4}: k_{14} = k, k_{2} \neq 0} \sum_{i_{1}, i_{2}, i_{3}, i_{4}, j_{1} =1}^{3} \theta(2^{-i} k_{1}) \theta(2^{-j} k_{4})\nonumber \\
& \times \int_{0}^{t} : \hat{X}_{s, i_{2}}^{u, \epsilon}(k_{1}) \hat{X}_{t, j_{0}}^{b, \epsilon}(k_{4}): e^{- \lvert k_{1} \rvert^{2}(t-s)} k_{1}^{j_{1}} \hat{\mathcal{P}}_{i_{0} i_{1}}(k_{1})  \int_{0}^{s} e^{- \lvert k_{12} \rvert^{2}(s-\sigma)}\nonumber \\
& \times \frac{ e^{- \lvert k_{2} \rvert^{2}(s-\sigma)} f( \epsilon k_{2})^{2}}{\lvert k_{2} \rvert^{2}} k_{12}^{i_{3}} \hat{\mathcal{P}}_{i_{1} i_{2}}(k_{12}) \hat{\mathcal{P}}_{i_{3} i_{4}}(k_{2}) \hat{\mathcal{P}}_{j_{1} i_{4}}(k_{2}) d\sigma ds e_{k}
\end{align} 
and 
\begin{align}\label{229}
\overline{\RomanIX}_{t}^{8,5} &\triangleq -\frac{1}{4(2\pi)^{\frac{9}{2}}} \sum_{k} \sum_{\lvert i-j \rvert \leq 1} \sum_{k_{1}, k_{4}: k_{14} = k, k_{2} \neq 0} \sum_{i_{1}, i_{2}, i_{3}, j_{1} =1}^{3} \theta(2^{-i} k_{1}) \theta(2^{-j} k_{4})\nonumber \\
& \times \int_{0}^{t} : \hat{X}_{s, i_{2}}^{u, \epsilon}(k_{1}) \hat{X}_{t, j_{0}}^{b, \epsilon}(k_{4}): e^{- \lvert k_{1} \rvert^{2}(t-s)} k_{1}^{j_{1}} \hat{\mathcal{P}}_{i_{0} i_{1}}(k_{1})  \int_{0}^{s} e^{- 2\lvert k_{2} \rvert^{2}(s-\sigma)}\nonumber \\
& \times \frac{ f( \epsilon k_{2})^{2}}{\lvert k_{2} \rvert^{2}} k_{2}^{i_{3}} \hat{\mathcal{P}}_{i_{1} i_{2}}(k_{2}) \hat{\mathcal{P}}_{i_{3} i_{4}}(k_{2}) \hat{\mathcal{P}}_{j_{1} i_{4}}(k_{2}) d\sigma ds e_{k} 
\end{align} 
so that $\overline{\RomanIX}_{t}^{8,5} = 0$. We define for $k_{2} \neq 0$, 
\begin{equation}\label{230}
d_{k_{12}, k_{2}} (s-\sigma) \triangleq \sum_{i_{2}, i_{3} =1}^{3} e^{- \lvert k_{12} \rvert^{2}(s-\sigma)} \frac{ e^{- \lvert k_{2} \rvert^{2}(s-\sigma)} f( \epsilon k_{2})^{2}}{\lvert k_{2} \rvert^{2}} \lvert k_{12}^{i_{3}} \hat{\mathcal{P}}_{i_{1} i_{2}}(k_{12}) \rvert. 
\end{equation} 
Then we see that
\begin{align}\label{231}
& \mathbb{E} [ : \hat{X}_{s, i_{2}}^{u, \epsilon}(k_{1}) \hat{X}_{t, j_{0}}^{b, \epsilon}(k_{4}): - : \hat{X}_{\sigma, i_{2}}^{u, \epsilon}(k_{1}) \hat{X}_{t, j_{0}}^{b, \epsilon}(k_{4}): \overline{: \hat{X}_{\overline{s}, i_{2}'}^{u, \epsilon} (k_{1}') \hat{X}_{t, j_{0}}^{b, \epsilon}(k_{4}'): - : \hat{X}_{\overline{\sigma}, i_{2}'}^{u, \epsilon}(k_{1}') \hat{X}_{t, j_{0}}^{b, \epsilon}(k_{4}'): }]\nonumber \\
=& \mathbb{E} [ \hat{X}_{s, i_{2}}^{u, \epsilon}(k_{1}) \hat{X}_{\overline{s}, i_{2}'}^{u, \epsilon}(k_{1}') ] \mathbb{E} [ \hat{X}_{t, j_{0}}^{b, \epsilon}(k_{4}) \hat{X}_{t, j_{0}}^{b, \epsilon}(k_{4}')] + \mathbb{E} [ \hat{X}_{s, i_{2}}^{u, \epsilon}(k_{1}) \hat{X}_{t, j_{0}}^{b, \epsilon}(k_{4}') ] \mathbb{E} [ \hat{X}_{t, j_{0}}^{b, \epsilon}(k_{4}) \hat{X}_{u,\overline{s}}^{\epsilon, i_{2}'}(k_{1}')]\nonumber \\
 -& \mathbb{E} [ \hat{X}_{s, i_{2}}^{u, \epsilon}(k_{1}) \hat{X}_{\overline{\sigma}, i_{2}'}^{u, \epsilon}(k_{1}') ] \mathbb{E} [ \hat{X}_{t, j_{0}}^{b, \epsilon}(k_{4}) \hat{X}_{t, j_{0}}^{b, \epsilon}(k_{4}')]  - \mathbb{E} [ \hat{X}_{s, i_{2}}^{u, \epsilon}(k_{1}) \hat{X}_{t, j_{0}}^{b, \epsilon}(k_{4}') ] \mathbb{E} [ \hat{X}_{t, j_{0}}^{b, \epsilon}(k_{4}) \hat{X}_{\overline{\sigma}, i_{2}'}^{u, \epsilon}(k_{1}')]\nonumber \\
 -& \mathbb{E} [ \hat{X}_{\sigma, i_{2}}^{u, \epsilon}(k_{1}) \hat{X}_{\overline{s}, i_{2}'}^{u, \epsilon}(k_{1}') ] \mathbb{E} [ \hat{X}_{t, j_{0}}^{b, \epsilon}(k_{4}) \hat{X}_{t, j_{0}}^{b, \epsilon}(k_{4}')] - \mathbb{E} [ \hat{X}_{\sigma, i_{2}}^{u, \epsilon}(k_{1}) \hat{X}_{t, j_{0}}^{b, \epsilon}(k_{4}') ] \mathbb{E} [ \hat{X}_{t, j_{0}}^{b, \epsilon}(k_{4}) \hat{X}_{u,\overline{s}}^{\epsilon, i_{2}'}(k_{1}')]\nonumber \\
+& \mathbb{E} [ \hat{X}_{\sigma, i_{2}}^{u, \epsilon}(k_{1}) \hat{X}_{\overline{\sigma}, i_{2}'}^{u, \epsilon}(k_{1}') ] \mathbb{E} [ \hat{X}_{t, j_{0}}^{b, \epsilon}(k_{4}) \hat{X}_{t, j_{0}}^{b, \epsilon}(k_{4}')]   + \mathbb{E} [ \hat{X}_{\sigma, i_{2}}^{u, \epsilon}(k_{1}) \hat{X}_{t, j_{0}}^{b, \epsilon}(k_{4}') ] \mathbb{E} [ \hat{X}_{t, j_{0}}^{b, \epsilon}(k_{4}) \hat{X}_{\overline{\sigma}, i_{2}'}^{u, \epsilon}(k_{1}')] \nonumber\\
& \hspace{10mm}  \triangleq \sum_{i=1}^{8} \RomanXII^{i}
\end{align} 
by $\mathbb{E} [ : \xi_{11} \xi_{12}: : \xi_{21} \xi_{22}:] = \mathbb{E} [ \xi_{11} \xi_{21}] \mathbb{E} [ \xi_{12} \xi_{22}] + \mathbb{E} [ \xi_{11} \xi_{22}]\mathbb{E}[\xi_{12} \xi_{21}]$ (see \cite{J97}). By interpolation, relying on \cite[Section 9.2]{GP17}, and using (\ref{187}), (\ref{228}), (\ref{230}), and \eqref{231} we obtain 
\begin{align}\label{234}
& \mathbb{E} [ \lvert \Delta_{q} ( \RomanIX_{t}^{8,5} - \widetilde{\RomanIX}_{t}^{8,5} ) \rvert^{2}] \nonumber \\
\lesssim& \sum_{k} \theta(2^{-q} k)^{2} \sum_{\lvert i-j \rvert \leq 1, \lvert i'-j'\rvert \leq 1} \sum_{k_{1}, k_{4} \neq 0: k_{14} = k, k_{2}, k_{3} \neq 0}  \theta(2^{-i} k_{1}) \theta(2^{-i'} k_{1}) \theta(2^{-j} k_{4}) \theta(2^{-j'} k_{4})\nonumber \\
& \times \int_{[0,t]^{2}}\int_{[0,s]\times [0, \overline{s}]} e^{- \lvert k_{1} \rvert^{2} [t- s + t - \overline{s}]} \frac{ \lvert k_{1} \rvert^{2\eta}}{\lvert k_{1} \rvert^{2} \lvert k_{4} \rvert^{2}} \lvert s- \sigma \rvert^{\frac{\eta}{2}} \lvert \overline{s} - \overline{\sigma} \rvert^{\frac{\eta}{2}} \nonumber \\
& \times d_{k_{12}, k_{2}}( s-\sigma) d_{k_{13}, k_{3}} (\overline{s} - \overline{\sigma}) \lvert k_{1} \rvert^{2} d \sigma d \overline{\sigma} ds d \overline{s}e_{k} 
\end{align} 
We can estimate for $k_{1}, k_{4}\neq 0$, 
\begin{align}\label{235}
& \sum_{k_{2}, k_{3} \neq 0} \int_{[0,t]^{2}} \int_{[0,s]\times [0,\overline{s}]} e^{- \lvert k_{1} \rvert^{2}(t- s + t - \overline{s})} \frac{ \lvert k_{1} \rvert^{2\eta + 2}}{\lvert k_{1} \rvert^{2} \lvert k_{4} \rvert^{2}} \nonumber \\
& \times \lvert s-\sigma \rvert^{\frac{\eta}{2}} \lvert \overline{s} - \overline{\sigma} \rvert^{\frac{\eta}{2}} d_{k_{12}, k_{2}}(s-\sigma) d_{k_{13}, k_{3}}(\overline{s} - \overline{\sigma}) d \sigma d \overline{\sigma} ds d \overline{s}\nonumber \\
\lesssim& \frac{1}{ \lvert k_{1} \rvert^{-2\eta} \lvert k_{4} \rvert^{2}} \sum_{k_{2}, k_{3} \neq 0} \frac{ \lvert k_{12} \rvert \lvert k_{13} \rvert}{ \lvert k_{2} \rvert^{2+ \eta} \lvert k_{3} \rvert^{2+ \eta}} \int_{[0,t]^{2}} \int_{[0,s]\times [0,\overline{s}]} e^{- \lvert k_{1} \rvert^{2} t} e^{ \frac{1}{2}[ \lvert k_{1} \rvert^{2} - \lvert k_{12} \rvert^{2} - \lvert k_{2} \rvert^{2} ]s}\nonumber \\
& \times e^{ \frac{1}{2}[ \lvert k_{1} \rvert^{2} - \lvert k_{13} \rvert^{2} - \lvert k_{3} \rvert^{2}] \overline{s}} e^{ \frac{1}{2}[ \lvert k_{12} \rvert^{2}+  \lvert k_{2} \rvert^{2}] \sigma} e^{ \frac{1}{2} [ \lvert k_{13} \rvert^{2} + \lvert k_{3} \rvert^{2} ] \overline{\sigma} } d \sigma d \overline{\sigma} ds d \overline{s}\nonumber \\
\lesssim& \frac{1}{ \lvert k_{1} \rvert^{-2\eta} \lvert k_{4} \rvert^{2}} \sum_{k_{2}, k_{3}\neq 0} \frac{1}{\lvert k_{2} \rvert^{3+ \eta} \lvert k_{3} \rvert^{3+ \eta}} \frac{(1 - e^{- \lvert k_{1} \rvert^{2} t})^{2}}{ \lvert k_{1} \rvert^{4}} \lesssim \frac{t^{2(\frac{\eta}{3} + \frac{\epsilon}{6})}}{\lvert k_{1} \rvert^{4 - 4(\frac{\eta}{3} + \frac{\epsilon}{6}) - 2 \eta} \lvert k_{4} \rvert^{2}} 
\end{align} 
by (\ref{230}) and (\ref{14}). Therefore, applying (\ref{235}) to (\ref{234}) gives 
\begin{align}\label{237}
\mathbb{E} [ \lvert \Delta_{q} ( \RomanIX_{t}^{8,5} - \widetilde{\RomanIX}_{t}^{8,5} ) \rvert^{2} ] 
\lesssim& t^{2(\frac{\eta}{3} + \frac{\epsilon}{6})} \sum_{k} \theta(2^{-q} k)^{2} \sum_{k_{1}, k_{4} \neq 0: k_{14} =k}  \sum_{q \lesssim i} \frac{ 2^{-i}}{\lvert k_{1} \rvert^{3- \frac{10\eta}{3} - \frac{2\epsilon}{3}} \lvert k_{4} \rvert^{2}}  \nonumber \\
\lesssim& t^{2( \frac{\eta}{3} + \frac{\epsilon}{6})} 2^{q(\frac{10\eta}{3} + \frac{2\epsilon}{3})} \sum_{k\neq 0} \frac{1}{\lvert k \rvert^{3}} \lesssim t^{2(\frac{\eta}{3} + \frac{\epsilon}{6})} 2^{q(\frac{10\eta}{3} + \frac{2\epsilon}{3})} 
\end{align} 
where we used that $2^{q} \lesssim 2^{i}$ so that $q \lesssim i$ and Lemma \ref{Lemma 2.9}. Next, within (\ref{227}) we estimate  
\begin{align}\label{238}
\mathbb{E} [ \lvert \Delta_{q} ( \widetilde{\RomanIX}_{t}^{8,5} - \overline{\RomanIX}_{t}^{8,5} ) \rvert^{2} ] 
\approx& \mathbb{E} [ \lvert \sum_{k} \theta(2^{-q} k) \sum_{\lvert i-j \rvert \leq 1} \sum_{k_{1}, k_{4}: k_{14} = k, k_{2} \neq 0} \sum_{i_{1}, i_{2}, i_{3}, i_{4}, j_{1} =1}^{3} \theta(2^{-i} k_{1}) \theta(2^{-j} k_{4}) \nonumber \\
& \times \int_{0}^{t}: \hat{X}_{s, i_{2}}^{u, \epsilon}(k_{1}) \hat{X}_{t, j_{0}}^{b, \epsilon} (k_{4}): e^{- \lvert k_{1} \rvert^{2}(t-s)} k_{1}^{j_{1}} \hat{\mathcal{P}}_{i_{0} i_{1}}(k_{1})\nonumber \\
& \times \int_{0}^{s} [ e^{- \lvert k_{12} \rvert^{2} (s-\sigma)} k_{12}^{i_{3}} \hat{\mathcal{P}}_{i_{1} i_{2}}(k_{12}) - e^{- \lvert k_{2} \rvert^{2} (s-\sigma)} k_{2}^{i_{3}} \hat{\mathcal{P}}_{i_{1} i_{2}}(k_{2})] \nonumber \\
& \times \frac{ e^{- \lvert k_{2} \rvert^{2} (s-\sigma)} f(\epsilon k_{2})^{2}}{ \lvert k_{2} \rvert^{2}} \hat{\mathcal{P}}_{i_{3} i_{4}}(k_{2}) \hat{\mathcal{P}}_{j_{1} i_{4}}(k_{2}) e_{k} d \sigma ds \rvert^{2}]
\end{align} 
due to (\ref{228})-(\ref{229}). For $k_{1}, k_{4} \neq 0$ we can compute $\mathbb{E} [ : \hat{X}_{s, i_{2}}^{u, \epsilon}(k_{1}) \hat{X}_{t, j_{0}}^{b, \epsilon}(k_{4}): : \hat{X}_{\overline{s}, i_{2}'}^{u, \epsilon} (k_{1}') \hat{X}_{t, j_{0}}^{b, \epsilon}(k_{4}'):]$ by $\mathbb{E} [ : \xi_{11} \xi_{12}: : \xi_{21} \xi_{22}:] = \mathbb{E} [ \xi_{11} \xi_{21}] \mathbb{E} [ \xi_{12} \xi_{22}] + \mathbb{E} [ \xi_{11} \xi_{22}]\mathbb{E}[\xi_{12} \xi_{21}]$ (see \cite{J97}) and (\ref{116}) to deduce
\begin{align}\label{240}
& \mathbb{E} [ \lvert \Delta_{q}( \widetilde{\RomanIX}_{t}^{8,5} - \overline{\RomanIX}_{t}^{8,5} ) \rvert^{2} ] \nonumber \\
\lesssim& \sum_{k, k' } \theta(2^{-q} k) \theta(2^{-q} k') \sum_{\lvert i-j \rvert \leq 1, \lvert i' -j' \rvert \leq 1} \sum_{k_{1}, k_{4} \neq 0: k_{14} =k, k_{2} \neq 0, k_{1}', k_{4}': k_{14}' = k', k_{2}' \neq 0} \nonumber  \\
& \times \sum_{i_{1}, i_{2}, i_{3}, i_{1}', i_{2}', i_{3}' = 1}^{3} \theta(2^{-i} k_{1}) \theta (2^{-i'} k_{1}') \theta(2^{-j} k_{4}) \theta (2^{-j'} k_{4}') \int_{[0,t]^{2}}\nonumber \\
& \times  1_{k_{1} + k_{1}' = 0, k_{4} + k_{4} ' = 0} \frac{ e^{- \lvert k_{1} \rvert^{2} \lvert s- \overline{s} \rvert}}{ \lvert k_{1} \rvert^{2} \lvert k_{4}\rvert^{2}} e^{- \lvert k_{1} \rvert^{2} ( t-s)} e^{- \lvert k_{1}' \rvert^{2}(t- \overline{s} )} \lvert k_{1} \rvert \lvert k_{1}' \rvert \nonumber \\
& \times \left( \int_{0}^{s} e^{- \lvert k_{12} \rvert^{2} ( s- \sigma)} k_{12}^{i_{3}} \hat{\mathcal{P}}_{i_{1} i_{2}}(k_{12}) - e^{- \lvert k_{2} \rvert^{2} ( s- \sigma)} k_{2}^{i_{3}} \hat{\mathcal{P}}_{i_{2} i_{3}}(k_{2}) \right) \nonumber \\
& \times \left( \int_{0}^{\overline{s}} e^{- \lvert k_{12} ' \rvert^{2} ( \overline{s} - \overline{\sigma})} ( k_{12}')^{i_{3}} \hat{\mathcal{P}}_{i_{1}' i_{2}'} (k_{12}') - e^{- \lvert k_{2}' \rvert^{2} ( \overline{s} - \overline{\sigma})} (k_{2}')^{i_{3}} \hat{\mathcal{P}}_{i_{2}' i_{3}'} (k_{2}') \right)\nonumber \\
& \times \frac{ e^{- \lvert k_{2} \rvert^{2} ( s-\sigma)} e^{- \lvert k_{2}' \rvert^{2} (\overline{s} - \overline{\sigma})}}{\lvert k_{2} \rvert^{2} \lvert k_{2}' \rvert^{2}} e_{k} e_{k}' d \sigma d \overline{\sigma} ds d \overline{s}  \nonumber\\
\lesssim& \sum_{k} \theta(2^{-q} k)^{2} \sum_{\lvert i-j \rvert \leq 1, \lvert i' - j' \rvert \leq 1} \sum_{k_{1}, k_{4} \neq 0: k_{14} = k, k_{2}, k_{3} \neq 0}  \theta(2^{-i} k_{1}) \theta(2^{-i'} k_{1}) \nonumber\\
& \times \theta(2^{-j} k_{4}) \theta(2^{-j'} k_{4}) \int_{[0,t]^{2}} \frac{ e^{- \lvert k_{1} \rvert^{2}( \lvert s- \overline{s} \rvert + 2t - s - \overline{s})}}{\lvert k_{1} \rvert^{2}  \lvert k_{4} \rvert^{2}} \lvert k_{1} \rvert^{2}\nonumber \\
& \times \int_{[0,s]\times [0,\overline{s}]} \frac{ \lvert k_{1} \rvert^{2\eta} \lvert s- \sigma \rvert^{- \frac{(1-\eta)}{2}} \lvert \overline{s} - \overline{\sigma} \rvert^{- \frac{(1-\eta)}{2}}}{ \lvert k_{2} \rvert^{2} \lvert k_{3} \rvert^{2}}e^{- \lvert k_{2} \rvert^{2} (s-\sigma) - \lvert k_{3} \rvert^{2} ( \overline{s} - \overline{\sigma} )} d \sigma d \overline{\sigma} ds d \overline{s} 
\end{align} 
for any $\eta \in (0,1)$ due to a change of variable of $k_{2}'$ with $k_{3}$ and Lemma \ref{Lemma 2.8}. By applying H$\ddot{\mathrm{o}}$lder's inequality we can bound furthermore as 
\begin{align}\label{243}
& \mathbb{E} [ \lvert \Delta_{q}(\widetilde{\RomanIX}_{t}^{8,5} - \overline{\RomanIX}_{t}^{8,5} ) \rvert^{2}]\nonumber \\
\lesssim& \sum_{k} \theta(2^{-q} k)^{2} \sum_{\lvert i-j \rvert \leq 1, \lvert i' - j'\rvert \leq 1} \sum_{k_{1}, k_{4} \neq 0: k_{14} = k, k_{2}, k_{3} \neq 0} \theta(2^{-i} k_{1}) \theta(2^{-i'} k_{1})\nonumber \\
&\times  \theta(2^{-j} k_{4}) \theta(2^{-j'} k_{4})  \int_{[0,t]^{2}} e^{- \lvert k_{1} \rvert^{2}( \lvert s - \overline{s} \rvert + 2t - s - \overline{s})} \frac{ \lvert k_{1} \rvert^{2\eta}}{ \lvert k_{2} \rvert^{2} \lvert k_{3} \rvert^{2} \lvert k_{4} \rvert^{2}} \nonumber \\
& \times \left( \int_{0}^{s} \lvert s- \sigma \rvert^{- (1-\eta)} d \sigma \right)^{\frac{1}{2}} \left( \int_{0}^{s} e^{-2 \lvert k_{2} \rvert^{2} (s-\sigma)} d \sigma \right)^{\frac{1}{2}} \nonumber \\
& \times \left( \int_{0}^{\overline{s}} \lvert \overline{s} - \overline{\sigma} \rvert^{- (1-\eta)} d \overline{\sigma} \right)^{\frac{1}{2}} \left( \int_{0}^{\overline{s}} e^{-2 \lvert k_{3} \rvert^{2} ( \overline{s} - \overline{\sigma})} d \overline{\sigma} \right)^{\frac{1}{2}} ds d \overline{s}\nonumber \\
\lesssim& \sum_{k} \theta(2^{-q} k)^{2} \sum_{\lvert i-j \rvert \leq 1, \lvert i'- j'\rvert \leq 1} \sum_{k_{1}, k_{4} \neq 0: k_{14} = k, k_{2}, k_{3} \neq 0} \theta(2^{-i} k_{1}) \theta(2^{-i'} k_{1}) \nonumber \\
& \times \theta(2^{-j} k_{4}) \theta(2^{-j'} k_{4})  e^{-2t \lvert k_{1} \rvert^{2}} \int_{[0,t]^{2}} e^{ \lvert k_{1} \rvert^{2}( s + \overline{s})} \frac{ \lvert k_{1} \rvert^{2\eta}}{\lvert k_{2} \rvert^{2} \lvert k_{3} \rvert^{2} \lvert k_{4} \rvert^{2}} \nonumber \\
& \times \frac{ \lvert s \rvert^{\frac{\eta}{2}} (1- e^{-2 \lvert k_{2} \rvert^{2} s})^{\frac{1}{2}}}{\lvert k_{2} \rvert} \frac{ \lvert \overline{s} \rvert^{\frac{\eta}{2}} (1- e^{-2 \lvert k_{3} \rvert^{2} \overline{s} })^{\frac{1}{2}}}{\lvert k_{3} \rvert} ds d \overline{s}.
\end{align} 
Finally, we continue to bound this by 
\begin{align*} 
& \mathbb{E} [ \lvert \Delta_{q} ( \widetilde{\RomanIX}_{t}^{8,5} - \overline{\RomanIX}_{t}^{8,5} ) \rvert^{2} ]  \\
\lesssim& t^{2( \frac{\eta}{3} + \frac{\epsilon}{6})} \sum_{k \neq 0} \theta(2^{-q} k)^{2} \sum_{q \lesssim i} 2^{-i} \sum_{k_{1}, k_{4} \neq 0: k_{14} = k}  \frac{1}{\lvert k_{1} \rvert^{3- \frac{10\eta}{3} - \frac{2\epsilon}{3}}\lvert k_{4} \rvert^{2}}  \lesssim t^{2( \frac{\eta}{3} + \frac{\epsilon}{6})} 2^{q(\frac{10\eta}{3} + \frac{2\epsilon}{3})} 
\end{align*}
due to the mean value theorem, that $2^{q} \lesssim 2^{i}$ and Lemma \ref{Lemma 2.9}. Combining this with (\ref{237}) in (\ref{227}) gives 
\begin{equation}\label{244}
\mathbb{E} [ \lvert \Delta_{q} \RomanIX_{t}^{8,5} \rvert^{2}] \lesssim t^{2( \frac{\eta}{3} + \frac{\epsilon}{6})} 2^{q(\frac{10\eta}{3} + \frac{2\epsilon}{3})}. 
\end{equation}
Similar estimates for $\RomanIX_{t}^{8,6}$ may be deduced as well. 

\subsubsection{Terms in the fourth chaos} 

We finally work on $\RomanIX_{t}^{8,1}$ of (\ref{194}), specifically the first term of (\ref{184}) where 
\begin{align}\label{245}
\RomanIX_{t}^{8,1} 
=& -\frac{1}{4(2\pi)^{\frac{9}{2}}} \sum_{k} \sum_{\lvert i-j \rvert \leq 1} \sum_{k_{1}, k_{2}, k_{3}, k_{4}: k_{1234} = k} \sum_{i_{1}, i_{2}, i_{3}, j_{1} = 1}^{3}\theta(2^{-i} k_{123}) \theta(2^{-j} k_{4}) \nonumber \\
& \times \int_{0}^{t} e^{- \lvert k_{123} \rvert^{2 }(t-s)} \int_{0}^{s} : \hat{X}_{\sigma, i_{2}}^{u, \epsilon}(k_{1}) \hat{X}_{\sigma, i_{3}}^{b, \epsilon}(k_{2}) \hat{X}_{s, j_{1}}^{b, \epsilon}(k_{3}) \hat{X}_{t, j_{0}}^{b, \epsilon}(k_{4}): \nonumber \\
& \times e^{- \lvert k_{12} \rvert^{2}(s-\sigma)} d\sigma ds k_{12}^{i_{3}} k_{123}^{j_{1}} \hat{\mathcal{P}}_{i_{1} i_{2}}(k_{12}) \hat{\mathcal{P}}_{i_{0} i_{1}}(k_{123}) e_{k}. 
\end{align} 
We can apply Lemma \ref{Lemma 5.7} (2) with ``$Y_{1}$'' = $: \hat{X}_{\sigma, i_{2}}^{u, \epsilon}(k_{1}) \hat{X}_{\sigma, i_{3}}^{b, \epsilon}(k_{2}) \hat{X}_{s, j_{1}}^{b, \epsilon}(k_{3}) \hat{X}_{t, j_{0}}^{b, \epsilon}(k_{4}):$ and ``$Y_{2}$'' = $ : \hat{X}_{\overline{\sigma}, i_{2}'}^{u, \epsilon}(k_{1}') \hat{X}_{\overline{\sigma}, i_{3}'}^{b, \epsilon}(k_{2}') \hat{X}_{\overline{s}, j_{1}'}^{b, \epsilon}(k_{3}') \hat{X}_{t, j_{0}}^{b, \epsilon}(k_{4}'):$ to explicitly compute $\mathbb{E} [Y_{1}Y_{2}]=\sum_{\gamma} v(\gamma)$ which consists of 24 terms (see \cite[Example 2.2]{Y20}), with 
\begin{align*}
 \mathbb{E} [ \hat{X}_{\sigma, i_{2}}^{u, \epsilon}(k_{1}) \hat{X}_{\overline{\sigma}, i_{2}'}^{u, \epsilon}(k_{1}')] \mathbb{E} [ \hat{X}_{\sigma, i_{3}}^{b, \epsilon}(k_{2}) \hat{X}_{\overline{\sigma}, i_{3}'}^{b, \epsilon}(k_{2}')] \mathbb{E} [ \hat{X}_{s, j_{1}}^{b, \epsilon}(k_{3}) \hat{X}_{\overline{s}, j_{1}'}^{b, \epsilon}(k_{3}')] \mathbb{E} [ \hat{X}_{t, j_{0}}^{b, \epsilon}(k_{4}) \hat{X}_{t, j_{0}}^{b, \epsilon}(k_{4}')]
\end{align*}
being one representative which can be bounded by a constant multiples of 
\begin{align}
\prod_{j=1}^{4} \frac{1}{\lvert k_{j} \rvert^{2}}  1_{k_{1} + k_{1}' = 0, k_{2} + k_{2}' = 0, k_{3} + k_{3}' = 0, k_{4} + k_{4}' = 0} \label{248}
\end{align}
when $k_{1}, k_{2}, k_{3}, k_{4} \neq 0$ and hence  
\begin{align}\label{249}
\mathbb{E} [ \lvert \Delta_{q} \RomanIX_{t}^{8,1} \rvert^{2}]  \lesssim& \sum_{k} \theta(2^{-q} k)^{2} \sum_{\lvert i-j \rvert \leq 1, \lvert i'-j'\rvert \leq 1} \sum_{k_{1}, k_{2}, k_{3}, k_{4}, k_{1}', k_{2}', k_{3}', k_{4}' \neq 0: k_{1234} = k_{1234}' = k}\nonumber \\
& \times  \theta(2^{-i} k_{123}) \theta(2^{-i'} k_{123}') \theta(2^{-j} k_{4}) \theta(2^{-j'} k_{4}')  \int_{[0,t]^{2}} e^{- \lvert k_{123} \rvert^{2}(t-s)} e^{- \lvert k_{123} ' \rvert^{2}(t- \overline{s})} \nonumber  \\
& \times \int_{[0,s]\times [0,\overline{s}]}  \frac{1}{\lvert k_{1} \rvert^{2} \lvert k_{2} \rvert^{2} \lvert k_{3} \rvert^{2} \lvert k_{4} \rvert^{2}} e^{- \lvert k_{12} \rvert^{2}(s-\sigma) - \lvert k_{12}' \rvert^{2}(\overline{s} - \overline{\sigma})} d \sigma d \overline{\sigma} \nonumber \\
& \times \lvert k_{12} \rvert \lvert k_{12}'\rvert \lvert k_{123} \rvert \lvert k_{123}' \rvert ds d \overline{s} 1_{k_{1} + k_{1}' = 0, k_{2} + k_{2}' = 0, k_{3} + k_{3}' = 0, k_{4} + k_{4} ' = 0} 
\end{align} 
due to \cite[Section 9.2]{GP17}. Within the right hand side of \eqref{249} we estimate 
\begin{align*} 
& \int_{[0,t]^{2}} e^{- \lvert k_{123} \rvert^{2}(t- s + t - \overline{s})} \int_{[0,s]\times [0,\overline{s}]} e^{- \lvert k_{12} \rvert^{2} ( s- \sigma + \overline{s} - \overline{\sigma})} d \sigma d \overline{\sigma} ds d \overline{s}\\
\lesssim& \frac{ e^{-2t \lvert k_{123} \rvert^{2}}}{\lvert k_{12} \rvert^{4}} [ \frac{ e^{\lvert k_{123} \rvert^{2} t} - 1}{\lvert k_{123} \rvert^{2}}]^{2}1_{k_{12}, k_{123} \neq 0} \lesssim \frac{1}{\lvert k_{12} \rvert^{4}} \frac{ t^{\eta}}{\lvert k_{123} \rvert^{4- 2\eta}}1_{k_{12}, k_{123} \neq 0}
\end{align*}
by mean value theorem so that 
\begin{align}\label{251} 
\mathbb{E} [ \lvert \Delta_{q} \RomanIX_{t}^{8,1} \rvert^{2}] \lesssim& t^{\eta} \sum_{k} \sum_{\lvert i-j \rvert \leq 1, \lvert i' - j'\rvert \leq 1} \sum_{k_{1}, k_{2}, k_{3}, k_{4} \neq 0: k_{1234} = k, k_{12} \neq 0, k_{123} \neq 0} \theta(2^{-q} k)^{2}  \\
& \times   \theta(2^{-i} k_{123})  \theta(2^{-j} k_{4}) \theta(2^{-i'} k_{123}) \theta(2^{-j'} k_{4})  \frac{1}{ \lvert k_{12} \rvert^{2} \lvert k_{1} \rvert^{2} \lvert k_{2} \rvert^{2} \lvert k_{3} \rvert^{2} \lvert k_{4} \rvert^{2} \lvert k_{123} \rvert^{2-2\eta}} \nonumber \\
\lesssim& t^{\eta} \sum_{k} \theta(2^{-q} k)^{2} \sum_{q \lesssim i'} 2^{-i' (3- 2 \eta - \epsilon)} 
\lesssim t^{\eta} 2^{q(2\eta + \epsilon)} \sum_{k\neq 0} \theta(2^{-q} k)^{2} \frac{1}{\lvert k \rvert^{3}} 
\lesssim t^{\eta} 2^{q(2\eta + \epsilon)} \nonumber
\end{align} 
by Lemma \ref{Lemma 2.9}, and that $2^{q}  \lesssim 2^{i'}$. By applying (\ref{201}), (\ref{226}), (\ref{244}) and (\ref{251}) to (\ref{194}) we have shown 
so that 
\begin{equation*}
\mathbb{E} [ \lvert \Delta_{q} \pi_{0,\diamond} (u^{
\scalebox{0.16}{

}\epsilon}_{j_{0}}) \rvert^{2}] \lesssim t^{2( \frac{\eta}{3} + \frac{\epsilon}{3})} 2^{q( \frac{10\eta}{3} + \frac{\epsilon}{3})}  
\end{equation*} 
due to (\ref{181}). Similarly to how we deduced (\ref{139}) from (\ref{138}), we can also prove 
\begin{align}\label{253}
& \mathbb{E} [ \lvert \Delta_{q}( \pi_{0, \diamond} (u^{
\scalebox{0.16}{%
}\epsilon_{2}}_{j_{0}})(t_{2})) \rvert^{2}] \lesssim (\epsilon_{1}^{2\gamma} + \epsilon_{2}^{2\gamma})\lvert t_{1} - t_{2} \rvert^{2( \frac{\eta}{3} + \frac{\epsilon}{6})} 2^{q( \frac{10\eta}{3} + \frac{\epsilon}{3})}. \nonumber 
\end{align} 
Recalling again that $B_{p,p}^{- \frac{5\eta}{3} - \epsilon} \hookrightarrow \mathcal{C}^{ -\frac{5\eta}{3} - \epsilon - \frac{3}{p}}$ as in (\ref{140}), we deduce 
\begin{align}\label{254}
& \mathbb{E} [ \lVert \pi_{0,\diamond} (u^{
\scalebox{0.16}{%
}\epsilon_{2}}_{j_{0}})(t_{2}) \rVert_{\mathcal{C}^{- \frac{5\eta}{3} - \epsilon - \frac{3}{p}}}^{p} \lesssim (\epsilon_{1}^{\gamma p} + \epsilon_{2}^{\gamma p}) \lvert t_{1} - t_{2} \rvert^{p( \frac{\eta}{3} + \frac{\epsilon}{6})} \nonumber
\end{align} 
by the Gaussian hypercontractivity theorem \cite[Theorem 3.50]{J97} and (\ref{26}) as we did in (\ref{141}).  If we choose $\eta, \epsilon, p > 0$ such that $\frac{5\eta}{3} + \epsilon + \frac{3}{p} \leq \delta$, we have proven that there exists $v^{
\scalebox{0.17}{%
}}_{16,i_{0}j_{0}}$ as $\epsilon \to 0$ in $L^{p}(\Omega; C([0,T]; \mathcal{C}^{-\delta}))$ as desired in (\ref{113}). 

\subsection{Group 4}
Among 
\begin{equation*}
\pi_{0, \diamond} (\mathcal{P}_{ii_{1}} \partial_{x_{j}}  K_{j}^{u, \epsilon}, u^{\scalebox{0.6}{%
}\epsilon}_{j_{1}}) 
\end{equation*}
from Group 4 of (\ref{114}), we can work on $\pi_{0,\diamond} (\mathcal{P}_{ ii_{1}}  \partial_{x_{j}}  K_{j}^{u, \epsilon}, b^{
\scalebox{0.6}{%
}\epsilon}_{j_{1}})$ and show the existence of $v_{20}^{ii_{1}, jj_{1}} \in C([0,T]; \mathcal{C}^{-\gamma})$ such that $\pi_{0,\diamond} (\mathcal{P}_{ii_{1}}  \partial_{x_{j}}  K_{j}^{u, \epsilon}, b^{
\scalebox{0.6}{%
}\epsilon}_{j_{1}}) \to v_{20}^{ii_{1}, jj_{1}}$ as $\epsilon \to 0$. Because the estimates are similar and straight-forward, we leave this in the Appendix. 

\section{Conclusion of the proof of Theorem \ref{Theorem 1.2}}
With these convergence results, we may now conclude the proof of Theorem \ref{Theorem 1.2}. By a similar argument that we showed already, and in particular (\ref{36a}) and (\ref{39}), we can prove the existence of $\gamma > 0$, $(u^{\scalebox{0.6}{%
}}) \in C([0,T]; \mathcal{C}^{\frac{1}{2} - \delta})^{2}$ such that for all $p > 0$, 
\begin{align}
& \mathbb{E} [ \lVert (u^{\scalebox{0.6}{%
}}) \rVert_{C([0,T]; \mathcal{C}^{-\frac{1}{2} - \frac{\delta}{2}})} > \epsilon \}) 
\lesssim \sum_{k=1}^{\infty} \frac{1}{\epsilon} (\epsilon_{k}^{\gamma}) \lesssim 1 
\end{equation} 
by Chebyshev's inequality and (\ref{262}) is standard. By Borel-Cantelli lemma, this implies that $(u^{\scalebox{0.6}{%
}}_{i})$ in $C([0,T]; \mathcal{C}^{- \frac{1}{2} - \frac{\delta}{2}})$ $\mathbb{P}$-a.s. as $k\to\infty$ and analogous conclusions hold for $(u^{
\scalebox{0.18}{%
}\epsilon_{k}}_{i})$. Hence, we have shown that $\sup_{\epsilon_{k} = 2^{-k}, k \in \mathbb{N}} \overline{C}_{\xi}^{\epsilon_{k}} < \infty$ $\mathbb{P}$-a.s. where $C_{\xi}^{\epsilon}$ is that of (\ref{37}), $(u^{F}, b^{F}) = \lim_{k\to\infty} (u^{F,\epsilon_{k}}, b^{F,\epsilon_{k}})$ in $[0, T_{0}]$, $y= (u,b) = (u^{\scalebox{0.6}{%
}} + b^{F})$ as the solution to (\ref{15}) on $[0, T_{0}]$ where $T_{0}$ is independent of $\epsilon$ and 
\begin{equation}\label{264}
\sup_{t \in [0, T_{0}]} \lVert (u^{\epsilon_{k}}, b^{\epsilon_{k}}) - (u,b) \rVert_{ \mathcal{C}^{-z}} \to 0 
\end{equation} 
as $k\to\infty$ $\mathbb{P}$-a.s. due to Proposition \ref{Proposition 3.1}. 

By identical proof to the case of the NSE on \cite[p. 4497--4498]{ZZ15} (because it does not rely on the precise structure of the equations), it follows that there exist the explosion time $\tau > 0$ and the maximal solution $y$ on $[0,\tau)$ such that 
\begin{equation}
\sup_{t \in [0,\tau)} \lVert y(t) \rVert_{\mathcal{C}^{-z}} = + \infty, 
\end{equation} 
and that if we define  
\begin{equation}\label{260} 
\tau_{L} \triangleq \inf\{t: \lVert y(t) \rVert_{\mathcal{C}^{-z}} \geq L \} \wedge L, \hspace{1mm} \tau_{L}^{\epsilon} \triangleq \inf \{t: \lVert y^{\epsilon}(t) \rVert_{\mathcal{C}^{-z}}  \geq L \} \wedge L, \hspace{1mm} \rho_{L}^{\epsilon} \triangleq \inf\{t: C_{\xi}^{\epsilon} \geq  L \}
\end{equation}  
for $C_{\xi}^{\epsilon}$ in \eqref{37} and $L \geq 0$, then $\tau_{L}$ increase to $L$ as $L \nearrow + \infty$, and for all $L, L_{1}, L_{2} > 0$, 
\begin{equation}\label{265}
\sup_{t\in [0, \rho_{L_{1}}^{\epsilon} \wedge \tau_{L} \wedge \tau_{L_{2}}^{\epsilon} ]} \lVert y^{\epsilon} - y \rVert_{\mathcal{C}^{-z}} \to 0 
\end{equation} 
as $\epsilon \to 0$ $\mathbb{P}$-a.s. Finally, we can compute 
\begin{align}\label{est 266} 
& \mathbb{P} ( \{ \sup_{t \in [0, \tau_{L}]} \lVert y^{\epsilon}- y \rVert_{\mathcal{C}^{-z}} > \epsilon \}) \\
\leq& \mathbb{P} ( \{ \sup_{t \in [0, \tau_{L} \wedge \rho_{L_{1}}^{\epsilon} \wedge \tau_{L_{2}}^{\epsilon}]} \lVert y^{\epsilon}- y \rVert_{\mathcal{C}^{-z}} > \epsilon \}) + \mathbb{P} ( \{ \rho_{L_{1}}^{\epsilon} < \tau_{L} \}) + \mathbb{P} ( \{ \tau_{L_{2}}^{\epsilon} < \tau_{L} \wedge \rho_{L_{1}}^{\epsilon} \}) \nonumber  
\end{align} 
where the right hand side can be shown to vanish as $\epsilon \searrow 0$ due to \eqref{265}. This completes the proof of \eqref{267} and Theorem \ref{Theorem 1.2}.

\section{Appendix}
\subsection{Preliminaries}
The following inequality is standard and was used many times: 
\begin{equation}\label{14}
\sup_{a \in \mathbb{R}} \lvert a \rvert^{r} e^{-a^{2}} \leq c \hspace{3mm} \text{ for all } r \geq 0. 
\end{equation} 
We also list useful lemmas which were used throughout, mostly from \cite{GIP15, ZZ15} (see also \cite[Appendix A]{FG19}). 
\begin{lemma}\label{Lemma 2.4}
\rm{(\cite[Lemma 2.4]{GIP15}, \cite[Lemma 3.3]{ZZ15})} Suppose $\alpha \in (0, 1), \beta, \gamma \in \mathbb{R}$ satisfy $\alpha + \beta + \gamma > 0$ and $\beta + \gamma < 0$. Then for smooth $f, g, h$, the tri-linear operator  $C(f,g,h) \triangleq \pi_{0} ( \pi_{<} (f,g), h) - f\pi_{0} (g,h)$ satisfies 
\begin{equation*}
\lVert C(f,g,h) \rVert_{\mathcal{C}^{\alpha + \beta + \gamma}} \lesssim \lVert f \rVert_{\mathcal{C}^{\alpha}} \lVert g \rVert_{\mathcal{C}^{\beta}} \lVert h \rVert_{\mathcal{C}^{\gamma}}, 
\end{equation*} 
and thus $C$ can be uniquely extended to a bounded tri-linear operator in $L^{3} ( \mathcal{C}^{\alpha}(\mathbb{T}^{3}) \times \mathcal{C}^{\beta}(\mathbb{T}^{3}) \times \mathcal{C}^{\gamma}(\mathbb{T}^{3}), \mathcal{C}^{\alpha + \beta + \gamma}(\mathbb{T}^{3}))$. 
\end{lemma} 
\begin{lemma}\label{Lemma 2.5}
\rm{(\cite[Lemma 3.4]{ZZ15})} Let $\mathcal{P}$ be the Leray projection, $f \in \mathcal{C}^{\alpha}(\mathbb{T}^{3})$, $g \in \mathcal{C}^{\beta}(\mathbb{T}^{3})$ for $\alpha < 1$ and $\beta \in \mathbb{R}$. Then for every $k, l \in \{1,2,3\}$, 
\begin{equation*}
\lVert \mathcal{P}_{kl} \pi_{<} (f,g) - \pi_{<} (f, \mathcal{P}_{kl} g ) \rVert_{\mathcal{C}^{\alpha + \beta}} \lesssim \lVert f \rVert_{\mathcal{C}^{\alpha}} \lVert g \rVert_{\mathcal{C}^{\beta}}. 
\end{equation*} 
\end{lemma} 
\begin{lemma}\label{Lemma 2.6}
\rm{(\cite[Lemma A.7]{GIP15}, \cite[Lemma 3.5]{ZZ15})} Let $P_{t}$ be the heat semigroup on $\mathbb{T}^{N}$. Then for $f \in \mathcal{C}^{\alpha}(\mathbb{T}^{3}), \alpha \in \mathbb{R}$ and $\delta \geq 0$, $P_{t} f$ satisfies 
\begin{equation*}
\lVert P_{t} f \rVert_{\mathcal{C}^{\alpha + \delta}} \lesssim t^{- \frac{\delta}{2}} \lVert f \rVert_{\mathcal{C}^{\alpha}}. 
\end{equation*} 
\end{lemma} 
\begin{lemma}\label{Lemma 2.7}
\rm{(\cite[Lemma 3.6]{ZZ15})} Let $\mathcal{P}$ be the Leray projection and $f \in \mathcal{C}^{\alpha}(\mathbb{T}^{N})$ for $\alpha \in \mathbb{R}$. Then for every $k, l \in \{1,2,3 \}$, 
\begin{equation*}
\lVert \mathcal{P}_{kl} f \rVert_{\mathcal{C}^{\alpha}} \lesssim \lVert f \rVert_{\mathcal{C}^{\alpha}}. 
\end{equation*} 
\end{lemma} 
\begin{lemma}\label{Lemma 2.8}
\rm{(\cite[Lemma 3.11]{ZZ15})} Let $\mathcal{P}$ be the Leray projection. Then for any $\eta \in (0,1), i, j, l \in \{1,2,3\}$ and $t > 0$, 
\begin{equation*}
\lvert e^{ - \lvert k_{12} \rvert^{2} t } k_{12}^{i} \hat{\mathcal{P}}_{jl} (k_{12}) - e^{- \lvert k_{2} \rvert^{2} t} k_{2}^{i} \hat{\mathcal{P}}_{jl} (k_{2}) \rvert \lesssim \lvert k_{1} \rvert^{\eta} \lvert t \rvert^{ - \frac{ (1-\eta)}{2}}. 
\end{equation*} 
\end{lemma} 
\begin{lemma}\label{Lemma 2.9}
\rm{(\cite[Lemma 3.10]{ZZ15})} For any $l, m \in (0, N)$ such that $l + m - N > 0$, 
\begin{equation*}
\sum_{k_{1}, k_{2} \in \mathbb{Z}^{N} \setminus \{0\}: k_{1} + k_{2} = k} \frac{1}{\lvert k_{1} \rvert^{l} \lvert k_{2} \rvert^{m}} \lesssim \frac{1}{\lvert k \rvert^{l + m - N}}.
\end{equation*} 
\end{lemma} 
Finally, recall from \cite[Definition 1.35]{J97} that a Feynman diagram of order $n \geq 0$ and rank $r\geq 0$ is a graph consisting of a set of $n$ vertices and a set of $r$ edges without common endpoints.  The Feynman diagram is complete if $r = \frac{n}{2}$. A Feynman diagram labelled by $n$ random variables $\xi_{1}, \hdots, \xi_{n}$ is a Feynman diagram of order $n$ with vertices $1, \hdots, n$. The value of such a labelled Feynman diagram $\gamma$ with edges $(i_{k}, j_{k}), k = 1, \hdots, r$, and unpaired vertices $\{i: i \in A \}$ is $v(\gamma) \triangleq \prod_{k=1}^{r} \mathbb{E} [ \xi_{i_{k}} \xi_{j_{k}}] \prod_{i\in A} \xi_{i}$. 
\begin{lemma}\label{Lemma 5.7}
\rm{(\cite[Lemma 3.4 and Theorem 3.12]{J97})} 
\begin{enumerate}
\item Wick products are given by 
\begin{equation*}
: \xi_{1}\hdots \xi_{n}: = \sum_{\gamma} (-1)^{r(\gamma)} v(\gamma), 
\end{equation*} 
where summation runs over all Feynman diagrams $\gamma$ labeled by $\{ \xi_{i}\}_{i=1}^{n}$. 
\item Let $Y_{i} = : \xi_{i 1} \hdots \xi_{i l_{i}}:$, where $\{ \xi_{ij}\}_{1 \leq i \leq k, 1 \leq j \leq l_{i}}$ are (real or complex) centered jointly normal variables, with $k \geq 0$ and $l_{1}, \hdots, l_{k} \geq 0$. Then 
\begin{equation*}
\mathbb{E} [ Y_{1} \hdots Y_{k}] = \sum_{\gamma} v(\gamma)
\end{equation*} 
where summation runs over all complete Feynman diagrams $\gamma$ labeled by $\{\xi_{ij}\}_{i,j}$ such that no edge joins two variables with $\xi_{i_{1} j_{1}}$ and $\xi_{i_{2} j_{2}}$ with $i_{1} = i_{2}$. 
\end{enumerate} 
\end{lemma}  

\subsection{Details of Renormalizations for Group 2}
Due to (\ref{17}), (\ref{24c}), (\ref{24b}) and relying on the representation of $u^{
\scalebox{0.18}{\begin{tikzpicture}
\draw[black, thick] (-0.7,0.9) -- (0,0);
\draw[black, thick] (0.7,0.9) -- (0,0);
\draw[snake=zigzag](0,0) -- (0,-0.9);
\filldraw[green] (-0.7,0.9) circle (6pt); 
\filldraw[green] (0.7,0.9) circle (6pt); 
\end{tikzpicture}
}\epsilon}_{j}(t)$ in (\ref{118}), we may compute 
\begin{align}\label{154}
&b^{
\scalebox{0.18}{\begin{tikzpicture}
\draw[black, thick] (-0.7,0.9) -- (0,0);
\draw[black, thick] (0.7,0.9) -- (0,0);
\draw[snake=zigzag](0,0) -- (0,-0.9);
\filldraw[violet] (-0.7,0.9) circle (6pt); 
\filldraw[violet] (0.7,0.9) circle (6pt); 
\end{tikzpicture}
}\epsilon}_{i} u^{
\scalebox{0.18}{\begin{tikzpicture}
\draw[black, thick] (-0.7,0.9) -- (0,0);
\draw[black, thick] (0.7,0.9) -- (0,0);
\draw[snake=zigzag](0,0) -- (0,-0.9);
\filldraw[green] (-0.7,0.9) circle (6pt); 
\filldraw[green] (0.7,0.9) circle (6pt); 
\end{tikzpicture}
}\epsilon}_{j} (t) = \frac{1}{4 (2\pi)^{\frac{9}{2}}} \sum_{i_{1}, i_{2}, j_{1}, j_{2} =1}^{3} \sum_{k} \sum_{k_{1}, k_{2}, k_{3}, k_{4}: k_{1234} = k}\\
& \times \int_{[0,t]^{2}} ds d \overline{s} e^{- \lvert k_{12} \rvert^{2}(t-s)} e^{- \lvert k_{34} \rvert^{2}(t- \overline{s})} \hat{\mathcal{P}}_{ ii_{1}}(k_{12}) \hat{\mathcal{P}}_{jj_{1}}(k_{34}) ik_{12}^{i_{2}} ik_{34}^{j_{2}} \nonumber\\
& \times [\hat{X}_{s,i_{1}}^{b, \epsilon}(k_{1}) \hat{X}_{s,i_{2}}^{u, \epsilon}(k_{2}) \hat{X}_{\overline{s}, j_{1}}^{u, \epsilon}(k_{3}) \hat{X}_{\overline{s}, j_{2}}^{u, \epsilon}(k_{4})  - \hat{X}_{s,i_{1}}^{u, \epsilon}(k_{1}) \hat{X}_{s,i_{2}}^{b, \epsilon}(k_{2}) \hat{X}_{\overline{s}, j_{1}}^{u, \epsilon}(k_{3}) \hat{X}_{\overline{s}, j_{2}}^{u, \epsilon}(k_{4})\nonumber\\ 
& - \hat{X}_{s,i_{1}}^{b, \epsilon}(k_{1}) \hat{X}_{s,i_{2}}^{u, \epsilon}(k_{2}) \hat{X}_{\overline{s}, j_{1}}^{b, \epsilon}(k_{3}) \hat{X}_{\overline{s}, j_{2}}^{b, \epsilon}(k_{4})+ \hat{X}_{s,i_{1}}^{u, \epsilon}(k_{1}) \hat{X}_{s,i_{2}}^{b, \epsilon}(k_{2}) \hat{X}_{\overline{s}, j_{1}}^{b, \epsilon}(k_{3}) \hat{X}_{\overline{s}, j_{2}}^{b, \epsilon}(k_{4})] e_{k}.     \nonumber      
\end{align}
We can apply Lemma \ref{Lemma 5.7} (1) with ``$\xi_{1}\xi_{2}\xi_{3}\xi_{4}:$'' $= \hat{X}_{s,i_{1}}^{b, \epsilon}(k_{1}) \hat{X}_{s,i_{2}}^{u, \epsilon}(k_{2}) \hat{X}_{\overline{s}, j_{1}}^{u, \epsilon}(k_{3}) \hat{X}_{\overline{s}, j_{2}}^{u, \epsilon}(k_{4})$ to write it as $\sum_{\gamma} (-1)^{r(\gamma)} v(\gamma)$ with the sum over all Feynman diagrams $\gamma$ labeled by $\{ \hat{X}_{s,i_{1}}^{b, \epsilon}(k_{1}), \hat{X}_{s,i_{2}}^{u, \epsilon}(k_{2}), \hat{X}_{\overline{s}, j_{1}}^{u, \epsilon}(k_{3})$, $\hat{X}_{\overline{s}, j_{2}}^{u, \epsilon}(k_{4})\}$, and split to groups of fourth, second, and zeroth Wiener chaos (see \cite[Example 2.2]{Y20}). We repeat for $ \hat{X}_{s,i_{1}}^{u, \epsilon}(k_{1}) \hat{X}_{s,i_{2}}^{b, \epsilon}(k_{2}) \hat{X}_{\overline{s}, j_{1}}^{u, \epsilon}(k_{3}) \hat{X}_{\overline{s}, j_{2}}^{u, \epsilon}(k_{4})$, $\hat{X}_{s,i_{1}}^{b, \epsilon}(k_{1}) \hat{X}_{s,i_{2}}^{u, \epsilon}(k_{2}) \hat{X}_{\overline{s}, j_{1}}^{b, \epsilon}(k_{3}) \hat{X}_{\overline{s}, j_{2}}^{b, \epsilon}(k_{4})$, and $\hat{X}_{s,i_{1}}^{u, \epsilon}(k_{1}) \hat{X}_{s,i_{2}}^{b, \epsilon}(k_{2}) \hat{X}_{\overline{s}, j_{1}}^{b, \epsilon}(k_{3}) \hat{X}_{\overline{s}, j_{2}}^{b, \epsilon}(k_{4})$ to write 
\begin{equation}\label{155}
b^{
\scalebox{0.18}{\begin{tikzpicture}
\draw[black, thick] (-0.7,0.9) -- (0,0);
\draw[black, thick] (0.7,0.9) -- (0,0);
\draw[snake=zigzag](0,0) -- (0,-0.9);
\filldraw[violet] (-0.7,0.9) circle (6pt); 
\filldraw[violet] (0.7,0.9) circle (6pt); 
\end{tikzpicture}
}\epsilon}_{i} u^{
\scalebox{0.18}{\begin{tikzpicture}
\draw[black, thick] (-0.7,0.9) -- (0,0);
\draw[black, thick] (0.7,0.9) -- (0,0);
\draw[snake=zigzag](0,0) -- (0,-0.9);
\filldraw[green] (-0.7,0.9) circle (6pt); 
\filldraw[green] (0.7,0.9) circle (6pt); 
\end{tikzpicture}
}\epsilon}_{j} (t) = \underbrace{\RomanVI_{t}^{1}}_{\text{4th chaos}} + \underbrace{\RomanVI_{t}^{2}}_{\text{2nd chaos}} + \underbrace{\RomanVI_{t}^{3}}_{\text{0th chaos}}
\end{equation} 
where 
\begin{align}\label{156}
&\RomanVI_{t}^{1} \triangleq \frac{1}{4(2\pi)^{\frac{9}{2}}} \sum_{i_{1}, i_{2}, j_{1}, j_{2} =1}^{3} \sum_{k} \sum_{k_{1}, k_{2}, k_{3}, k_{4}: k_{1234} = k} \\
& \times \int_{[0,t]^{2}} ds d \overline{s} e_{k} e^{- \lvert k_{12} \rvert^{2} ( t -s ) - \lvert k_{34} \rvert^{2}(t- \overline{s})} \hat{\mathcal{P}}_{ ii_{1}}(k_{12}) \hat{\mathcal{P}}_{ jj_{1}}(k_{34}) i k_{12}^{i_{2}} i k_{34}^{j_{2}}\nonumber\\
&\times [ : \hat{X}_{s,i_{1}}^{b, \epsilon}(k_{1}) \hat{X}_{s,i_{2}}^{u, \epsilon}(k_{2}) \hat{X}_{\overline{s}, j_{1}}^{u, \epsilon}(k_{3}) \hat{X}_{\overline{s}, j_{2}}^{u, \epsilon}(k_{4}):  -  : \hat{X}_{s,i_{1}}^{u, \epsilon}(k_{1}) \hat{X}_{s,i_{2}}^{b, \epsilon}(k_{2}) \hat{X}_{\overline{s}, j_{1}}^{u, \epsilon}(k_{3}) \hat{X}_{\overline{s}, j_{2}}^{u, \epsilon}(k_{4}):\nonumber\\
& -  : \hat{X}_{s,i_{1}}^{b, \epsilon}(k_{1}) \hat{X}_{s,i_{2}}^{u, \epsilon}(k_{2}) \hat{X}_{\overline{s}, j_{1}}^{b, \epsilon}(k_{3}) \hat{X}_{\overline{s}, j_{2}}^{b, \epsilon}(k_{4}):+  : \hat{X}_{s,i_{1}}^{u, \epsilon}(k_{1}) \hat{X}_{s,i_{2}}^{b, \epsilon}(k_{2}) \hat{X}_{\overline{s}, j_{1}}^{b, \epsilon}(k_{3}) \hat{X}_{\overline{s}, j_{2}}^{b, \epsilon}(k_{4}):], \nonumber 
\end{align} 
$\RomanVI_{t}^{2}$ consists of 16 terms with 
\begin{align}
&\RomanVI_{t}^{2,\star} \triangleq \frac{1}{4(2\pi)^{\frac{9}{2}}} \sum_{i_{1}, i_{2}, j_{1}, j_{2} = 1}^{3} \sum_{k} \sum_{k_{2}, k_{4}: k_{24} = k, k_{1} \neq 0} \int_{[0,t]^{2}} \label{266} \\
& \times e^{- \lvert k_{12} \rvert^{2} (t-s) - \lvert k_{4} - k_{1} \rvert^{2}(t- \overline{s})}\hat{\mathcal{P}}_{ii_{1}}(k_{12}) \hat{\mathcal{P}}_{jj_{1}} (k_{4} - k_{1}) i k_{12}^{i_{2}} i(k_{4}^{j_{2}} - k_{1}^{j_{2}})\nonumber\\
& \times \sum_{j_{5} =1}^{3} \frac{e^{- \lvert k_{1} \rvert^{2} \lvert s - \overline{s} \rvert} f(\epsilon k_{1})^{2}}{2 \lvert k_{1} \rvert^{2}} \hat{\mathcal{P}}_{ i_{4} j_{5}}(k_{1}) \hat{\mathcal{P}}_{j_{4} j_{5}}(k_{1}): \hat{X}_{s,i_{3}}^{b, \epsilon}(k_{2}) \hat{X}_{\overline{s}, j_{3}}^{b, \epsilon}(k_{4}): ds d \overline{s} e_{k} 1_{i_{3} = i_{2}, i_{4} = i_{1}, j_{3} = j_{2}, j_{4} = j_{1}} \nonumber 
\end{align}
being a representative, and 
\begin{align}\label{158}
\RomanVI_{t}^{3} \triangleq& \frac{1}{4(2\pi)^{\frac{9}{2}}} \sum_{i_{1}, i_{2}, j_{1}, j_{2} =1}^{3} \sum_{k} \sum_{k_{1}, k_{2} \neq 0} \int_{[0,t]^{2}} e^{- \lvert k_{12} \rvert^{2}(2t-s - \overline{s})} \hat{\mathcal{P}}_{ ii_{1}}(k_{12}) \hat{\mathcal{P}}_{jj_{1}}(k_{12}) \nonumber\\
& \times k_{12}^{i_{2}} k_{12}^{j_{2}}  \frac{ f( \epsilon k_{1})^{2} f(\epsilon k_{2})^{2} e^{- ( \lvert k_{1} \rvert^{2} + \lvert k_{2} \rvert^{2}) \lvert s - \overline{s} \rvert}}{4 \lvert k_{1} \rvert^{2} \lvert k_{2} \rvert^{2}} ds d \overline{s} \\
& \times \sum_{j_{3}, j_{4} =1}^{3} [ \hat{\mathcal{P}}_{ i_{2} j_{4}}(k_{2}) \hat{\mathcal{P}}_{ j_{1} j_{4}}(k_{2}) \hat{\mathcal{P}}_{i_{1} j_{3}}(k_{1}) \hat{\mathcal{P}}_{j_{2} j_{3}}(k_{1})  + \hat{\mathcal{P}}_{ i_{2} j_{4}}(k_{2}) \hat{\mathcal{P}}_{ j_{2} j_{4}}(k_{2}) \hat{\mathcal{P}}_{i_{1} j_{3}}(k_{1}) \hat{\mathcal{P}}_{j_{1} j_{3}}(k_{1})\nonumber\\
& - \hat{\mathcal{P}}_{ i_{2} j_{4}}(k_{2}) \hat{\mathcal{P}}_{ j_{1} j_{4}}(k_{2}) \hat{\mathcal{P}}_{i_{1} j_{3}}(k_{1}) \hat{\mathcal{P}}_{j_{2} j_{3}}(k_{1}) - \hat{\mathcal{P}}_{ i_{2} j_{4}}(k_{2}) \hat{\mathcal{P}}_{ j_{2} j_{4}}(k_{2}) \hat{\mathcal{P}}_{i_{1} j_{3}}(k_{1}) \hat{\mathcal{P}}_{j_{1} j_{3}}(k_{1})\nonumber\\
& - \hat{\mathcal{P}}_{ i_{2} j_{4}}(k_{2}) \hat{\mathcal{P}}_{ j_{1} j_{4}}(k_{2}) \hat{\mathcal{P}}_{i_{1} j_{3}}(k_{1}) \hat{\mathcal{P}}_{j_{2} j_{3}}(k_{1})  - \hat{\mathcal{P}}_{ i_{2} j_{4}}(k_{2}) \hat{\mathcal{P}}_{ j_{2} j_{4}}(k_{2}) \hat{\mathcal{P}}_{i_{1} j_{3}}(k_{1}) \hat{\mathcal{P}}_{j_{1} j_{3}}(k_{1})\nonumber\\
& + \hat{\mathcal{P}}_{ i_{2} j_{4}}(k_{2}) \hat{\mathcal{P}}_{ j_{1} j_{4}}(k_{2}) \hat{\mathcal{P}}_{i_{1} j_{3}}(k_{1}) \hat{\mathcal{P}}_{j_{2} j_{3}}(k_{1}) + \hat{\mathcal{P}}_{ i_{2} j_{4}}(k_{2}) \hat{\mathcal{P}}_{ j_{2} j_{4}}(k_{2}) \hat{\mathcal{P}}_{i_{1} j_{3}}(k_{1}) \hat{\mathcal{P}}_{j_{1} j_{3}}(k_{1})]. \nonumber        
\end{align} 
Finally, from (\ref{158}) we define 
\begin{equation}\label{162} 
\RomanVI_{t}^{3} \triangleq C_{2,3}^{\epsilon, ij}. 
\end{equation} 

\subsubsection{Terms in the second chaos} 
In order to estimate $\mathbb{E} [ \lvert \Delta_{q} \RomanVI_{t}^{2} \rvert^{2}]$, we consider only $\RomanVI_{t}^{2,\star}$ in \eqref{266} as others are similarly estimated. We use $\mathbb{E} [ : \xi_{11} \xi_{12}: : \xi_{21} \xi_{22}:]$ $=$ $\mathbb{E} [ \xi_{11} \xi_{21}] \mathbb{E} [ \xi_{12} \xi_{22}] + \mathbb{E} [ \xi_{11} \xi_{22}]\mathbb{E}[\xi_{12} \xi_{21}]$ (see \cite{J97}) to compute $ \mathbb{E} [ : \hat{X}_{s,i_{2}}^{b, \epsilon}(k_{2}) \hat{X}_{\overline{s}, j_{2}}^{b, \epsilon}(k_{4}): : \hat{X}_{\sigma, i_{2}'}^{b, \epsilon} (k_{2}') \hat{X}_{\overline{\sigma}, j_{2}'}^{b, \epsilon} (k_{4}') : ]$ and deduce 
\begin{align}\label{164}
\mathbb{E} [ \lvert \Delta_{q} \RomanVII_{t}^{15} \rvert^{2}] \lesssim& \sum_{k} \theta (2^{-q} k)^{2} \sum_{k_{2}, k_{4} \neq 0: k_{24} = k, k_{1} \neq 0, k_{2}', k_{4}': k_{24}' = k, k_{1}' \neq 0}\nonumber\\
& \times \int_{[0,t]^{4}} e^{- \lvert k_{12} \rvert^{2} (t-s) - \lvert k_{4} - k_{1} \rvert^{2} (t- \overline{s})} e^{- \lvert k_{12}' \rvert^{2}(t- \sigma) - \lvert k_{4}' - k_{1}' \rvert^{2}( t-  \overline{\sigma})} \nonumber\\
& \times \lvert k_{12} (k_{4} - k_{1}) \rvert \lvert k_{12}' (k_{4}' - k_{1}') \rvert  \frac{ e^{- \lvert k_{1} \rvert^{2} \lvert s - \overline{s} \rvert}}{\lvert k_{1} \rvert^{2}} \frac{ e^{- \lvert k_{1}' \rvert^{2} \lvert \sigma - \overline{\sigma} \rvert}}{\lvert k_{1}' \rvert^{2}} \frac{1}{\lvert k_{2} \rvert^{2} \lvert k_{4} \rvert^{2}}\nonumber\\
& \times 1_{k_{2} + k_{2}' = 0, k_{4} + k_{4}' = 0} ds d \overline{s} d \sigma d \overline{\sigma}
\end{align} 
where we denoted $k_{12}' \triangleq k_{1}' + k_{2}'$. Considering the characteristic function $1_{k_{2} + k_{2}' = 0, k_{4} + k_{4}' = 0}$, we see that it may be further estimated as 
\begin{align}\label{165}
& \sum_{k} \theta(2^{-q} k)^{2} \sum_{k_{2}, k_{4} \neq 0: k_{24} = k, k_{1}, k_{1}' \neq 0}  \int_{[0,t]^{4}} e^{- \lvert k_{12} \rvert^{2}(t-s) - \lvert k_{4} - k_{1} \rvert^{2} (t - \overline{s})} e^{- \lvert k_{1}' - k_{2} \rvert^{2} (t - \sigma) - \lvert k_{1}' + k_{4} \rvert^{2}(t- \overline{\sigma})} \nonumber\\
&\times \lvert k_{12} (k_{4} - k_{1}) \rvert \lvert (k_{1}' - k_{2})(k_{1}' + k_{4}) \rvert \frac{1}{\lvert k_{1} \rvert^{2}} \frac{1}{\lvert k_{1} ' \rvert^{2}} \frac{1}{ \lvert k_{2} \rvert^{2} \lvert k_{4} \rvert^{2}} ds d \overline{s} d \sigma d \overline{\sigma} \nonumber\\
\lesssim& t^{\epsilon} \sum_{k} \sum_{k_{2}, k_{4} \neq 0: k_{24} = k, k_{1}, k_{3}\neq 0} \theta(2^{-q} k)^{2} \prod_{j=1}^{4} \frac{1}{ \lvert k_{j} \rvert^{2}} \frac{1}{\lvert k_{4} - k_{1} \rvert^{2-\epsilon} \lvert k_{4} - k_{3} \rvert^{2-\epsilon}} \nonumber\\
\lesssim& t^{\epsilon} 2^{2q \epsilon} \sum_{k\neq 0} \theta (2^{-q} k)^{2} \frac{1}{\lvert k \rvert^{3}} \lesssim t^{\epsilon} 2^{2q \epsilon} 
\end{align} 
by a change of variable $k_{1}'$ with $-k_{3}$, mean value theorem, and Lemma \ref{Lemma 2.9}. 

\subsubsection{Terms in the fourth chaos}
We wish to estimate 
\begin{equation}\label{166}
\mathbb{E} [ \lvert \Delta_{q}  \RomanVI_{t}^{1} \rvert^{2}] = \mathbb{E} [ \lvert \sum_{k} \theta(2^{-q} k) \widehat{\RomanVI}_{t}^{1} (k) e_{k} \rvert^{2}]
\end{equation} 
where $\RomanVI_{t}^{1}$ is that of (\ref{156}) of which it suffices to estimate for example a mix term such as second and third terms multiplied; i.e. 
\begin{align}\label{167}
& \mathbb{E} [ \lvert \sum_{k} \theta(2^{-q} k) \sum_{i_{1}, i_{2}, j_{1}, j_{2} =1}^{3}  \sum_{k_{1}, k_{2}, k_{3}, k_{4}: k_{1234} = k} \int_{[0,t]^{2}} e^{- \lvert k_{12} \rvert^{2} (t-s) - \lvert k_{34} \rvert^{2} (t - \overline{s})}\\
& \times : \hat{X}_{s,i_{1}}^{u, \epsilon} (k_{1}) \hat{X}_{s,i_{2}}^{b, \epsilon}(k_{2}) \hat{X}_{\overline{s}, j_{1}}^{u, \epsilon}(k_{3}) \hat{X}_{\overline{s}, j_{2}}^{u, \epsilon} (k_{4}): ds d \overline{s} e_{k} \hat{\mathcal{P}}_{ ii_{1}} (k_{12}) \hat{\mathcal{P}}_{ jj_{1}} (k_{34}) ik_{12} ^{i_{2}} i k_{34}^{j_{2}} \rvert \nonumber\\
& \times \lvert \sum_{k'} \theta (2^{-q} k') \sum_{i_{1}', i_{2}', j_{1}', j_{2}' = 1}^{3} \sum_{k_{1}',k_{2}', k_{3}', k_{4}', k_{1234}' = k'} \int_{[0,t]^{2}} e^{- \lvert k_{12} ' \rvert^{2} (t- \sigma) - \lvert k_{34} ' \rvert^{2} (t - \overline{\sigma})}\nonumber\\
& \times : \hat{X}_{\sigma, i_{1}'}^{b, \epsilon} (k_{1}') \hat{X}_{\sigma, i_{2}'}^{u, \epsilon} (k_{2}') \hat{X}_{\overline{\sigma}, j_{1}'}^{b, \epsilon} (k_{3}') \hat{X}_{\overline{\sigma}, j_{2}'}^{b, \epsilon} (k_{4}') : d \sigma d \overline{\sigma} e_{k'}\hat{\mathcal{P}}_{ i' i_{1}'} (k_{12}') \hat{\mathcal{P}}_{j' j_{1}'} (k_{34}') i (k_{12}') ^{ i_{2}'} i (k_{34}')^{j_{2}'} \rvert ].\nonumber
\end{align} 
We can apply Lemma \ref{Lemma 5.7} (2) with ``$Y_{1}$'' $=  : \hat{X}_{s,i_{1}}^{u, \epsilon}(k_{1}) \hat{X}_{s, i_{2}}^{b, \epsilon}(k_{2}) \hat{X}_{\overline{s}, j_{1}}^{u, \epsilon}(k_{3}) \hat{X}_{\overline{s}, j_{2}}^{u, \epsilon}(k_{4}):$ and ``$Y_{2}$'' $= : \hat{X}_{\sigma, i_{1}'}^{b, \epsilon}(k_{1}') \hat{X}_{\sigma, i_{2}'}^{u, \epsilon} (k_{2}') \hat{X}_{\overline{\sigma}, j_{1}'}^{b, \epsilon} (k_{3}') \hat{X}_{\overline{\sigma}, j_{2}'}^{b, \epsilon} (k_{4}') : $ to compute $\mathbb{E}[Y_{1}Y_{2}] = \sum_{\gamma} v(\gamma)$ explicitly (see \cite[Example 2.2]{Y20} for details) and see that it consists of 24 terms, one representative being 
\begin{align}
\RomanVI_{t}^{1,1} \triangleq& 1_{k_{1} + k_{1} ' = 0, k_{2} + k_{2} ' = 0, k_{3} + k_{3} ' = 0, k_{4} + k_{4} ' = 0} \sum_{i_{3}, i_{4}, i_{5}, i_{6} = 1}^{3} 1_{k_{1}, k_{2}, k_{3}, k_{4} \neq 0} \label{169}\\
& \times  \frac{e^{ - \lvert k_{1} \rvert^{2} \lvert s- \sigma \rvert} f( \epsilon k_{1})^{2}}{2 \lvert k_{1} \rvert^{2}} \hat{\mathcal{P}}_{ i_{1} i_{3}}(k_{1}) \hat{\mathcal{P}}_{ i_{1}' i_{3}}(k_{1}) \frac{ e^{- \lvert k_{2} \rvert^{2} \lvert s- \sigma \rvert} f(\epsilon k_{2})^{2}}{2 \lvert k_{2} \rvert^{2}} \hat{\mathcal{P}}_{ i_{2} i_{4}}(k_{2}) \hat{\mathcal{P}}_{ i_{2}' i_{4}}(k_{2}) \nonumber\\
& \times \frac{ e^{ - \lvert k_{3} \rvert^{2} \lvert \overline{s} - \overline{\sigma} \rvert} f(\epsilon k_{3})^{2}}{2 \lvert k_{3} \rvert^{2}} \hat{\mathcal{P}}_{ j_{1} i_{5}}(k_{3}) \hat{\mathcal{P}}_{ j_{1}' i_{5}}(k_{3}) \frac{ e^{- \lvert k_{4} \rvert^{2} \lvert \overline{s} - \overline{\sigma} \rvert} f( \epsilon k_{4})^{2}}{2 \lvert k_{4} \rvert^{2}} \hat{\mathcal{P}}_{ j_{2} i_{6}}(k_{4}) \hat{\mathcal{P}}_{ j_{2}' i_{6}}(k_{4}) \nonumber 
\end{align}
where $k = k_{1234} =  -k_{1234}' = -k'$ so that we can bound it by 
\begin{align*}
 \sum_{k} \theta (2^{-q} k)^{2} \sum_{k_{1}, k_{2}, k_{3}, k_{4} \neq 0: k_{1234} = k} \int_{[0,t]^{4}} & e^{- \lvert k_{12} \rvert^{2} ( 2t - s - \sigma) - \lvert k_{34} \rvert^{2} ( 2t - \overline{s} - \overline{\sigma})} \\
 & \hspace{20mm} \times \frac{ \lvert k_{12} \rvert^{2} \lvert k_{34} \rvert^{2}}{ \lvert k_{1} \rvert^{2} \lvert k_{2} \rvert^{2} \lvert k_{3} \rvert^{2} \lvert k_{4} \rvert^{2}} ds d \overline{s} d \sigma d \overline{\sigma}. 
\end{align*} 
By relying on \cite[Section 9.2]{GP17}, this estimate leads us to 
\begin{align}\label{175}
\mathbb{E} [ \lvert \Delta_{q} \RomanVI_{t}^{1} &\rvert^{2}] \lesssim \sum_{k} \theta (2^{-q} k)^{2} \sum_{k_{1}, k_{2}, k_{3}, k_{4} \neq 0: k_{1234} = k}\\
& \times \int_{[0,t]^{4}} [ e^{ - \lvert k_{12} \rvert^{2} ( 2t - s - \sigma) - \lvert k_{34} \rvert^{2} (2t - \overline{s} - \overline{\sigma})}\frac{ \lvert k_{12} \rvert^{2} \lvert k_{34} \rvert^{2} }{ \lvert k_{1} \rvert^{2} \lvert k_{2} \rvert^{2} \lvert k_{3} \rvert^{2} \lvert k_{4} \rvert^{2}} \nonumber\\
& \hspace{8mm}+ e^{- \lvert k_{12} \rvert^{2} (t-s) - \lvert k_{34} \rvert^{2} (t- \overline{s}) - \lvert k_{14} \rvert^{2} ( t - \overline{\sigma}) - \lvert k_{23} \rvert^{2} ( t- \sigma)} \frac{ \lvert k_{12} \rvert \lvert k_{34} \rvert \lvert k_{14} \rvert \lvert k_{23} \rvert}{ \lvert k_{1} \rvert^{2} \lvert k_{2} \rvert^{2} \lvert k_{3} \rvert^{2} \lvert k_{4} \rvert^{2}}] ds  d \overline{s}  d \sigma d \overline{\sigma}.\nonumber
\end{align} 
Within (\ref{175}) we may further estimate for $k_{1}, k_{2}, k_{3}, k_{4}\neq 0$, 
\begin{align}\label{176}
& \int_{[0,t]^{4}} e^{ - \lvert k_{12} \rvert^{2}(2t - s - \sigma) - \lvert k_{34} \rvert^{2} (2t - \overline{s} - \overline{\sigma})}\frac{  \lvert k_{12} \rvert^{2} \lvert k_{34} \rvert^{2}  }{ \lvert k_{1} \rvert^{2} \lvert k_{2} \rvert^{2} \lvert k_{3} \rvert^{2} \lvert k_{4} \rvert^{2}} ds d \overline{s} d \sigma d \overline{\sigma} \nonumber\\
\lesssim& 1_{k_{12}, k_{34} \neq 0}\frac{ t^{\epsilon}}{ \lvert k_{1} \rvert^{2} \lvert k_{2} \rvert^{2} \lvert k_{3} \rvert^{2} \lvert k_{4} \rvert^{2} \lvert k_{12} \rvert^{2- \epsilon} \lvert k_{34} \rvert^{2- \epsilon}}
\end{align} 
where we used mean value theorem, while for $k_{1}, k_{2}, k_{3}, k_{4}\neq 0$, 
\begin{align}\label{177}
& \int_{[0,t]^{4}} e^{ - \lvert k_{12} \rvert^{2} (t-s) - \lvert k_{34} \rvert^{2}(t -\overline{s}) - \lvert k_{14} \rvert^{2}(t - \overline{\sigma}) - \lvert k_{23} \rvert^{2}(t - \sigma)}  \left( \frac{ \lvert k_{12} \rvert \lvert k_{34} \rvert \lvert k_{14} \rvert \lvert k_{23} \rvert}{ \lvert k_{1} \rvert^{2} \lvert k_{2} \rvert^{2} \lvert k_{3} \rvert^{2} \lvert k_{4} \rvert^{2}} \right) ds d \overline{s} d \sigma d \overline{\sigma}\nonumber\\
\lesssim& 1_{k_{12}, k_{34}, k_{14}, k_{23} \neq 0} \frac{t^{\epsilon}}{ \lvert k_{1} \rvert^{2} \lvert k_{2} \rvert^{2} \lvert k_{3} \rvert^{2} \lvert k_{4} \rvert^{2}} \frac{1}{ \lvert k_{12} \rvert^{1- \frac{\epsilon}{2}} \lvert k_{34} \rvert^{1- \frac{\epsilon}{2}} \lvert k_{14} \rvert^{1- \frac{\epsilon}{2}} \lvert k_{23} \rvert^{1- \frac{\epsilon}{2}}}
\end{align} 
by mean value theorem. Therefore, applying (\ref{176}) and (\ref{177}) to (\ref{175}) gives 
\begin{equation}\label{178}
\mathbb{E} [ \lvert \Delta_{q} \RomanVI_{t}^{1} \rvert^{2}] 
\lesssim t^{\epsilon} \sum_{k} \theta (2^{-q} k)^{2} [ \RomanVII^{1} + \sqrt{ \RomanVII^{1}} \sqrt{\RomanVII^{2}}]
\end{equation} 
where 
\begin{align*} 
& \RomanVII^{1} \triangleq \sum_{k_{1}, k_{2}, k_{3}, k_{4} \neq 0: k_{1234} = k} \frac{1_{k_{12}, k_{34} \neq 0}}{ \prod_{j=1}^{4} \lvert k_{j} \rvert^{2} \lvert k_{12} \rvert^{2-\epsilon} \lvert k_{34} \rvert^{2-\epsilon} },\\
& \RomanVII^{2} \triangleq \sum_{k_{1}, k_{2}, k_{3}, k_{4} \neq 0: k_{1234} = k} \frac{1_{k_{14} \neq 0, k_{23} \neq 0}}{ \prod_{j=1}^{4} \lvert k_{j} \rvert^{2} \lvert k_{14} \rvert^{2-\epsilon} \lvert k_{23} \rvert^{2-\epsilon}},
\end{align*}
due to H$\ddot{\mathrm{o}}$lder's inequality. We may estimate 
\begin{align*}
t^{\epsilon} \sum_{k} \theta(2^{-q} k)^{2} \sqrt{\RomanVII^{1}}\sqrt{\RomanVII^{2}} 
\lesssim 2^{2q\epsilon} t^{\epsilon} \sum_{k} \theta(2^{-q} k)^{2} \frac{1}{\lvert k \rvert^{2\epsilon}} \left( \frac{1}{\lvert k \rvert^{12 - 2 \epsilon - 9}}\right)^{\frac{1}{2}}\left( \frac{1}{\lvert k \rvert^{12 - 2 \epsilon - 9}}\right)^{\frac{1}{2}} \lesssim  2^{2q\epsilon} t^{\epsilon} 
\end{align*} 
by Lemma \ref{Lemma 2.9}. We may apply identical estimates to $ \sum_{k} \theta (2^{-q} k)^{2} \RomanVII^{1}$ in (\ref{178}) to deduce 
\begin{equation}\label{180}
\mathbb{E} [ \lvert \Delta_{q} \RomanVI_{t}^{1} \rvert^{2}] \lesssim t^{\epsilon} 2^{2q \epsilon}.
\end{equation} 
Similarly to how we deduced (\ref{139}) from (\ref{138}), we can obtain an analogous Lipschitz bound on  
\begin{align*} 
 \mathbb{E} [ \lvert \Delta_{q} ( b^{
\scalebox{0.18}{\begin{tikzpicture}
\draw[black, thick] (-0.7,0.9) -- (0,0);
\draw[black, thick] (0.7,0.9) -- (0,0);
\draw[snake=zigzag](0,0) -- (0,-0.9);
\filldraw[violet] (-0.7,0.9) circle (6pt); 
\filldraw[violet] (0.7,0.9) circle (6pt); 
\end{tikzpicture}
}\epsilon_{1}}_{i} \diamond u^{
\scalebox{0.18}{\begin{tikzpicture}
\draw[black, thick] (-0.7,0.9) -- (0,0);
\draw[black, thick] (0.7,0.9) -- (0,0);
\draw[snake=zigzag](0,0) -- (0,-0.9);
\filldraw[green] (-0.7,0.9) circle (6pt); 
\filldraw[green] (0.7,0.9) circle (6pt); 
\end{tikzpicture}
}\epsilon}_{j} (t_{1}) - b^{
\scalebox{0.18}{\begin{tikzpicture}
\draw[black, thick] (-0.7,0.9) -- (0,0);
\draw[black, thick] (0.7,0.9) -- (0,0);
\draw[snake=zigzag](0,0) -- (0,-0.9);
\filldraw[violet] (-0.7,0.9) circle (6pt); 
\filldraw[violet] (0.7,0.9) circle (6pt); 
\end{tikzpicture}
}\epsilon_{1}}_{i} \diamond u^{
\scalebox{0.18}{\begin{tikzpicture}
\draw[black, thick] (-0.7,0.9) -- (0,0);
\draw[black, thick] (0.7,0.9) -- (0,0);
\draw[snake=zigzag](0,0) -- (0,-0.9);
\filldraw[green] (-0.7,0.9) circle (6pt); 
\filldraw[green] (0.7,0.9) circle (6pt); 
\end{tikzpicture}
}\epsilon_{1}}_{j} (t_{2}) - b^{
\scalebox{0.18}{\begin{tikzpicture}
\draw[black, thick] (-0.7,0.9) -- (0,0);
\draw[black, thick] (0.7,0.9) -- (0,0);
\draw[snake=zigzag](0,0) -- (0,-0.9);
\filldraw[violet] (-0.7,0.9) circle (6pt); 
\filldraw[violet] (0.7,0.9) circle (6pt); 
\end{tikzpicture}
}\epsilon_{2}}_{i} \diamond u^{
\scalebox{0.18}{\begin{tikzpicture}
\draw[black, thick] (-0.7,0.9) -- (0,0);
\draw[black, thick] (0.7,0.9) -- (0,0);
\draw[snake=zigzag](0,0) -- (0,-0.9);
\filldraw[green] (-0.7,0.9) circle (6pt); 
\filldraw[green] (0.7,0.9) circle (6pt); 
\end{tikzpicture}
}\epsilon_{2}}_{j} (t_{1}) + b^{
\scalebox{0.18}{\begin{tikzpicture}
\draw[black, thick] (-0.7,0.9) -- (0,0);
\draw[black, thick] (0.7,0.9) -- (0,0);
\draw[snake=zigzag](0,0) -- (0,-0.9);
\filldraw[violet] (-0.7,0.9) circle (6pt); 
\filldraw[violet] (0.7,0.9) circle (6pt); 
\end{tikzpicture}
}\epsilon_{2}}_{i} \diamond u^{
\scalebox{0.18}{\begin{tikzpicture}
\draw[black, thick] (-0.7,0.9) -- (0,0);
\draw[black, thick] (0.7,0.9) -- (0,0);
\draw[snake=zigzag](0,0) -- (0,-0.9);
\filldraw[green] (-0.7,0.9) circle (6pt); 
\filldraw[green] (0.7,0.9) circle (6pt); 
\end{tikzpicture}
}\epsilon_{2}}_{j} (t_{2})) \rvert^{2}], 
\end{align*}
with which similar arguments using Besov embedding, Gaussian hypercontractivity \cite[Theorem 3.50]{J97}, as we did in (\ref{139})-(\ref{141}), imply that there exists $v^{
\scalebox{0.25}{\begin{tikzpicture}
\draw[black, thick] (-1.05,1.4) -- (-0.84,0.84);
\draw[black, thick] (-0.35,1.4) -- (-0.84,0.84);
\draw[black, thick] (0.35,1.4) -- (0.84,0.84);
\draw[black, thick] (1.05,1.4) -- (0.84,0.84);
\draw[snake=zigzag](0,0) -- (-0.84,0.84);
\draw[snake=zigzag](0,0) -- (0.84,0.84);
\filldraw[violet] (-1.05,1.4) circle (4pt); 
\filldraw[violet] (-0.35,1.4) circle (4pt); 
\filldraw[green] (0.35,1.4) circle (4pt); 
\filldraw[green] (1.05,1.4) circle (4pt); 
\end{tikzpicture}
}}_{13,ij} \in C([0,T]; \mathcal{C}^{-\gamma})$ for $i,j \in \{1,2,3 \}$ such that for all $p \in (1,\infty)$, $b^{
\scalebox{0.18}{\begin{tikzpicture}
\draw[black, thick] (-0.7,0.9) -- (0,0);
\draw[black, thick] (0.7,0.9) -- (0,0);
\draw[snake=zigzag](0,0) -- (0,-0.9);
\filldraw[violet] (-0.7,0.9) circle (6pt); 
\filldraw[violet] (0.7,0.9) circle (6pt); 
\end{tikzpicture}
}\epsilon}_{i} \diamond u^{
\scalebox{0.18}{\begin{tikzpicture}
\draw[black, thick] (-0.7,0.9) -- (0,0);
\draw[black, thick] (0.7,0.9) -- (0,0);
\draw[snake=zigzag](0,0) -- (0,-0.9);
\filldraw[green] (-0.7,0.9) circle (6pt); 
\filldraw[green] (0.7,0.9) circle (6pt); 
\end{tikzpicture}
}\epsilon}_{j} \to v^{
\scalebox{0.25}{\begin{tikzpicture}
\draw[black, thick] (-1.05,1.4) -- (-0.84,0.84);
\draw[black, thick] (-0.35,1.4) -- (-0.84,0.84);
\draw[black, thick] (0.35,1.4) -- (0.84,0.84);
\draw[black, thick] (1.05,1.4) -- (0.84,0.84);
\draw[snake=zigzag](0,0) -- (-0.84,0.84);
\draw[snake=zigzag](0,0) -- (0.84,0.84);
\filldraw[violet] (-1.05,1.4) circle (4pt); 
\filldraw[violet] (-0.35,1.4) circle (4pt); 
\filldraw[green] (0.35,1.4) circle (4pt); 
\filldraw[green] (1.05,1.4) circle (4pt); 
\end{tikzpicture}
}}_{13,ij}$ in $L^{p} (\Omega; C([0,T]; \mathcal{C}^{-\delta}))$ as desired in (\ref{113}).

\subsection{Details of Renormalizations for Group 4}
Due to (\ref{26}) we can write down 
\begin{align}\label{255}
& \pi_{0} (\mathcal{P}_{i_{1} i_{2}}  \partial_{x_{j_{0}}}  K_{j_{0}}^{u, \epsilon},  b^{
\scalebox{0.6}{\begin{tikzpicture}
\draw[black, thick] (0,0.5) -- (0,0);
\filldraw[blue] (0,0.5) circle (2pt); 
\end{tikzpicture}
}\epsilon}_{j_{1}})(t) \nonumber \\
=&  \frac{1}{(2\pi)^{\frac{3}{2}}}  \sum_{k}  \sum_{\lvert i-j \rvert \leq 1} \sum_{k_{1}, k_{2}: k_{12} = k} \theta(2^{-i} k_{1}) \theta(2^{-j} k_{2}) \int_{0}^{t} e^{- \lvert k_{1} \rvert^{2}(t-s)} ik_{1}^{j_{0}} \nonumber \\
& \hspace{20mm} \times : \hat{X}_{s, j_{0}}^{u, \epsilon}(k_{1}) \hat{X}_{t, j_{1}}^{b, \epsilon}(k_{2}): ds e_{k} \hat{\mathcal{P}}_{i_{1} i_{2}}(k_{1}) \nonumber \\
&+  \frac{1}{(2\pi)^{\frac{3}{2}}}  \sum_{k} \sum_{\lvert i-j \rvert \leq 1} \sum_{k_{1} \neq 0, k_{2}: k_{12} =k} \theta(2^{-i} k_{1}) \hat{\mathcal{P}}_{i_{1} i_{2}}(k_{1}) \int_{0}^{t} e^{- \lvert k_{1} \rvert^{2}(t-s)} ik_{1}^{j_{0}}  \nonumber \\
&  \hspace{20mm} \times 1_{k_{12} = 0} \sum_{j_{2} =1}^{3} \frac{ e^{- \lvert k_{1} \rvert^{2} (t-s)} f(\epsilon k_{1})^{2}}{2 \lvert k_{1} \rvert^{2}} \hat{\mathcal{P}}_{j_{0} j_{2}}(k_{1}) \hat{\mathcal{P}}_{j_{1} j_{2}}(k_{1}) ds \theta(2^{-j} k_{2}) e_{k} 
\end{align} 
by $:\xi_{1}\xi_{2}: = \xi_{1}\xi_{2} - \mathbb{E}[\xi_{1}\xi_{2}]$ (see \cite{J97}) where the second term can be shown to be actually zero. Thus, 
\begin{align}\label{256}
& \mathbb{E} [ \lvert \Delta_{q}  \pi_{0} ( \mathcal{P}_{i_{1} i_{2}}  \partial_{x_{j_{0}}}  K_{j_{0}}^{u, \epsilon}, b^{
\scalebox{0.6}{\begin{tikzpicture}
\draw[black, thick] (0,0.5) -- (0,0);
\filldraw[blue] (0,0.5) circle (2pt); 
\end{tikzpicture}
}\epsilon}_{j_{1}})(t) \rvert^{2}]\nonumber \\
\approx& \sum_{k, k'} \sum_{\lvert i-j \rvert \leq 1, \lvert i' - j' \rvert \leq 1} \sum_{k_{1}, k_{2}: k_{12} = k, k_{1}', k_{2}': k_{12}' = k'} \theta(2^{-i} k_{1}) \theta(2^{-i'} k_{1}') \theta(2^{-j} k_{2})\nonumber \\
& \times \theta(2^{-j'} k_{2}') \theta(2^{-q} k)^{2} \int_{[0,t]^{2}} e^{- \lvert k_{1} \rvert^{2} (t-s) - \lvert k_{1}' \rvert^{2} (t- \overline{s})} \lvert k_{1} \rvert \lvert k_{1}' \rvert \nonumber \\
& \times \mathbb{E} [ : \hat{X}_{s, j_{0}}^{u, \epsilon}(k_{1}) \hat{X}_{t, j_{1}}^{b, \epsilon}(k_{2}): : \hat{X}_{\overline{s}, j_{0}}^{u, \epsilon}(k_{1}') \hat{X}_{t, j_{1}'}^{b, \epsilon} (k_{2}'): ] e_{k} e_{k}' \hat{\mathcal{P}}_{i_{1} i_{2}}(k_{1})\hat{\mathcal{P}}_{i_{1}' i_{2}'} (k_{1}'). 
\end{align} 
We may compute $\mathbb{E} [ : \hat{X}_{s, j_{0}}^{u, \epsilon}(k_{1}) \hat{X}_{t, j_{1}}^{b, \epsilon}(k_{2}): : \hat{X}_{\overline{s}, j_{0}}^{u, \epsilon}(k_{1}') \hat{X}_{t, j_{1}'}^{b, \epsilon} (k_{2}'):]$ for $k_{1}, k_{2} \neq 0$ using the identity $\mathbb{E} [ : \xi_{11} \xi_{12}: : \xi_{21} \xi_{22}:] = \mathbb{E} [ \xi_{11} \xi_{21}] \mathbb{E} [ \xi_{12} \xi_{22}] + \mathbb{E} [ \xi_{11} \xi_{22}]\mathbb{E}[\xi_{12} \xi_{21}]$ (see \cite{J97}) and (\ref{116}) to deduce from \eqref{256} 
\begin{align}\label{258}
&\mathbb{E} [ \lvert \Delta_{q} \pi_{0} (\mathcal{P}_{i_{1} i_{2}}  \partial_{x_{j_{0}}}  K_{j_{0}}^{u, \epsilon}, b^{
\scalebox{0.6}{\begin{tikzpicture}
\draw[black, thick] (0,0.5) -- (0,0);
\filldraw[blue] (0,0.5) circle (2pt); 
\end{tikzpicture}
}\epsilon}_{j_{1}})(t) \rvert^{2}]  \nonumber \\
\lesssim& \sum_{k} \sum_{\lvert i-j \rvert \leq 1, \lvert i'-j'\rvert \leq 1} \sum_{k_{1}, k_{2} \neq 0: k_{12} = k} \nonumber  \\
& \times  \theta(2^{-i} k_{1}) \theta(2^{-i'} k_{1}) \theta(2^{-j} k_{2}) \theta(2^{-j'} k_{2}) \theta(2^{-q} k)^{2} \int_{[0,t]^{2}} \frac{e^{- \lvert k_{1} \rvert^{2}(2t - s - \overline{s} + \lvert s - \overline{s} \rvert)}}{\lvert k_{2} \rvert^{2}} ds d \overline{s} \nonumber  \\
\lesssim& t^{\eta} 2^{2q\eta} \sum_{k \neq 0} \theta(2^{-q} k)^{2} \frac{1}{\lvert k \rvert^{3}}\lesssim t^{\eta} 2^{2q\eta} 
\end{align} 
where we used mean value theorem, Lemma \ref{Lemma 2.9} and that $2^{q}  \lesssim 2^{i}$. Similarly to how we deduced (\ref{139}) from (\ref{138}) we can also show 
\begin{align}\label{259}
& \mathbb{E} [ \lvert \Delta_{q} ( \pi_{0,\diamond} (\mathcal{P}_{i_{1} i_{2}}  \partial_{x_{j_{0}}}  K_{j_{0}}^{u, \epsilon_{1}}, b^{
\scalebox{0.6}{\begin{tikzpicture}
\draw[black, thick] (0,0.5) -- (0,0);
\filldraw[blue] (0,0.5) circle (2pt); 
\end{tikzpicture}
}\epsilon_{1}}_{j_{1}})(t_{1}) - \pi_{0,\diamond} (\mathcal{P}_{i_{1} i_{2}}  \partial_{x_{j_{0}}}  K_{j_{0}}^{u, \epsilon_{1}}, b^{
\scalebox{0.6}{\begin{tikzpicture}
\draw[black, thick] (0,0.5) -- (0,0);
\filldraw[blue] (0,0.5) circle (2pt); 
\end{tikzpicture}
}\epsilon_{1}}_{j_{1}})(t_{2})\nonumber \\
& \hspace{5mm} - \pi_{0,\diamond} (\mathcal{P}_{i_{1} i_{2}} \partial_{x_{j_{0}}}  K_{j_{0}}^{u, \epsilon_{2}}, b^{
\scalebox{0.6}{\begin{tikzpicture}
\draw[black, thick] (0,0.5) -- (0,0);
\filldraw[blue] (0,0.5) circle (2pt); 
\end{tikzpicture}
}\epsilon_{2}}_{j_{1}}(t_{1}) + \pi_{0,\diamond} (\mathcal{P}_{i_{1} i_{2}} \partial_{x_{j_{0}}}  K_{j_{0}}^{u, \epsilon_{2}}, b^{
\scalebox{0.6}{\begin{tikzpicture}
\draw[black, thick] (0,0.5) -- (0,0);
\filldraw[blue] (0,0.5) circle (2pt); 
\end{tikzpicture}
}\epsilon_{2}}_{j_{1}})(t_{2}) ) \rvert^{2}]\nonumber \\
\lesssim& (\epsilon_{1}^{2\gamma} + \epsilon_{2}^{2\gamma}) \lvert t_{1} - t_{2} \rvert^{\eta} 2^{q (\epsilon + 2 \eta)}  
\end{align}  
so that applications of Besov embedding and Gaussian hypercontractivity theorem \cite[Theorem 3.50]{J97} as we did in (\ref{139})-(\ref{141}) implies that there exists $v_{20}^{i_{1}i_{2}, j_{0}j_{1}}\in C([0,T]; \mathcal{C}^{-\delta})$ for $i_{1}, i_{2}, j_{0}, j_{1} \in \{1,2,3\}$ such that for all $p \in [1,\infty)$, we have $
\pi_{0,\diamond} (\mathcal{P}_{i_{1}i_{2}}  \partial_{x_{j_{0}}}  K_{j_{0}}^{u, \epsilon}, b^{
\scalebox{0.6}{\begin{tikzpicture}
\draw[black, thick] (0,0.5) -- (0,0);
\filldraw[blue] (0,0.5) circle (2pt); 
\end{tikzpicture}
}\epsilon}_{j_{1}}) \to v_{20}^{i_{1}i_{2}, j_{0}j_{1}}$  as $\epsilon \to 0$ in $L^{p}(\Omega; C([0,T]; \mathcal{C}^{-\delta}))$.

\textbf{Acknowledgments}
The author expresses deep gratitude to Prof. Carl Mueller and Prof. Marco Romito for valuable discussions. He also thanks Prof. Jared Whitehead for suggesting references \cite{ACHS81, GP75, HS92, SH77} on the Boussinesq system. Moreover, the author expresses deep gratitude to the anonymous referee and the editor for the valuable suggestions and comments that have improved this manuscript significantly. This work was supported by a grant from the Simons Foundation (962572, KY).


\end{document}